\documentclass[12pt,a4paper,reqno]{amsart}
\usepackage[utf8]{inputenc}
\usepackage[T1]{fontenc}
\usepackage{indentfirst}
\usepackage[english]{babel}
\usepackage{csquotes}
\usepackage{enumitem}
\usepackage{amsmath,amsfonts,amssymb,amsopn,amscd,amsthm}
\usepackage{mathrsfs,shuffle,bbm}
\usepackage{stmaryrd}
\usepackage{mathpazo}
\usepackage{graphicx}
\usepackage[dvipsnames]{xcolor}
\usepackage[colorlinks=true,citecolor=DarkOrchid,linkcolor=NavyBlue]{hyperref}
\usepackage{pgfplots}
\usepackage{tikz}
\usetikzlibrary{patterns,arrows.meta}
\pgfplotsset{compat=1.14}
\usepgfplotslibrary{colormaps}
\usepackage[a4paper,twoside,margin=0.8in]{geometry}

    \newtheorem{definition}{Definition}[section]
    \newtheorem{lemma}[definition]{Lemma}
    \newtheorem{theorem}[definition]{Theorem}
    \newtheorem{proposition}[definition]{Proposition}
    \newtheorem{corollary}[definition]{Corollary}
    \newtheorem{conjecture}[definition]{Conjecture}
    \theoremstyle{remark}
    \newtheorem{example}[definition]{Example}
    \newtheorem{remark}[definition]{Remark}

    \makeatletter
    \def\thm@space@setup{\thm@preskip=0.5cm   \thm@postskip=0.5cm}
    \makeatother

\newcommand{\N}{\mathbb{N}}
\newcommand{\Z}{\mathbb{Z}}

\newcommand{\R}{\mathbb{R}}
\newcommand{\C}{\mathbb{C}}
\newcommand{\Hq}{\mathbb{H}}
\newcommand{\Sph}{\mathbb{S}}
\newcommand{\Proj}{\mathbb{P}}
\newcommand{\Octo}{\mathbb{O}}
\newcommand{\Tor}{\mathbb{T}}

\newcommand{\E}{\mathrm{e}}
\newcommand{\I}{\mathrm{i}}
\newcommand{\J}{\mathrm{j}}
\newcommand{\esper}{\mathbb{E}}

\newcommand{\proba}{\mathbb{P}}
\newcommand{\sph}{\mathbb{RS}}
\newcommand{\sym}{\mathfrak{S}}
\newcommand{\SO}{\mathrm{SO}}
\newcommand{\SU}{\mathrm{SU}}
\newcommand{\SP}{\mathrm{USp}}
\newcommand{\GL}{\mathrm{GL}}
\newcommand{\Endom}{\mathrm{End}}
\newcommand{\Unit}{\mathrm{U}}
\newcommand{\Spin}{\mathrm{Spin}}
\newcommand{\Span}{\mathrm{Span}}
\newcommand{\leb}{\mathscr{L}}
\newcommand{\meas}{\mathscr{M}}
\newcommand{\gbullet}{\mathfrak{G}_\bullet}
\newcommand{\card}{\mathrm{card}}
\newcommand{\vol}{\mathrm{vol}}
\newcommand{\diag}{\mathrm{diag}}
\newcommand{\eps}{\varepsilon}
\newcommand{\sinc}{\mathrm{sinc}}
\newcommand{\tr}{\mathrm{tr}}
\newcommand{\ch}{\mathrm{ch}}
\newcommand{\zon}{\mathrm{zon}}
\newcommand{\rank}{\mathrm{rank}}
\newcommand{\glie}{\mathfrak{g}}
\newcommand{\tlie}{\mathfrak{t}}
\newcommand{\hatG}{\widehat{G}}
\newcommand{\cryst}{\mathfrak{C}}
\newcommand{\tang}{\mathrm{T}}
\newcommand{\tildeJ}{\widetilde{J}}
\newcommand{\lle}{\left[\!\left[} 
\newcommand{\rre}{\right]\!\right]} 
\newcommand{\GEOM}{\Gamma_{\mathrm{geom}}}
\newcommand{\GRID}{\Gamma_{\mathrm{grid}}}
\newcommand{\DD}[1]{\,d\hspace{-0.3mm}{#1}}
\newcommand{\scal}[2]{\left\langle #1\vphantom{#2}\,\right |\left.#2 \vphantom{#1}\right\rangle}
\newcommand{\comment}[1]{}

\setlist[enumerate]{itemsep=10pt,topsep=10pt}
\setlist[itemize]{itemsep=5pt,topsep=5pt}
\renewcommand{\arraystretch}{1.3}

\title[Spectrum of a random geometric graph on a compact Lie group]{Asymptotic representation theory and the spectrum of a random geometric graph on a compact Lie group}
\author{Pierre-Loïc Méliot}
\date{\today}

\begin{document}
\begin{abstract}
Let $G$ be a compact Lie group, $N\geq 1$ and $L>0$. The random geometric graph on $G$ is the random graph $\GEOM(N,L)$ whose vertices are $N$ random points $g_1,\ldots,g_N$ chosen under the Haar measure of $G$, and whose edges are the pairs $\{g_i,g_j\}$ with $d(g_i,g_j)\leq L$, $d$ being the distance associated to the standard Riemannian structure on $G$. In this paper, we describe the asymptotic behavior of the spectrum of the adjacency matrix of $\GEOM(N,L)$, when $N$ goes to infinity.
\begin{enumerate}
    \item If $L$ is fixed and $N \to + \infty$ (Gaussian regime), then the largest eigenvalues of $\GEOM(N,L)$ converge after an appropriate renormalisation towards certain explicit linear combinations of values of Bessel functions. 
    \item If $L = O(N^{-\frac{1}{\dim G}})$ and $N \to +\infty$ (Poissonian regime), then the geometric graph $\GEOM(N,L)$ converges in the local Benjamini--Schramm sense, which implies the weak convergence in probability of the spectral measure of $\GEOM(N,L)$. 
\end{enumerate}
In both situations, the representation theory of the group $G$ provides us with informations on the limit of the spectrum, and conversely, the computation of this limiting spectrum involves many classical tools from representation theory: Weyl's character formula and the weight lattice in the Gaussian regime, and a degeneration of these objects in the Poissonian regime. The representation theoretic approach allows one to understand precisely how the degeneration from the Gaussian to the Poissonian regime occurs, and the article is written so as to highlight this degeneration phenomenon. In the Poissonian regime, this approach leads us to an algebraic conjecture on certain functionals of the irreducible representations of $G$.
\end{abstract}
\maketitle

\hrule
\tableofcontents

\hrule
\bigskip
\bigskip

\section{Random geometric graphs on compact Lie groups}
In this paper, $\Z$, $\R$, $\C=\R \oplus \I \R$ and $\Hq = \R \oplus \I\R \oplus \J\R \oplus \mathrm{k}\R = \C \oplus \J \C$ denote respectively the set of integers, the field of real numbers, the field of complex numbers, and the division algebra of quaternionic numbers. 
\medskip

\subsection{Spectrum of large random graphs}
We call \emph{graph} a pair $\Gamma=(V,E)$ with $V$ finite set, and $E$ finite subset of the set of pairs $\{v,w\}$ with $v\neq w$ in $V$. In particular, the \emph{random} graphs that we shall consider in this paper will always be unoriented and simple, that is without loop or multiple edge. Our computations will also involve (deterministic) oriented or labeled graphs, possibly with loops or with multiple edges, but this will be recalled each time by using in particular the terminologies of circuits and reduced circuits (Section \ref{sec:ART}). The \emph{size} of a graph is the cardinality $N=|V|$ of its vertex set. 

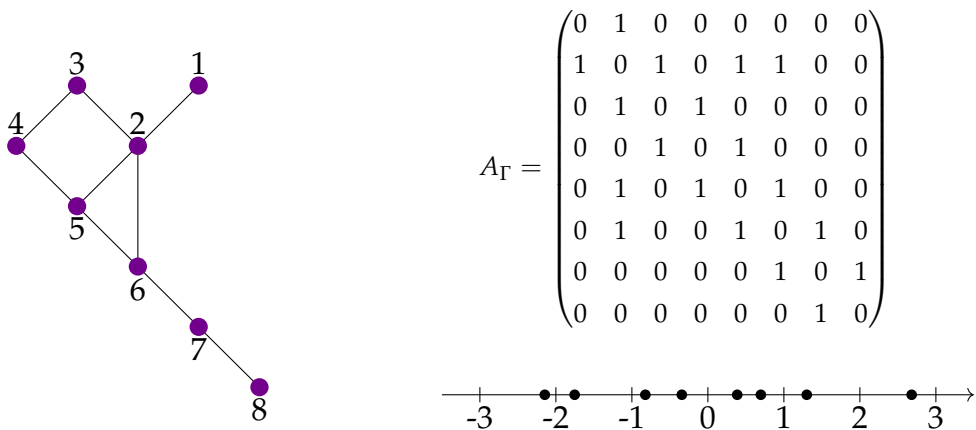
\begin{figure}[ht]
\begin{center}
\begin{tikzpicture}[scale=1]
\draw (-2.9,-2.4) -- (-3.7,-1.6) -- (-4.5,-0.8) -- (-4.5,0.8) -- (-5.3,0) -- (-4.5,-0.8); 
\draw (-5.3,0) -- (-6.1,0.8) -- (-5.3,1.6) -- (-4.5,0.8) -- (-3.7,1.6);
\foreach \x in {(-5.3,0),(-6.1,0.8),(-5.3,1.6),(-4.5,-0.8),(-4.5,0.8),(-3.7,1.6),(-2.9,-2.4),(-3.7,-1.6)}
\fill [color=red!45!blue] \x circle [radius=1.2mm];
\draw (-3.7,1.9) node {$1$};
\draw (-4.5,1.1) node {$2$};
\draw (-5.3,1.9) node {$3$};
\draw (-6.1,1.1) node {$4$};
\draw (-5.3,-0.3) node {$5$};
\draw (-4.5,-1.1) node {$6$};
\draw (-3.7,-1.9) node {$7$};
\draw (-2.9,-2.7) node {$8$};
\draw (2.7,0.5) node {\footnotesize $A_\Gamma=\begin{pmatrix}
0 & 1 & 0 & 0 & 0 & 0 & 0 & 0 \\
1 & 0 & 1 & 0 & 1 & 1 & 0 & 0 \\
0 & 1 & 0 & 1 & 0 & 0 & 0 & 0 \\
0 & 0 & 1 & 0 & 1 & 0 & 0 & 0 \\
0 & 1 & 0 & 1 & 0 & 1 & 0 & 0 \\
0 & 1 & 0 & 0 & 1 & 0 & 1 & 0 \\
0 & 0 & 0 & 0 & 0 & 1 & 0 & 1 \\
0 & 0 & 0 & 0 & 0 & 0 & 1 & 0 
\end{pmatrix}$};
\begin{scope}[shift={(3,-2.5)}]
\draw [->] (-3.5,0) -- (3.5,0);
\foreach \x in {2.681,1.301,0.697,0.386,-0.343,-0.823,-1.753,-2.146}
\fill (\x,0) circle (2pt);
\foreach \x in {-3,-2,-1,0,1,2,3}
{\draw (\x,0.1) -- (\x,-0.1);
\draw (\x,-0.3) node {\x};}
\end{scope}
\end{tikzpicture}
\end{center}
\caption{A (simple, unoriented) graph, its adjacency matrix and its spectrum.\label{fig:spectrum}}
\end{figure}

The \emph{adjacency matrix} of a graph $\Gamma$ of size $N$ is the matrix $A_\Gamma$ of size $N \times N$, with rows and columns labeled by the vertices of $\Gamma$, and with coefficients
$$A_\Gamma(v,w) = \begin{cases}
    1 &\text{if }\{v,w\} \in E,\\
    0 & \text{otherwise}.
\end{cases}$$
In particular, the diagonal coefficients of the adjacency matrix of a simple graph are all equal to zero.
The adjacency matrix of a graph $\Gamma$ being a real symmetric matrix, its \emph{spectrum} $\mathrm{Spec}(\Gamma)$ consists of $N$ real eigenvalues $c_1 \geq c_2 \geq \cdots \geq c_N$ (see Figure \ref{fig:spectrum}). The knowledge of the spectrum yields many informations on the geometry of the graph: mean and maximal number of neighbors of a vertex; chromatic number; number of edges, triangles, spanning trees; expansion properties, Cheeger constant; \emph{etc.} We refer to \cite{Chung97,GR01} for an introduction to this algebraic graph theory.
The purpose of this paper is to study the spectrum of a class of random graphs drawn on certain Riemannian manifolds $X$, by using the representation theory of the isometry group of $X$. The simplest example of random graphs that one can think of is when each possible edge $\{i,j\}$ between points of $V=\lle 1,N\rre = \{1,2,\ldots,N\}$ is kept at random, according to a Bernoulli law of parameter $p$, independently for each pair. One obtains the Erdös--Rényi random graphs (\cite{ER59}), and if $p$ is not too small (\emph{e.g.} larger than $N^{-1/3}$), the eigenvalue distribution $\frac{1}{N}\sum_{i=1}^N \delta_{c_i}$ of $\Gamma_{\mathrm{ER}}(N,p)$ admits after appropriate renormalisation a deterministic continuous limit, which is the Wigner semicircle law $\frac{1}{2\pi}\sqrt{4-x^2}\,1_{-2<x<2}\,\DD{x}$; see Figure \ref{fig:erdosrenyi}. We refer to the recent papers \cite{EKYY13,EKYY12} for a detailed study of this model, including results on the spacing of eigenvalues and on the edge of the spectrum.

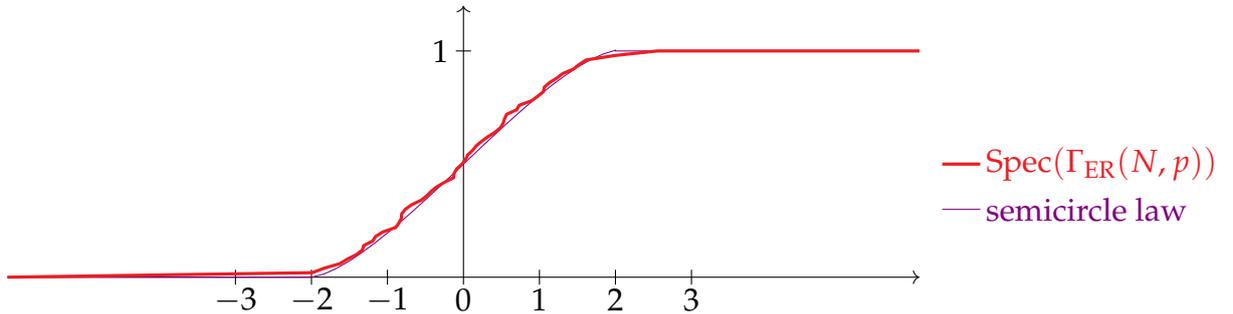
\begin{figure}[ht]
\begin{center}      
\begin{tikzpicture}[xscale=1,yscale=3]
\draw [->] (0,-0.033) -- (0,1.2);
\draw (-3,-0.033) -- (-3,0.033);
\draw (-2,-0.033) -- (-2,0.033);
\draw (-1,-0.033) -- (-1,0.033);
\draw (1,-0.033) -- (1,0.033);
\draw (2,-0.033) -- (2,0.033);
\draw (3,-0.033) -- (3,0.033);
\draw [->] (-6,0) -- (6,0);
\draw (-0.1,1) -- (0.1,1);
\draw (-0.3,1) node {$1$};
\draw (-3,-0.1) node {$-3$};
\draw (-2,-0.1) node {$-2$};
\draw (-1,-0.1) node {$-1$};
\draw (0,-0.1) node {$0$};
\draw (1,-0.1) node {$1$};
\draw (2,-0.1) node {$2$};
\draw (3,-0.1) node {$3$};
\draw [violet,domain=-2:2] (-6,0) -- (-2,0) -- plot ({\x}, {0.5 + \x*sqrt(4-\x*\x)/(4*3.14) + asin(\x/2)/180}) -- (2,1) -- (6,1);
\begin{scope}[xscale=0.5,yscale=1]
\draw [Red, very thick] (-12,0) -- (-3.9777,0.02) -- (-3.6641,0.04) -- (-3.2494,0.06) -- (-3.0544,0.08) -- (-2.8162,0.1) -- (-2.6615,0.12) -- (-2.6247,0.14) -- (-2.3871,0.16) -- (-2.3131,0.18) -- (-2.1209,0.2) -- (-1.7797,0.22) -- (-1.6983,0.24) -- (-1.6354,0.26) -- (-1.6262,0.28) -- (-1.5310,0.3) -- (-1.3654,0.32) -- (-1.1155,0.34) -- (-0.96336,0.36) -- (-0.84541,0.38) -- (-0.69802,0.4) -- (-0.49199,0.42) -- (-0.26059,0.44) -- (-0.24276,0.46) -- (-0.17149,0.48) -- (-0.037806,0.5) -- (0.053687,0.52) -- (0.10896,0.54) -- (0.23763,0.56) -- (0.33236,0.58) -- (0.47438,0.6) -- (0.63487,0.62) -- (0.82310,0.64) -- (0.96932,0.66) -- (1.0434,0.68) -- (1.0814,0.7) -- (1.1374,0.72) -- (1.4004,0.74) -- (1.4677,0.76) -- (1.7928,0.78) -- (1.9564,0.8) -- (2.1038,0.82) -- (2.1299,0.84) -- (2.2638,0.86) -- (2.4576,0.88) -- (2.6106,0.9) -- (2.8944,0.92) -- (3.0423,0.94) -- (3.2256,0.96) -- (3.9846,0.98) -- (5.1057,1) -- (12,1);
\end{scope}
\draw [Red, very thick] (6.8,0.5) -- (6.3,0.5);
\draw [Red] (8.4,0.5) node {$\mathrm{Spec}(\Gamma_{\mathrm{ER}}(N,p))$};
\draw [violet] (6.8,0.3) -- (6.3,0.3);
\draw [violet] (8.2,0.3) node {semicircle law};
\end{tikzpicture}
\caption{The cumulative distribution function of a renormalisation of the spectrum of a random Erdös--Rényi graph of parameter $p=0.1$, with $N=50$ vertices; and the limiting semicircle distribution.\label{fig:erdosrenyi}}
\end{center}
\end{figure}

\subsection{Random geometric graphs} A more complicated model consists in geometric graphs, that can be defined on any measured metric space. Let $(X,d)$ be a metric space, and $m$ be a Borel probability measure on $X$. Given a positive real number $L>0$, the \emph{random geometric graph} with $N$ points and level $L$ is the random graph $\Gamma=(V,E)$ with \vspace{2mm}
\begin{itemize}
    \item the $N$ vertices $v_1,\ldots,v_N$ of $V$ chosen randomly in $X$ according to the probability measure $m^{\otimes N}$ on $X^N$.\vspace{2mm}
    \item an edge between $v_i$ and $v_j$ if and only if $d(v_i,v_j)\leq L$.\vspace{2mm}
\end{itemize}
For the Euclidian case, when $X=\R^p$, $d(x,y)=\|x-y\|$ is the Euclidian distance and $m$ is a measure on $\R^p$, we refer to the monograph \cite{Pen03}, where most of the classical questions on random graphs (subgraph counts, threshold for connectivity, existence of a giant connected component, \emph{etc.}) are answered. In this setting, the spectrum of the adjacency matrix has been studied in \cite{BEJJ06}, see also \cite{Bor08,DGK16} for the case where $\R^p$ is replaced by the torus $\Tor^p = (\R/\Z)^p$. In \cite{DGK16}, the distance between points of the torus is the Euclidean distance, whereas in \cite{Bor08} this is the $\ell^\infty$-distance (which does not correspond to a Riemannian structure). The goal of this article is to extend this work to a general setting of compact Riemannian manifolds that satisfy a certain symmetry property. An important point is that in the regime where $L$ is fixed and $N$ goes to infinity, the asymptotics of the spectrum are \emph{discrete} instead of \emph{continuous}. We only get a continuous limiting distribution in the thermodynamic limit where $L = L_N \to 0$ is chosen so that the mean number of neighbors of a given vertex is a $O(1)$. 

\begin{example}
In Figure \ref{fig:stereo}, we have drawn in stereographic projection a random geometric graph on the real sphere $\sph^2$, with $m$ equal to Lebesgue's spherical measure, $N=100$ points, and $L=\frac{\pi}{8}$ (one eighth of the diameter of the space).
\end{example}
\medskip

\begin{figure}[ht]
\begin{center}
\vspace{2.2cm}        
\begin{tikzpicture}[scale=0.55]
\useasboundingbox (-3.5,-3.5) rectangle (3.5,3.5);
\draw[color=blue,dashed] (0,0) circle [radius=2];
\draw (-0.42017,0.48707) arc (59.793:61.882:10.532);
\draw (-0.12249,0.21036) arc (46.693:47.486:29.390);
\draw (-0.46113,0.80236) arc (-174.00:-171.19:6.4791);
\draw (-0.42017,0.48707) arc (167.64:172.14:4.2876);
\draw (-0.42017,0.48707) -- (-0.18172,0.21405);
\draw (0.079583,-4.8841) arc (-55.168:-24.497:3.4358);
\draw (-0.47528,-3.3536) arc (-167.32:-152.82:6.4488);
\draw (-0.0029384,-4.8829) arc (-91.650:-89.991:2.8522);
\draw (-0.36045,-5.2298) arc (-56.060:-47.623:3.8044);
\draw (-1.8125,-1.1759) arc (174.30:185.71:2.5694);
\draw (-2.4592,-1.0427) arc (-107.54:-95.725:3.2084);
\draw (-0.46113,0.80236) arc (110.06:116.98:2.6629);
\draw (0.093436,-1.0592) arc (-125.36:-124.03:3.1682);
\draw (-0.061524,-0.81852) arc (-148.89:-145.55:4.9137);
\draw (-0.19663,-1.0193) arc (-101.25:-94.390:2.4457);
\draw (-0.039316,-1.5607) arc (-16.728:-12.926:7.8209);
\draw (0.093436,-1.0592) arc (-72.500:-67.640:2.4708);
\draw (-0.18928,2.9204) arc (142.03:151.40:3.2285);
\draw (-0.18928,2.9204) arc (10.470:15.601:18.232);
\draw (0.40202,4.1072) arc (146.53:160.51:5.4486);
\draw (0.17033,2.0960) arc (18.502:28.633:5.0939);
\draw (-0.75589,-1.2081) arc (-50.999:-47.551:6.5103);
\draw (-0.74502,-0.79631) arc (-184.60:-178.42:3.8227);
\draw (-0.55744,4.5106) arc (99.479:137.61:2.7149);
\draw (-2.1156,3.6629) arc (130.72:167.29:2.6337);
\draw (-2.1156,3.6629) arc (121.39:161.33:2.5886);
\draw (-2.1156,3.6629) arc (-170.38:-161.55:7.4266);
\draw (-1.0361,0.23286) arc (-118.52:-116.06:8.8386);
\draw (-1.0361,0.23286) arc (-167.34:-157.52:2.6710);
\draw (-1.0361,0.23286) arc (113.29:119.78:4.1119);
\draw (-2.4592,-1.0427) arc (-147.34:-122.41:2.1144);
\draw (-2.9438,-1.5807) arc (-105.28:-85.426:3.2957);
\draw (-0.74502,-0.79631) arc (-122.51:-114.79:2.4205);
\draw (-0.45901,-0.95256) arc (-75.138:-67.591:3.1870);
\draw (-0.45901,-0.95256) arc (-107.45:-101.10:2.4456);
\draw (-0.43825,-0.74365) arc (173.44:175.21:6.7712);
\draw (1.0841,1.3211) arc (135.32:136.39:21.907);
\draw (1.6378,1.6953) arc (121.73:126.38:8.2293);
\draw (1.0841,1.3211) arc (102.49:109.78:3.2787);
\draw (-4.9947,-4.3753) arc (-54.189:-50.715:38.434);
\draw (-2.9438,-1.5807) arc (-200.18:-176.71:3.4186);
\draw (-0.12249,0.21036) arc (97.564:99.853:9.0345);
\draw (0.17747,0.21499) arc (89.977:91.791:9.4805);
\draw (-0.12249,0.21036) arc (86.262:86.601:10.051);
\draw (-0.47851,2.4798) arc (127.04:141.99:2.2884);
\draw (-0.90315,2.0623) arc (-175.01:-168.40:6.2951);
\draw (-0.90315,2.0623) arc (169.41:180.78:3.5792);
\draw (-0.90315,2.0623) arc (54.581:70.065:3.9183);
\draw (3.6273,2.0998) arc (109.08:115.18:13.497);
\draw (2.2958,1.5583) arc (71.622:79.865:2.6557);
\draw (2.2958,1.5583) arc (70.899:85.579:2.6304);
\draw (0.028208,1.9359) arc (36.500:49.437:3.2995);
\draw (0.17033,2.0960) arc (50.506:68.280:2.4401);
\draw (0.98678,-0.63643) arc (-59.757:-48.329:2.5691);
\draw (0.98559,-1.0642) arc (-4.4402:4.1200:2.8656);
\draw (1.9258,1.6523) arc (38.832:66.554:2.0579);
\draw (1.6378,1.6953) arc (31.579:52.061:2.0929);
\draw (1.1416,2.2499) arc (33.351:38.798:2.3660);
\draw (1.9375,2.9193) arc (125.47:134.66:6.4933);
\draw (0.52680,5.7291) arc (127.66:149.01:4.4027);
\draw (0.52680,5.7291) -- (0.40202,4.1072);
\draw (2.3911,4.9227) arc (48.248:84.969:3.2243);
\draw (0.40202,4.1072) arc (56.774:77.622:2.8762);
\draw (3.5488,-1.9515) arc (-66.291:-39.828:3.1742);
\draw (3.9841,-4.2666) arc (-4.9064:26.203:4.3924);
\draw (4.6138,-3.5561) arc (26.031:41.115:7.3369);
\draw (3.9315,-3.5301) arc (2.5701:24.689:4.2341);
\draw (2.5469,-2.9317) arc (-61.782:-29.475:2.5190);
\draw (-0.061524,-0.81852) arc (-36.016:-31.517:5.4154);
\draw (-0.79844,1.3442) arc (84.185:87.889:2.5647);
\draw (-2.9665,2.2465) arc (80.634:83.480:5.1468);
\draw (-2.9665,2.2465) arc (-138.65:-132.53:12.019);
\draw (-1.8380,2.5526) arc (92.402:117.95:2.6441);
\draw (-3.2197,2.2818) arc (115.76:160.62:2.8331);
\draw (-3.2439,0.53165) arc (83.509:85.290:45.823);
\draw (-4.6609,0.67060) arc (-171.73:-161.25:2.8979);
\draw (-1.8380,2.5526) arc (86.748:115.43:2.8422);
\draw (-8.4505,-1.6103) arc (-157.18:-100.15:4.6356);
\draw (-8.4505,-1.6103) arc (105.52:119.71:54.932);
\draw (-7.1739,4.5398) arc (-232.78:-150.67:4.7819);
\draw (0.84018,0.34542) arc (-35.698:-30.792:5.0183);
\draw (0.84018,0.34542) arc (102.29:103.03:15.213);
\draw (1.8230,-0.86146) arc (-75.286:-73.810:3.1110);
\draw (1.4075,-1.1508) arc (-65.254:-50.280:2.2350);
\draw (1.2437,-3.4885) arc (-83.625:-50.103:2.4572);
\draw (1.2437,-3.4885) arc (13.150:14.277:21.221);
\draw (-0.0029384,-4.8829) arc (-57.930:-25.665:3.3657);
\draw (0.63312,-2.5303) arc (-151.44:-143.55:8.2631);
\draw (0.63433,-2.6091) arc (-149.53:-141.03:7.2194);
\draw (4.6138,-3.5561) arc (-23.943:19.489:3.3510);
\draw (7.0663,-2.9549) arc (42.614:60.309:9.7927);
\draw (3.9315,-3.5301) arc (-42.820:7.5842:3.0208);
\draw (7.6505,-2.1114) arc (68.343:72.935:38.901);
\draw (-2.0688,1.3303) arc (122.22:142.91:2.2549);
\draw (-2.4592,-1.0427) arc (134.36:141.62:5.7174);
\draw (-2.1335,-0.24066) arc (148.15:167.66:2.5546);
\draw (2.3911,4.9227) arc (94.994:129.59:3.6149);
\draw (1.4075,-1.1508) arc (-61.830:-48.467:2.1756);
\draw (1.8230,-0.86146) arc (33.815:43.740:3.8956);
\draw (0.26647,1.7584) arc (49.728:56.927:2.3658);
\draw (0.26647,1.7584) arc (13.977:17.815:5.2415);
\draw (2.2180,0.027106) arc (79.153:81.499:10.045);
\draw (1.0757,0.70464) arc (34.850:45.945:2.1992);
\draw (0.80011,1.0284) arc (30.822:35.993:2.3466);
\draw (0.98559,-1.0642) arc (-104.75:-98.451:3.9217);
\draw (3.9841,-4.2666) arc (-49.912:-33.197:3.2657);
\draw (3.9841,-4.2666) arc (-89.269:-44.627:4.4100);
\draw (3.9841,-4.2666) arc (-0.38469:8.5405:4.7446);
\draw (2.5469,-2.9317) arc (-134.64:-131.14:32.156);
\draw (3.9841,-4.2666) arc (-131.36:-126.14:33.445);
\draw (3.9841,-4.2666) arc (-88.677:-30.429:4.3691);
\draw (4.6138,-3.5561) arc (-89.939:-62.514:5.3261);
\draw (3.9315,-3.5301) arc (-95.988:-88.381:5.1469);
\draw (2.5469,-2.9317) arc (-116.50:-97.123:6.4150);
\draw (6.3640,-6.1767) arc (27.684:39.791:14.942);
\draw (4.6138,-3.5561) arc (-85.019:-44.096:4.8098);
\draw (1.9258,1.6523) arc (78.299:84.724:2.5986);
\draw (1.9258,1.6523) arc (-11.442:10.384:3.3463);
\draw (3.9315,-3.5301) arc (-97.125:-62.081:5.2929);
\draw (6.3640,-6.1767) arc (-32.580:7.9839:4.7564);
\draw (7.0663,-2.9549) arc (-41.523:-27.889:4.3221);
\draw (0.17033,2.0960) arc (136.49:140.34:3.1915);
\draw (2.5469,-2.9317) arc (-119.32:-107.43:7.2872);
\draw (6.3640,-6.1767) arc (40.536:44.635:50.268);
\draw (3.9315,-3.5301) arc (-93.366:-44.873:4.8462);
\draw (1.4009,-0.33605) arc (-50.130:-37.086:2.6279);
\draw (1.8126,0.096203) arc (41.866:54.574:2.5794);
\draw (6.1486,6.4384) arc (95.526:128.41:7.1572);
\draw (13.682,11.506) arc (117.83:130.03:42.705);
\draw (-7.1739,4.5398) arc (80.050:178.91:11.769);
\draw (1.0091,2.4317) arc (101.41:122.21:2.5027);
\draw (1.1448,-3.0831) arc (-99.573:-68.094:2.5995);
\draw (-0.19051,-2.2061) arc (-14.802:-11.564:11.732);
\draw (-0.19051,-2.2061) arc (-122.71:-100.25:2.2722);
\draw (-0.19051,-2.2061) arc (-127.17:-104.90:2.3768);
\draw (-0.19051,-2.2061) arc (-109.01:-84.954:2.0726);
\draw (1.6378,1.6953) arc (27.405:53.571:2.1386);
\draw (-0.061524,-0.81852) arc (-144.94:-140.40:4.4889);
\draw (-0.19663,-1.0193) arc (-107.32:-98.933:2.4610);
\draw (-0.039316,-1.5607) arc (-25.656:-19.971:5.0269);
\draw (0.15399,-1.1011) arc (-52.150:-48.407:2.7160);
\draw (-0.74502,-0.79631) arc (-82.647:-77.873:3.7364);
\draw (-0.19663,-1.0193) arc (-35.095:-32.773:5.9715);
\draw (-0.061524,-0.81852) arc (-119.24:-112.11:3.1399);
\draw (-0.43825,-0.74365) arc (-105.10:-97.379:2.8511);
\draw (0.63312,-2.5303) arc (-50.188:-29.652:2.2995);
\draw (0.63433,-2.6091) arc (-47.187:-25.946:2.3901);
\draw (0.66680,-2.3111) arc (-58.996:-41.495:2.1050);
\draw (0.63312,-2.5303) arc (-141.53:-132.91:5.0089);
\draw (0.63433,-2.6091) arc (-137.50:-128.26:4.3277);
\draw (0.66680,-2.3111) arc (-151.08:-145.40:9.1616);
\draw (-0.19663,-1.0193) arc (-91.947:-80.634:2.4759);
\draw (-0.43825,-0.74365) arc (-142.32:-135.22:2.9604);
\draw (-2.6653,0.78249) arc (108.53:118.34:3.6866);
\draw (-0.47528,-3.3536) arc (-169.31:-156.37:7.1019);
\draw (-0.47528,-3.3536) arc (-179.86:-173.13:16.008);
\draw (1.3869,0.47658) arc (49.208:58.325:2.4274);
\draw (-0.47927,0.15571) arc (112.84:114.97:6.4539);
\draw (-0.47927,0.15571) arc (136.04:139.32:4.6947);
\draw (-0.18172,0.21405) arc (100.06:102.13:8.4098);
\draw (-0.69830,0.058617) arc (-160.36:-158.60:3.5197);
\draw (-0.69830,0.058617) arc (140.24:145.07:3.9099);
\draw (1.9375,2.9193) arc (-19.101:-6.4123:9.2946);
\draw (-0.66026,-0.043023) arc (122.77:125.15:6.9081);
\draw (1.9375,2.9193) arc (110.01:125.41:3.9123);
\draw (-0.36045,-5.2298) arc (-49.240:-42.479:4.2245);
\draw (0.17747,0.21499) arc (89.049:91.249:9.3567);
\draw (6.3640,-6.1767) arc (-44.786:9.6649:4.6601);
\draw (0.68352,1.2052) arc (105.91:114.78:3.0461);
\draw (0.63433,-2.6091) arc (0.63538:1.1239:9.2592);
\draw (0.63312,-2.5303) arc (-9.9857:-7.4872:5.0848);
\draw (-3.2439,0.53165) arc (100.41:111.99:6.6746);
\draw (0.63433,-2.6091) arc (-7.7174:-4.7172:5.7254);
\fill (-0.42017,0.48707) circle [radius=0.7mm];
\fill (0.079583,-4.8841) circle [radius=0.7mm];
\fill (-1.8125,-1.1759) circle [radius=0.7mm];
\fill (-0.75554,0.67422) circle [radius=0.7mm];
\fill (0.093436,-1.0592) circle [radius=0.7mm];
\fill (-0.18928,2.9204) circle [radius=0.7mm];
\fill (-0.75589,-1.2081) circle [radius=0.7mm];
\fill (-2.1156,3.6629) circle [radius=0.7mm];
\fill (-1.0361,0.23286) circle [radius=0.7mm];
\fill (-1.8124,-1.6867) circle [radius=0.7mm];
\fill (-0.45901,-0.95256) circle [radius=0.7mm];
\fill (1.0841,1.3211) circle [radius=0.7mm];
\fill (-3.1479,-2.9558) circle [radius=0.7mm];
\fill (-0.12249,0.21036) circle [radius=0.7mm];
\fill (-0.90315,2.0623) circle [radius=0.7mm];
\fill (2.2958,1.5583) circle [radius=0.7mm];
\fill (-0.47851,2.4798) circle [radius=0.7mm];
\fill (0.98678,-0.63643) circle [radius=0.7mm];
\fill (1.1416,2.2499) circle [radius=0.7mm];
\fill (0.52680,5.7291) circle [radius=0.7mm];
\fill (-0.55744,4.5106) circle [radius=0.7mm];
\fill (3.5488,-1.9515) circle [radius=0.7mm];
\fill (0.17470,-0.46521) circle [radius=0.7mm];
\fill (-0.79844,1.3442) circle [radius=0.7mm];
\fill (-2.9665,2.2465) circle [radius=0.7mm];
\fill (-4.6609,0.67060) circle [radius=0.7mm];
\fill (-3.2197,2.2818) circle [radius=0.7mm];
\fill (-8.4505,-1.6103) circle [radius=0.7mm];
\fill (0.84018,0.34542) circle [radius=0.7mm];
\fill (1.9002,-0.84012) circle [radius=0.7mm];
\fill (1.2437,-3.4885) circle [radius=0.7mm];
\fill (4.7102,-1.0782) circle [radius=0.7mm];
\fill (-2.0688,1.3303) circle [radius=0.7mm];
\fill (-2.4592,-1.0427) circle [radius=0.7mm];
\fill (0.40202,4.1072) circle [radius=0.7mm];
\fill (3.6273,2.0998) circle [radius=0.7mm];
\fill (-4.9947,-4.3753) circle [radius=0.7mm];
\fill (1.8230,-0.86146) circle [radius=0.7mm];
\fill (0.26647,1.7584) circle [radius=0.7mm];
\fill (2.2180,0.027106) circle [radius=0.7mm];
\fill (0.80011,1.0284) circle [radius=0.7mm];
\fill (1.4075,-1.1508) circle [radius=0.7mm];
\fill (3.9841,-4.2666) circle [radius=0.7mm];
\fill (4.6138,-3.5561) circle [radius=0.7mm];
\fill (1.9258,1.6523) circle [radius=0.7mm];
\fill (7.0663,-2.9549) circle [radius=0.7mm];
\fill (0.028208,1.9359) circle [radius=0.7mm];
\fill (3.9315,-3.5301) circle [radius=0.7mm];
\fill (1.8126,0.096203) circle [radius=0.7mm];
\fill (6.1486,6.4384) circle [radius=0.7mm];
\fill (1.4009,-0.33605) circle [radius=0.7mm];
\fill (-0.22968,-0.20894) circle [radius=0.7mm];
\fill (-20.974,-6.8276) circle [radius=0.7mm];
\fill (0.17033,2.0960) circle [radius=0.7mm];
\fill (2.5469,-2.9317) circle [radius=0.7mm];
\fill (-0.46113,0.80236) circle [radius=0.7mm];
\fill (-0.19051,-2.2061) circle [radius=0.7mm];
\fill (1.6378,1.6953) circle [radius=0.7mm];
\fill (-0.96381,1.3556) circle [radius=0.7mm];
\fill (0.15399,-1.1011) circle [radius=0.7mm];
\fill (-0.74502,-0.79631) circle [radius=0.7mm];
\fill (-0.061524,-0.81852) circle [radius=0.7mm];
\fill (1.1592,-1.9015) circle [radius=0.7mm];
\fill (1.1448,-3.0831) circle [radius=0.7mm];
\fill (-1.8380,2.5526) circle [radius=0.7mm];
\fill (-0.19663,-1.0193) circle [radius=0.7mm];
\fill (-2.6653,0.78249) circle [radius=0.7mm];
\fill (-0.47528,-3.3536) circle [radius=0.7mm];
\fill (1.3869,0.47658) circle [radius=0.7mm];
\fill (-0.039316,-1.5607) circle [radius=0.7mm];
\fill (-0.47927,0.15571) circle [radius=0.7mm];
\fill (-7.1739,4.5398) circle [radius=0.7mm];
\fill (-0.69830,0.058617) circle [radius=0.7mm];
\fill (2.3911,4.9227) circle [radius=0.7mm];
\fill (-2.9438,-1.5807) circle [radius=0.7mm];
\fill (1.0757,0.70464) circle [radius=0.7mm];
\fill (-0.66026,-0.043023) circle [radius=0.7mm];
\fill (1.0091,2.4317) circle [radius=0.7mm];
\fill (0.64964,0.30261) circle [radius=0.7mm];
\fill (-0.0029384,-4.8829) circle [radius=0.7mm];
\fill (0.17747,0.21499) circle [radius=0.7mm];
\fill (-0.89804,-0.20316) circle [radius=0.7mm];
\fill (6.3640,-6.1767) circle [radius=0.7mm];
\fill (0.68352,1.2052) circle [radius=0.7mm];
\fill (0.63312,-2.5303) circle [radius=0.7mm];
\fill (7.6505,-2.1114) circle [radius=0.7mm];
\fill (-3.2439,0.53165) circle [radius=0.7mm];
\fill (13.682,11.506) circle [radius=0.7mm];
\fill (-2.1335,-0.24066) circle [radius=0.7mm];
\fill (0.98559,-1.0642) circle [radius=0.7mm];
\fill (-0.36045,-5.2298) circle [radius=0.7mm];
\fill (0.63433,-2.6091) circle [radius=0.7mm];
\fill (-4.5373,0.15585) circle [radius=0.7mm];
\fill (0.29042,-0.98774) circle [radius=0.7mm];
\fill (1.9375,2.9193) circle [radius=0.7mm];
\fill (-1.4528,0.024769) circle [radius=0.7mm];
\fill (-0.43825,-0.74365) circle [radius=0.7mm];
\fill (0.24191,1.0415) circle [radius=0.7mm];
\fill (-0.18172,0.21405) circle [radius=0.7mm];
\fill (0.66680,-2.3111) circle [radius=0.7mm];
\end{tikzpicture}

\vspace{1cm}
\caption{A random geometric graph on $\sph^2$, with $N=100$ and $L=\frac{\pi}{8}$ (the blue circle is the equator).\label{fig:stereo}}
\end{center}
\end{figure}
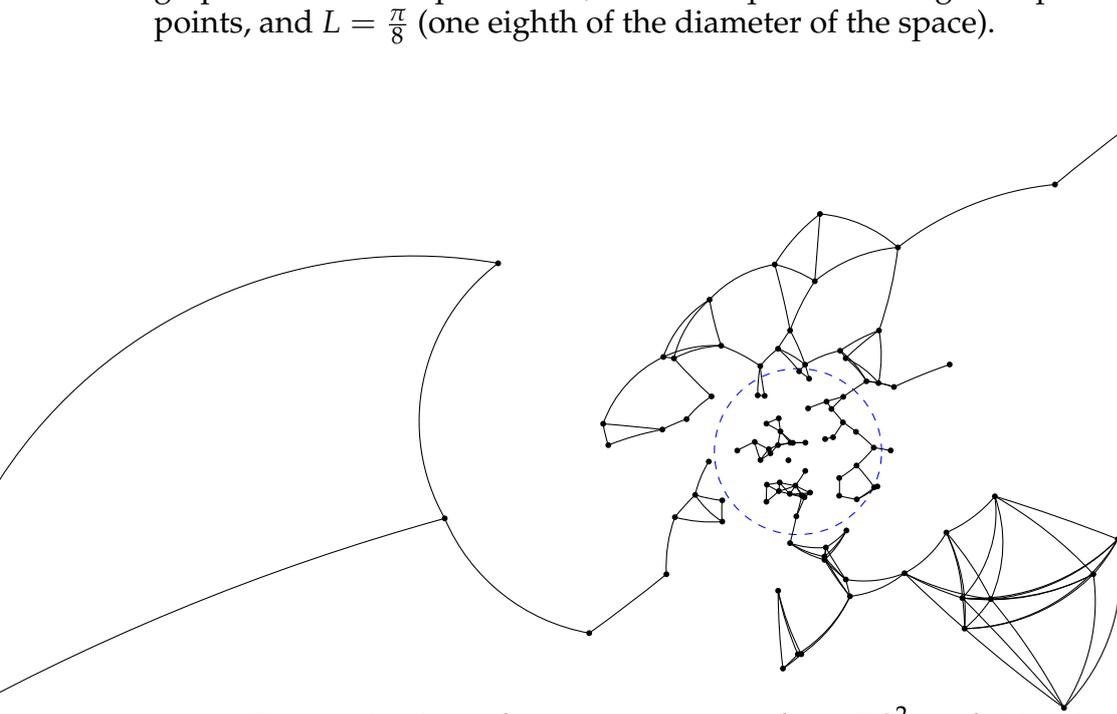

\subsection{Compact symmetric spaces}
Let us now explain which measured metric spaces $(X,d,m)$ will be allowed in this paper. We want $(X,d)$ to be a \emph{Riemannian manifold} (\emph{cf.} \cite{Jost11}), that is a smooth manifold $X$ endowed with scalar products $\scal{\cdot}{\cdot}_{\tang_xX}$ on each tangent space $\tang_xX$, these scalar products varying smoothly with $x$. The Riemannian structure allows one to measure the distance between two points:
$$d(x,y) = \inf_{\substack{\gamma : [0,1] \to X \text{ smooth path} \\ \gamma(0)=x,\,\,\gamma(1)=y}} \left(\int_{0}^1 \sqrt{\scal{\gamma'(t)}{\gamma'(t)}_{\tang_{\gamma(t)}X}}\DD{t}\right).$$
A geodesic on a Riemannian manifold $X$ is a (smooth) path that minimises locally the distances; it is the solution of an order $2$ differential equation, the Euler--Lagrange equation (\cite[Lemma 1.4.4]{Jost11}). If $X$ is a compact Riemannian manifold, then for any $x \in X$ and any vector $v$ of norm $1$ in $\tang_xX$, there is a unique geodesic $\gamma_{x,v} : \R \to X$ with $\gamma_{x,v}(0)=x$, $\gamma_{x,v}'(0)=v$ and $\|\gamma_{x,v}'(t)\|_{\tang_{\gamma_{x,v}(t)}X}=1$ for any $t \in \R$.\medskip

The computation of the spectrum of a random geometric graph on a Riemannian manifold $X$ relies on the harmonic analysis of this space. If $X$ has some symmetry properties, then this harmonic analysis turns into algebraic combinatorics, which allow exact calculations. Thus, in the sequel, we shall restrict ourselves to the more convenient setting of compact \emph{symmetric} Riemannian manifolds. A compact Riemannian manifold $X$ is called a (globally) \emph{symmetric space} if, for any $x\in X$, there exists a (unique) involutive isometry $s_x \in \mathrm{Isom}(X)$ that reverses the geodesics, that is to say that it sends $\gamma_{x,v}(t)$ to $\gamma_{x,v}(-t)$ for any $v \in \mathbb{S}(\tang_xX)$. Intuitively, this means that the geodesics meeting at a point $x \in X$ are arranged in a nice symmetric way around their starting point, see Figure \ref{fig:symmetric}. 
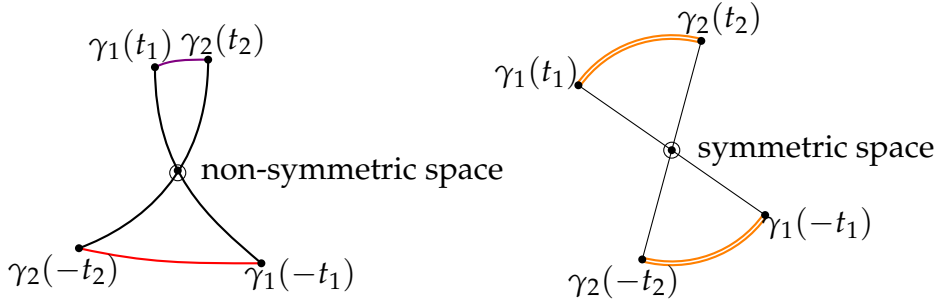
\begin{figure}[ht]
\begin{center}
\begin{tikzpicture}[scale=1]
\draw [thick,red] (1.1,-1.2) .. controls (0.1,-1.2) and (-0.4,-1.2) .. (-1.3,-1); 
\draw [thick,violet] (-0.3,1.4) .. controls (-0.1,1.5) and (0.1,1.5) .. (0.4,1.5);
\foreach \x in {(0,0.03),(-1.3,-1),(1.1,-1.2),(0.4,1.5),(-0.3,1.4)}
\fill \x circle (1.5pt);
\draw (0,0) circle (3pt);
\draw [thick] (-1.3,-1) .. controls (-0.5,-0.7) and (0.4,0) .. (0.4,1.5);
\draw [thick] (1.1,-1.2) .. controls (0.4,-0.6) and (-0.3,0) .. (-0.3,1.4);
\draw (2.3,0) node {non-symmetric space};
\draw (0.6,1.75) node {$\gamma_2(t_2)$};
\draw (110:1.8) node {$\gamma_1(t_1)$};
\draw (320:2.1) node {$\gamma_1(-t_1)$};
\draw (-1.5,-1.3) node {$\gamma_2(-t_2)$};
\begin{scope}[shift={(6.5,0.3)}]
\draw [thick,red!50!yellow,double] (75:1.5) arc (75:145:1.5);
\draw [thick,red!50!yellow,double] (255:1.5) arc (255:325:1.5);
\foreach \x in {(0,0),(75:1.5),(255:1.5),(145:1.5),(325:1.5)}
\fill \x circle (1.5pt);
\draw (0,0) circle (3pt);
\draw (75:1.5) -- (255:1.5);
\draw (145:1.5) -- (325:1.5);
\draw (70:1.85) node {$\gamma_2(t_2)$};
\draw (150:2.05) node {$\gamma_1(t_1)$};
\draw (333:2.15) node {$\gamma_1(-t_1)$};
\draw (250:1.85) node {$\gamma_2(-t_2)$};
\draw (1.9,0) node {symmetric space};
\end{scope}
\end{tikzpicture}
\end{center}
\caption{Geodesics on a symmetric space. \label{fig:symmetric}} 
\end{figure}
There is a complete classification of the (compact) symmetric spaces due to Cartan, see \cite[Chapter X]{Hel78}. To simplify a bit the discussion, we shall assume $X$ to be simply connected. In the general case, a connected but non simply connected compact symmetric space $X$ admits a universal cover $\widetilde{X}$ which is still a compact symmetric space, and whose covering map $\pi : \widetilde{X} \to X$ has finite degree. This allows one to transfer most results and techniques from $\widetilde{X}$ to $X$. Later, we shall for instance explain how to deal with the case of the special orthogonal groups $\SO(n\geq 3)$, which are symmetric spaces with fundamental group $\pi_1(\SO(n))=\Z/2\Z$, and which are covered by the simply connected spin groups $\Spin(n)$.\medskip

If $X$ is a simply connected compact symmetric space, then it is isometric to a unique product $X_1\times X_2\times \cdots \times X_r$ of such spaces, with each $X_i$ that cannot be split further. The $X_i$'s are called \emph{irreducible} or \emph{simple} simply connected compact symmetric spaces. Then, the classification of \emph{simple and simply connected compact symmetric spaces} (in short, ssccss) is the following:
\begin{enumerate}
\item either $X=G$ is one of the classical simple and simply connected compact Lie groups, associated to the root systems of type $A_{n\geq 1}$, $B_{n \geq 2}$, $C_{n \geq 3}$, $D_{n \geq 4}$, $\mathrm{G}_2$, $\mathrm{F}_4$, $\mathrm{E}_6$, $\mathrm{E}_7$ or $\mathrm{E}_8$. In this case, $(\mathrm{Isom}(X))^0 = G \times G$.
\item or, $X=G/K$, with $G$ simple and simply connected compact Lie group, and $K$ closed subgroup with $G^{\theta,0} \subset K \subset G^\theta$, where $G^\theta$ denotes the set of fixed points of an involutive automorphism $\theta : G\to G$, and $G^{\theta,0}$ is the connected component of the neutral element in $G^\theta$. In this case, $(\mathrm{Isom}(X))^0=G$.
\end{enumerate}
We call a ssccss of group type or of non-group type according to the aforementioned classification. The Riemannian structure on each ssccss $X=G$ or $X=G/K$ is unique up to a scalar multiple, and we shall explain in a moment how to construct it. Moreover, this Riemannian structure on $X$ yields a natural volume form $d\omega$ with finite mass.  After renormalisation, this volume form produces a probability measure $m= d\omega/(\int_X d\omega)$ on $X$ that is invariant by the group of isometries of $X$. Therefore, every ssccss is naturally endowed with a distance $d$ and a probability measure $m$.\medskip

Our objective is to study random geometric graphs in the general setting of ssccss. The harmonic analysis of the two types (group and non-group) is in theory quite similar, and in each case there is an explicit description of the spherical functions of the space (see the works of Helgason \cite{Hel70,Hel78,Hel84}). However, the manipulation of the spherical functions in the non-group case (which includes Grassmannian manifolds and Lagrangian Grassmannian manifolds) is in practice more difficult. \emph{Therefore, in this paper, we shall in many cases restrict our study to the group type, hence to the classical sscc Lie groups. We shall only treat the non-group type when the results extend almost immediately to this case.} More precisely, the non-group type ssccss will appear in the following sections:
\begin{itemize}
    \item Section \ref{subsec:gaussnongroup}: when studying the Gaussian regime in the symmetric spaces with rank one, the irreducible characters are replaced by the zonal spherical functions, which are in this case the orthogonal Laguerre or Jacobi polynomials, hence explicit and easy to manipulate.
    \item Section \ref{sec:poisson}: the Benjamini--Schramm local convergence holds for all the symmetric spaces, and the argument is exactly the same in the group and non-group case.
\end{itemize}
\medskip

\subsection{Compact Lie groups and normalisation of the Riemannian structure}\label{subsec:normalisation}
In the following we fix a simple simply connected compact (in short sscc) Lie group $G$. Given a compact Lie group $G$, the tangent space $\glie=\tang_{e_G}G$ at the neutral element $e_G$ is endowed with a structure of Lie algebra; we denote $[X,Y] = (\mathrm{ad}\,X)(Y)$. The opposite of the Killing form
\begin{equation}
\scal{X}{Y}_{\glie} = - \,\tr(\mathrm{ad}\,X \circ \mathrm{ad}\,Y).\label{eq:normalisation}
\end{equation}
is a symmetric and positive-definite bilinear form on $\glie$ which is invariant by the adjoint action of $G$ on $\glie$. We transport this scalar product to any tangent space $\tang_gG$ by the rule
$$\scal{V}{W}_{\tang_gG} = \scal{d_gL_{g^{-1}} (V)}{d_gL_{g^{-1}}(W)}_{\tang_{e_G}G}$$
where $L_{g^{-1}} : G \to G$ is the multiplication on the left by $g^{-1}$, and $d_gL_{g^{-1}}$ is the differential of this map at $g$. By construction, the Riemannian structure thus obtained is $G$-invariant on the left, and it is also $G$-invariant on the right since $\scal{\cdot}{\cdot}_\glie$ is $\mathrm{Ad}(G)$-invariant. \emph{In the sequel, the Riemannian structure on a classical sscc Lie group will always be the one associated to the bilinear form of Equation \eqref{eq:normalisation}}. The corresponding balls $B_{(x,L)} = \{y \in G\,|\,d(x,y)\leq L\}$ for the geodesic distance $d$ will be described in Section \ref{subsec:distancesliegroup}. We shall recall in a moment that almost all the sscc Lie groups are classical groups of matrices over the real, the complex or the quaternionic numbers. The Killing form writes then as 
$\tr(\mathrm{ad}\,X\circ \mathrm{ad}\,Y) = c\,\mathrm{Re}(\tr(XY)),$
with $c=n-2$ when $\glie=\mathfrak{so}(n)$; $c=2n$ when $\glie=\mathfrak{su}(n)$; and $c=4n+4$ when $\glie=\mathfrak{sp}(n)$ (the real part of the trace is only needed in this last case).

\begin{example}
Consider the group of special unitary matrices
$$\SU(2)=\{M \in \mathrm{M}(2,\C)\,\,|\,\,MM^* = M^*M = I_2,\,\,\det M = 1\}.$$ 
It is a sscc Lie group with real dimension $3$, and it is diffeomorphic to the unit sphere $\mathbb{S}^3$. Its Lie algebra is the space of anti-hermitian matrices $\mathfrak{su}(2) = \{M \in \mathrm{M}(2,\C)\,\,|\,\,M^*=-M,\,\,\tr(M)=0\}$, and the opposite Killing form is equal to $\scal{M}{N}_{\mathfrak{su}(2)} = -4\,\tr(MN)=4\,\tr(M^*N)$. For the corresponding Riemannian structure, the distance in $\SU(2)$ between two unitary matrices $U_1$ and $U_2$ is
$$d(U_1,U_2) = 2\sqrt{2}\,|\theta(U_1U_2^{-1})|,$$
where $\E^{\pm \I\theta(M)}$ are the two eigenvalues of a unitary matrix $M \in \SU(2)$, with $\theta(M) \in [0,\pi]$. Indeed, $d(U_1,U_2)=d(U_1U_2^{-1},I_2) = d(\diag(\E^{\I \theta(U_1U_2^{-1})},\E^{-\I \theta(U_1U_2^{-1})}),I_2)$, and a geodesic connecting $I_2$ to the diagonal matrix $\diag(\E^{\I \theta},\E^{-\I\theta})$ is
$$\gamma : t\in [0,1]\mapsto \exp \begin{pmatrix}
t\I \theta & 0 \\ 0 & -t\I\theta
\end{pmatrix},$$
which has constant speed $\|\gamma'(t)\|=2\sqrt{2}\,\theta$. In particular, the diameter of $\SU(2)$ with this normalisation is $2\sqrt{2}\,\pi$.
\end{example}\medskip

If $X=G/K$ is a ssccss of non-group type, we denote $\pi_X : G \to X$ the canonical projection, and $o=\pi_X(e_G)=K$. The tangent space $\tang_oX$ identifies through $\tang_{e_G}\pi_X$ with the Killing orthogonal complement $\mathfrak{p}$ of the Lie subalgebra $\mathfrak{k}$ of $K$ in $\glie$. The restriction of the scalar product from Equation \eqref{eq:normalisation} to $\mathfrak{p}$ can be transported to any tangent space $\tang_xX$ by using the action of $G$:
$$\scal{V}{W}_{\tang_xX} = \scal{d_xA_{g^{-1}} (V)}{d_xA_{g^{-1}}(W)}_{\tang_{o}X},$$
where $x = g\cdot o$ and $A_g : X \to X$ is the action of $g$ on $X$. By construction, the Riemannian structure thus obtained makes $G$ act on $X$ by isometries. \emph{In the sequel, the Riemannian structure on a ssccss of non-group type will be the one obtained by this construction. However, in the specific case of ssccss of rank one (Section \ref{subsec:gaussnongroup}), we shall multiply this Riemannian metric by a multiplicative constant so as to fit the classical definitions}. The following example explains why this modification is natural.

\begin{example}
Suppose that $X=\mathbb{C}\mathbb{P}^n=\SU(n+1)/\mathrm{S}(\Unit(n)\times\Unit(1))$ is the complex projective space. If $Z=(z_0,\ldots,z_n)$ belongs to $\C^{n+1}\setminus \{0\}$, we denote $[Z]=[z_0,z_1,\ldots,z_n]$ the corresponding line in $\C\mathbb{P}^n$. The reference point in $X$ is $o=[0,\ldots,0,1]$. The standard Riemannian metric on $\mathbb{C}\mathbb{P}^n$ is the Fubini--Study metric, defined by
$$\scal{V}{W}_{[Z]} = \frac{\|Z\|^2\,(V,W)-(V,Z)\,(Z,W)}{\|Z\|^4},$$
where $\|Z\|^2=\sum_{i=0}^n |z_i|^2$ and $(Y,Z)$ is the real scalar product on $\C^n$ corresponding to this norm. In this formula, a vector $V \in \C^{n+1}$ is sent to the element of $\tang_{[Z]}\C\mathbb{P}^{n}$ which is the derivative at $t=0$ of the smooth curve $([Z+tV]=[z_0+tv_0,\ldots,z_n+tv_n])_{t \in \R}$; the kernel of the linear map $\C^{n+1}\to \tang_{[Z]}\C\mathbb{P}^{n}$ is the line $[Z]$. In particular, if $[Z]=o$, then we have the identification $\C^n = \tang_o\C\mathbb{P}^n$, and the scalar product on $\C^n$ inherited from the Fubini--Study metric is simply $\scal{V}{V}=\sum_{i=0}^{n-1}|V_i|^2$. Now, the Riemannian structure obtained by $\SU(n+1)$-transport of the restriction to $\mathfrak{p} = \tang_o\C\mathbb{P}^n$ of the opposite Killing form is a scalar multiple of this metric. Indeed, we have
$$ \mathfrak{p} = \left\{ \begin{pmatrix}
0 & \cdots & 0 & z_0 \\
\vdots & & \vdots & \vdots \\
0 & \cdots & 0 &z_{n-1}\\
-\overline{z_0}&\cdots &-\overline{z_{n-1}} & 0
\end{pmatrix},\,\,\,(z_0,\ldots,z_{n-1})\in\C^n \right\},$$
and the tangent map $\tang_{I_{n+1}}\pi_{\C\mathbb{P}^n}$ sends the skew-Hermitian matrix $M(z_0,\ldots,z_{n-1})$ to the vector $(z_0,\ldots,z_{n-1})$ in $\C^n = \tang_o\C\mathbb{P}^n$. As the Killing form of $\SU(n+1)$ is $(A,B)\mapsto (2n+2)\,\tr (AB)$, we conclude that the scalar product on $\C^n$ given by the structure of symmetric space is $\scal{V}{V}=(4n+4)\sum_{i=0}^{n-1}|V_i|^2$, hence $(4n+4)$ times the "standard" scalar product.
\end{example}
\medskip

To conclude this section, let us detail a bit more the classification of sscc Lie groups. They are:
\begin{itemize}
    \item the special unitary groups $\SU(n)$ with $n \geq 2$: 
    $$\SU(n) = \{M \in \mathrm{M}(n,\C)\,\,|\,\, M^*M=MM^*=I_n,\,\,\det M = 1\};$$
    \item the compact symplectic groups $\SP(n)$ with $n \geq 2$: 
    $$\SP(n) = \{M \in \mathrm{M}(n,\Hq)\,\,|\,\, M^\dagger M=MM^\dagger=I_n\},$$
    where $(M^\dagger)_{ij} = \overline{M_{ji}}$, the conjugate of a quaternionic number $a + \I b + \J c + \mathrm{k} d$ being  $a - \I b - \J c - \mathrm{k} d$;
    \item the spin groups $\Spin(n)$ with $n \geq 7$, which are double covers of the special orthogonal groups
    $$\SO(n) = \{M \in \mathrm{M}(n,\R)\,\,|\,\, M^tM=MM^t=I_n,\,\,\det M = 1\},$$
    and which are simply connected (whereas $\pi_1(\SO(n))=\Z/2\Z$ for any $n \geq 3$).
\end{itemize}
There are also $5$ exceptional cases which are associated to the root systems $\mathrm{G}_2$, $\mathrm{F}_4$, $\mathrm{E}_6$, $\mathrm{E}_7$ and $\mathrm{E}_8$, and which all related to the geometry of the algebra of octonions (see \cite{Baez02}). For instance, consider the exceptional Jordan algebra $A(3,\Octo)$ (the so-called Albert algebra), which is the algebra of real dimension $27$ that consists in Hermitian $3\times 3$ octonionic matrices, endowed with the Jordan product
$$A\circ B = \frac{AB+BA}{2}.$$
One can show that the automorphism group of this algebra is a simply connected simple compact Lie group of real dimension $52$, associated to the root system $\mathrm{F}_4$. As the exceptional Lie groups do not possess adequate systems of (matrix) coordinates, it is quite difficult to express distances on them. Thus, in these cases, our theoretical results will remain mainly abstract. On the other hand, for the "classical" sscc groups 
$$ \SU(n \geq 2),\,\,\SP(n \geq 2),\,\,\Spin(n \geq 7),$$
all our results will be explicit; see the appendix (Section \ref{sec:appendix}) for explanations and computations on these groups. Note that one can extend many of our results to a slightly more general setting, with reductive connected Lie groups instead of sscc Lie groups. The case of the special orthogonal groups $\SO(n)$, which are not simply connected, is for instance explained in Remark \ref{remark:fromspintoso}.
\medskip

\subsection{Main results and outline of the paper}
When studying the random geometric graphs $\Gamma=\GEOM(N,L)$ on a compact Riemannian manifold $X$, there are two interesting asymptotic regimes which one can consider:
\begin{enumerate}
    \item the \emph{Gaussian} regime, where $L$ is fixed but $N$ goes to infinity; in this setting the adjacency matrix $A_\Gamma$ is dense.

    \item the \emph{Poissonian} regime, where $L=L_N$ decreases to zero in such a way that each vertex of $\Gamma$ has a $O(1)$ number of vertices; in this setting the adjacency matrix $A_\Gamma$ is sparse.
\end{enumerate}

\paragraph{\textbf{Gaussian regime}} The adjacency matrix $A_\Gamma$ can be considered as a finite-dimensional (random) approximation of the operator of convolution by the kernel $h(x,y) = 1_{d(x,y)\leq L}$. In particular, a result due to Giné and Koltchinskii \cite{GK00} relates the asymptotics of the spectrum of $A_\Gamma$ to the eigenvalues of the operator of convolution by $h$ (Section \ref{subsec:ginekoltchinskii}). Suppose that $X=G$ is a ssccss of group type. By using the representation theory of compact Lie groups, one can compute these eigenvalues, which drive the highest frequencies of the random geometric graph, that is the asymptotic behavior of the largest eigenvalues of $A_{\GEOM(N,L)}$. In Sections \ref{subsec:convolution} and \ref{subsec:weightlattice}, we present the arguments from representation theory that show that there is one limiting eigenvalue $c_\lambda$ of 
$$\frac{1}{N}\,\mathrm{Spec}(A_{\GEOM(N,L)}) = \left(\frac{c_1(N)}{N} \geq \frac{c_2(N)}{N} \geq \cdots \geq \frac{c_N(N)}{N} \right)$$ for each dominant weight $\lambda$ of the group $G$. This eigenvalue $c_\lambda$ has a multiplicity related to the dimension of the corresponding irreducible representation of $G$. In Section \ref{sec:gaussian}, we complete this theoretical result by an explicit calculation of $c_\lambda$ (Theorem \ref{thm:gaussianlimit}). Thus, each limiting eigenvalue $c_\lambda$ is given by a finite linear combination of values of Bessel functions of the first kind $J_\beta$, see Section \ref{subsec:gaussianlimit}. In Section \ref{subsec:spectralgap}, we deduce from this result an estimate of the spectral radius and of the spectral gap of the matrix $A_{\GEOM(N,L)}$ when $L$ is fixed and $N$ goes to infinity. If instead of a group $G$ we consider a ssccss of non-group type $G/K$, the same techniques apply in theory, but with the irreducible representations replaced by the spherical representations of the pair $(G,K)$, and the irreducible characters by the zonal spherical functions. These functions can be cumbersome to deal with in the general case, but if $G/K$ has rank one (meaning that there are no totally geodesic flat submanifold of dimension strictly larger than $1$), then they are simply the Laguerre or Jacobi polynomials, and the computations can be explicitly performed; we explain this in Section \ref{subsec:gaussnongroup}. \medskip

\paragraph{\textbf{Poissonian regime}} We consider again a general ssccss $X$. The connection distance $L_N$ is normalised as follows:
$$L_N = \left(\frac{\ell}{N}\right)^{\frac{1}{\dim X}},$$
with $\ell>0$ fixed. Then, the number of neighbors of any vertex $v_i$ of $\GEOM(N,L_N)$ follows a binomial law 
$$\mathcal{B}\!\left(N-1\,,\,p_{L_N}=\frac{\vol(B(v_i,L_N))}{\vol(X)}\right),$$
where $\vol(X)$ is the volume of the symmetric space $X$ for the volume form associated to the Riemannian structure given by Equation \eqref{eq:normalisation}, and $\vol(B(v_i,L_N))$ is the volume of the ball in $X$ with center $v_i$ and radius $L_N$. As $N$ goes to infinity, this volume behaves like the volume of a Euclidean ball with the same dimension, which is
$$\frac{\pi^{\frac{\dim X}{2}}}{\Gamma(1+\frac{\dim X}{2})}\,(L_N)^{\dim X} = c(\dim X)\,\frac{\ell}{N},\quad \text{with } c(\dim X) = \frac{\pi^{\frac{\dim X}{2}}}{\Gamma(1+\frac{\dim X}{2})} $$
Therefore, in the limit $N \to \infty$, the number of neighbors of any vertex $v_i$ of $\GEOM(N,L_N)$ has a law close to a Poisson law of parameter $\frac{c(\dim X)}{\vol(X)}\,\ell$; in particular it is a $O(1)$. More generally, for any $n \geq 1$ and any fixed vertex $v_i$, one can show that the subgraph of $\GEOM(N,L_N)$ which consists in vertices at distance smaller than $n$ from $v_i$ has a limit in law in the set of rooted finite graphs. This is the \emph{convergence in the local Benjamini--Schramm sense} \cite{BS01}, and the limit only depends on the dimension $\dim X$ and on the parameter $\ell/\vol(X)$; see Sections \ref{subsec:benjaminischramm}--\ref{subsec:infernal}, and in particular our Theorem \ref{thm:poisson_BS}. We prove this result by developing a general theory relating the convergence of pointed metric spaces to the local Benjamini--Schramm convergence of the random geometric graphs drawn on such spaces; see Theorem \ref{thm:convergencespacesandgraphs}.
\medskip

It is then known from \cite{BL10,ATV11,BLS11,Bor16} that under appropriate assumptions, the local convergence of random graphs implies the convergence in law of the spectral measures
$$\nu_N = \frac{1}{N}\sum_{i=1}^N \delta_{c_i(N)}$$
of the graphs $\GEOM(N,L_N)$ towards a limiting probability measure $\mu$. We check the conditions to apply this result in Section \ref{subsec:poissonlimit}; one has in particular to verify that two random roots in $\GEOM(N,L_N)$ give rise to two \emph{independent} local limits, and this is a consequence of the structure of group or homogeneous (symmetric) space. We also prove that the limiting measure $\mu$ of the spectral measures $\nu_N$ is determined by its moments, and that we have convergence in probability of the moments. \medskip

\paragraph{\textbf{From random graphs to a conjecture in representation theory}}
Moving on from there, one can try to obtain more information on the limiting distribution $\mu$. For instance, one expects it to be compactly supported, but this result does not follow from the abstract link between local Benjamini--Schramm convergence and convergence of the spectral measure. The crude upper bounds proving that $\mu$ is determined by its moments also do not imply the compactness of the support. This leads one to try to improve these bounds, and to develop techniques that enable one to compute all the moments $M_s=\int_\R x^s\,\mu(\!\DD{x})$. In the sequel, we focus on the case where $X=G$ is a sscc Lie group.
\begin{enumerate}
    \item In Section \ref{subsec:circuit}, we start by giving a circuit expansion of the expected moments, which is a combinatorial expansion of $\esper[\int_{\R} x^s \,\nu_N(\!\DD{x})]$ involving certain labeled graphs. This expansion implies that each moment $M_s$ is a polynomial with degree $s-1$ in the parameter $\ell$ (see Theorem \ref{thm:circuitexpansion}).
    \item Since the graph limit in the local sense does not depend on the group $G$ and only depends on $\dim G$ and $\ell/\vol(G)$, the same is true for the limiting spectral measure $\mu$, and therefore one can replace $G$ by a simpler group, namely, the torus $\Tor^{\dim G}$. We plan to pursue this approach in a forthcoming paper; even with this simplification, it is not easy to obtain good bounds on the moments $M_s$, as it amounts to count (reduced) circuits with certain weights (see Remark \ref{remark:lastone}). 
\end{enumerate}
Aside from the search for good upper bounds on the moments $M_s$, there is actually an interest in keeping the base model $G$ instead of the flat model $\Tor^{\dim G}$. It turns out that the Poissonian regime of random geometric graphs, which we approach in Section \ref{sec:poisson} with the geometric notion of local Benjamini--Schramm convergence, can also be studied with representation theoretic tools (Section \ref{sec:ART}). In this setting, the computation of the moments sheds a different light on the degeneration from the Gaussian to the Poissonian regime, and it eventually leads to an algebraic conjecture which we state below, and which concerns certain joint integrals of characters of $G$. Let us explain briefly how one is led to it:
\begin{enumerate}
    \item The formulas in the Gaussian regime (Section \ref{sec:gaussian}) rely mainly on the \emph{Weyl formula} for the characters of the irreducible representations of $G$. When going from the Gaussian to the Poissonian regime and trying to compute the first moments $M_s$ in the model $G$ (specifically, for $s \leq 5$), the Weyl formula degenerates into a product of partial derivatives, and the \emph{sums over dominant weights} become integrals over Weyl chambers and products thereof; see Section \ref{subsec:onepoint}. This is a typical result from asymptotic representation theory, and as far as we know this explicit degeneration has not been pointed at previously in a study of random objects associated to groups.
    \item The previous degeneration concerns the terms of the circuit expansion of a moment $M_s$ which corresponds to a reduced circuit with one vertex. When $s \in \{6,7\}$, one starts to see contributions from reduced circuits with two vertices, and their asymptotics is related to asymptotic formulas for the \emph{Littlewood--Richardson coefficients} associated to large dominant weights. In the general case, these asymptotic formulas come from the Kashiwara--Lusztig theory of crystal bases and the Berenstein--Zelevinsky theory of string polytopes; they involve positive measures with piecewise polynomial densities, against which one integrates partial derivatives of Bessel functions in order to compute the contributions of the reduced circuits on two vertices; see Section \ref{subsec:twopoints}. A similar kind of degeneration has been observed when studying Brownian motions in Weyl chambers, see \cite{BBO05}. 
    \item Starting with $s \geq 8$, the circuit expansion of $M_s$ involves some reduced circuits with more than $3$ vertices. These contributions are limits of certain series whose terms involve graph functionals of the irreducible characters of $G$ (Section \ref{subsec:graphfunctionals}). If we suppose that the limiting process happens in the same way as for $2$-vertices reduced circuits, then we obtain the following conjecture. Suppose that $G$ is a sscc Lie group and that $S=(V(S),E(S))$ is a finite graph, possibly with multiple edges or loops and with an arbitrary orientation $a\to b$ of each edge $\{a,b\} \in E(S)$. We associate to each edge $e \in E(S)$ a dominant weight $\lambda_e$, which parametrises an irreducible finite-dimensional representation of $G$; see Section \ref{subsec:weightlattice} for a reminder on this theory. The \emph{graph functional} associated to $G$, $S$ and to this choice of dominant weights is:
\begin{equation}
    \mathrm{GF}_S((\lambda_e)_{e \in E(S)}) = \int_{G^{k}}\left( \prod_{(a \to b) \in E(S)} \ch^{\lambda_e}(g_a(g_b)^{-1})\right)\DD{g_1}\,\cdots \DD{g_{k}}, \label{eq:graphfunctional}
\end{equation}
where $V(S)=\{1,2,\ldots,k\}$, $\DD{g}$ is the Haar measure on $G$, and $\ch^{\lambda_e}$ is the character of the irreducible representation $V^{\lambda_e}$ with highest weight $\lambda_e$. When $S$ has one or two vertices and several edges or loops, one recovers classical quantities such as the dimensions $\dim V^\lambda$ or the Littlewood--Richardson coefficients $c_{\nu}^{\lambda,\mu}$. The graph functionals defined by Equation \eqref{eq:graphfunctional} are generalisations of these quantities, and thus it is natural to try to compute them. Our study of random geometric graphs in the Poissonian regime led us to the following conjecture, which seems important:
\end{enumerate}
\begin{conjecture}\label{conj:graphfunctional}
Fix a sscc Lie group $G$ and a connected graph $S$ as above, with $k$ vertices and $r$ edges. We denote $\Z\Omega$ the weight lattice of $G$, see Section \ref{subsec:weightlattice}. There exists a sublattice $A_S \subset (\Z\Omega)^r$ with maximal rank $rd$ and such that:
\begin{itemize}
     \item If the integrality condition $(\lambda_e)_{e \in E(S)} \in A_S$ is not satisfied, then $\mathrm{GF}_S((\lambda_e)_{e \in E(S)})$ vanishes.
     \item If the integrality condition $(\lambda_e)_{e \in E(S)} \in A_S$ is satisfied, then $\mathrm{GF}_S((\lambda_e)_{e \in E(S)})$ equals the number of integer points in a polytope $\mathscr{P}((\lambda_e)_{e \in E(S)})$ whose generic dimension is 
     $$\dim \mathscr{P}((\lambda_e)_{e \in E(S)}) = lr-(2l+d)(k-1),$$ 
with $d = \rank(G)$ and $l=\frac{\dim (G) - \rank(G)}{2}$. Here by generic we mean that the dimension of the polytope is equal to the right-hand side as soon as the dominant weights $\lambda_e$ are in the interior of the Weyl chamber. The equations that determine the polytope $\mathscr{P}((\lambda_e)_{e \in E(S)})$ are affine functions of the weights, and $\mathscr{P}((\lambda_e)_{e \in E(S)})$ is a part of the string cone $\mathscr{S}\!\mathscr{C}(G^r)$ of the sscc Lie group $G^r$.
 \end{itemize} 
\end{conjecture}

This conjecture implies some vanishing results which do not seem trivial at all; see Remark \ref{remark:vanishing}. We probably would never have obtained this conjecture without examining this concrete problem of computation of the moments $M_s$; it is a typical example of the interplay between random objects considered on spaces which admit a group of symmetry, and the asymptotic representation theory of these groups. Note that the conjecture is interesting in itself, but not at all for the original problem stated at the beginning of the paragraph (computing bounds on $M_s$), which is more of a combinatorial nature and which we do not intend to solve here (it is then required to consider the flat model $\Tor^{\dim G}$). Our last Section \ref{sec:ART} is devoted to the presentation of this conjecture, following the arguments that we have briefly exposed above. We also found it essential to explain how the degeneration from the Gaussian to the Poissonian regime of geometric graphs can be followed in representation theoretic terms, with degenerations of the Weyl formula, of sums over dominant weights and of Littlewood--Richardson coefficients; these results will certainly be interesting for specialists of asymptotic representation theory. A reader with a probabilistic background might not be familiar with the arguments from representation theory. He will find in this case:
\begin{itemize}
     \item in Section \ref{sec:representationtheory}, a reminder of the classical Cartan--Weyl representation theory of sscc Lie groups; we also use this section to fix notations.

     \item an appendix (Section \ref{sec:appendix}) with a list of conventions and results (choice of the maximal tori, description of the root systems and of the weight lattices, computation of the volumes of the groups, \emph{etc.}); it allows one to apply concretely our results to the classical sscc Lie groups ($\SU(n)$, $\SP(n)$, $\SO(n)$ and $\Spin(n)$).

     \item a second appendix (Section \ref{sec:crystal}) with an explanation of the theory of crystals of representations and string polytopes; one can skip these explanations if one is not interested in the algebraic details that leads to Conjecture \ref{conj:graphfunctional}.  
 \end{itemize}
The reader with a more advanced knowledge of these algebraic results can safely skip these sections.

\subsection*{Acknowledgments}
I would like to address my thanks to the organizers of the seminar on kernel random matrices at University Paris-Sud (Orsay), who introduced me to the problem of the spectrum of random geometric graphs. In particular, I am much indebted to Édouard Maurel-Segala, who explained to me the somewhat easier case of geometric graphs on $\ell^\infty$-tori. I also thank Reda Chhaibi for several discussions that we had on the subject, and for his precious comments. I learned about the link between Benjamini--Schramm convergence and the convergence of spectral measures from talks given in the Workgroup on random matrices and graphs at the Institut Henri Poincaré (MEGA), and I am very thankful to its organizers. A significant progress on this project was made during a conference in Les Diablerets (Switzerland) in January 2017, and I would like to thank the organizers of this conference for their invitation. Finally, I am indebted to an anonymous referee for many constructive remarks on a previous version of this paper, which allowed to improve a lot the presentation of our results.
\bigskip

\section{Ingredients from representation theory}\label{sec:representationtheory}
In this section, $G$ is a fixed sscc Lie group, and $L>0$ is a fixed level. The uniform probability measure $m$ on this space (Haar measure) will be denoted $\!\DD{g}$ or $\!\DD{x}$. By combining a result of Giné and Koltchinskii and the representation theory of compact groups, we relate the spectrum of the random adjacency matrix $A(N,L)$ of $\GEOM(N,L)$ to the spectrum of an integral operator on $\leb^2(G,\!\DD{g})$. This integral operator will be explicitly diagonalised in Section \ref{sec:gaussian}. The present section will also allow us to introduce many ingredients from representation theory that we shall use throughout the paper.\medskip

\subsection{The Giné--Koltchinskii law of large numbers}\label{subsec:ginekoltchinskii}
We denote $\leb^2(G,\!\DD{g})$ the set of complex-valued measurable and square-integrable functions on $G$. Let $h(x,y)$ be a real symmetric function on $G$, such that $\iint_{G^2} (h(x,y))^2 \DD{x}\DD{y} < +\infty$. The convolution by $h$ induces a integral operator $T_h$ on $\leb^2(G,\!\DD{g})$:
\begin{align*}
T_h : \leb^2(G,\!\DD{g}) & \to\leb^2(G,\!\DD{g}) \\
f &\mapsto \left(T_h(f) : x \mapsto \int_G h(x,y)\,f(y)\DD{y}\right).
\end{align*}
This operator is auto-adjoint and of Hilbert--Schmidt class: given any (countable) orthonormal basis $(e_i)_{i \in I}$ of $\leb^2(G,\!\DD{g})$, 
$ (\|T_h\|_{\mathrm{HS}})^2 = \sum_{i \in I} (\|T_h(e_i)\|_{\leb^2(G)})^2 =  (\|h\|_{\leb^2(G^2)})^2$.
Therefore, $T_h$ is a compact operator, and it admits a discrete real spectrum, which we label by integers:
$$\mathrm{Spec}(T_h) = (c_{-1}\leq c_{-2}\leq \cdots \leq 0 \leq \cdots \leq c_2 \leq c_1 \leq c_0),$$
with $\lim_{|k| \to \infty} c_k= 0$ (here, we add an infinity of zeroes to the sequence $(c_k)_{k \in \Z}$ if needed, for instance when $T_h$ is of finite rank). The Hilbert--Schmidt class ensures that $\sum_{k\in \Z} (c_k)^2 < +\infty$. Now, a general result due to Giné and Kolchinskii (see \cite{GK00}) ensures that one can approximate the operator $T_h$ by the random matrices
$$T_{h}(N) = \frac{1}{N}\,((1-\delta_{ij})\, h(v_i,v_j))_{1 \leq i,j \leq N} ,$$
where the $v_i$'s are independent random variables chosen according to the Haar measure $\!\DD{g}$ on $G$. The spectrum of $T_h(N)$ is a random set
$$\mathrm{Spec}(T_h(N)) = (c_{-1}(N) \leq c_{-2}(N) \leq \cdots \leq 0 \leq \cdots \leq c_2(N) \leq c_1(N) \leq c_0(N)),$$
which approximates $\mathrm{Spec}(T_h)$ in the following sense:

\begin{theorem}[Giné--Koltchinskii, Theorem 3.1 in \cite{GK00}]\label{thm:GKlawoflargenumbers}
Under the previous assumptions,
$$\delta\left(\mathrm{Spec}(T_h(N)),\mathrm{Spec}(T_h)\right) = \sqrt{\sum_{k \in \Z} (c_k(N) - c_k)^2} \to 0\quad\text{almost surely}.$$
\end{theorem}

\noindent This result yields readily the asymptotics of the spectrum of $A(N,L)$ when $L$ is fixed and $N$ goes to infinity. Indeed, 
$$\frac{A(N,L)}{N} = T_h(N) \quad \text{with }h(x,y)=1_{d(x,y) \leq L}.$$
Now, notice that the operator $T_h$ is in fact an operator of convolution by a function of \emph{one} variable:  for $f \in \leb^2(G,\!\DD{g})$,
\begin{align*}
(T_h(f))(g) & = \int_G 1_{d(g,u)\leq L}\,f(u)\DD{u} = \int_G 1_{d(gu^{-1},e_G)\leq L}\,f(u) \DD{u} = (Z_L*f)(g),
\end{align*}
where $Z_L(g) = 1_{d(g,e_G)\leq L}$. Here we used the invariance of the distance $d$ by the action of the group $G$. Hence, $T_h$ is an operator of convolution on $\leb^2(G,\!\DD{g})$ by a function in $\leb^2(G,\!\DD{g})$ which is invariant by conjugation.
The next paragraphs explain how to use the representation theory of $G$ in order to compute the eigenvalues of such a convolution operator (and therefore, the asymptotics of $\mathrm{Spec}(\GEOM(N,L))$ in the regime where $L$ is fixed and $N \to +\infty$).
\medskip

\subsection{Convolution on a compact Lie group}\label{subsec:convolution}
Let $G$ be a compact topological group endowed with its Haar measure $\!\DD{g}$. We denote $\hatG$ the set of classes of isomorphism of irreducible finite-dimensional complex representations of $G$; it is always countable, and for any element $\lambda \in \hatG$ corresponding to a representation $(V^\lambda,\rho^\lambda : G \to \GL(V^\lambda))$, one can find an Hermitian scalar product $\scal{\cdot}{\cdot}_{V^\lambda}$ on $V^\lambda$ which is invariant by $G$. This scalar product induces an adjunction $u \mapsto u^*$ on $\Endom(V^\lambda)$, and we then endow $\Endom(V^\lambda)$ with the scalar product
$\scal{u}{v}_{\Endom(V^\lambda)} = d_\lambda\,\tr(u^*v),$
where $d_\lambda$ is the complex dimension of $V^\lambda$. The basic theorem which allows to understand convolution in $\leb^2(G,\!\DD{g})$ is:

\begin{theorem}[Peter--Weyl, 1927]\label{thm:peterweyl}
For $\lambda \in \hatG$ and $f \in \leb^2(G,\!\DD{g})$, denote 
$$\widehat{f}(\lambda) = \left(\int_G f(g)\,\rho^\lambda(g)\DD{g}\right) \in \Endom(V^\lambda)$$
the  Fourier transform of $f$. The map $f\mapsto \widehat{f}$ from $\leb^2(G)$ to 
$\leb^2(\hatG) = \bigoplus_{\lambda \in \hatG}^\perp \Endom(V^\lambda)$
is an isometry of Hilbert spaces and an isomorphism of algebras (with $\leb^2(G,\!\DD{g})$ endowed with the convolution product).
\end{theorem}

We refer to \cite[Chapter 4]{Bump13} for a proof of this important result. It implies that the eigenspaces for the convolution on the left by $Z_L$ correspond \emph{via} the Fourier transform to subspaces of the endomorphism spaces $\Endom(V^\lambda)$, that are eigenspaces for the multiplication on the left by $\widehat{Z_L}(\lambda)$. Moreover, as $Z_L$ is invariant by conjugation, the convolution on the left by $Z_L$ is the same as the convolution on the right by $Z_L$. In the Fourier world, this means that each endomorphism $\widehat{Z_L}(\lambda)$ is in the center of $\Endom(V^\lambda)$, hence a scalar matrix $c_\lambda \,\mathrm{id}_{V^\lambda}$. Therefore, the eigenspaces for the convolution on the left by $Z_L$ are exactly the spaces $\Endom(V^\lambda)$, and the corresponding eigenvalues are the
$$c_\lambda = \frac{\tr (\widehat{Z_L}(\lambda))}{d_\lambda} = \int_G Z_L(g)\,\chi^\lambda(g)\,dg,$$
where $\chi^\lambda(g) = \frac{\tr (\rho^\lambda(g))}{d_\lambda}$ is the \emph{normalised character} of the irreducible representation $\lambda$. Thus, to summarise:

\begin{proposition}\label{prop:abstractgaussian}
Denote $Z_L(g)=1_{d(g,e_G)\leq L}$ with $G$ sscc Lie group. The eigenvalues of the operator on $\leb^2(G)$ of convolution on the left or on the right by $Z_L$ are in bijection with the irreducible representations $\lambda \in \hatG$. Each eigenvalue $c_\lambda$ has multiplicity $(d_\lambda)^2$ and is given by the formula $c_\lambda = \int_G Z_L(g)\,\chi^\lambda(g)\DD{g}$.
\end{proposition}

\noindent The next paragraph will allow us to identify the set $\hatG$, and to compute the dimensions $d_\lambda$. Proposition \ref{prop:abstractgaussian} extends readily to the case of symmetric spaces $X=G/K$, see Section \ref{subsec:gaussnongroup}.

\begin{remark}
In the following, we denote $\ch^\lambda$ the non-normalised character $\tr\, \rho^\lambda$. A direct consequence of the Peter--Weyl theorem \ref{thm:peterweyl} is that the collection of non-normalised characters $(\ch^\lambda)_{\lambda \in \hatG}$ forms an orthonormal basis of $\leb^2(G)^G$, the space of square-integrable and conjugacy-invariant functions on $G$. Moreover, one has the convolution rule $\ch^\lambda * \ch^\mu = \frac{\delta_{\lambda,\mu}}{d_\lambda}\,\ch^\lambda$.
\end{remark}
\medskip

\subsection{Weight lattice and combinatorics of the highest frequencies}\label{subsec:weightlattice}
When $G$ is a (semi)simple simply connected compact Lie group, the set $\hatG$ is classically described by Weyl's highest weight theorem, see for instance \cite[Theorems 3.2.5 and 3.2.6]{GW09}. Let $T$ be a maximal torus in $G$, and $\Z\Omega$ be the \emph{lattice of weights}, a weight of $G$ being a character $\omega : T \to \mathrm{U}(1) = \{z \in \C\,\,|\,\,|z|=1\}$ such that there exists a unitary representation $(V,\rho)$ of $G$ with
$$V_\omega = \{v \in V\,\,|\,\,\forall t \in T,\,\,(\rho(t))(v) = \omega(t)\,v\} \neq \{0\}.$$
The weights form a free module over $\Z$ for the operation of pointwise product. A standard convention is to denote additively the composition law in $\Z\Omega$, and to write evaluations of weights as $t \mapsto \E^{\omega}(t)$ (instead of $\omega(t)$).
Let $\glie_\C$ be the complexification of the Lie algebra $\glie$ of $G$, and $\tlie_{\C}$ the complexification of the Lie algebra $\tlie$ of $T$. The map $\omega \mapsto \tang_{e_G}(\E^\omega)$ allows one to see the weights as elements of $\tlie_{\C}^*$. The dual of the Killing form restricted to $\R\Omega=\R \otimes_{\Z} \Z\Omega$ is positive-definite. Hence, one has a natural scalar product $\scal{\cdot}{\cdot}$ on the lattice of weights, which can be shown to be $W$-invariant, where $W=\mathrm{Norm}(T)/T$ is the Weyl group. We decompose the roots of $G$ (non-zero weights of the adjoint representation of $G$ on $\glie_{\C}$) in two disjoint sets $\Phi_+$ and $\Phi_-=-\Phi_+$ of positive and negative roots; then, 
$$\Z\Omega = \left\{\omega \in \R\Omega\,\,|\,\,\forall \alpha \in \Phi_+,\,\,2\,\frac{\scal{\alpha}{\omega}}{\scal{\alpha}{\alpha}}\in \Z \right\},$$
and on the other hand, the positive roots determine a cone 
$$C =\{\omega \in \R\Omega\,\,|\,\,\forall \alpha \in \Phi_+,\,\,\scal{\alpha}{\omega} \geq 0\}$$ 
in $\R\Omega$ which is a fundamental domain for the action of the Weyl group (the Weyl chamber). The intersection of the two aforementioned sets is then in bijection with $\hatG$:

\begin{theorem}[Weyl, 1925]
An irreducible unitary representation $(V,\rho)$ of $G$ admits a unique highest weight $\lambda$, which is maximal with respect to the partial order on weights induced by the cone $C$. This highest weight has multiplicity one and enables one to reconstruct the irreducible representation $(V,\rho)$. Moreover, $\lambda$ is an arbitrary dominant weight in $C\cap \Z\Omega$, so
$$\hatG = C \cap \Z\Omega.$$
The dimension of the representation $V^\lambda$ with highest weight $\lambda$ is given by the formula
$$d_\lambda = \frac{\prod_{\alpha \in \Phi_+} \scal{\alpha}{\rho+\lambda}}{\prod_{\alpha \in \Phi_+} \scal{\alpha}{\rho}},\quad\text{with }\rho = \frac{1}{2}\sum_{\alpha \in \Phi_+}\alpha.$$
\end{theorem}

\begin{example}
Suppose $G = \SU(3)$. A maximal torus is
$$T = \{\diag(t_1,t_2,t_3)\,\,|\,\,t_1t_2t_3=1,\,\,|t_1|=|t_2|=|t_3|=1\}.$$ 
The lattice of weights $\Z\Omega$ is spanned by the two fundamental weights $\E^{\omega_1}(t)=t_1$ and $\E^{\omega_2}(t)=(t_3)^{-1}$. The positive roots are $\E^{\alpha_1}(t)=t_1(t_2)^{-1}$, $\E^{\alpha_2}(t) = t_2(t_3)^{-1}$ and $\E^{\alpha_1+\alpha_2}(t)=t_1(t_3)^{-1}$. The dominant weights, which label the irreducible representations of $\SU(3)$, are the linear combinations $n_1\omega_1+n_2\omega_2$ with $n_1,n_2 \in \N$; on Figure \ref{fig:latticeweight}, they correspond to the dots that are included in the cone $C$.

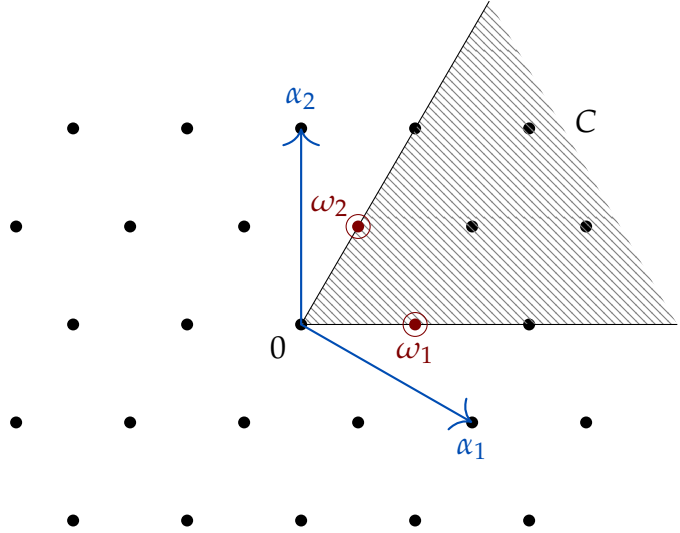
\begin{figure}[ht]
\begin{center}
\begin{tikzpicture}[scale=1.5]
\foreach \x in {(0,0),(2,0),(-1,0),(-2,0),(0,1.733),(1,1.733),(2,1.733),(-1,1.733),(-2,1.733),(0,-1.733),(1,-1.733),(2,-1.733),(-1,-1.733),(-2,-1.733),(2.5,0.866),(2.5,-0.866),(-2.5,0.866),(-2.5,-0.866),(1.5,0.866),(-0.5,0.866),(-1.5,0.866),(0.5,-0.866),(1.5,-0.866),(-0.5,-0.866),(-1.5,-0.866)}
\fill \x circle (1.5pt);
\draw [pattern = north west lines, pattern color=black!50!white] (60:3.3) -- (0,0) -- (3.3,0);
\draw (2.5,1.8) node {$C$};
\draw (-0.2,-0.2) node {$0$};
\draw [{<[scale=2]}-{>[scale=2]},thick,blue!70!green] (0,1.733) -- (0,0) -- (1.5,-0.866);
\draw [blue!70!green] (1.5,-1.1) node {$\alpha_1$};
\draw [blue!70!green] (0,2) node {$\alpha_2$};
\foreach \x in {(1,0),(0.5,0.866)}
{\draw [red!50!black] \x circle (3pt);
\fill [red!50!black] \x circle (1.5pt);}
\draw [red!50!black] (1,-0.25) node {$\omega_1$};
\draw [red!50!black] (0.25,1.05) node {$\omega_2$};
\end{tikzpicture}
\end{center}
\caption{The lattice of weights of the group $\SU(3)$.\label{fig:latticeweight}}
\end{figure}

\noindent The dimension of $V^\lambda$ with $\lambda = n_1\omega_1+n_2\omega_2$ is $d_\lambda=\frac{(n_1+1)(n_2+1)(n_1+n_2+2)}{2}$. For instance, the adjoint representation of $\SU(3)$ on $\mathfrak{sl}(3,\C)$ has highest weight $\lambda=\omega_1+\omega_2$, and dimension $8$. If one replaces the coordinates $(n_1,n_2)$ by the integer partition $\lambda=(\lambda_1\geq \lambda_2 \geq \lambda_3)$ with $\lambda_1 = n_1+n_2$, $\lambda_2=n_2$ and $\lambda_3=0$, one gets the classical formula
$$d_\lambda = \prod_{1\leq i<j\leq 3} \frac{\lambda_i-\lambda_j+j-i}{j-i}$$
which generalises to higher dimensions.
\end{example}

\begin{corollary}\label{cor:pregaussianlimit}
Let $\Gamma = \GEOM(N,L)$ be a random geometric graph of fixed level $L$ on a sscc Lie group $G$, and $A(N,L)$ be its adjacency matrix. In the sense of Theorem \ref{thm:GKlawoflargenumbers}, the limit of $\mathrm{Spec}(\frac{A(N,L)}{N})$ consists of one limiting eigenvalue $c_\lambda$ for each dominant weight $\lambda \in C \cap \Z\Omega$. The multiplicity of $c_\lambda$ is 
$$m_\lambda = \left(\frac{\prod_{\alpha \in \Phi_+} \scal{\alpha}{\rho+\lambda}}{\prod_{\alpha \in \Phi_+} \scal{\alpha}{\rho}}\right)^2,$$
and the value of $c_\lambda = \int_G Z_L(g)\,\chi^\lambda(g)\DD{g}$ will be given in Theorem \ref{thm:gaussianlimit}.
\end{corollary}

In the appendix (Section \ref{sec:appendix}), we give for each classical case (unitary groups, compact symplectic groups, spin groups): \begin{itemize}
\item a maximal torus $T$;
\item the corresponding weight lattice $\Z\Omega$ and the root system $\Phi$;
\item the dimension $d_\lambda$ of an irreducible representation $V^\lambda$ with $\lambda \in \hatG = C \cap \Z\Omega$.
\end{itemize}  
This allows one to make explicit Corollary \ref{cor:pregaussianlimit} and all the forthcoming theorems. In the examples hereafter, we shall focus on the groups $\SU(2)$ and $\SU(3)$. For $\SU(2)$, the weight lattice is drawn in Figure \ref{fig:lineweight}, and it is one-dimensional; many intuitions come from a detailed study of this toy-model.

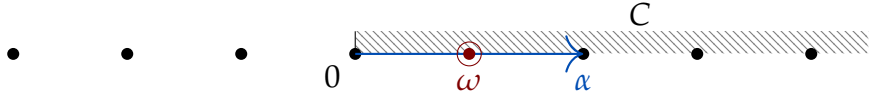
\begin{figure}[ht]
\begin{center}      
\begin{tikzpicture}[scale=1.5]
\draw [white,pattern = north west lines, pattern color=black!50!white] (0,0.2) -- (0,0) -- (4.5,0) -- (4.5,0.2);
\foreach \x in {(0,0),(2,0),(3,0),(-1,0),(-2,0),(-3,0),(4,0)}
\fill \x circle (1.5pt);
\draw (0,0) -- (0,0.2);
\draw (-0.2,-0.2) node {$0$};
\draw (2.5,0.35) node {$C$};
\draw [-{>[scale=2]},thick,blue!70!green] (0,0) -- (2,0);
\draw [blue!70!green] (2,-0.25) node {$\alpha$};
\foreach \x in {(1,0)}
{\draw [red!50!black] \x circle (3pt);
\fill [red!50!black] \x circle (1.5pt);}
\draw [red!50!black] (1,-0.25) node {$\omega$};
\end{tikzpicture}
\caption{The lattice of weights of the group $\SU(2)$. \label{fig:lineweight}}
\end{center}
\end{figure}

\begin{remark}
Corollary \ref{cor:pregaussianlimit} generalises readily to more general compact Lie groups, by replacing the set of dominant weights $C \cap \Z\Omega$ by an adequate sublattice of it. In particular, one can treat without additional work the case of the unitary groups $\mathrm{U}(n)$, which are not simple, since they have a non-trivial center; and the case of the special orthogonal groups $\SO(n)$, which are simple Lie groups but are not simply connected. In the appendix we detail this last case, where $\Z\Omega$ is replaced by an index $2$ sublattice (see Remark \ref{remark:fromspintoso}). Thus, though we shall not mention it again hereafter, every result obtained in the sequel whose statement starts by "Given a sscc Lie group\ldots" also holds \emph{mutatis mutandis} for the non-sscc but classical Lie groups $\mathrm{U}(n)$ and $\SO(n)$.
\end{remark}
\bigskip

\section{Asymptotics of the spectrum in the Gaussian regime}\label{sec:gaussian}

In this section, we compute the limiting eigenvalues $c_\lambda$ introduced in Corollary \ref{cor:pregaussianlimit}. We obtain a formula which involves Bessel functions of the first kind and an alternate sum over elements of the Weyl group, see Theorem \ref{thm:gaussianlimit}. These computations allow one for instance to estimate the spectral radius and the spectral gap of a random geometric graph $\GEOM(N,L)$ with fixed level $L$; see Section \ref{subsec:spectralgap}. On the other hand, we shall see in Section \ref{sec:ART} that the alternate sums involved in the formula for $c_\lambda$ degenerate in the Poissonian regime into certain partial derivatives. Therefore, the calculation of the eigenvalues $c_\lambda$ will be useful for studying both asymptotic regimes (Gaussian and Poissonian). In Section \ref{subsec:gaussnongroup}, we also explain how to extend our results to ssccss of non-group type; the computations become explicit in rank one and they involve Laguerre of Jacobi orthogonal polynomials.\medskip

\subsection{Distances on a compact Lie group}\label{subsec:distancesliegroup}
Since $c_\lambda = \int_G 1_{d(g,e_G) \leq L}\,\chi^\lambda(g)\DD{g}$, we need to explain how to deal with distances on a sscc compact Lie group $G$. We fix as before a maximal torus $T\subset G$, and we denote $\tlie \subset \glie$ the corresponding Lie subalgebra. Every element $g \in G$ is conjugated to an element $t \in T$, which is unique up to action of the Weyl group $W$. Consequently, as $Z_L$ is a function invariant by conjugation, in order to compute the function $Z_L(g)=1_{d(g,e_G)\leq L}$, it suffices to know its values on $T$. Now, the maximal torus is a totally geodesic flat submanifold of $G$, and the exponential map $\exp :\tlie \to T$ is locally isometric from a neighborhood of $0$ to a neighborhood of $e_G$. In all the classical cases, the injectivity radius of the exponential map is at least equal to $\pi$ (this is clear from the description of the maximal tori given in Section \ref{sec:appendix}). This enables one to reduce the calculation of $c_\lambda$ to an integration over a ball in the Euclidean space $\tlie$. Indeed, by Weyl's integration formula (see \cite[Chapters 17 and 22]{Bump13}), since $Z_L$ is invariant by conjugation,
$$
c_\lambda = \int_G 1_{d(g,e_G) \leq L}\,\frac{\mathrm{ch}^\lambda(g)}{d_\lambda}\DD{g} = \frac{1}{d_\lambda\,|W|}\int_T 1_{d(t,e_G)\leq L}\,\mathrm{ch}^\lambda(t)\,|\Delta(t)|^2\DD{t} 
$$
where 
\begin{itemize}
    \item $\!\DD{t}$ is the uniform probability over the torus $T$;
    \item $\Delta(t) = \sum_{w \in W} \eps(w)\,\E^{\rho}(w(t))$;
    \item for any $w \in W$ viewed as an element of $\SO(\R\Omega)$, $\eps(w)=(-1)^{\ell(w)}$ is the determinant of the transformation $w$, or equivalently the parity of the number of reflections with respect to the walls of the Weyl chamber that are needed to write $w$. 
\end{itemize}
Suppose $L< \pi$. Then, the integral can be taken over $\tlie$ instead of $T$:
\begin{equation}
c_\lambda = \frac{1}{d_\lambda\,|W|\,\vol(\tlie/\tlie_\Z)\,(2\pi)^{\dim \tlie}}\int_{\tlie} 1_{\|X\|_\glie \leq L}\, \mathrm{ch}^\lambda(\E^X)\,|\Delta(\E^X)|^2\DD{X}.
\label{eq:formulaclambda}
\end{equation}
Indeed, the probability measure $\!\DD{t}$ corresponds \emph{via} the exponential map to the rescaled Lebesgue measure 
$$\frac{1}{\vol(\tlie / 2\pi \tlie_\Z)} \DD{X} = \frac{1}{\vol(\tlie / \tlie_\Z)\,(2\pi)^{\dim \tlie}} \DD{X},$$ 
where $\!\DD{X}$ is the volume form on $\tlie$ which is associated to the Riemannian structure given by $\scal{\cdot}{\cdot}_\glie$; and $2\pi \tlie_\Z$ is the kernel of the exponential map $\exp : \tlie \to T$, and a lattice with maximal rank in $\tlie$. In the classical cases, the volumes $\vol(\tlie/\tlie_\Z)$ are computed in Section \ref{subsec:volumes}.
\medskip

\subsection{Asymptotics of the largest eigenvalues}\label{subsec:gaussianlimit}
In the sequel we always denote $d= \dim T=\rank\, G$ the \emph{rank} of the group $G$; in geometric terms, it is the dimension of a totally geodesic flat submanifold, and for a compact Lie group this is the dimension of a maximal torus. In the classical cases, we have $\rank(\SU(n)) = n-1$, and $\rank(\mathrm{Spin}(2n)) = \rank(\mathrm{Spin}(2n+1)) =\rank(\SP(n))=n$. If $\lambda$ is the highest weight of an irreducible representation $V^\lambda$, then the restriction of the corresponding character to the torus $T$ is given by Weyl's formula
$$\mathrm{ch}^\lambda(t) = \frac{\sum_{w \in W} \eps(w)\,\E^{\rho+\lambda}(w(t))}{\sum_{w \in W} \eps(w)\,\E^{\rho}(w(t))}.$$
Notice that the denominator in Weyl's character formula is the quantity $\Delta(t)$ previously introduced. Therefore, in Equation \eqref{eq:formulaclambda}, writing $|\Delta(t)|^2 = \Delta(t)\,\overline{\Delta(t)}$ makes $\Delta(t)$ appear in the numerator and the denominator. We can simplify it to get:
$$c_\lambda = \frac{1}{d_\lambda\,|W|\,\vol(\tlie/\tlie_{\Z})\,(2\pi)^d} \int_{X \in \tlie ,\,\|X\|_{\glie} \leq L} \left(\sum_{w_1,w_2 \in W} \eps(w_1)\,\eps(w_2)\,\E^{(w_1(\lambda+\rho)-w_2(\rho))(X)}\right) \DD{X} ,$$
As the measure $\!\DD{X}$ is invariant by $W$, one can gather the terms of the double sum according to the value $w=w_2w_1^{-1}$, and one obtains:
\begin{align*}
c_\lambda =\frac{1}{d_\lambda\,\vol(\tlie/\tlie_{\Z})\,(2\pi)^d} \sum_{w \in W} \eps(w) \int_{X \in \tlie ,\,\|X\|_{\glie} \leq L} \E^{(\lambda+\rho-w(\rho))(X)}\DD{X}.
\end{align*}
Each integral is a value of the Fourier transform of the unit ball in $\R^{d}$, that is a value of a Bessel function of the first kind. Indeed, recall that if $B^d$ is the unit ball in $\R^d$, we have
$$\int_{B^d} \E^{\I \scal{x}{\xi}} dx = \left(\frac{2\pi}{\|\xi\|}\right)^{\!d/2}\,J_{d/2}(\|\xi\|),$$
where $J_\beta$ is the Bessel function of the first kind of index $\beta$, defined by the power series
$$J_\beta(z) = \sum_{m=0}^\infty \frac{(-1)^m}{m!\,\Gamma(m+\beta+1)}\,\left(\frac{z}{2}\right)^{2m+\beta};$$
see Figure \ref{fig:bessel} for the case $\beta=1$.
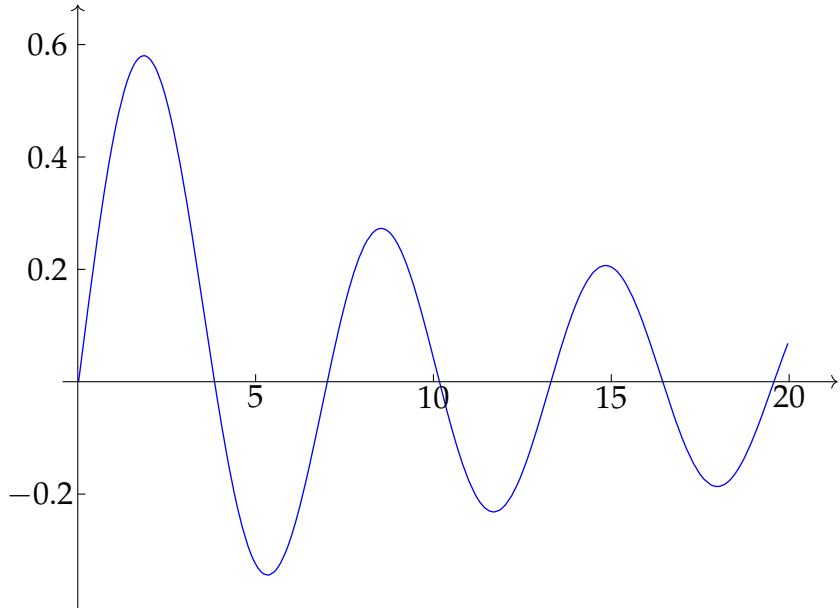
\begin{figure}[ht]
\begin{center}
\begin{tikzpicture}[scale=1]
\node[right] at (0,0) {\includegraphics[scale=0.5]{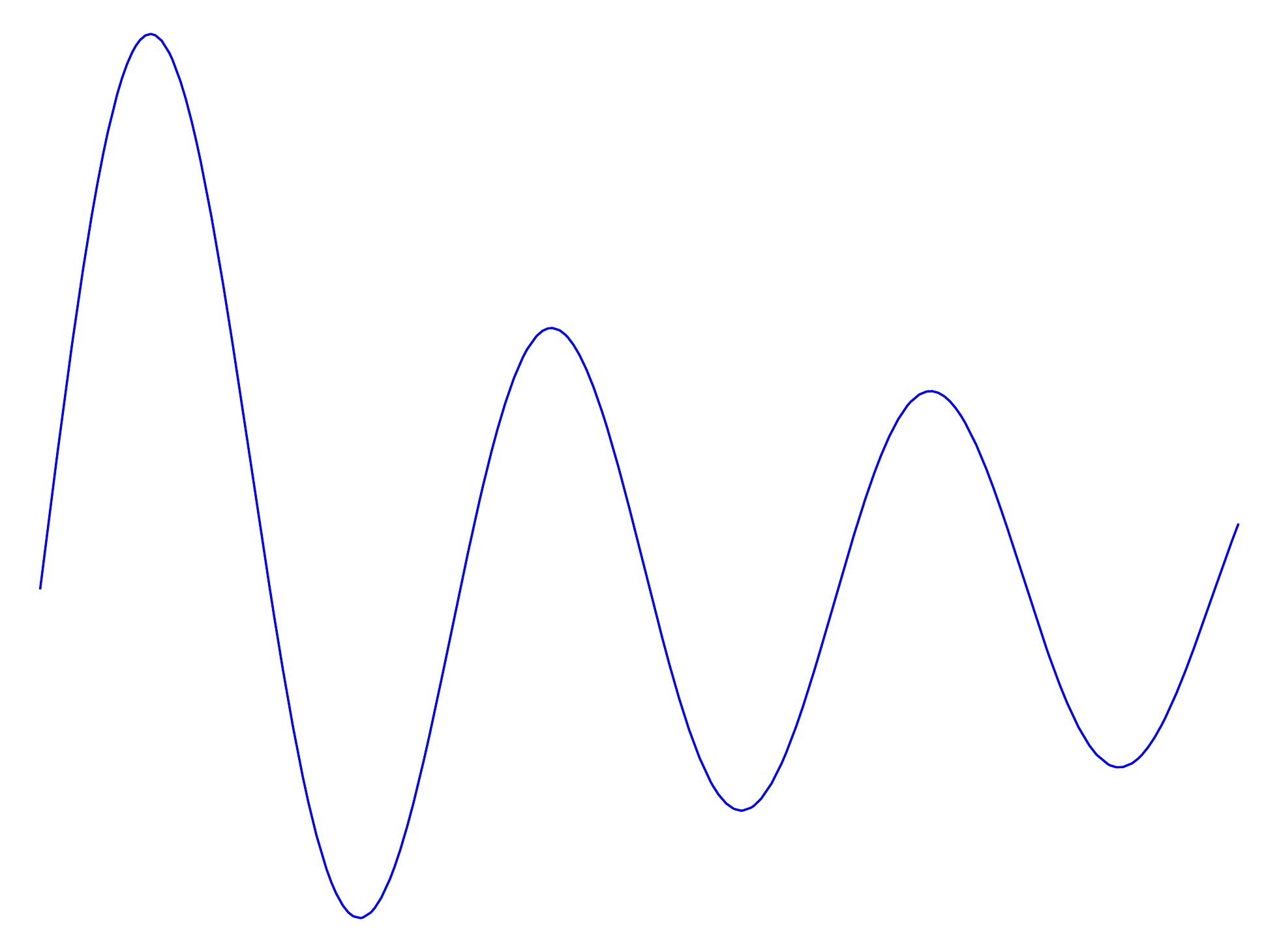}};
\begin{scope}[shift={(0.45,-0.88)}]
\draw [->] (-0.2,0) -- (10,0);
\draw [->] (0,-3) -- (0,5);
\draw (2.34,0) -- (2.34,0.1);
\draw (2.34,-0.2) node {$5$};
\draw (4.68,0) -- (4.68,0.1);
\draw (4.68,-0.2) node {$10$};
\draw (7.02,0) -- (7.02,0.1);
\draw (7.02,-0.2) node {$15$};
\draw (9.36,0) -- (9.36,0.1);
\draw (9.36,-0.2) node {$20$};
\draw (0,-1.49) -- (0.1,-1.49);
\draw (-0.5,-1.49) node {$-0.2$};
\draw (0,1.49) -- (0.1,1.49);
\draw (-0.4,1.49) node {$0.2$};
\draw (0,2.98) -- (0.1,2.98);
\draw (-0.4,2.98) node {$0.4$};
\draw (0,4.47) -- (0.1,4.47);
\draw (-0.4,4.47) node {$0.6$};
\end{scope}
\end{tikzpicture}
\end{center}
\caption{The Bessel function $J_1(x)$.\label{fig:bessel}}
\end{figure}

\noindent Given a weight lattice $\R\Omega$ of a sscc Lie group with rank $d=\rank(G)$, it is convenient to introduce the modified Bessel function 
$$\tildeJ_{\R\Omega}(x) = \frac{J_{d/2}(\|x\|)}{(\|x\|)^{d/2}} = \frac{1}{2^{d/2}}\sum_{m=0}^\infty \frac{(-1)^m}{m!\,\Gamma(m+\frac{d}{2}+1)}\,\left(\frac{\|x\|^2}{4}\right)^{m},$$
which is a $W$-invariant analytic function on $\R\Omega$. Then,
$$\frac{1}{(2\pi)^{d}}\int_{X \in \tlie,\, \|X\|_\glie \leq L} \E^{(\lambda+\rho-w(\rho))(X)}\DD{X} = \left(\frac{L}{\sqrt{2\pi}}\right)^{\!d}\,\tildeJ_{\R\Omega}(L\,(\lambda+\rho-w(\rho))) .$$
In this formula, the modified Bessel function involves the norm $\|\lambda+\rho-w(\rho)\|$, which is the norm of the weight lattice introduced in Section \ref{subsec:weightlattice}, and which is computed in the appendix for the classical cases. We have finally shown:

\begin{theorem}\label{thm:gaussianlimit}
Suppose that the level $L$ is smaller than $\pi$. Let $\lambda$ be a highest weight in $\hatG$. The eigenvalue $c_\lambda$ is given by the following formula:
$$c_\lambda = \frac{1}{d_\lambda\,\vol(\tlie/\tlie_{\Z})}\,\left(\frac{L}{\sqrt{2\pi}}\right)^{\!d} \sum_{w \in W} \eps(w)\,\tildeJ_{\R\Omega}(L\,(\lambda+\rho-w(\rho))),$$
where $d=\rank(G)$ is the dimension of a maximal torus $T \subset G$, and $\tildeJ_{\R\Omega}$ is the modified Bessel function on the weight space $\R\Omega$.
\end{theorem}

\begin{example}\label{ex:su2gaussian}
Consider $G=\SU(2)$. Its weight lattice $\Z\Omega$ is spanned by the fundamental weight $\E^\omega(\diag(\E^{\I\theta},\E^{-\I\theta})) = \E^{\I\theta}$. The norm of a weight $k\omega$ is $\frac{|k|}{2\sqrt{2}}$, and on the other hand, $\rho=\omega$, and $W = \sym(2) = \{\pm 1\}$. The volume $\vol(\tlie/\tlie_\Z)$ is $2\sqrt{2}$. Therefore, for $k \geq 1$ and $L <\pi$, 
\begin{align*}
c_k &= \frac{1}{4(k+1)} \left(\frac{L}{\sqrt{\pi}}\right) \left(\tildeJ_{\R\Omega}(kL\omega) - \tildeJ_{\R\Omega}((k+2)L\omega) \right)\\
&=\frac{1}{\pi (k+1)}  \left( \frac{1}{k} \sin \left(\frac{kL}{2\sqrt{2}}\right)  - \frac{1}{(k+2)} \sin \left(\frac{(k+2)L}{2\sqrt{2}}\right)  \right)
\end{align*}
since $J_{\frac{1}{2}}(x)=\sqrt{\frac{2}{\pi x}}\,\sin x$ and $\tildeJ_{\R\Omega}(x\omega) = \frac{4}{x\sqrt{\pi}}\sin (\frac{x}{2\sqrt{2}})$. For $k=0$, the formula specialises to
$$c_0 = \frac{1}{2\pi}\,\left(\frac{L}{\sqrt{2}} - \sin \left(\frac{L}{\sqrt{2}}\right)\right).$$
The multiplicity of the eigenvalue $c_k$ is equal to $(k+1)^2$ for any $k \geq 0$. 
\end{example}

\begin{example}
Suppose $G=\SU(3)$. The formula for $c_\lambda$ with $\lambda=n_1\omega_1+n_2\omega_2$ dominant weight in the Weyl chamber involves $6$ weights close to $\lambda$, namely, all the weights $\lambda+\mu$ with 
$\mu \in \{0, 3\omega_1,3\omega_2,2\omega_2-\omega_1,2\omega_1-\omega_2,2\omega_1+2\omega_2\},$
see Figure \ref{fig:besseldiscretelaplacian}. Thus,
$$c_{n_1,n_2} = \frac{L^2}{6\pi\sqrt{3}(n_1+1)(n_2+1)(n_1+n_2+2)} \left(\sum_{w \in \sym(3)} \eps(w) \,\frac{J_1(L\,\|\lambda+\rho-w(\rho)\|)}{L\,\|\lambda+\rho-w(\rho)\|}\right)$$
and each eigenvalue $c_{n_1,n_2}$ has multiplicity 
$m_{n_1,n_2}=(\frac{(n_1+1)(n_2+1)(n_1+n_2+2)}{2})^2.$
In this formula, the norm of a weight $k_1\omega_1+k_2\omega_2$ is $$\|k_1\omega_1+k_2\omega_2\| = \frac{1}{3}\, \sqrt{(k_1)^2 + k_1k_2 + (k_2)^2} .$$
\end{example}

\begin{figure}[ht]
\begin{center}      
\begin{tikzpicture}[scale=1]
\draw [thick] (3.5,2.6) -- (5,1.733) -- (6.5,2.6) -- (6.5,4.333) -- (5,5.2) -- (3.5,4.333) -- (3.5,2.6);
\foreach \x in {(6.5,2.6),(3.5,2.6),(5,5.2),(6.5,4.333),(3.5,4.333),(5,1.733)}
\fill [white] \x circle (3mm);
\foreach \x in {(0,0),(2,0),(3,0),(4,0),(5,0),(6,0),(7,0),(8,0),(9,0),(10,0),(9.5,0.866),(8.5,0.866),(7.5,0.866),(6.5,0.866),(5.5,0.866),(4.5,0.866),(3.5,0.866),(2.5,0.866),(1.5,0.866),(1,1.733),(2,1.733),(3,1.733),(4,1.733),(6,1.733),(7,1.733),(8,1.733),(9,1.733),(1.5,2.6),(2.5,2.6),(4.5,2.6),(5.5,2.6),(7.5,2.6),(8.5,2.6),(2,3.466),(3,3.466),(4,3.466),(5,3.466),(6,3.466),(7,3.466),(8,3.466),(2.5,4.333),(4.5,4.333),(5.5,4.333),(7.5,4.333),(3,5.2),(4,5.2),(6,5.2),(7,5.2),(3.5,6.066),(4.5,6.066),(5.5,6.066),(6.5,6.066),(4,6.933),(5,6.933),(6,6.933),(4.5,7.8),(5.5,7.8),(5,8.666)}
\fill \x circle (1.5pt);
\draw [pattern = north west lines, pattern color=black!50!white] (60:10.5) -- (0,0) -- (10.5,0);
\draw (10.3,1.2) node {$C$};
\draw (-0.2,-0.2) node {$0$};
\foreach \x in {(1,0),(0.5,0.866)}
{\draw [red!50!black] \x circle (3pt);
\fill [red!50!black] \x circle (1.5pt);}
\draw [red!50!black] (1,-0.3) node {$\omega_1$};
\draw [red!50!black] (0.2,1) node {$\omega_2$};
\foreach \x in {(6.5,2.6),(3.5,2.6),(5,5.2)}
\draw [blue!70!green] \x node {\Large $\boxplus$}; 
\foreach \x in {(6.5,4.333),(3.5,4.333),(5,1.733)}
\draw [blue!70!green] \x node {\Large $\boxminus$}; 
\draw [blue!70!green] (3.1,2.6) node {$\lambda$};
\end{tikzpicture}
\end{center}
\caption{The weights involved in the computation of $c_\lambda$ for $G=\SU(3)$.\label{fig:besseldiscretelaplacian}}
\end{figure}
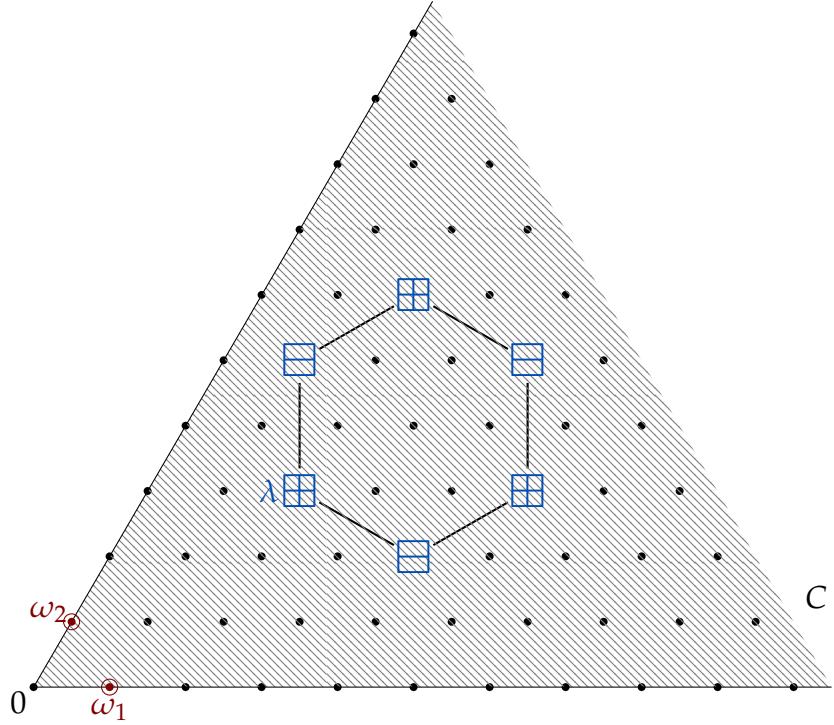

\subsection{Spectral radius and spectral gap}\label{subsec:spectralgap}
One thing that is not entirely clear from Theorem \ref{thm:gaussianlimit} is that the largest eigenvalues $c_\lambda$ \emph{correspond roughly} to the smallest dominant weights $\lambda$ in the Weyl chamber $C$. This is not a perfect correspondence: for instance, when $G=\SU(2)$, the dominant weights $k\omega$ with $k \geq 0$ yields constants $c_k=c_{k\omega}$ whose modules are not strictly decreasing with $k$. However, the two largest eigenvalues in this case are always $c_0$ and $c_1$, see the discussion later in this paragraph. One thing that is always true and easy to prove is that the largest eigenvalue corresponds to the zero weight:
\begin{proposition}
For any $L < \pi$, the eigenvalue $c_\lambda$ with the largest absolute value is obtained when $\lambda = 0$ is the trivial weight. Hence, the spectral radius of the graph $\GEOM(N,L)$ is asymptotically equivalent to
$$ \frac{N}{\vol(\tlie/\tlie_\Z)}\left( \frac{L}{\sqrt{2\pi}}\right)^{\rank(G)} \left(\sum_{w \in W} \eps(w)\,\tildeJ_{\R\Omega}\!\left(L\,(\rho-w(\rho))\right)\right).$$
\end{proposition}
\begin{proof}
The eigenvalue $c_\lambda$ is given by the integral $c_\lambda = \int_G Z_{L}(g)\,\chi^\lambda(g) \DD{g}$, with $Z_L(g)$ non-negative function, and $\chi^\lambda$ renormalised character that has always its module smaller than $1$. The maximum value is obtained when $\chi^\lambda(g)=1$ for every $g$, that is for the trivial representation of $G$.
\end{proof}
\medskip

\begin{example}
When $G=\SU(2)$, it is easy to prove that the two largest eigenvalues $c_k$ are always 
\begin{align*}
c_0 &= \frac{1}{2\pi}\,\frac{L}{\sqrt{2}}\left(1 - \sinc\!\left( \frac{L}{\sqrt{2}} \right)\right) ;\\
c_1 &= \frac{1}{2\pi}\,\frac{L}{2\sqrt{2}}\left(\sinc\!\left(\frac{L}{2\sqrt{2}}\right)-\sinc\!\left(\frac{3L}{2\sqrt{2}}\right)\right),
\end{align*}
where $\sinc(x) = \frac{\sin x}{x}$. Indeed, $L$ being fixed, the eigenvalue $c_k$ is proportional to the function
$$g_L = k\mapsto \frac{\sinc(k\ell)-\sinc((k+2)\ell)}{(k+1)\ell}$$
with $\ell = \frac{L}{2\sqrt{2}}$. For any value of $L \in (0,\pi)$, the function $g_L$ looks like the one of Figure \ref{fig:spectralgap}, and the two values of $g_L$ at $k=0$ and $k=1$ always fall on the first decreasing section of the curve. Hence, they yield the asymptotic spectral gap
$$\Delta_N=c_0(N)-c_1(N) \simeq \frac{N L }{2\sqrt{2}\pi} \left(1 - \sinc\!\left(\frac{L}{\sqrt{2}}\right) - \frac{1}{2}\left(\sinc\!\left(\frac{L}{2\sqrt{2}}\right)-\sinc\!\left(\frac{3L}{2\sqrt{2}}\right)\right)\right) $$
of a random geometric graph $\GEOM(N,L)$ on $\SU(2)$, with $L$ fixed and $N$ going to infinity. In general, a level $L$ being fixed, the map $\lambda \mapsto c_\lambda$ is proportional to
$$g_L(\lambda)=\frac{1}{\prod_{\alpha \in \Phi_+} \scal{L(\lambda + \rho)}{\alpha}}\,\sum_{w \in W} \eps(w)\,\tildeJ_{\R\Omega}\left(L(\lambda+\rho-w(\rho))\right).$$ 
In the definition of $g_L$, the alternate sum of modified Bessel functions is a discretisation of the partial derivative $((\prod_{\alpha \in \Phi_+} \partial_\alpha)\tildeJ_{\beta})(Lx)$; see Section \ref{subsec:onepoint}, where this argument will be made rigorous for the Poisson regime. The discrete partial derivative $g_L$ can be extended to the whole Weyl chamber $C$, and it is then an oscillating function that goes to $0$ as the norm of its parameter grows to infinity, in a fashion very similar to what happens for $\SU(2)$. 
We have drawn in Figure \ref{fig:beautifulGL} a function $g_L$ for $L>0$ and the group $G=\SU(3)$; here the oscillations are very small and barely visible. In this case, it is clear that the non-zero weights that yield the largest eigenvalues are $\omega_1$ and $\omega_2$, which correspond to the fundamental representations of the group.
\end{example}
\vspace*{-5mm}
\begin{figure}[ht]
\begin{center}      
\begin{tikzpicture}[scale=0.8]
\begin{axis}[xmin=-0.6,xmax=21,ymax=0.7,ymin=-0.2,
  axis lines=center,xscale=1.5,yscale=1.5]
\addplot [domain=0.01:20,thick,violet,smooth] {(sin(x*31.82)/(x*0.555) - sin((x+2)*31.82)/((x+2)*0.555) )/((x+1)*0.555)};
\fill [Red] (0,0.348) circle (0.7mm);
\draw [Red] (0,0.348) circle (1.4mm);
\fill [Red] (1,0.315) circle (0.7mm);
\draw [Red] (1,0.315) circle (1.4mm);
\end{axis}
\end{tikzpicture}
\caption{The function $g_L(x)$ for the group $G=\SU(2)$, with $L=\frac{\pi}{2}$; and the two largest integer values.\label{fig:spectralgap}}
\end{center}
\end{figure}
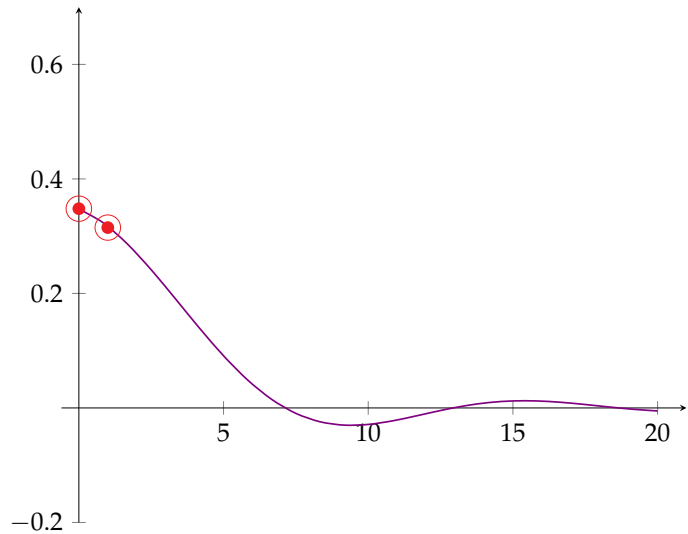
\begin{figure}[ht]
\includegraphics[scale=0.8]{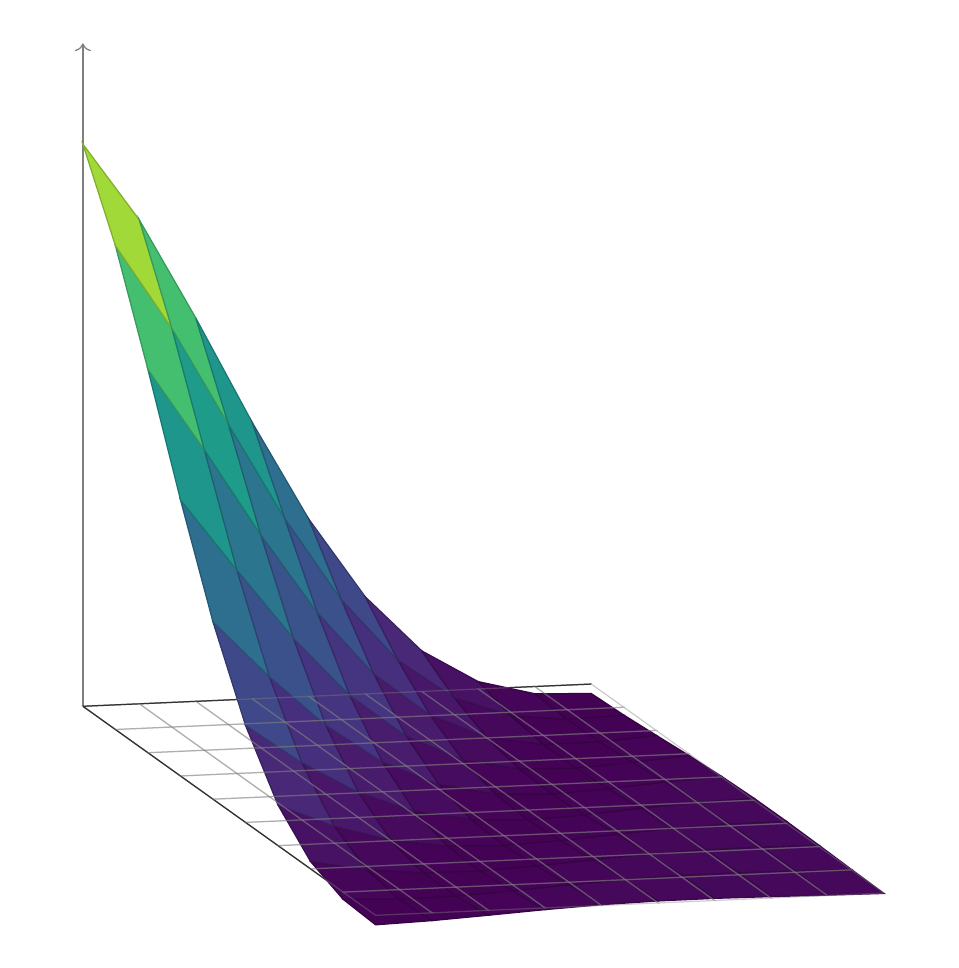}
\caption{A function $g_L$ for the group $G=\SU(3)$.\label{fig:beautifulGL}}
\end{figure}
\medskip

The only thing that might prevent one of the fundamental representations of the group to provide the second largest eigenvalue is if $L$ is too large, forcing the points of the lattice $L (\Z\Omega)$ that are neighbors of the weight $0$ to be at the bottom of the first oscillation of $g_L$. This forbids us to state a universal theorem for the spectral gap, though we are also unable to provide a counterexample. In practice, a level $L$ being fixed, one can use the asymptotic behavior of the Bessel functions to get rid of the points of the lattice $L (\Z\Omega)$ that are too far from $0$, and then there is only a finite number of values of $g_L$ to examine in order to determine the spectral gap. Thus, in most cases, the spectral gap of $\GEOM(N,L)$ is asymptotically equivalent to
$$\frac{N}{\vol(\tlie/\tlie_\Z)}\left( \frac{L}{\sqrt{2\pi}}\right)^{\rank(G)} \left(\sum_{w \in W} \eps(w)\,\left(\tildeJ_{\R\Omega}(L\,(\rho-w(\rho)))-\frac{1}{\dim \omega_1}\,\tildeJ_{\R\Omega}(L\,(\omega_1+\rho-w(\rho)))\right)\right),$$
where $\omega_1$ is the fundamental weight:
\begin{itemize}
    \item that corresponds in the classical cases to the geometric representation of the group of matrices,
    \item and that maximises $g_L$ most of the time.
 \end{itemize}  
\medskip

\subsection{Extension to symmetric spaces with rank one}\label{subsec:gaussnongroup}
In this paragraph, we explain how to adapt the arguments of the two previous sections to the case of a ssccss $X=G/K$ of non-group type. Roughly speaking, all the theoretical arguments from Section \ref{sec:representationtheory} adapt readily by replacing the irreducible representations of $G$ by the spherical representations of the pair $(G,K)$; on the other hand, the concrete computations from this section can be performed without too much additional work if the space $G/K$ has rank one, because in this case the zonal spherical functions are polynomials of one parameter.
\medskip

\paragraph{\textbf{Spherical representations and zonal functions}}
In the sequel, we fix a compact symmetric space $X=G/K$, and we denote $\!\DD{x}$ the unique $G$-invariant probability measure on $X$, which is the image of the Haar measure by the canonical projection from $G$ to $X$. We call \emph{spherical} an irreducible representation $V^\lambda$ of $G$ which admits a non-zero $K$-fixed vector, so $$(V^\lambda)^K=\{v \in V^\lambda\,|\,\forall k \in K,\,\,(\rho^\lambda(k))(v) =v\}$$ is not reduced to $\{0\}$. It can be shown that $(V^\lambda)^K$ has then dimension $1$, and also that the subset $\hatG^K \subset \hatG$ of spherical representations is the intersection of the Weyl chamber $C$ with a sublattice of the lattice of weights $\Z\Omega$; this is the Cartan--Helgason theorem, see \cite[Chapter V, Theorem 4.1]{Hel84}, as well as \cite{Sug62} and \cite[Section 12.3]{GW09}. Later, we shall only be interested in the case of compact symmetric spaces with rank one, in which case this sublattice has also rank one and will be explicitly described by Proposition \ref{prop:sphericalrepresentationsrankone}. Given a spherical representation $V^\lambda$ with $\lambda \in \hatG^K$, we fix a spherical vector $e^\lambda$ in $(V^\lambda)^K$ with $\scal{e^\lambda}{e^\lambda}_{V^\lambda}=1$. The vector $e^\lambda$ is unique up to multiplication by a complex number with modulus $1$. The (normalised) \emph{zonal spherical function} on $X$ associated to the spherical representation $V^\lambda$ is
 $$\zon^\lambda(g) = \scal{e^\lambda}{(\rho^\lambda(g))(e^\lambda)}_{V^\lambda} ; $$
 this function is bi-$K$-invariant on $G$, and it yields a $K$-invariant function on $X$. We have
 $$\int_G |\zon^\lambda(g)|^2\DD{g} = \int_X |\zon^\lambda(x)|^2\DD{x} = \frac{1}{d_\lambda}.$$
 The spherical transform of a bi-$K$-invariant function is defined for $f$ bi-$K$-invariant and $\lambda \in \hatG^K$ by
 $$f^\mathrm{sph}(\lambda) = \scal{e^\lambda}{(\widehat{f}(\lambda))(e^\lambda)}_{V^\lambda} = \int_G f(g)\,\zon^\lambda(g)\DD{g}.$$
We endow the space $\leb^2(\hatG^K,d_\bullet)=\bigoplus_{\lambda \in \hatG^K}^\perp\C$ with the coordinatewise product and with the Hilbert structure coming from the scalar product
$$\scal{a}{b} =\sum_{\lambda \in \hatG^K} d_\lambda \scal{a(\lambda)}{b(\lambda)},$$
where on the right-hand side we have the usual scalar product on $\C$.
The analogue of Theorem \ref{thm:peterweyl} in this setting is:
\begin{theorem}[Cartan]\label{thm:cartanpeterweyl}
The map $f\mapsto f^{\mathrm{sph}}$ from $\leb^2(K \backslash G / K,\!\DD{g})$ to $\leb^2(\hatG^K,d_\bullet)$
 is an isometry of Hilbert spaces and an isomorphism of commutative algebras. Moreover, if $c \in \leb^2(K\backslash G/K,\!\DD{g})$, then the convolution on the right
 \begin{align*}
 R_c : \leb^2(G/K,\!\DD{g}) &\to \leb^2(G/K,\!\DD{g})\\
        f &\mapsto f*c
 \end{align*}
 is a Hilbert--Schmidt operator; its eigenvalues $c^\mathrm{sph}(\lambda)$ are in correspondence with the spherical weights $\lambda \in \hatG^K$, each $c^\mathrm{sph}(\lambda)$ having multiplicity $d_\lambda$.
\end{theorem}
\noindent A reformulation of the first part of this theorem is that the zonal spherical functions form an orthogonal basis of $\leb^2(K \backslash G / K,\!\DD{g})$, with the convolution rule $\zon^\lambda * \zon^\mu = \frac{\delta_{\lambda,\mu}}{d_\lambda}\,\zon^\lambda$. We refer to \cite[Chapter V]{Hel84} for a proof of this result and a study of the spherical functions of a compact symmetric space; an analogous treatment for finite Gelfand pairs is provided by \cite[Chapter 4]{CSST08}, and the whole discussion from \emph{loc.~cit.}~adapts readily to compact symmetric spaces by replacing the finite sums by integrals against Haar measures. Now, in the setting of random geometric graphs with fixed level $L$ on a symmetric space of non-group type, the Giné--Koltchinskii theorem still applies, so the limit in the sense of Theorem \ref{thm:GKlawoflargenumbers} of the spectrum of $\frac{A(N,L)}{N}$ is the spectrum of the integral operator
\begin{align*}
T_h : \leb^2(X,\!\DD{x}) &\to \leb^2(X,\!\DD{x}) \\
f &\mapsto \left(T_h(f) : x \mapsto \int_X h(x,y)\,f(y)\DD{y}\right)
\end{align*}
with $h(x,y)=1_{d(x,y) \leq L}$. This operator writes as the right-convolution $R_c$ with $c(g)=Z_L(g)=1_{d(gK,K)\leq L}$ bi-$K$-invariant function on $G$. Consequently, the analogue in the setting of ssccss of non-group type of Proposition \ref{prop:abstractgaussian} is:

\begin{proposition}\label{prop:abstractgaussiannongauss}
Denote $Z_L(g)=1_{d(gK,K)\leq L}$ with $X=G/K$ ssccss of non-group type. The eigenvalues of the operator on $\leb^2(X)$ of convolution on the right by $Z_L$ are in bijection with the spherical representations $\lambda \in \hatG^K$. Each eigenvalue $c_\lambda$ has multiplicity $d_\lambda$ and is given by the formula $c_\lambda = \int_G Z_L(g)\,\zon^\lambda(g)\DD{g}$.
\end{proposition}
\medskip

\paragraph{\textbf{Symmetric spaces with rank one}}
The abstract result from Proposition \ref{prop:abstractgaussiannongauss} still holds if $X$ is connected but not necessarily simply connected, so in the sequel of this subsection we remove this assumption. Then, the zonal integrals from the previous proposition can be computed when $X$ has rank one, which is equivalent to one of the following assertions:
\begin{itemize}
    \item $X$ is a compact symmetric space with rank one, meaning that $X$ does not contain a totally geodesic flat submanifold with dimension at least $2$;
    \item $X$ is a $2$-point homogeneous compact connected Riemannian manifold: if $x_1,y_1,x_2,y_2 \in X$ satisfy $d(x_1,y_1)=d(x_2,y_2)$, then there exists an isometry $i : X \to X$ such that $i(x_1)=x_2$ and $i(y_1)=y_2$;
    \item $X$ is one of the following spaces: the real spheres $\R\Sph^{n \geq 1} = \SO(n+1)/\SO(n)$; the real, complex and quaternionic projective spaces 
    \begin{align*}
    &\R\Proj^{n\geq 2} = \SO(n+1)/\mathrm{O}(n);\\
    &\C\Proj^{n \geq 2}=\SU(n+1)/\mathrm{U}(n);\\
    &\Hq\Proj^{n \geq 2} = \SP(n+1)/(\SP(n)\times \SP(1));
    \end{align*}
     and the exceptional octonionic projective plane $\Octo\Proj^2 = \mathrm{F}_4 / \Spin(9)$.
\end{itemize}
We refer to \cite[Chapter 8]{Wolf67} for a proof of the equivalence between the two first assertions; the classification and the harmonic analysis of these spaces can be found in \cite{Grin83,AH10} and \cite[Chapter 3]{Vol09}. All these spaces are simply connected but the one-dimensional sphere $\R\Sph^1=\Tor$ and the real projective spaces $\R\Proj^{n\geq 2}$, which are twofold-covered by the real spheres $\R\Sph^n$. In the following, we endow $\R\Sph^{n}$ (respectively, $\mathbb{K}\Proj^n$ with $\mathbb{K} \in \{\R,\C,\Hq\}$) with its usual Euclidean coordinates $(x_1,x_2,\ldots,x_{n+1})$ with $\sum_{i=1}^{n+1}(x_i)^2 = 1$ (respectively, with its usual homogeneous coordinates $[x_1:x_2:\cdots:x_{n+1}]$). For the exceptional space $\Octo\Proj^2$, we have to be careful because of the non-associativity of the product of octonions, but there is a dense affine chart of $\Octo\Proj^2$ whose points $[\theta_1:\theta_2:1]$ are labeled by pairs $(\theta_1,\theta_2)$ of octonions, such that one can manipulate these homogeneous coordinates in exactly the same way as for the other projective spaces. In this affine chart, we set $\theta_3=1$; in general, a set of homogeneous coordinates $[\theta_1:\theta_2:\theta_3]$ is allowed for a point in $\Octo\Proj^2$ if the algebra spanned by the three octonions $\theta_1,\theta_2,\theta_3$ is associative, see \cite{John76,Adam96,Baez02} for details.\medskip

With these choices of coordinates, the group $K$ stabilises the base point $b=(0,0,\ldots,1)$ or $b=[0:0:\cdots:1]$, and the distance from $x=(x_1,x_2,\ldots,x_{n+1})$ or $x=[x_1:x_2:\cdots:x_{n+1}]$ to the base point is given by
$$d(x,b) = \begin{cases}
    \arccos (x_{n+1}) &\text{in the case of spheres},\\
    \arccos \left(\frac{|x_{n+1}|}{\sqrt{|x_1|^2+|x_2|^2+\cdots+|x_{n+1}|^2}}\right) &\text{in the case of projective spaces}.
\end{cases}$$
This formula differs by a multiplicative constant from the canonical Riemannian metric on the symmetric space $G/K$ which we detailed in Section \ref{subsec:normalisation}. On the other hand, with the same choice of coordinates, a bi-$K$-invariant function on the group $G$, which is a $K$-invariant function on the symmetric space $X$, is also a function of this single parameter $$
x= x_{n+1}\quad \text{or} \quad s=\frac{|x_{n+1}|^2}{|x_1|^2+\cdots+|x_{n+1}|^2}.$$
Consequently, the integrals from Proposition \ref{prop:abstractgaussiannongauss} are integrals over one single real parameter in $[-1,1]$ (real spheres) or in $[0,1]$ (projective spaces), whose distribution under the Haar measure is:
\vspace{2mm}
$$\renewcommand{\arraystretch}{2}
\begin{tabular}{|c|c|c|}
\hline $X$ & spherical coordinate & law of the spherical coordinate \\
\hline\hline $\R\Sph^n$ & $x=x_{n+1}$ & $\frac{\Gamma(\frac{n+1}{2})}{\Gamma(\frac{1}{2})\,\Gamma(\frac{n}{2})}\,(1-x^2)^{\frac{n}{2}-1}\,1_{x \in [-1,1]} \DD{x}$ \\
\hline $\R\Proj^n$ &$s = \frac{|x_{n+1}|^2}{|x_1|^2+\cdots+|x_{n+1}|^2}$ & $\frac{\Gamma(\frac{n+1}{2})}{\Gamma(\frac{1}{2})\,\Gamma(\frac{n}{2})}\,s^{-\frac{1}{2}}\,(1-s)^{\frac{n}{2}-1}\,1_{s\in[0,1]}\DD{s}\quad (\beta^{(\frac{1}{2},\frac{n}{2})})$ \\
\hline $\C\Proj^n$ &$s = \frac{|x_{n+1}|^2}{|x_1|^2+\cdots+|x_{n+1}|^2}$ & $\qquad n\,(1-s)^{n-1}\,1_{s \in [0,1]}\DD{s}\qquad\quad\,\,\, (\beta^{(1,n)})$  \\
\hline $\Hq\Proj^n$ &$s = \frac{|x_{n+1}|^2}{|x_1|^2+\cdots+|x_{n+1}|^2}$ & $2n(2n+1)\,s(1-s)^{2n-1}\,1_{s \in [0,1]}\DD{s}\quad (\beta^{(2,2n)})$ \\
\hline $\Octo\Proj^2$ &$s = \frac{|\theta_3|^2}{|\theta_1|^2+|\theta_2|^2+|\theta_3|^2}$ & $\quad \,\,1320\,s^3(1-s)^7\,1_{s \in [0,1]}\DD{s}\qquad\quad (\beta^{(4,8)})$ \\
\hline
\end{tabular}$$
\vspace{2mm}

\noindent In the case of a projective space $\mathbb{K}\Proj^n$, one obtains a $\beta$-distribution
$$\beta^{(a,b)}(\!\DD{s}) = \frac{\Gamma(a+b)}{\Gamma(a)\,\Gamma(b)}\,s^{a-1}\,(1-s)^{b-1}\DD{s}$$
with parameters $a$ and $b$ which depends on the field $\mathbb{K}$ of the projective space and on the rank $n$. The distribution of the spherical coordinate $x$ on the real sphere $\R\Sph^n$ will be denoted $\theta^n(\!\DD{x})$; its image by the map $x \mapsto x^2$ is the distribution $\beta^{(\frac{1}{2},\frac{n}{2})}$.\bigskip

By the remark stated just after the Cartan analogue \ref{thm:cartanpeterweyl} of the Peter--Weyl theorem, the zonal spherical functions $\zon^\lambda$ with $\lambda \in \hatG^K$ form an orthogonal basis of the space of $K$-invariant functions on $X$, with the normalisation condition $\zon^\lambda(e_G)=1$. On the other hand, the orthogonal polynomials with respect to the distribution $\theta^n(\!\DD{s})$ or $\beta^{a,b}(\!\DD{x})$ form an orthogonal basis of the space of functions of the parameter $x$ or $s$ which are square-integrable. Since the bi-$K$-invariant functions are functions of this single parameter, the two orthogonal bases correspond, and the following proposition describes these functions and the associated spherical representations.

\begin{proposition}\label{prop:sphericalrepresentationsrankone}
Consider a symmetric space $X=G/K$ with rank one. There is in each case a dominant weight $\omega_0 \in C \cap \Z\Omega$ such that $\hatG^K = \N \omega_0$:
\vspace{2mm}
\renewcommand{\arraystretch}{1.5}
$$\begin{tabular}{|c|c|c|}
\hline $X$ & $V^{\omega_0}$ & $d_{k\omega_0}$ \\
\hline \hline $\R\Sph^{n}$& geometric representation on $\C^{n+1}$& $\frac{2k+n-1}{k+n-1}\binom{k+n-1}{n-1}$\\
\hline $\R\Proj^{n}$ & $\mathfrak{so}^\perp(n+1,\C)\subset \mathfrak{sl}(n+1,\C)$ & $\frac{4k+n-1}{2k+n-1}\binom{2k+n-1}{n-1}$ \\
\hline $\C\Proj^{n}$  & adjoint representation $\mathfrak{sl}(n+1,\C)$ & $\frac{2k+n}{n} \binom{k+n-1}{n-1}^2$\\
\hline $\Hq\Proj^{n}$ & $\mathfrak{sp}^\perp(2n+2,\C) \subset \mathfrak{sl}(2n+2,\C)$ & $\frac{2k+2n+1}{(2n+1)(k+1)}\,\binom{k+2n}{2n}\binom{k+2n-1}{2n-1}$ \\
\hline $\Octo\Proj^{2}$ & tracefree part of the $27$-dimensional Albert algebra $\mathrm{A}(3,\mathbb{O})$ & $\frac{2k+11}{385}\binom{k+7}{4}\binom{k+10}{10}$\\
\hline
\end{tabular}$$ 
\vspace{2mm}

\noindent The corresponding zonal spherical functions are the \emph{Legendre polynomials} in the case of real spheres
$$P^{n,k}(x) = \frac{(-1)^k\,\Gamma(\frac{n}{2})}{2^k\,\Gamma(\frac{n}{2}+k)}\,\frac{1}{(1-x^2)^{\frac{n}{2}-1}}\,\frac{d^k}{\!\DD{x}^k}(1-x^2)^{\frac{n}{2}+k-1},\quad k\geq 0,$$
and the \emph{Jacobi polynomials}
$$J^{(a,b),k}(s)=\frac{(-1)^k\,\Gamma(b)}{\Gamma(b+k)}\,\frac{1}{s^{a-1}(1-s)^{b-1}}\,\frac{d^k}{\!\DD{s}^k}(s^{a+k-1}(1-s)^{b+k-1}),\quad k\geq 0,$$
both formulas being instances of Rodrigues' formula for orthogonal polynomials.
\end{proposition}

The first part of the proposition is an immediate application of the Cartan--Helgason theorem which identifies the spherical dominant weights; the second part is treated in \cite[Chapter 2]{AH10} in the case of spheres, and in \cite{Vol09} for the other spaces.
\medskip

\paragraph{\textbf{Asymptotics of the largest eigenvalues}}
By combining Proposition \ref{prop:abstractgaussiannongauss} and the explicit formula for zonal spherical functions from Proposition \ref{prop:sphericalrepresentationsrankone}, we can now compute the limiting eigenvalues of $\frac{A(N,L)}{N}$. Let us for instance treat the case of projective spaces. We have for $k \geq 1$:
\begin{align*}
c_{k\omega_0} &= \int_0^1 1_{\arccos(\sqrt{s})\leq L}\,J^{(a,b),k}(s)\,\beta^{(a,b)}(\!\DD{s}) \\
&= \int_{\cos^2 L}^1 \frac{(-1)^k\,\Gamma(a+b)}{\Gamma(a)\,\Gamma(b+k)}\,\frac{d^k}{\!\DD{s}^k}(s^{a+k-1}(1-s)^{b+k-1})  \DD{s} \\
&= \frac{(-1)^{k-1}\,\Gamma(a+b)}{\Gamma(a)\,\Gamma(b+k)}\,\left.\frac{d^{k-1}}{\!\DD{s}^{k-1}}\right|_{s=\cos^2 L}(s^{a+k-1}(1-s)^{b+k-1})\\
&=\frac{\Gamma(a+b)}{\Gamma(a)\,\Gamma(b+1)} \left. \left(s^a(1-s)^b\,J^{(a+1,b+1),k-1}(s)\right)\right|_{s=\cos^2 L}\\
& = \frac{\Gamma(a+b)}{\Gamma(a)\,\Gamma(b+1)} \,(\cos L)^{2a} (\sin L)^{2b}\,J^{(a+1,b+1),k-1}(\cos^2 L).
\end{align*}
For $k=0$, the spherical representation is the trivial one and $c_0$ is the mean of the function $1_{d(x,b)\leq L}$, hence the normalised volume in $X$ of this ball. The computations are analogous in the case of real spheres, with Laguerre polynomials instead of Jacobi polynomials. We conclude:
\begin{theorem}
Let $X$ be a sscc with rank one, and $L$ a level in $(0,\frac{\pi}{2})$. In the sense of Theorem \ref{thm:GKlawoflargenumbers}, the limit of $\mathrm{Spec}(\frac{A(N,L)}{N})$ consists of one eigenvalue $c_k$ for each $k \geq 0$, the multiplicity of $c_k$ being the dimension $d_{k\omega_0}$ computed in Proposition \ref{prop:sphericalrepresentationsrankone}. The eigenvalues $c_k$ are provided by the following table:

\renewcommand{\arraystretch}{2}
$$\begin{tabular}{|c|c|c|}
\hline $X$ & eigenvalue $c_0$ & eigenvalue $c_{k \geq 1}$  \\
\hline\hline $\R\mathbb{S}^n$ & $\frac{\int_0^L \sin^{n-1} \theta\,d\theta}{\int_0^\pi \sin^{n-1} \theta\,d\theta}$ & $\frac{(\sin L)^n}{ n  \,\int_0^\pi \sin^{n-1} \theta\,d\theta}\,P^{n+2,k-1}(\cos L)$ \\
\hline $\R\mathbb{P}^n$ & $\frac{\int_0^L \sin^{n-1} \theta\,d\theta}{\int_0^{\frac{\pi}{2}} \sin^{n-1} \theta\,d\theta}$ & $\frac{\Gamma(\frac{n+1}{2})}{\Gamma(\frac{1}{2})\,\Gamma(\frac{n+2}{2})}\,\cos L\,(\sin L)^n\,J^{(\frac{3}{2},\frac{n}{2}+1),k-1}(\cos^2 L)$ \\
\hline $\C\mathbb{P}^n$ & $(\sin L)^{2n}$ & $ (\cos L)^2\,(\sin L)^{2n}\,J^{(2,n+1),k-1}(\cos^2 L)$ \\
\hline $\Hq\mathbb{P}^n$ & $(\sin L)^{4n}\,(1+2n\,\cos^2 L)$ & $(2n+1)\,(\cos L)^4\,(\sin L)^{4n}\, J^{(3,2n+1),k-1}(\cos^2 L)$ \\
\hline $\mathbb{O}\mathbb{P}^2$ & $(\sin L)^{16}\,(\substack{1+8\cos^2 L \qquad\qquad\!\\+ 36 \cos^4 L + 120 \cos^6 L})$ & $165\,(\cos L)^{8}\,(\sin L)^{16}\,J^{(5,9),k-1}(\cos^2 L)$ \\
\hline
\end{tabular}$$
\end{theorem}
\vspace{2mm}

In particular, one can as in Section \ref{subsec:spectralgap} use these formulas in order to compute the asymptotic spectral radius and spectral gap of $\GEOM(N,L)$ in these cases.

\begin{example}
Let us treat the example from Figure \ref{fig:stereo}. The spherical representations for the real sphere $\R\Sph^2$ have dimension $2k+1$, $k \geq 0$; the corresponding zonal spherical functions are the classical Legendre polynomials $P^{2,k}(x)=\frac{1}{2^k\,k!}\,\frac{d^k}{\!\DD{x}^k}(x^2-1)^k$. The limiting eigenvalues of the rescaled adjacency matrix $\frac{A(N,L)}{N}$ of a random geometric graph with level $L$ on the sphere are
\begin{align*}
c_0 &= \frac{1-\cos L}{2}=\frac{\sin^2 L}{4}\,\frac{2}{1+\cos L};\\
c_k &= \frac{\sin^2 L}{4}\,P^{4,k-1}(\cos L),\quad\text{with }P^{4,k-1}(x) = \frac{1}{2^{k-1}\,k!}\,\frac{1}{x^2-1}\,\frac{d^{k-1}}{\!\DD{x}^{k-1}}(x^2-1)^k.
\end{align*}
Thus, up to the multiplicative factor $\frac{\sin^2 L}{4}$, all the limiting eigenvalues of the random geometric graph of level $L$ can be obtained by looking at the values at $x = \cos L$ of the family of functions
$$ \left\{f(x)=\frac{2}{1+x}\right\}\sqcup \{P^{4,k}(x),\,\,k\geq 0\}$$
see Figure \ref{fig:eigenvaluesrealsphere}.

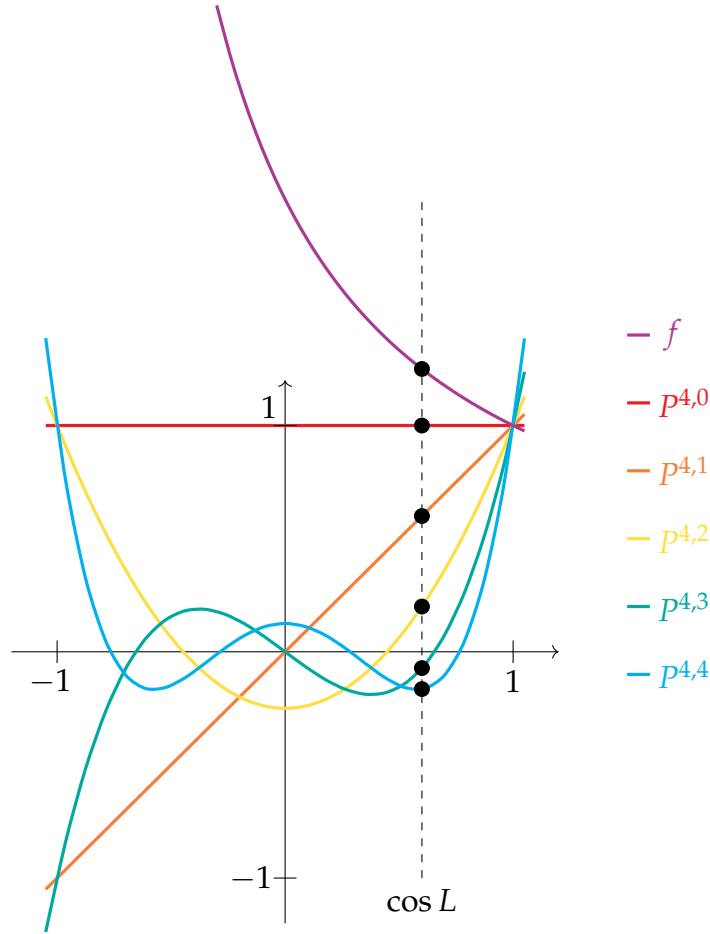
\begin{figure}[ht]
\begin{center}      
\begin{tikzpicture}[scale=3]
    \draw[->] (-1.2,0) -- (1.2,0);
    \draw[->] (0,-1.2) -- (0,1.2);
    \draw[line width=1.2pt,domain=-1.05:1.05,smooth,color = Red] plot (\x,1); 
    \draw[line width=1.2pt,domain=-1.05:1.05,smooth,color = Orange] plot (\x,\x); 
    \draw[line width=1.2pt,domain=-1.05:1.05,smooth,color = Goldenrod] plot (\x,{5/4*\x*\x - 1/4}); 
    \draw[line width=1.2pt,domain=-1.05:1.05,smooth,color = Emerald] plot (\x,{7/4*\x*\x*\x - 3/4*\x}); 
    \draw[line width=1.2pt,domain=-1.05:1.05,smooth,color = ProcessBlue] plot (\x,{21/8*\x*\x*\x*\x - 7/4*\x*\x + 1/8}); 
    \draw[line width=1.2pt,domain=-0.3:1.05,smooth,color = Mulberry] plot (\x,{2/(1+\x)});
    \draw (-0.05,1) -- (0.05,1) ;
    \draw (-0.05,-1) -- (0.05,-1) ;
    \draw (1,-0.05) -- (1,0.05) ;
    \draw (-1,-0.05) -- (-1,0.05) ;
    \draw[line width=1.2pt,color = Mulberry] (1.5,1.4) -- (1.6,1.4);
    \draw (1.7,1.4) node {\textcolor{Mulberry}{$f$}};
    \draw[line width=1.2pt,color = Red] (1.5,1.1) -- (1.6,1.1);
    \draw (1.75,1.1) node {\textcolor{Red}{$P^{4,0}$}};
    \draw[line width=1.2pt,color = Orange] (1.5,0.8) -- (1.6,0.8);
    \draw (1.75,0.8) node {\textcolor{Orange}{$P^{4,1}$}};
    \draw[line width=1.2pt,color = Goldenrod] (1.5,0.5) -- (1.6,0.5);
    \draw (1.75,0.5) node {\textcolor{Goldenrod}{$P^{4,2}$}};
    \draw[line width=1.2pt,color = Emerald] (1.5,0.2) -- (1.6,0.2);
    \draw (1.75,0.2) node {\textcolor{Emerald}{$P^{4,3}$}};
    \draw[line width=1.2pt,color = ProcessBlue] (1.5,-0.1) -- (1.6,-0.1);
    \draw (1.75,-0.1) node {\textcolor{ProcessBlue}{$P^{4,4}$}};
    \draw (-0.07,1.08) node {$1$};
    \draw (-0.15,-1) node {$-1$};
    \draw (1,-0.13) node {$1$};
    \draw (-1.03,-0.13) node {$-1$};
    \draw [dashed] (0.6,-1) -- (0.6,2);
    \draw (0.6,-1.1) node {$\cos L$};
    \foreach \x in {1.25,1,0.6,0.2,-0.072,-0.1648}
    \fill (0.6,\x) circle (0.35mm); 
\end{tikzpicture}
\caption{The limiting eigenvalues of a random geometric graph on the $2$-dimensional real sphere in the Gaussian regime.\label{fig:eigenvaluesrealsphere}}
\end{center}
\end{figure}
\end{example}

\section{Asymptotics of the graph and of its spectrum in the Poissonian regime}\label{sec:poisson}

In this section, we fix a ssccss $X$ (of group or non-group type), and we are interested in the asymptotic behavior of the spectrum of $\GEOM(N,L_N)$ when $N$ goes to infinity and $L_N$ goes to $0$ in the following prescribed way:
\begin{equation}
    L_N = \left( \frac{\ell}{N} \right)^{\frac{1}{\dim X}},\label{eq:poissonnorm}
\end{equation}
with $\ell >0$ fixed. As explained in the introduction, with this normalisation of $L_N$, the expected number of neighbors of a fixed vertex of $\GEOM(N,L_N)$ (for instance $v_1$) is asymptotic to
$$ \frac{c(\dim X)}{\vol(X)}\,\ell,$$
where $\vol(X)$ is the volume of the space $X$, and $c(\dim X)$ is the volume of a Euclidean unit ball in $\R^{\dim X}$. When $X=G$ is a Lie group, its volume $\vol(G)$ is computed in Section \ref{subsec:volumes}; for the other cases, we refer to \cite{AY97}. We now set 
$$\nu_N = \frac{1}{N} \sum_{i=1}^N \delta_{c_i(N)},$$
where $c_1(N) \geq c_2(N) \geq \cdots \geq c_N(N)$ are the eigenvalues of the adjacency matrix of $\GEOM(N,L_N)$. For each $N$, $\nu_N$ is a random element of $\meas^1(\R)$, the set of Borel probability measures on the real line. 
The remainder of this article focuses on studying the asymptotic behavior of the random spectral measures $\nu_N$. We shall in particular prove that there exists a probability measure $\mu \in \meas^1(\R)$ which depends only on $\ell$, $\vol(X)$ and $\dim X$, and such that
\begin{equation}
    \nu_N \rightharpoonup_{N \to +\infty} \mu,\label{eq:spectralconvergence}
\end{equation}
where $\rightharpoonup$ denotes the convergence in law; see Theorem \ref{thm:poissonlimit}. In Equation \eqref{eq:spectralconvergence}, the convergence occurs in probability; this makes sense since $\meas^1(\R)$ is a polish space for the topology of weak convergence, so in particular it is metrisable; see \cite[Chapter 1]{Bil99}. \medskip

\noindent There are at least two possible approaches in order to prove the convergence in law \eqref{eq:spectralconvergence}:
\begin{itemize}
    \item \textbf{Local Benjamini--Schramm convergence of the graphs} (this section). The notion of local convergence of graphs has been introduced formally in \cite{BS01} and \cite[Section 2]{AS04}; the idea appeared in several previous works for specific examples, see for instance \cite{Ald91}. More recently, a connection between this notion of convergence and the weak convergence of the spectral measures of the adjacency matrices has been established. We refer to \cite{BL10,ATV11,BLS11}, and to \cite[Proposition 2.2]{Bor16} for the most general result, which relies on arguments from the theory of von Neumann algebras. We recall briefly this theory in Section \ref{subsec:benjaminischramm}. In the setting of Poissonian random geometric graphs:
    \begin{enumerate}[label=(\roman*)]
        \item We have a random point process $(v_n)_{n \in \N}$ (the random vertices) which takes place on a space which is locally \emph{almost} isometric to an Euclidean vector space.
        \item As $n$ goes to infinity, this random point process has locally \emph{almost} the same statistics as a Poisson point process.
        \item The geometric graph built from this random point process converges then in the local Benjamini--Schramm sense, and this implies the weak convergence of the spectral measures.
    \end{enumerate}
    In the \emph{almost} correspondences listed above, the geometric graphs can be locally modified with a positive probability, so we have to be very careful if we want to prove rigorously the local convergence of our random geometric graphs. To this purpose, we solve a more general problem by giving a sufficient condition for a convergent random point process on a convergent sequence of metric spaces to give rise to a sequence of random graphs which is locally convergent (Theorem \ref{thm:convergencespacesandgraphs}).  In Section \ref{subsec:lipschitz}, we recall the notion of pointed Lipschitz convergence for proper metric spaces, and we present a similar notion of convergence for proper metric spaces endowed with a random point process. In Section \ref{subsec:infernal}, we relate these notions of convergence to the local convergence of random geometric graphs  under a mild regularity hypothesis. Our result implies in particular the Benjamini--Schramm convergence of the Poissonian random geometric graphs on a compact connected symmetric space $X$, the limit being the geometric graph drawn from a Poisson point process on $\R^{\dim X}$ (Theorem \ref{thm:poisson_BS}). This geometric argument combined with the aforementioned result from \cite{Bor16} implies the convergence of the spectral measures $\nu_N$ towards some probability measure $\mu$ (see Theorem \ref{thm:poissonlimit}). 
    For this result, we shall use in addition to the previous argument the fact that in a Poissonian random geometric graph, the neighborhoods of two vertices which are at macroscopic distance are asymptotically independent when $N$ goes to infinity.

    \item \textbf{Method of moments} (Section \ref{sec:ART}). Another more naive approach is to try to compute the moments of the measure $\nu_N$, and to prove that they all converge in probability towards the moments of a measure which is determined by its moments. We shall prove at the end of Section \ref{subsec:poissonlimit} that the limiting measure $\mu$ is indeed determined by its moments. Section \ref{sec:ART} proposes then a combinatorial method in order to compute these limiting moments, and it explains how the computation of these moments is related to the asymptotic representation theory of the Lie group $G$. As detailed in the introduction, we do not solve entirely the problem of the computation of the moments of $\mu$, but this alternative approach leads quite surprisingly to a general conjecture on certain functionals of the irreducible representations of the group.
\end{itemize}

\subsection{Benjamini--Schramm local convergence and continuity of the spectral map}\label{subsec:benjaminischramm}
The notion of local convergence of random graphs concerns random \emph{rooted} graphs. Since all the vertices of a random geometric graph play the same role, looking at rooted graphs $(\Gamma,r)$ instead of simple graphs $\Gamma$ will not be a problem hereafter. We denote $\gbullet$ the set of all \emph{connected locally finite rooted graphs} $(\Gamma,r)=(V,E,r)$:
\begin{itemize}
    \item $\Gamma = (V,E)$ is a simple graph, with $V$ possibly infinite but countable;
    \item $r \in V$ is a distinguished vertex and all the vertices of $\Gamma$ are connected to $v$ by a finite path;
    \item any vertex $v \in V$ has finite degree: $\deg v = |\{w\in V\,\,|\,\,v \sim w\}|<+\infty$.
\end{itemize}
We identify two connected locally finite rooted graphs $(\Gamma_1,r_1)$ and $(\Gamma_2,r_2)$ if there exists a bijective map $\phi : V_1 \to V_2$ such that $\phi(r_1)=r_2$ and such that $(v,w) \in E_1 $ if and only if $ (\phi(v),\phi(w)) \in E_2$. On the other hand, for $n \in \N$, we denote $\gbullet(n)$ the subset of $\gbullet$ that consists in rooted graphs $((V,E),r)$ where all the vertices $v \in V$ satisfy $d_\Gamma(v,r) \leq n.$
Here and in the sequel, the distance $d_\Gamma$ is the graph distance $d_\Gamma(v,w) = \min \{r \geq 0\,\,|\,\,v=v_0 \sim v_1 \sim v_2 \sim \cdots \sim v_r = w\}$. It should not be confused with the geodesic distance $d$ if the vertices of $\Gamma$ are points in a Riemannian manifold $X$.
We have a natural homomorphism of rooted graphs
\begin{align*}
\pi_n : \gbullet &\to \gbullet(n)\\
(\Gamma,r) &\mapsto (\Gamma(r,n),r),
\end{align*}
where $\Gamma(r,n)$ is the subgraph of $\Gamma$ whose vertices are the $v$'s in $ V$ such that $d_\Gamma(r,v)\leq n$, and whose edges are those of $\Gamma$ that connect vertices $v,w$ such that $d_\Gamma(r,v)\leq n$ and $d_\Gamma(r,w)\leq n$. For instance, if $(\Gamma,r)$ is the lattice $\Z^2$ rooted at the origin $r = (0,0)$, then 
\begin{center}
\begin{tikzpicture}[scale=1]
\draw (0,0) node {$\pi_3(\Z^2,(0,0)) = ~$};
\draw (4.8,0) node {$.$};
\draw (1.5,0) -- (4.5,0);
\draw (3,-1.5) -- (3,1.5);
\draw (2.5,-1) -- (2.5,1) -- (3.5,1) -- (3.5,-1) -- (2.5,-1);
\draw (2,-0.5) -- (2,0.5) -- (4,0.5) -- (4,-0.5) -- (2,-0.5);
\fill (3,0) circle (0.8mm);
\end{tikzpicture}
\end{center}
We endow $\gbullet$ with the following distance:
$$d_\bullet((\Gamma,r),(\Gamma',r')) = \frac{1}{1+\min \{n \in \N\,\,|\,\,\pi_n(\Gamma,r) = \pi_n(\Gamma',r')\}}.$$
It is known that $(\gbullet,d_\bullet)$ is a complete separable metric space. Moreover, the topology corresponding to $d_\bullet$ is the projective limit of the discrete topologies on the sets $\gbullet(n)$:
$$ \big((\Gamma_N,r_N) \to_{d_\bullet,\,N\to \infty} (\Gamma,r) \big) \iff \big(\forall n \in \N,\,\,\exists N_n \,\text{ such that }\,\forall N \geq N_n,\,\, \pi_n(\Gamma_N,r_N)  =  \pi_n(\Gamma,r)\big). $$ 
Let us now introduce randomness in this framework. Since $\gbullet$ is a polish space, the space $\meas^1(\gbullet)$ of Borel probability measures on the space of rooted graphs is again a polish space. We say that a sequence of random rooted graphs $(\Gamma_N,r_N)_{N \in \N}$ converges in the \emph{local Benjamini--Schramm sense} towards a random rooted graph $(\Gamma,r)$ if the probability distributions $L_{(\Gamma_N,r_N)}$ and $L_{(\Gamma,r)}$ of these random rooted graphs satisfy
$$L_{(\Gamma_N,r_N)} \rightharpoonup_{N \to \infty} L_{(\Gamma,r)}.$$
We have the following characterisation of the local Benjamini--Schramm convergence:
\begin{proposition}\label{prop:characterisation_BSconvergence}
A sequence of random rooted graphs $(\Gamma_N,r_N)_{N \in \N}$ converges in the local sense if and only if, for any $n \in \N$ and any rooted finite graph $\gamma_n \in \gbullet(n)$,
$$\lim_{N \to + \infty} \proba[\pi_n(\Gamma_N,r_N)=\gamma_n] = \proba[\pi_n(\Gamma,r)=\gamma_n].$$
\end{proposition}
\noindent This equivalence is stated without proof at the beginning of \cite{BS01}; it is relatively easy to prove once one remarks that any open subset of $\gbullet$ is a finite or countable \emph{disjoint} union of open balls.
\medskip

If $\Gamma$ is a finite graph on $N$ vertices, its spectral measure $\nu_\Gamma$ is $\frac{1}{N}\sum_{i=1}^N \delta_{c_i}$, where the $c_i$'s are the eigenvalues of the symmetric adjacency matrix $A_\Gamma$; thus, $\nu_N = \nu_{\GEOM(N,L_N)}$. This definition can be extended to certain infinite (random) rooted graphs as follows. Given $(\Gamma,r)=(V,E,r) \in \gbullet$, we can consider the adjacency operator
\begin{align*}
A_\Gamma : \ell^2_{\mathrm{c}}(V) &\to \ell^2_{\mathrm{c}}(V) \\
f &\mapsto \left(A_\Gamma f : w \mapsto \sum_{(v,w) \in E} f(v)\right)
\end{align*}
where $\ell^2_{\mathrm{c}}(V)$ is the space of finitely supported functions on $V$, which is dense in $\ell^2(V)$. This operator is self-adjoint, and it admits at least one self-adjoint extension to $\ell^2(V)$. If $\Gamma$ has a uniformly bounded degree, then the self-adjoint extension is unique and it is a continuous linear operator $\ell^2(V) \to \ell^2(V)$. However, in general, there might be several different self-adjoint extensions of $A_\Gamma$, and these extensions can be unbounded operators. We say that the graph $\Gamma$ or its adjacency operator $A_\Gamma$ is \emph{essentially self-adjoint} if the self-adjoint extension $A_\Gamma : \ell^2(V) \to \ell^2(V)$ is unique. In this case, given a root $r$ of $\Gamma$, we define the spectral measure $\mu_{(\Gamma,r)}$ of the rooted graph by the following formula:
$$\forall z \in \C_+,\,\,\,\scal{1_r}{(A_\Gamma-z)^{-1}(1_r)}_{\ell^2(V)} = \int_{\R} \frac{1}{x-z}\,\mu_{(\Gamma,r)}(\!\DD{x}),$$
where $\C_+$ denotes the upper half-plane. The measure $\mu_{(\Gamma,r)}$ is a Borel probability measure in $\meas^1(\R)$, and its existence and unicity is obtained by using Herglotz's representation theorem of holomorphic functions on the upper half-plane, and the standard properties of the resolvent of a self-adjoint (possibly unbounded) linear operator. We refer to \cite{Sch12} for the spectral theory of unbounded operators. 
\begin{remark}
 In this general setting, we cannot \emph{a priori} use the moments $M_{s\geq 1}=\scal{1_r}{(A_\Gamma)^{s}(1_r)}$ in order to define $\mu_{(\Gamma,r)}$. Indeed, without additional assumptions, these quantities might correspond to several different probability measures; see however the end of Section \ref{subsec:poissonlimit}.
\end{remark} 

Given a distribution $L \in \meas^1_{\mathrm{ess}}(\gbullet)\subset \meas^1(\gbullet)$ supported by essentially self-adjoint rooted graphs, we can finally define its spectral measure $\mu_L$ by $\mu_L=\esper_L[\mu_{(\Gamma,r)}]$. The process of taking the expectation of a random probability measure is what one expects: for any bounded continuous function $f$ on $\R$, $\mu_L(f) = \esper_L[\mu_{(\Gamma,r)}(f)]$. This formula defines a positive linear functional $\mu_L$ on $\mathscr{C}_\mathrm{b}(\R)$, and by \cite[Chapter IX, \S2, Theorem 2.3]{Lang93} this functional is uniquely determined by a Borel probability measure $\mu_L$ in $\meas^1(\R)$. We have thus defined a spectral map
\begin{align*}
\meas^1_{\mathrm{ess}}(\gbullet) &\to \meas^1(\R)\\
L &\mapsto \mu_L.
\end{align*}
This construction extends the notion of spectral measure of a finite graph. Indeed, given a finite graph $\Gamma$, let us denote $U(\Gamma)$ the \emph{uniformly pointed graph} constructed from $\Gamma$: it is the random connected finite graph $(\Gamma,r)$ with $r$ uniformly chosen among the vertices of $\Gamma$, and where we only keep the connected component of the root $r$. This measure is supported by connected finite graphs, which are of course essentially self-adjoint. An easy computation shows then that the measure $\mu_{U(\Gamma)}$ defined above is simply equal to $\nu_\Gamma$; see \emph{e.g.} \cite[beginning of Section 2.3]{Bor16}.
\begin{remark}
Above and also in the sequel, the spectral measures of finite graphs are denoted by the letter $\nu$, whereas the expected spectral measures of random essentially self-adjoint rooted graphs are denoted by the letter $\mu$. The reader should pay attention to the fact that in the first case the spectral measure $\nu$ can be random if the graph is random, whereas the notation $\mu$ is always used for deterministic measures.
\end{remark}

As far as we know, it is unknown whether the spectral map $\mu$ is continuous on the whole space $\meas^1_{\mathrm{ess}}(\gbullet)$ of essentially self-adjoint random rooted graphs, but the restriction to a smaller subspace is known to be continuous. A distribution $L\in \meas^1(\gbullet)$ is said \emph{unimodular} if, for any positive measurable function $f:\mathfrak{G}_{\bullet\bullet} \to \R_+$ on the set of locally finite bi-rooted graphs $(\Gamma,r_1,r_2)$, one has
$$\esper_L\!\left[\sum_{v \in V(\Gamma,r)} f(\Gamma,r,v)\right] = \esper_L\!\left[\sum_{v \in V(\Gamma,r)} f(\Gamma,v,r)\right].$$
Here, $\mathfrak{G}_{\bullet\bullet}$ is endowed with the smallest topology which makes the two projections $(\Gamma,r_1,r_2)\to (\Gamma,r_1)$ and $(\Gamma,r_1,r_2)\to (\Gamma,r_2)$ continuous towards $\gbullet$. For any finite graph $\Gamma$, $U(\Gamma)$ is unimodular, and conversely, a unimodular distribution $L$ of random rooted graphs which is supported by connected finite graphs is necessarily a mixture of uniformly pointed graphs $U(\Gamma)$.  It was shown by Benjamini and Schramm that the unimodular distributions form a closed subset $\meas^1_{\mathrm{uni}}(\gbullet) \subset \meas^1(\gbullet)$ for the local convergence; see \cite[Section 3.2]{BS01} or \cite[Lemma 2.1]{Bor16}.
\begin{theorem}\label{thm:bordenave}
    If $(\Gamma,r)$ is a random rooted graph chosen according to a unimodular distribution $L$, then $\Gamma$ is $L$-almost surely essentially self-adjoint. In other words, $\meas^1_{\mathrm{uni}}(\gbullet) \subset \meas^1_\mathrm{ess}(\gbullet)$. Then, the restriction of the spectral map $\mu$ to $\meas^1_{\mathrm{uni}}(\gbullet)$ is continuous with respect to the Benjamini--Schramm local convergence and to the weak convergence of measures.
\end{theorem}

These facts are proven in \cite[Proposition 2.2]{Bor16}. They imply in particular that if $(\Gamma_N)_{N \in \N}$ is a sequence of random graphs such that $U(\Gamma_N) \to L$ for some $L $ in $\meas^1(\gbullet)$ (and in fact in $\meas_\mathrm{uni}^1(\gbullet)$), then we have $\esper[\nu_{\Gamma_N}] \rightharpoonup_{N \to \infty} \mu_L$. Then, there is a simple criterion which allows one to get rid of the expectation and to obtain the convergence in probability of the spectral measures:
\begin{proposition}\label{prop:bordenave2}
Given a finite graph $\Gamma$, we denote $U^2(\Gamma)$ the law in $\meas^1(\mathfrak{G}_{\bullet\bullet})$ of a random bi-rooted graph $(\Gamma,r_1,r_2)$ with $(r_1,r_2)$ chosen uniformly among the pairs of vertices of $\Gamma$. Let $(\Gamma_N)_{N \in \N}$ be a sequence of random finite graphs such that $U^2(\Gamma_N) \to L \otimes L$, with $L \in \meas^1(\gbullet)$. Note that this implies in particular the local convergence $U(\Gamma_N) \to L$. Then, $\nu_{\Gamma_N} \rightharpoonup_{N \to \infty} \mu_L$ in probability.
\end{proposition}
\begin{proof}
See \cite[Theorem 1]{BL10} or \cite[Corollary 12]{BLS11}; the arguments in \emph{loc.~cit.}~are used in a less general setting, but they hold without modification.
\end{proof}
\medskip
 

\subsection{Pointed Lipschitz and random pointed Lipschitz convergence}\label{subsec:lipschitz}
In the remainder of this section, we fix a ssccss $X$, a parameter $\ell>0$, we consider the sequence of random geometric graphs $\Gamma_N = \GEOM(N,L_N)$, with $L_N$ as in Equation \eqref{eq:poissonnorm}. We denote $U(\Gamma_N)=(\Gamma_N,r_N)$ with $r_N$ uniformly chosen among the vertices of $\Gamma_N$, and where it is understood that we then only look at the connected component of this root $r_N$. We recall that $\nu_N$ is the (random) spectral measure of the random geometric graph $\GEOM(N,L_N)$, and we shall also denote $\mu_N=\esper[\nu_N]$; with the notations previously introduced, $\mu_N=\mu_{U(\GEOM(N,L_N))}$. The discussion of the previous section shows that the convergence in probability $\nu_N \rightharpoonup \mu$ and the deterministic convergence $\mu_N \rightharpoonup \mu$ are quite close results, and that the second (weaker) result is an immediate consequence of:

\begin{theorem}[Local convergence]\label{thm:poisson_BS}~

\begin{enumerate}
    \item The sequence $(\Gamma_N,r_N)_{N \in \N}$ converges in the local Benjamini--Schramm sense towards an infinite random rooted graph $(\Gamma_\infty,r)$. As a consequence, there exists a Borel probability measure $\mu$ on $\R$ such that $\mu_N \rightharpoonup_{N \to \infty} \mu$.
    \item The limit $(\Gamma_\infty,r)$ has the following distribution. We consider a Poisson point process $P$ on $\R^{\dim X}$ with intensity $\frac{\ell}{\vol (X)}\,\mathrm{Leb}$, where $\mathrm{Leb}$ is the standard Lebesgue measure. We take $r=0$ and we connect points of $P \sqcup \{0\}$ when their Euclidean distance is smaller than $1$. Then, $\Gamma_\infty$ has the distribution of the connected component of the root vertex $\{0\}$. In particular, the local limit depends only on $\dim X$ and on the parameter $\frac{\ell}{\vol(X)}$.
\end{enumerate}  

\end{theorem}
\begin{remark}
The arguments used hereafter adapt readily to any connected compact homogeneous Riemannian manifold $X=G/H$; in particular, since we do not use any argument from representation theory in this section, the assumption of simple connectedness on the symmetric spaces is here superfluous.
\end{remark}

Let us give an intuitive explanation of Theorem \ref{thm:poisson_BS}. When looking at the $s$-neighborhood (in the sense of graph distance) of a random root $r_N$ in $\GEOM(N,L_N)$, this $s$-neighborhood only depends on what happens in a small ball of radius $O(sL_N)$ around $r_N$, and this ball is \emph{almost} isometric to its Euclidean counterpart in $\R^{\dim X}$. Then, the restriction of the point process $\{v_1,\ldots,v_N\}$ to the small ball converges towards a Poisson point process, since each $v_i$ has a probability $O(\frac{1}{N})$ to be in the small ball, and since there are $N$ independent random points $v_1,\ldots,v_N$. Theorem \ref{thm:poisson_BS} is therefore a natural result, but let us insist on the fact that its rigorous proof cannot be made short, for the following reason. Since we only have a quasi-isometry between the small ball $B^X$ in $X$ and its tangent projection $B^{\mathrm{Euclidean}}$ in $\R^{\dim X}$:
\begin{quote}
 \emph{The projection in $B^{\mathrm{Euclidean}}$ of a geometric graph with level $L_N$ in $B^X$ \emph{is not} a geometric graph with level $L_N$ in $B^{\mathrm{Euclidean}}$. }
\end{quote}
Therefore, we need to be very careful with the various approximations involved in the previous intuitive explanation. 
To overcome the aforementioned difficulties, we shall see Theorem \ref{thm:poisson_BS} as a particular case of a more general result, which states roughly that if a sequence of pointed metric spaces converges in a suitable way, and if one chooses random points on these spaces in a way that is also convergent, then the corresponding geometric graphs converge under adequate assumptions in the local Benjamini--Schramm sense. The existence of such a result is not really surprising, but we could not find in the literature a set of sufficient conditions for the local convergence of graphs in this setting. The remainder of this subsection is devoted to the introduction of all the required hypotheses. In Section \ref{subsec:infernal}, we shall then show that these hypotheses lead to the aforementioned connection between convergence of metric spaces and convergence of random geometric graphs (Theorem \ref{thm:convergencespacesandgraphs}). We shall also prove in this paragraph and in the next one that all the required hypotheses for Theorem \ref{thm:convergencespacesandgraphs} are fulfilled in the Poissonian regime of random geometric graphs on a ssccss, leading to a proof of Theorem \ref{thm:poisson_BS}. \bigskip

We start by recalling the notion of pointed Lipschitz convergence \cite{Gro07}. Given two compact metric spaces $(X,d)$ and $(X',d')$, we say that they are \emph{Lipschitz equivalent} if there exists an homeomorphism $f : X \to X'$ such that
$$\forall x\neq y \in X,\,\,\,c \leq \frac{d'(f(x),f(y))}{d(x,y)} \leq C,$$
with $c$ and $C$ strictly positive constants. The Lipschitz distance between two compact metric spaces is then defined by
$$d_{\mathrm{L}}((X,d),(X',d')) = \inf_{f : X \to X'\text{ homeomorphism}} \left( |\log \mathrm{dil}(f)| + |\log \mathrm{dil}(f^{-1})|  \right),$$
where $\mathrm{dil}(f)$ denotes the dilation constant of an homeomorphism, defined by
$$\mathrm{dil}(f) = \sup_{x \neq y \in X}\frac{d'(f(x),f(y))}{d(x,y)}.$$
We convene that $d_{\mathrm{L}}((X,d),(X',d'))=+\infty$ if the two spaces $(X,d)$ and $(X',d')$ are not Lipschitz equivalent. Obviously, $d_{\mathrm{L}}((X,d),(X',d'))=0$ if and only if the two spaces $(X,d)$ and $(X',d')$ are isometric, so $d_{\mathrm{L}}$ is well defined on the set $\mathbf{CMS}$ of isometry classes of compact metric spaces. The topology associated to this notion of convergence does not have nice properties (polish space, \emph{etc.}), and it is much finer than the Gromov--Hausdorff topology (\emph{cf.} \cite[Section 3.11]{Gro07} and \cite[Chapter 10]{Pet06}). However, it is adequate in order to compare random point processes on metric spaces, and the corresponding random geometric graphs.\medskip

It is easy to adapt the definition of the Lipschitz distance to \emph{pointed} compact metric spaces: if $(X,x,d)$ and $(X',x',d')$ are two pointed compact metric spaces (compact metric spaces with a distinguished point), we define their pointed Lipschitz distance by
\begin{align*}
 &d_{\mathrm{L},\bullet} ((X,x,d),(X',x',d')) \\
 &= \inf_{f : X \to X'\text{ homeomorphism}} \left( |\log \mathrm{dil}(f)| + |\log \mathrm{dil}(f^{-1})| + d'(f(x),x') + d(x,f^{-1}(x'))\right).
\end{align*}
This metric yields a topology on the set $\mathbf{CMS}_\bullet$ of pointed isometry classes of pointed compact metric spaces. Next, we consider pointed \emph{proper} metric spaces, that is to say metric spaces $(X,d)$ with a distinguished point $x$ and such that every closed ball $B^X(x,R)$ with $R \geq 0$ is compact. Note that these hypotheses imply that every closed ball in $X$ is compact. We denote $\mathbf{PMS}_\bullet$ the set of pointed isometry classes of such spaces. For every $R \geq 0$, we have a natural map
\begin{align*}
\pi_R : \mathbf{PMS}_\bullet &\to\mathbf{CMS}_\bullet \\
(X,x,d) &\mapsto (B^X(x,R),x,d_{|B^X(x,R)}),
\end{align*}
and these maps allow one to endow $\mathbf{PMS}_\bullet$ with the \emph{topology of pointed Lipschitz convergence}: a sequence of pointed proper metric spaces $(X_N,x_N,d_N)_{N \in \N}$ converges to $(X,x,d)$ if and only if $\pi_R((X_N,x_N,d_N)) \to_{d_{\mathrm{L},\bullet}}  \pi_R((X,x,d))$ for any $R \geq 0$. This is the adequate definition that we shall use hereafter for convergence of metric spaces.

\begin{example}[Proposition 3.15 in \cite{Gro07}]
Let $(X,o)$ be a Riemannian manifold with a distinguished point $o$. Let $(t_N)_{N \in \N}$ be a sequence growing to infinity, and $d_N = t_N d$, $d$ being the geodesic distance on $X$. When $N \to \infty$, $(X,o,d_N)_{N \in \N}$ converges in the pointed Lipschitz topology towards $(\tang_oX, 0, d_{\tang_oX})$, where $d_{\tang_oX}$ is the Euclidean distance associated to the scalar product $\scal{\cdot}{\cdot}_{\tang_oX}$. 
\end{example}

Let $(X_N,x_N,d_N)_{N \in \N}$ be a sequence of pointed proper metric spaces, and $(M_N)_{N\in \N}$ be a sequence of random point processes on these spaces. We now want to define a notion of convergence for the whole sequence $(X_N,x_N,d_N,M_N)_{N \in \N}$. Let us first recall briefly the general theory of point processes; see \cite[Chapter 12]{Kal02} for more details. We warn the reader that in the following, all the measures considered will be positive measures; and all the results claimed hold only for proper metric spaces. If $(X,d)$ is a proper metric space, then:
\begin{itemize}
    \item It is locally compact, therefore, the positive \emph{Radon measures} (Borel measures that are locally finite and regular) are exactly the positive linear forms on the space of compactly supported continuous functions $\mathscr{C}_{\mathrm{c}}(X)$ (see again \cite[Chapter IX]{Lang93}).
    \item It is also $\sigma$-compact, hence polish, and in particular a Radon space. Thus, any locally finite measure on $X$ is regular, so the space $\meas_{\mathrm{Radon}}(X)$ of Radon measures on $X$ is simply the space of locally finite positive Borel measures.
\end{itemize} 
In this setting, we endow $\meas_{\mathrm{Radon}}(X)$ with the $*$-weak topology: a sequence of positive Radon measures $(\mu_n)_{n \in \N}$ converges towards a Radon measure $\mu$ if, for any $\phi \in \mathscr{C}_{\mathrm{c}}(X)$, $\mu_n(\phi) \to \mu(\phi)$. This topology is also called the vague topology, and we refer to \cite[Chapter 3]{Bour81} for a detailed study of it. With respect to the $\sigma$-field spanned by the vague topology, for any Borel subset $A \subset X$, the map 
\begin{align*}
\meas_{\mathrm{Radon}}(X) &\to \R_+ \sqcup \{+\infty\} \\
\mu &\mapsto \mu(A)
\end{align*}
is measurable. On the other hand, if $X$ is a proper metric space, then one can show that the vague topology on $\meas_{\mathrm{Radon}}(X)$ makes it a polish space (see \cite[Chapter 3, ex. 1.14.a]{Bour81} for the metrisability, and [\emph{loc.~cit.}, Chapter 3, Proposition 14] for the completeness). A \emph{random point process} is a random element of the measurable subset $\meas_{\mathrm{atomic}}(X) \subset \meas_{\mathrm{Radon}}(X)$ of atomic measures, which are the locally finite sums of Dirac measures
$$\left( M \in \meas_{\mathrm{atomic}}(X)  \right) \iff \left( M = \sum_{i\in I} \delta_{x_i}, \text{ with }\mu(K)<+\infty  \text{ for any compact subset }K\subset X \right).$$
A Poisson point process $P$ associated to an intensity $\mu \in \meas_{\mathrm{Radon}}(X)$ is an example of random point process, with $\esper[P(\phi)] = \mu(\phi)$ for any $\phi \in \mathscr{C}_c(X)$; see again \cite[Chapter 12]{Kal02}. We denote $\mathbf{PMS}_{\bullet,\star}$ the set of pointed proper metric spaces $(X,x,d)$ endowed with a random point process $M$. We identify two such objects $(X,x,d,M)$ and $(X',x',d',M')$ if there exists a bijective isometry $i : X \to X'$ such that $i(x)=x'$ and $i_\star M =_{\mathrm{(law)}} M'$. We say that a sequence $(X_N,x_N,d_N,M_N)_{N \in \N}$ in $\mathbf{PMS}_{\bullet,\star}$ \emph{converges in the random pointed Lipschitz sense} towards an element $(X,x,d,M)$ if, for any $R>0$, there exists an integer $N_0(R)$ such that for $N \geq N_0(R)$, one can find homeomorphisms $f_{N,R} : B^{X_N}(x_N,R) \to B^X(x,R)$ with the following properties:
\begin{enumerate}
    \item We have 
    $$\lim_{N \to \infty} \left( |\log \mathrm{dil}(f_{N,R})| + |\log \mathrm{dil}(f^{-1}_{N,R})| + d'(f_{N,R}(x_N),x) + d(x_N,f_{N,R}^{-1}(x))\right) = 0.$$
    In particular, this implies that $(X_N,x_N,d_N) \to (X,x,d)$ in the pointed Lipschitz topology.
    \item Consider a continuous function $\phi : X \to \R$ which is compactly supported on a ball $B^X(x,R)$. If $(M_N)_{|B^{X_N}(x_N,R)} = \sum_{i \in I_{N,R}} \delta_{x_{N,R,i}}$ and $M_{|B^X(x,R)} = \sum_{i \in I_R} \delta_{x_{R,i}}$, then we have the convergence in law
$$ \sum_{i \in I_{N,R}} \phi(f_{N,R}(x_{N,R,i})) \rightharpoonup_{N \to \infty} \sum_{i \in I_{R}} \phi(x_{R,i}). $$  
\end{enumerate}
Beware that the second condition is weaker than the statement $(f_{N,R})_{\star}((M_N)_{|B^{X_N}(x_N,R)}) \rightharpoonup M_{|B^{X}(x,R)}$ (in law and for the vague topology), because we only allow test functions $\phi$ that vanish on the boundary (and outside) of $B^{X}(x,R)$. 

\begin{proposition}\label{prop:rescalingliegroup}
Let $X$ be a ssccss. We denote $d$ the geodesic distance on $X$ with the normalisation given by Equation \eqref{eq:normalisation}; $t_N = N^{\frac{1}{\dim X}}$ and $d_N = t_N\,d$; $o=e_G$ the neutral element if $X=G$ is a group, and $o=\pi_X(e_G)$ the reference point if $X=G/K$ is not a group. We denote $M_N$ the point process on $X$ obtained by taking $N$ independent points at random according to the Haar measure. As $N$ goes to infinity,
\begin{equation}
    (X,o,d_N,M_N) \to \left(\tang_{o}X,0,d_{\tang_oX},\mathcal{P}\!\left(\frac{1}{\vol(X)}\right)\right),\label{eq:randompointedlipschitz}
\end{equation}
where $d_{\tang_oX}$ is the Euclidean distance on $\tang_oX$ associated to the opposite Killing form or its restriction, and $\mathcal{P}(\lambda)$ is the Poisson point process on $\tang_oX$ whose intensity is $\lambda\,\mathrm{Leb}$, $\mathrm{Leb}$ being the Lebesgue measure associated to the distance $d_{\tang_oX}$. The convergence in Equation \eqref{eq:randompointedlipschitz} is in the random pointed Lipschitz sense.
\end{proposition}

\begin{proof}
In the sequel, since we shall consider families of distances on $X$, in order to avoid any ambiguity, we shall indicate the distance $d$ of a ball $B(x,r)=B^{(X,d)}(x,r)$ in $X$. We also denote $\mathfrak{x}=\tang_oX$, which is the Lie algebra of $X$ if $X$ is a group, and a subspace of the Lie algebra of $G$ if $X=G/K$ is of non-group type. Fix $R>0$. By the aforementioned result from \cite[Proposition 3.15]{Gro07}, there exist some bijective maps $f_{N,R} : B^{(X,t_N d)}(o,R) \to B^\mathfrak{x}(0,R)$ which are smooth, which send the reference point $o$ to $0$, and such that
$$\lim_{N \to \infty} \big( |\log \mathrm{dil}(f_{N,R})| + |\log \mathrm{dil}(f^{-1}_{N,R})| \big)=0.$$
Let $\phi$ be a bounded measurable function compactly supported on $B^{\mathfrak{x}}(0,R)$. If $P=\mathcal{P}(\frac{1}{\vol(X)}) $ is the Poisson process on $\mathfrak{x}$ with intensity $\frac{1}{\vol(X)}\,\mathrm{Leb}=\frac{\!\DD{t}}{\vol(X)}$, then the Laplace transform of $P(\phi)$ is given by the Campbell formula:
$$\esper[\E^{z\,P(\phi)}] = \exp\left(\frac{1}{\vol(X)}\int_{B^\mathfrak{x}(0,R)} (\E^{z\,\phi(t)}-1)\,\DD{t}\right) .$$
On the other hand, the Laplace transform of $((f_{N,R})_\star(M_N)_{|B^{(X,d_N)}(o,R)})(\phi)$ is
\begin{align*}
\esper\!\left[\exp \left(z\, ((f_{N,R})_\star(M_N)_{|B^{(X,d_N)}(o,R)})(\phi) \right)  \right] 
&= \esper\!\left[ \exp\left(z\sum_{i=1}^N 1_{d_N(o,v_i) \leq R}\,\phi(f_{N,R}(v_i))\right) \right] \\
&= \left(\int_X \exp\big(z\,1_{d_N(o,x) \leq R}\,\phi(f_{N,R}(x))\big) \DD{x}\right)^N\\
&= \left(1+\int_{B^{(X,d_N)}(o,R)} (\E^{z\,\phi(f_{N,R})(x)}-1) \DD{x}\right)^N.
\end{align*}
Since $f_{N,R}$ is a smooth quasi-isometry, by the change of variables formula, the image by the map $f_{N,R}$ of the restriction of the Haar measure to the ball $B^{(X,d_N)}(o,R) = B^{(X,d)}(o,R(t_N)^{-1})$  is a measure
$$ m_{N,R}(t)\,1_{\|t\|_\mathfrak{x} \leq R} \,\frac{1}{N\,\vol (X)}\,\DD{t},$$
where $m_{N,R}(t)$ is a smooth positive function that converges uniformly to $1$ on $B^\mathfrak{x}(0,R)$. Then, the change of variables $t = f_{N,R}(x)$ yields 
\begin{align*}
\esper\!\left[\exp\! \left( z\,((f_{N,R})_\star(M_N)_{|B^{(X,d_N)}(o,R)})(\phi) \right)  \right]
&= \left(1+\frac{1}{N\,\vol(X)}\int_{B^{\mathfrak{x}}(0,R)} \!\!\!m_{N,R}(t)\,(\E^{z\,\phi(t)}-1)\DD{t} \right)^{\!N} \\
&= \exp\!\left(\frac{1}{\vol(X)}\int_{B^{\mathfrak{x}}(0,R)} (\E^{z\,\phi(t)}-1)\DD{t}\right)\,(1+o(1)).
\end{align*}
This ensures the convergence in law of the restricted point processes, hence the convergence in the random pointed Lipschitz sense. Notice that, since the previous convergence holds for any bounded measurable function $\phi$, given a family $(\phi_1,\ldots,\phi_s)$ of bounded measurable functions compactly supported on $B^\mathfrak{x}(0,R)$, we also have
$$\esper\!\left[\exp \left( \sum_{i=1}^s z_i ((f_{N,R})_\star(M_N)_{|B^{(X,d_N)}(o,R)})(\phi_i) \right) \right] = \esper\!\left[\exp\left(\sum_{i=1}^s z_i\,P(\phi_i)\right)\right] (1+o(1)), $$
hence the convergence in law of the whole vector of observables of the random point process $(f_{N,R})_\star(M_N)_{|B^{(X,d_N)}(o,R)}$ towards the vector of observables of the Poisson point process.
\end{proof}
\medskip

\subsection{Convergence of metric spaces and convergence of random geometric graphs}\label{subsec:infernal}
Given a pointed proper metric space $(X,x,d)$ endowed with a random point process $M$, for any $L>0$, we can consider the geometric graph $\GEOM(X,x,d,M,L)$ whose vertices are the points of $\{x\} \cup \{\text{atoms of }M\}$, and whose edges are the pairs of points $(y,z) \in \{x\}\cup \{\text{atoms of }M\}$ such that $d(y,z) \leq L$. Here and in the sequel, we assume that $M$ is a simple random point process (with probability $1$, $M$ does not involve atoms with multiplicity greater than $1$), and that $M(x)=0$ almost surely. This assumption ensures that $\GEOM(X,x,d,M,L)$ is a simple graph without multiple edge. The geometric graph $\GEOM(X,x,d,M,L)$ is naturally rooted at $x$. Since $M$ is almost surely locally finite and $(X,d)$ is proper, it is easy to see that (the connected component of $x$ in) $\GEOM(X,x,d,M,L)$ is almost surely locally finite, hence an element of $\gbullet$. Now, given a sequence $(X_N,x_N,d_N,M_N)_{N \in \N}$ in $\mathbf{PMS}_{\bullet,\star}$ that converges in the random pointed Lipschitz sense towards $(X,x,d,M)$, it is natural to ask whether this ensures the local Benjamini--Schramm convergence of the random rooted graphs $\GEOM(X_N,x_N,d_N,M_N,L)$ towards $\GEOM(X,x,d,M,L)$. Under an additional hypothesis of regularity on the limiting point process $M$, the answer is yes.

\begin{definition}
A family $(X,x,d,M) \in \mathbf{PMS}_{\bullet,\star}$ consisting in a pointed proper metric space and a random point process on it is said \emph{regular} if, for any fixed $L>0$, the increasing map
    \begin{align*}
    \R_+ &\to \gbullet\\
    l &\mapsto \GEOM(X,x,d,M,l)
    \end{align*}
is almost surely continuous at $l=L$ with respect to the local Benjamini--Schramm topology on $\gbullet$.
\end{definition}

\begin{theorem}\label{thm:convergencespacesandgraphs}
Let $(X_N,x_N,d_N,M_N)_{N \in \N}$ be a sequence in $\mathbf{PMS}_{\bullet,\star}$ that converges in the random pointed Lipschitz sense towards $(X,x,d,M)$. We assume that the pointed  proper metric space $(X,x,d)$ and its point process $M$ are regular. Then, for any $L>0$ fixed, the sequence of random graphs $(\GEOM(X_N,x_N,d_N,M_N,L))_{N \in \N}$ converges in the local Benjamini--Schramm sense towards $\GEOM(X,x,d,M,L)$.
\end{theorem}

Note that the regularity assumption is equivalent to the fact that almost surely, no pair of points $(y,z)$ of $\{x\}\sqcup \{\text{atoms of }M\}$ are exactly at distance $L$ for a fixed positive real number $L$. Let us first see why this general result implies Theorem \ref{thm:poisson_BS}:
\begin{proof}[Proof of Theorem \ref{thm:poisson_BS}]
With $t_N = N^{\frac{1}{\dim X}}$, we already know that $(X,r_N,t_N d,\,M_{N})_{N \in \N}$ converges in $\mathbf{PMS}_{\bullet,\star}$ towards 
$$\left(\R^{\dim X},0,d_{\mathrm{Euclidean}},\delta_0 + \mathcal{P}\!\left(\frac{1}{\vol(X)}\right)\right).$$
The differences between this statement and Proposition \ref{prop:rescalingliegroup} are the following:
\begin{itemize}
     \item We have replaced the reference point $o$ by a random atom $r_N$ of $M_N$, which is added to $M_{N-1}$.
     \item We place ourselves on $(\R^{\dim X},0,d_{\mathrm{Euclidean}})$ instead of $(\tang_oX,0,d_{\tang_oX})$.
 \end{itemize} 
 However, the second modification amounts to an isometry between $\mathfrak{x}=\tang_oX$ and $\R^{\dim X}$, whereas the first point is clearly solved by using the transitive action of the compact Lie group $G$ associated to $X$. Now, the limiting space is obviously regular with respect to random geometric graphs, because given $L > 0$, there is almost surely no atom of the Poisson point process exactly at distance $L$ from another atom, or at distance $L$ from $0$. Therefore, we have the local convergence
 $$\GEOM(X,r_N,t_N d,M_N,\ell^{\frac{1}{\dim X}}) \to \GEOM\left(\R^{\dim X},0,d_{\mathrm{Euclidean}},\delta_0+\mathcal{P}\!\left(\frac{1}{\vol(X)}\right),\ell^{\frac{1}{\dim X}}\right).$$
 By scaling, the left-hand side is also $\GEOM(X,r_N,d,M_N,L_N)=(\Gamma_N,r_N)$, whereas the right-hand side has the same law as $\GEOM\left(\R^{\dim X},0,d_{\mathrm{Euclidean}},\delta_0+\mathcal{P}\!\left(\frac{\ell}{\vol(X)}\right),1\right)$.
\end{proof}
\medskip

We now turn to the proof of Theorem \ref{thm:convergencespacesandgraphs}, which we split in several lemmas. A first consequence of the regularity assumption is that the random increasing map
$r \mapsto M(B^X(x,r))$
is almost surely continuous at $r=R$, for any fixed radius $R$. Indeed, this amounts to the almost sure continuity of the map
 $l \mapsto \card (\pi_1(\GEOM(X,x,d,M,l)))$
at $l=R$. A generalisation of this property will be stated in Lemma \ref{lem:emptyboundary}. A less trivial consequence of the assumptions of Theorem \ref{thm:convergencespacesandgraphs} is the following:

\begin{lemma}\label{lem:separated}
Suppose that the atoms of $M+\delta_x$ are simple, and consider $(X_N,x_N,d_N,M_N)_{N \in \N}$ sequence in $\mathbf{PMS}_{\bullet,\star}$ that converges in the random pointed Lipschitz sense to $(X,x,d,M)$. Then, each of the random point processes $M_N+\delta_{x_N}$ and $M+\delta_x$ is uniformly separated: for any $R> 0$ and any $\eps>0$, there exists $\eta>0$ and an integer $N_0=N_0$ such that 
\begin{align*}
\forall N \geq N_0,\,\,\,\proba[\text{$M_N+\delta_{x_N}$ has two atoms in $B^{X_N}(x_N,R)$ at distance smaller than $\eta$}] &\leq \eps;\\
\proba[\text{$M+\delta_{x}$ has two atoms in $B^{X}(x,R)$ at distance smaller than $\eta$}] &\leq \eps.
\end{align*}
\end{lemma}

\begin{proof}
In $B^{X}(x,R)$, we fix a finite sequence $(y_k)_{1\leq k \leq K}$ that is $\eta$-dense: $B^{X}(x,R) \subset \bigcup_{k =1}^K B^X(y_k,\eta)$. This is possible since $B^{X}(x,R)$ is compact. We set $\phi_k(t) = (1-\frac{d(y_k,t)}{4\eta})_+$. Note then that if two points $p_1$ and $p_2$ are at distance smaller than $\frac{3\eta}{2}$, then there is at least one $y_k$ such that $d(p_1,y_k) \leq \eta$ and $d(p_2,y_k) \leq \frac{5\eta}{2}$, and therefore such that
$$(\delta_{p_1}+\delta_{p_2})(\phi_k) \geq \frac{3}{4} + \frac{3}{8} = \frac{9}{8}.$$
Conversely, if an atomic measure $P$ satisfies $P(\phi_k) \geq \frac{9}{8}$, then there are at least two atoms $p_1$ and $p_2$ of $P$ in the support of $\phi_k$, and therefore at distance $d(p_1,p_2)\leq 8\eta$. We have thus shown, for any atomic measure $P$:
\begin{align*}
&\left(\text{the minimal distance between atoms of $P_{| B^X(x,R)}$ is less than $\frac{3\eta}{2}$}\right) \\
&\Rightarrow \left(\text{$(P(\phi_1),\ldots,P(\phi_K))$ belongs to the closed set }\bigcup_{k=1}^K (\R_+)^{k-1} \times \left[\frac{9}{8},+\infty\right) \times (\R_+)^{K-k}\right) \\
&\Rightarrow \big(\text{the minimal distance between atoms of $P_{| B^X(x,R+4\eta)}$ is less than $8\eta$}\big). 
\end{align*}
The parameters $R$ and $\eps$ being fixed, the probability that $M+\delta_x$ has two atoms in $B^X(x,R+4\eta)$ at distance smaller than $8\eta$ goes to $0$ as $\eta$ goes to $0$, because $M+\delta_x$ is supposed without multiplicity (we also use the fact that the random point processes that we are studying are assumed to be locally finite). So, one can find $\eta$ such that 
$$\proba\!\left[\text{the minimal distance between atoms of $(M+\delta_x)_{| B^X(x,R+4\eta)}$ is less than $8\eta$}\right] \leq \frac{\eps}{2}.$$
\emph{A fortiori},
$$\proba\left[\text{$((M+\delta_x)(\phi_1),\ldots,(M+\delta_x)(\phi_K))$ belongs to }\bigcup_{k=1}^K (\R_+)^{k-1} \times \left[\frac{9}{8},+\infty\right) \times (\R_+)^{K-k}\right] \leq \frac{\eps}{2}.$$
We introduce the maps $f_{N,R+4\eta}$ which are almost isometries between the balls $B^{X_N}(x_N,R+4\eta)$ and $B^{X}(x,R+4\eta)$. By assumption, $((M_N+\delta_{x_N})(\phi_1 \circ f_{N,R+4\eta}),\ldots,(M_N+\delta_{x_N})(\phi_K \circ f_{N,R+4\eta}))$ converges in law towards $((M+\delta_x)(\phi_1),\ldots,(M+\delta_x)(\phi_K))$, so by Portmanteau theorem,
$$\limsup_{N \to \infty} \left(\proba\!\left[((M_N+\delta_{x_N})(\phi_1 \circ f_{N,R+4\eta}),\ldots)\in \bigcup_{k=1}^K (\R_+)^{k-1} \times \left[\frac{9}{8},+\infty\right) \times (\R_+)^{K-k}\right] \right)\leq \frac{\eps}{2}.$$
Therefore, for $N$ large enough, these probabilities are smaller than $\eps$, and this implies:
$$\proba\!\left[\text{there are two atoms of $(f_{N,R+4\eta})_\star((M_N+\delta_{x_N})_{| B^{X_N}(x_N,R+4\eta)})$ at distance less than $\frac{3\eta}{2}$}\right] \leq \eps.$$
However, for $N$ large enough, $f_{N,R+4\eta}$ modifies the distances by a factor smaller than $\frac{3}{2}$, therefore,
$$\proba\!\left[\text{there are two atoms of $(M_N+\delta_{x_N})_{| B^{X_N}(x_N,R+4\eta)}$ at distance less than $\eta$}\right] \leq \eps.$$
This clearly implies the result.
\end{proof}\medskip

A similar result that we shall use later is a property of uniform continuity of the maps $R \mapsto M(B^{X}(x,R))$ and $R \mapsto M(B^{X_N}(x_N,R))$:
\begin{lemma}\label{lem:emptyboundary}
For any $R>0$ and $\eps>0$, there exists $\eta>0$ and an integer $N_0$ such that 
\begin{align*}
\forall N \geq N_0,\,\,\,\proba[(M_N+\delta_{x_N})(B^{X_N}(x_N,R+\eta)\setminus B^{X_N}(x_N,R-\eta)) \geq 1] &\leq \eps;\\
\proba[(M+\delta_{x})(B^{X}(x,R+\eta)\setminus B^{X}(x,R-\eta)) \geq 1] &\leq \eps.
\end{align*}
\end{lemma}
\noindent The proof of this second lemma is entirely similar to the one of Lemma \ref{lem:separated}, and relies on the use of adequate test functions. In the sequel, we fix a rooted finite graph $\gamma_n \in \gbullet(n)$, and $\eps>0$. The symbols $\Gamma_N$ and $\Gamma$ stand for $\GEOM(X_N,x_N,d_N,M_N,L)$ and $\GEOM(X,x,d,M,L)$; in the sequel we shall deal with numerous approximations of these random graphs. Note that if $R \geq nL$, then the structure of $\pi_n(\Gamma,x)$ only depends on the restriction of the point process $M$ to the ball $B^{X}(x,R)$. We fix $R\geq nL+4$, and then $\eta< \min\left(1,\frac{L}{4}\right)$ sufficiently small such that with probability at least $1-\eps$,
\begin{itemize}
\item the random rooted graphs $\pi_n(\GEOM(X,x,d,M,L+c\eta),x) $ with $c \in [-4,4]$ are all the same (regularity condition);
\item the atoms of $(M+\delta_x)_{|B^X(x,R+\eta)}$ are all separated by strictly more than $2\eta$ (Lemma \ref{lem:separated});
\item the cardinality $(M+\delta_x)(B^{X}(x,R-\eta))$ is the same as $(M+\delta_x)(B^{X}(x,R+\eta))$ (Lemma \ref{lem:emptyboundary}).
\end{itemize} 
We denote $A_\eps$ the event corresponding to these three conditions; $\proba[A_\eps] \geq 1-\eps$. If $\eta$ is sufficiently small and $N_0$ is sufficiently large, then for $N \geq N_0$, on an event $A_{N,\eps}$ with probability larger than $1-\eps$, we also have the same two last conditions satisfied by $M_N$ in $X_N$:
\begin{itemize}
\item the atoms of $(M_N+\delta_{x_N})_{|B^{X_N}(x_N,R+\eta)}$ are all separated by strictly more than $2\eta$;
\item the cardinality $(M_N+\delta_{x_N})(B^{X_N}(x_N,R-\eta))$ is the same as $(M_N+\delta_{x_N})(B^{X_N}(x_N,R+\eta))$.
\end{itemize} 
We now proceed to a kind of discretisation of the random geometric graph $\Gamma_N$. We fix a set partition $\Psi= \Psi_{1}\sqcup \Psi_{2}\sqcup \cdots \sqcup \Psi_{\ell}$ of the ball $B^{X}(x,R)$ such that $\mathrm{diam}(\Psi_{l}) \leq \frac{\eta}{2}$ for any $l \in \lle 1,\ell\rre$, and we set $\Pi_{N,l} = f_{N,R+\eta}^{-1}(\Psi_l)$, where $f_{N,R+\eta} : B^{X_N}(x_N,R+\eta)\to B^X(x,R+\eta)$ is almost an isometry. If $N$ is taken large enough, then $f_{N,R+\eta}$ modifies the distance between two points $p_1$ and $p_2$ by a factor $c=c(p_1,p_2)$ with 
$$  \max\left(\frac{2}{3},\frac{R-\eta}{R}\right) \leq c \leq \min\left(\frac{3}{2}, \frac{R+\eta}{R} \right).$$ 
Therefore, $\Pi_N$ is a set partition with
$$ B^{X_N}(x_N,nL) \subset B^{X_N}(x_N,R-\eta) \subset \left(\bigsqcup_{l=1}^\ell \Pi_{N,l} \right) \subset B^{X_N}(x_N,R+\eta)$$
and such that $\mathrm{diam}(\Pi_{N,l}) \leq \eta$ for any $l \in \lle 1,\ell\rre$. If we place ourselves on the event $A_{N,\eps}$, then $(M_N+\delta_{x_N})(\Pi_{N,l}) \leq 1$ for any $l \in \lle 1,\ell\rre$, because otherwise $(M_N+\delta_{x_N})_{|B^{X_N}(x_N,R+\eta)}$ would have two atoms at distance smaller than $\eta$. Hence, we have fixed for any $N \geq N_0$ a grid $\Pi_N$ with arbitrary small size and such that, with very high probability, the atoms of $(M_N + \delta_{x_N})_{|\Pi_N}$ fall into the cases of this grid with at most one atom in each case. In the following, we use the same notation $\Pi_N$ for the set partition $( \Pi_{N,1},\ldots,\Pi_{N,\ell})$ and for the disjoint union of its parts.\bigskip

We call \emph{configuration} associated to the random point process $M_N$ the subset
$$C(\Pi_N,M_N) = \{l \in \lle 1,\ell\rre\,\,|\,\,(M_N+\delta_{x_N})(\Pi_{N,l})\geq 1\}$$ 
of the set $\lle 1,\ell\rre$ of parts of $\Pi_N$ that indicates in which cases of the grid the points of $M_N+\delta_{x_N}$ fall. This configuration is well-defined on the whole probability space $(\Omega,\mathcal{F},\proba)$ on which the random point process $M_N$ is constructed, and it is a measurable function of it. Besides, on the event $A_{N,\eps}$, $C(\Pi_N,M_N) = \{l \in \lle 1,\ell\rre\,\,|\,\,(M_N+\delta_{x_N})(\Pi_{N,l})= 1\}$. We can associate to the discrete configuration the random rooted graph $\GRID(\Pi_N,M_N,L)$
\begin{itemize}
     \item whose vertices are the $l$'s in $C(\Pi_N,M_N)$,
     \item whose edges connect two indices $l$ and $m$ if $d_N(\Pi_{N,l},\Pi_{N,m}) \leq L$,
     \item whose root is the index $l=l(x_N)$ such that $\delta_{x_N}(\Pi_{N,l})=1$, that is to say that $x_N$ falls in $\Pi_{N,l}$. 
 \end{itemize}   
We refer to Figure \ref{fig:grid} for a drawing of the configuration $C(\Pi_N,M_N)$ and of the two random rooted graphs $\GRID(\Pi_N,M_N,L)$ and $\GEOM(X_N,x_N,d_N,M_N,L)$. On this drawing, the space $X_N$ is a part of the plane $\R^2$, the distance comes from the norm $\|\cdot\|_\infty$, the cases of the grid are of size $\eta \times \eta$, and $L=3\eta$.

\begin{figure}[ht]
\begin{center}
\begin{tikzpicture}[scale=1]
\draw (-3,-3) -- (2.5,-3) -- (2.5,2.5) -- (-3,2.5) -- (-3,-3);
\draw (-2.5,-3) -- (-2.5,2.5);
\draw (-2,-3) -- (-2,2.5);
\draw (-1.5,-3) -- (-1.5,2.5);
\draw (-1,-3) -- (-1,2.5);
\draw (-0.5,-3) -- (-0.5,2.5);
\draw (0,-3) -- (0,2.5);
\draw (2,-3) -- (2,2.5);
\draw (1.5,-3) -- (1.5,2.5);
\draw (1,-3) -- (1,2.5);
\draw (0.5,-3) -- (0.5,2.5);
\draw (-3,-2.5) -- (2.5,-2.5);
\draw (-3,-2) -- (2.5,-2);
\draw (-3,-1.5) -- (2.5,-1.5);
\draw (-3,-1) -- (2.5,-1);
\draw (-3,-0.5) -- (2.5,-0.5);
\draw (-3,0) -- (2.5,0);
\draw (-3,2) -- (2.5,2);
\draw (-3,1.5) -- (2.5,1.5);
\draw (-3,1) -- (2.5,1);
\draw (-3,0.5) -- (2.5,0.5);
\foreach \p in {(-0.25,-0.25),(-0.98213,0.44287),(0.84706,2.3578),(2.0369,-2.1604),
(1.6868,0.21227),(0.17989,1.0986),(-1.7677,2.4852),(-0.41979,-1.9),(-2.2473,-2.0997),(0.63321,-0.87695)}
 \draw \p node {$\times$};
\draw (0.3,-0.3) node {$x_N$}; 
\draw (-0.25,-3.5) node {Grid $\Pi_N$ and point process $M_N+\delta_{x_N}$.};

\begin{scope}[shift={(9,0)}]
\draw [gray] (-3,-3) -- (2.5,-3) -- (2.5,2.5) -- (-3,2.5) -- (-3,-3);
\draw [gray] (-2.5,-3) -- (-2.5,2.5);
\draw [gray] (-2,-3) -- (-2,2.5);
\draw [gray] (-1.5,-3) -- (-1.5,2.5);
\draw [gray] (-1,-3) -- (-1,2.5);
\draw [gray] (-0.5,-3) -- (-0.5,2.5);
\draw [gray] (0,-3) -- (0,2.5);
\draw [gray] (2,-3) -- (2,2.5);
\draw [gray] (1.5,-3) -- (1.5,2.5);
\draw [gray] (1,-3) -- (1,2.5);
\draw [gray] (0.5,-3) -- (0.5,2.5);
\draw [gray] (-3,-2.5) -- (2.5,-2.5);
\draw [gray] (-3,-2) -- (2.5,-2);
\draw [gray] (-3,-1.5) -- (2.5,-1.5);
\draw [gray] (-3,-1) -- (2.5,-1);
\draw [gray] (-3,-0.5) -- (2.5,-0.5);
\draw [gray] (-3,0) -- (2.5,0);
\draw [gray] (-3,2) -- (2.5,2);
\draw [gray] (-3,1.5) -- (2.5,1.5);
\draw [gray] (-3,1) -- (2.5,1);
\draw [gray] (-3,0.5) -- (2.5,0.5);
\foreach \p in {(-0.25,-0.25),(-0.75,0.25),(0.75,2.25),(2.25,-2.25),
(1.75,0.25),(0.25,1.25),(-1.75,2.25),(-0.25,-1.75),(-2.25,-2.25),(0.75,-0.75)}
 {\fill[shift={\p},color=red!50!white] (-0.2,-0.2) rectangle (0.2,0.2) ;
 \draw[shift={\p},color=red] (-0.2,-0.2) rectangle (0.2,0.2);}
\draw [color=red] (-0.45,-0.45) -- (-0.05,-0.05);
\draw [color=red] (-0.45,-0.05) -- (-0.05,-0.45);
\draw (-0.25,-3.5) node {Discrete configuration $C(\Pi_N,M_N)$.};
\end{scope}

\begin{scope}[shift={(0,-7)}]
\foreach \p in {(-0.25,-0.25),(-0.75,0.25),(0.75,2.25),(2.25,-2.25),
(1.75,0.25),(0.25,1.25),(-1.75,2.25),(-0.25,-1.75),(-2.25,-2.25),(0.75,-0.75)}
 {\fill \p circle (1.5pt) ;}
 \draw (-0.25,-0.25) circle (4pt);
 \draw (-0.25,-0.25) -- (-0.75,0.25) -- (0.25,1.25) -- (-0.25,-0.25); 
 \draw (0.75,2.25) -- (0.25,1.25) -- (1.75,0.25) -- (0.75,-0.75) -- (-0.25,-0.25);
 \draw (-0.75,0.25) -- (0.75,-0.75) -- (-0.25,-1.75) -- (-0.25,-0.25);
 \draw (0.25,1.25) -- (0.75,-0.75) -- (2.25,-2.25);
 \draw (-0.75,0.25) -- (-0.25,-1.75) -- (-2.25,-2.25) -- (-0.25,-0.25);
 \draw (-0.25,-1.75) .. controls (1,-1) .. (1.75,0.25) -- (-0.25,-0.25);
 \draw (1.75,0.25) -- (0.75,2.25);
 \draw (0.25,1.25) -- (-1.75,2.25) -- (-0.75,0.25) -- (0.75,2.25);
 \draw (-0.25,-3.5) node {Graph $\GRID(\Pi_N,M_N,L)$.};
\end{scope}

\begin{scope}[shift={(9,-7)}]
\foreach \p in {(-0.25,-0.25),(-0.98213,0.44287),(0.84706,2.3578),(2.0369,-2.1604),
(1.6868,0.21227),(0.17989,1.0986),(-1.7677,2.4852),(-0.41979,-1.9),(-2.2473,-2.0997),(0.63321,-0.87695)}
{\fill \p circle (1.5pt) ;}
\draw (-0.25,-0.25) circle (4pt);
\draw (-0.25,-0.25) -- (-0.98213,0.44287);
\draw (-0.25,-0.25) -- (0.17989,1.0986);
\draw (-0.25,-0.25) -- (0.63321,-0.87695);
\draw (-0.98213,0.44287) -- (0.17989,1.0986);
\draw (0.84706,2.3578) -- (0.17989,1.0986);
\draw (2.0369,-2.1604) -- (0.63321,-0.87695);
\draw (1.6868,0.21227) -- (0.63321,-0.87695);
\draw (-0.41979,-1.9) -- (0.63321,-0.87695);
\draw (-0.25,-3.5) node {Graph $\GEOM(X_N,x_N,d_N,M_N,L)$.};
\end{scope}
\end{tikzpicture}
\caption{The grid $\Pi_N$, the discrete configuration $C(\Pi_N,M_N)$ and the two random graphs $\GRID(\Pi_N,M_N,L)$ and $\GEOM(X_N,x_N,d_N,M_N,L)$.\label{fig:grid}}
\end{center}
\end{figure}

\begin{remark}
In all the proofs hereafter, we shall manipulate atoms of the point processes $M_N+\delta_{x_N}$ or $M+\delta_x$, and indices of configurations $C(\Pi_N,M_N)$ or $C(\Psi,M)$; and we shall discuss whether they are connected in a graph $\GEOM$ or $\GRID$. \emph{When discussing the property of being connected, implicitly, we shall only consider the atoms and the indices that are at graph distance smaller than $n$ from the root of the graph.} We ask the reader to keep this convention in mind, which we shall not recall each time and which if omitted might lead to imprecise arguments.
\end{remark}

\begin{lemma}\label{lem:discretisationgraphs}
On the event $A_{N,\eps}$, for any $L>0$, we have a sequence of inclusions 
\begin{align*}
\pi_n(\GEOM(X_N,x_N,d_N,M_N,L),x_N) &\subset \pi_n (\GRID(\Pi_N,M_N,L),l(x_N)) \\
&\subset \pi_n(\GEOM(X_N,x_N,d_N,M_N,L+2\eta),x_N).
\end{align*}
\end{lemma}

\begin{proof}
Let $a$ and $b$ be two atoms of $(M_N+\delta_{x_N})_{|\Pi_N}$, and $l$ and $m$ be the indices of the parts $\Pi_{N,l}$ and $\Pi_{N,m}$ such that $a\in \Pi_{N,l}$ and $b \in \Pi_{N,m}$. We say that $l$ and $m$ are the elements of the configuration $C(\Pi_N,M_N)$ associated to $a$ and $b$; this correspondence is well-defined on $A_{N,\eps}$. Now, if $a$ and $b$ are connected in $\GEOM(X_N,x_N,d_N,M_N,L)$, then we have two points of $\Pi_{N,l}$ and $\Pi_{N,m}$ at distance smaller than $L$, so $l$ and $m$ are connected in $\GRID(\Pi_N,M_N,L)$. On the other hand, if $l$ is connected to $m$ in $\GRID(\Pi_N,M_N,L)$, then since $\Pi_{N,l}$ and $\Pi_{N,m}$ have diameter smaller than $\eta$, $a$ and $b$ are connected in $\GEOM(X_N,x_N,d_N,M_N,L+2\eta)$. Therefore, on the event $A_{N,\eps}$, we have indeed the two inclusions stated, being understood that on this event we can identify the atoms of $(M_N+\delta_{x_N})_{|\Pi_N}$ and the integers in $C(\Pi_N,M_N) \subset \lle 1,\ell\rre$. 
\end{proof}

As a consequence of this lemma, if the two discretisations $\pi_n(\GRID(\Pi_N,M_N,L-2\eta),l(x_N))$ and $\pi_n(\GRID(\Pi_N,M_N,L),l(x_N))$ are the same and are equal to $\gamma_n$, then on $A_{N,\eps}$, we also have $\pi_n(\Gamma_N,x_N) = \gamma_n$. This leads to the inequality
\begin{align}
&\proba[\pi_n(\Gamma_N,x_N) = \gamma_n] \nonumber \\
&\geq \proba[A_{N,\eps}\cap (\pi_n(\GRID(\Pi_N,M_N,L-2\eta),l(x_N)) = \pi_n(\GRID(\Pi_N,M_N,L),l(x_N)) = \gamma_n)] - \eps \label{eq:infernalinequality}
\end{align}
for any $N \geq N_0$. From now on, we shall work on $X_N$ with the discretised random graphs $\GRID$, and the next step of the proof of Theorem \ref{thm:convergencespacesandgraphs} consists in relating their distribution to events that can be expressed in terms of observables $(M_N+\delta_{X_N})(\theta_{N,l})$, where the $\theta_{N,l}$'s are compactly supported continuous functions on $X_N$. To construct these functions, we start from functions compactly supported on $B^X(x,R+\eta)$:
\begin{align*}
\phi_0(t) &= \left(1-\frac{d(t,B^X(x,R-\eta))}{2\eta}\right)_{\!+};\\
\forall l \in \lle 1,\ell\rre, \,\,\,\phi_{l}(t) &= \left(1-\frac{d(t,\Psi_l)}{\eta}\right)_{\!+} 
\end{align*}
We then set 
$$\theta_{N,l}(t \in X_N) = \begin{cases}
     \phi_l \circ f_{N,R+\eta}(t) &\text{if } t \in B^{X_N}(x_N,R+\eta),\\
     0 &\text{otherwise}.
\end{cases} $$
Given a configuration $C \subset \lle 1,\ell\rre$, we denote $I_l(C) = (\frac{2}{3},+\infty)$ if $l \in C$, and $I_l(C) = \R$ if $l \in \lle 1,\ell \rre \setminus C$. We set
$$ U_C = \left(\left(\card (C)-\frac{1}{2},\card (C) + \frac{1}{2}\right) \times \prod_{l =1}^\ell I_l(C) \right) \subset \R^{\ell+1}.$$
The set $U_C$ is open in $\R^{\ell + 1}$. On the other hand, for any $L>0$, there is a finite set $\mathfrak{C}(\gamma_n,\Pi_N,L)$ of configurations $C \subset \lle 1,\ell \rre$ such that $\GRID(\Pi_N,M_N,L) = \gamma_n$ if and only if  $C \in \mathfrak{C}(\gamma_n,\Pi_N,L)$. The following lemma relates the event $A_{N,\eps}\cap (\pi_n(\GRID(\Pi_N,M_N,L),l(x_N)) = \gamma_n)$ to the values of the random vector of observables 
$$(M_N+\delta_{x_N})(\theta_N) = \left((M_N+\delta_{x_N})(\theta_{N,0}) , (M_N+\delta_{x_N})(\theta_{N,1}) , \ldots, (M_N+\delta_{x_N})(\theta_{N,\ell})\right)$$
and to its belonging to certain unions of open sets $U_C$.

\begin{lemma}
We then have the following inclusions of events:
\begin{align*}
&\left(A_{N,\eps} \cap \left((M_N+\delta_{x_N})(\theta_N) \in \bigcup_{\substack{C \in \mathfrak{C}(\gamma_n,\Pi_N,L-\eta) \\ \text{and }C \in \mathfrak{C}(\gamma_n,\Pi_N,L+\eta)}} U_C\right)\right) \\
&\subset \big( A_{N,\eps}\cap (\pi_n(\GRID(\Pi_N,M_N,L),l(x_N)) = \gamma_n) \big)  \\
&\subset \left(A_{N,\eps} \cap \left((M_N+\delta_{x_N})(\theta_N) \in \bigcup_{C \in \mathfrak{C}(\gamma_n,\Pi_N,L)} U_C\right)\right).
\end{align*}
\end{lemma}

\begin{proof}
On $A_{N,\eps}$, suppose first that there exists a configuration $C_0$ such that $(M_N+\delta_{x_N})(\theta_N) \in U_{C_0}$, with $C_0 \in \mathfrak{C}(\gamma_n,\Pi_N,L-\eta)  \cap \mathfrak{C}(\gamma_n,\Pi_N,L+\eta)$. For any $l \in C_0$, $(M_N+\delta_{x_N})(\theta_{N,l}) \in (\frac{2}{3},1)$, so
$(f_{N,R+\eta})_\star((M_N + \delta_{x_N})_{|\Pi_N})$ has at least one atom at distance smaller than $\frac{\eta}{3}$ from $\Psi_l$. Since $f_{N,R+\eta}$ cannot modify the distances by a factor larger than $\frac{3}{2}$, this implies that $M_N + \delta_{x_N}$ has at least one atom at distance strictly smaller than $\frac{\eta}{2}$ from $\Pi_{N,l}$, and in fact there is exactly one such atom with this property: otherwise, since $\Pi_{N,l}$ has diameter smaller than $\eta$, we would have two distinct atoms at distance smaller than $2\eta$, and this is not allowed on $A_{N,\eps}$. We thus have an injection from $C_0$ to $\{\text{atoms of }(M_N+\delta_{x_N})_{|\Pi_N}\}$, and this is actually a bijection, because if there were other atoms, then one would have
$$(M_N+\delta_{x_N})(\theta_{N,0}) = (M_N+\delta_{x_N})(\Pi_N) \geq \card (C_0) +1,$$
which contradicts the assumption $(M_N + \delta_{x_N})(\theta_N)\in U_{C_0}$. So, on $A_{N,\eps}$ we have a perfect correspondence $C_0 \leftrightarrow \{\text{atoms of }(M_N+\delta_{x_N})_{|\Pi_N}\}$. Beware that this does not imply $C(\Pi_N,M_N) = C_0$ (the two configurations might have occupied cases in the grid $\Pi_N$ that are adjacent but distinct). Let $l$ and $m$ two indices in $C_0$, $a$ and $b$ the corresponding atoms, and $l'$ and $m'$ the indices in $C(\Pi_N,M_N)$ such that $a \in \Pi_{N,l'}$ and $b \in \Pi_{N,m'}$. If $l$ is connected to $m$ by $\gamma_n$, then $d_N(\Pi_l,\Pi_m) \leq L-\eta$ since $C_0 \in \mathfrak{C}(\gamma_n,\Pi_N,L-\eta)$, so $d_N(a,b) \leq L$ and $d_N(\Pi_{N,l'},\Pi_{N,m'}) \leq L$. Conversely, if $d_N(\Pi_{N,l'},\Pi_{N,m'}) \leq L$, then $d_N(\Pi_l,\Pi_m) \leq L+\eta$, and as $C_0 \in \mathfrak{C}(\gamma_n,\Pi_N,L+\eta)$, this implies that $l$ and $m$ are connected by $\gamma_n$. We conclude that the assumption made at the beginning implies that $\pi_n(\GRID(\Pi_N,M_N,L),l(x_N)) = \gamma_n$, whence the first inclusion of events.\medskip

The second inclusion is much simpler. On $A_{N,\eps}$, if $\pi_n(\GRID(\Pi_N,M_N,L),l(x_N)) = \gamma_n$, then $C(\Pi_N,M_N) \in \mathfrak{C}(\gamma_n,\Pi_N,L)$, and
\begin{equation*}
(M_N+\delta_{x_N})(\theta_N) \in U_{C(\Pi_N,M_N)} \subset \bigcup_{C \in \mathfrak{C}(\gamma_n,\Pi_N,L)} U_C.
\qedhere\end{equation*}
\end{proof}

By adapting the previous lemma to the events that appear in Inequality \eqref{eq:infernalinequality}, we obtain the following:
$$ \liminf_{N \to \infty} (\proba[\pi_n(\Gamma_N,x_N) = \gamma_n])  \geq \liminf_{N \to \infty}\left( \proba\!\left[ (M_N + \delta_{x_N})(\theta_N) \in \bigcup_{\substack{C \in \mathfrak{C}(\gamma_n,\Pi_N,L-3\eta) \\ \text{and }C \in \mathfrak{C}(\gamma_n,\Pi_N,L-\eta) \\ \text{and }C \in \mathfrak{C}(\gamma_n,\Pi_N,L+\eta)}} U_C \right] \right) - \eps .
$$
Since we assume the convergence in the random pointed Lipschitz sense, and since the $U_C$'s are open sets, by the Portmanteau theorem, the right-hand side is larger than the analogue probability involving the limiting point process $M+\delta_x$, so
$$ \liminf_{N \to \infty} (\proba[\pi_n(\Gamma_N,x_N) = \gamma_n])  \geq \liminf_{N \to \infty}\left( \proba\!\left[ (M + \delta_{x})(\phi) \in \bigcup_{\substack{C \in \mathfrak{C}(\gamma_n,\Pi_N,L-3\eta) \\ \text{and }C \in \mathfrak{C}(\gamma_n,\Pi_N,L-\eta) \\ \text{and }C \in \mathfrak{C}(\gamma_n,\Pi_N,L+\eta)}} U_C \right] \right) - \eps .
$$
The following final lemma relates the event on the right-hand side to properties of the discretised random geometric graphs on $(X,x,d,M)$:
\begin{lemma}
We place ourselves on the event $A_{\eps}$ specified before the introduction of the discretised graphs $\GRID$. If $\pi_n(\GRID(\Psi,M,L-4\eta),l(x)) $ and $ \pi_n(\GRID(\Psi,M,L+2\eta),l(x))$ are the same graph and are equal to $\gamma_n$, then $(M+\delta_x)(\phi)$ belongs to 
$$\bigcup_{\substack{C \in \mathfrak{C}(\gamma_n,\Pi_N,L-3\eta) \\ \text{and }C \in \mathfrak{C}(\gamma_n,\Pi_N,L-\eta) \\ \text{and }C \in \mathfrak{C}(\gamma_n,\Pi_N,L+\eta)}} U_C.$$
\end{lemma}

\begin{proof}
We suppose that
$$
\pi_n(\GRID(\Psi,M,L-4\eta),l(x)) = \pi_n(\GRID(\Psi,M,L+2\eta),l(x)) = \gamma_n,
$$
and we are going to prove that $C_0=C(\Psi,M)$ belongs to $\mathfrak{C}(\gamma_n,\Pi_N,L-3\eta)$, to $\mathfrak{C}(\gamma_n,\Pi_N,L-\eta)$ and $\mathfrak{C}(\gamma_n,\Pi_N,L+\eta)$. This will imply the result, since by the same argument as in the proof of the previous lemma, $(M+\delta_x)(\phi) \in U_{C(\Psi,M)}$. Let $l$ and $m$ be two indices of $C_0$ such that $\Pi_{N,l}$ and $\Pi_{l,m}$ are at distance smaller than $L-3\eta$. Then, as $f_{N,R+\eta}$ does not modify the distances by a factor larger than $\frac{L-2\eta}{L-3\eta}$, $\Psi_l$ and $\Psi_m$ are at distance smaller than $L-2\eta$, so $l$ and $m$ are connected in $\gamma_n$. Conversely, if $\Psi_l$ and $\Psi_m$ are at distance smaller than $L-4\eta$, then since $f_{N,R+\eta}^{-1}$ does not modify the distances by a factor larger than $\frac{L-3\eta}{L-4\eta}$, $\Pi_{N,l}$ and $\Pi_{N,m}$ are at distance smaller than $L-3\eta$. We conclude that $C_0 \in \mathfrak{C}(\gamma_n,\Pi_N,L-3\eta)$, and the two other sets of configurations are treated with similar arguments.
\end{proof}

\begin{proof}[Proof of Theorem \ref{thm:convergencespacesandgraphs}] 
The previous lemma ensures that
\begin{align*}
 & \liminf_{N \to \infty} (\proba[\pi_n(\Gamma_N,x_N) = \gamma_n]) \\
& \geq  \proba[A_\eps \cap (\pi_n(\GRID(\Psi,M,L-4\eta),l(x)) = \pi_n(\GRID(\Psi,M,L+2\eta),l(x)) = \gamma_n)]- \eps. 
 \end{align*}
However, we have on $A_\eps$ the inclusion
\begin{align*}
\pi_n(\GEOM(X,x,d,M,L-4\eta),x) &\subset \pi_n (\GRID(\Pi,M,L-4\eta),l(x)) \\
&\subset \pi_n (\GRID(\Pi,M,L+2\eta),l(x))\\
&\subset \pi_n(\GEOM(X,x,d,M,L+4\eta),x),
\end{align*}
for the same reasons as in Lemma \ref{lem:discretisationgraphs}. Since the two bounds given by geometric graphs are the same on $A_\eps$ and are equal to $\pi_n(\GEOM(X,x,d,M,L),x)$, we have thus shown:
$$ \liminf_{N \to \infty} (\proba[\pi_n(\Gamma_N,x_N) = \gamma_n]) \geq \proba[\pi_n(\Gamma,x) = \gamma_n]-2\eps.$$
As this is true for any $\eps>0$, and as both sides are probability measures on $\gbullet(n)$, this proves that there is no mass of the distributions of the graphs $\pi_n(\Gamma_N,x_N)$ that escapes at infinity, and that we have in fact $\lim_{N \to \infty} (\proba[\pi_n(\Gamma_N,x_N) = \gamma_n]) = \proba[\pi_n(\Gamma,x) = \gamma_n]$. This amounts to the local Benjamini--Schramm convergence by Proposition \ref{prop:characterisation_BSconvergence}.
\end{proof}
\medskip

The limiting random graph that appears in Theorem \ref{thm:poisson_BS} is called the (rooted) \emph{Poisson Boolean model} in \cite{MR96}, and it is studied from the point of view of continuous percolation in Chapters 3-5 of \emph{loc.~cit.}, as well as in \cite[Section 9.6]{Pen03}. The most important result is the existence of a critical parameter $\lambda_c(\dim G)>0$ for the Poisson point process $P = \mathcal{P}(\lambda)$, such that the resulting random geometric graph with connection distance $1$ has no unbounded connected component almost surely if $\lambda<\lambda_c(\dim X)$, and has exactly one unbounded connected component if $\lambda> \lambda_c(\dim X)$; see \emph{e.g.} \cite[Theorem 9.19]{Pen03}. As the random geometric graphs $\GEOM(N,L_N)$ on $X$ converge locally towards the Poisson Boolean model, this implies the following result:
\begin{corollary}
Consider a random geometric graph $\GEOM(N,L_N)$ on a ssccss $X$, with as usual $L_N=(\ell/N)^{\frac{1}{\dim X}}$. There exists a critical parameter $\ell_c(\dim X)>0$ such that, if $\ell<\ell_c(\dim X)$, then
$$\lim_{n \to \infty} \left( \limsup_{N \to \infty}\,\proba[\mathrm{diam}(\text{connected component of the vertex $v_1$ in $\GEOM(N,L_N)$}) \geq n] \right)= 0.$$
\end{corollary}

\subsection{Convergence in probability of the spectral measures}\label{subsec:poissonlimit}
By Theorem \ref{thm:bordenave}, the local convergence $U(\GEOM(N,L_N))=(\Gamma_N,r_N) \to (\Gamma_\infty,r)$ implies the weak convergence of the expected spectral measures $\mu_N=\esper[\nu_N]$ towards a probability measure $\mu$ on $\R$. If we want instead to prove the convergence in probability of the spectral measures $\nu_N$ (without taking the expectation), then taking into account Proposition \ref{prop:bordenave2}, we need to prove the following extension of our Theorem \ref{thm:poisson_BS}:

\begin{proposition}\label{prop:macroscopicindependence}
Let $X$ be a ssccss, $M_N$ the point process on $X$ obtained by taking $N$ independent points $v_1,\ldots,v_N$ according to the Haar measure, and $r_N$ and $r_N'$ two independent random vertices in the set of atoms of $M_N$. We denote as before $t_N = N^{\frac{1}{\dim X}}$ and $d_N = t_N d$, $d$ being the geodesic distance. As $N$ goes to infinity,
\begin{itemize}
     \item the pair of random pointed proper metric spaces $((X,r_N,d_N,M_N),(X,r_N',d_N,M_N))$ converges in the Lipschitz sense towards two independent copies of $(\R^{\dim X},0,d_\mathrm{Euclidean},\delta_0+\mathcal{P}(\frac{1}{\vol(X)}))$;
     \item the pair of random rooted graphs $((\Gamma_N,r_N),(\Gamma_N,r_N'))$ converges in the Benjamini--Schramm sense towards two independent copies of the random graph $(\Gamma_\infty,r)$ from Theorem \ref{thm:poisson_BS}.
 \end{itemize} 
\end{proposition}

\begin{proof}
In order to lighten a bit the notations, we shall prove the first item of the proposition when $X=G$ is a Lie group; the proof adapts readily to the non-group case by using the transitive action of the isometry group of $X$.
Fix $R>0$, and denote $h_{N,R} : B^{(G,d_N)}(e_G,R) \to B^\glie(0,R)$ a bijective map which is a quasi-isometry (its dilation constant goes to $1$ as $N$ goes to infinity). We then set
$f_{N,R}(g) = h_{N,R}(g\,(r_N)^{-1})$ and $f_{N,R}'(g) = h_{N,R}(g\,(r_N')^{-1})$; these maps are quasi-isometries from $B^{(G,d_N)}(r_N,R)$ and $B^{(G,d_N)}(r_N',R)$ to $B^\glie(0,R)$. We also consider two continuous and compactly supported functions $\phi$ and $\phi'$ on $B^\glie(0,R)$. We have to show that
\begin{align}
&\esper\left[\exp \left(z_1\, ((f_{N,R})_\star(M_{N})_{|B^{(G,d_N)}(r_N,R)})(\phi) + z_2\,((f_{N,R}')_\star(M_{N})_{|B^{(G,d_N)}(r_N',R)})(\phi')\right)\right] \label{eq:bilaplacetransform}\\
&\to_{N \to \infty} \exp\left(z_1\phi(0)+z_2\phi'(0)+ \frac{1}{\vol(G)} \int_{B^{\glie}(0,R)} (\E^{z_1\phi(t)} - 1) + (\E^{z_2\phi'(t)}-1) \DD{t}\right)\nonumber
\end{align}
for any complex numbers $z_1$ and $z_2$; by replacing $\phi$ and $\phi'$ by $z_1\phi$ and $z_2\phi'$, we can take them equal to $1$ in the following. The expectation in Equation \eqref{eq:bilaplacetransform} is
\begin{align}
&\frac{1}{N^2}\int_{G^N}\sum_{i,j=1}^N \exp\left(\sum_{k=1}^n \phi(h_{N,R}(v_k(v_i)^{-1})) + \phi'(h_{N,R}(v_k(v_j)^{-1})) \right) \DD{v_1}\DD{v_2}\,\cdots \DD{v_N} \nonumber\\
&= \frac{1}{N}\int_{G^N} \exp\left(\sum_{k=1}^n \phi(h_{N,R}(v_k(v_1)^{-1})) + \phi'(h_{N,R}(v_k(v_1)^{-1})) \right) \DD{v_1}\DD{v_2}\,\cdots \DD{v_N} \label{eq:bilaplacetransform2}\\
&\quad + \frac{N-1}{N}\int_{G^N} \exp\left(\sum_{k=1}^n \phi(h_{N,R}(v_k(v_1)^{-1})) + \phi'(h_{N,R}(v_k(v_2)^{-1})) \right) \DD{v_1}\DD{v_2}\,\cdots \DD{v_N}\label{eq:bilaplacetransform3}
\end{align}
by using the symmetry of the roles played by the variables $v_1,\ldots,v_N$. Here, we convene that $\phi(h_{N,R}(g))=0$ if $g$ does not belong to $B^{(G,d_N)}(e_G,R)$. 
\begin{itemize}
    \item The first term \eqref{eq:bilaplacetransform2} corresponding to the case where $r_N=r_N'$ yields a contribution which is a $O(\frac{1}{N})$, hence negligeable in the limit $N \to +\infty$. Indeed, it rewrites as
    \begin{align*}
    &\frac{\exp(\phi(0)+\phi'(0))}{N} \int_{G^{N-1}} \exp\left(\sum_{k=2}^n \phi(h_{N,R}(w_k)) + \phi'(h_{N,R}(w_k)) \right)\DD{w_2}\,\cdots \DD{w_N}\\
    &=\frac{\exp(\phi(0)+\phi'(0))}{N} \left(\int_{G} \exp(\phi(h_{N,R}(g)) + \phi'(h_{N,R}(g)))\DD{g}\right)^{N-1}\\
    &= \frac{\exp(\phi(0)+\phi'(0))}{N} \,\exp\left(\frac{1}{\vol(G)}\int_{B^\glie(0,R)} (\E^{\phi(t)+\phi'(t)}-1) \DD{t}\right)\,(1+o(1)),
    \end{align*}
    by using on the third line the same arguments as in the proof of Proposition \ref{prop:rescalingliegroup}.
    \item The second term \eqref{eq:bilaplacetransform3} is asymptotically equivalent to
    $$\exp(\phi(0)+\phi'(0))\int_{G^2} \exp(\phi(h_{N,R}(v_2(v_1)^{-1})) + \phi'(h_{N,R}(v_1(v_2)^{-1}))) \,f_N(v_1(v_2)^{-1}) \DD{v_1}\DD{v_2} $$
    where
    \begin{align*}
    f_N(v_1(v_2)^{-1}) &= \int_{G^{N-2}}\exp\left(\sum_{k=3}^n \phi(h_{N,R}(v_k(v_1)^{-1})) + \phi'(h_{N,R}(v_k(v_2)^{-1}))\right)\DD{v_3}\,\cdots \DD{v_N}\\
    &= \left(\int_G \exp(\phi(h_{N,R}(g) + \phi'(h_{N,R}(gv_1(v_2)^{-1})))\DD{g}\right)^{N-2};
    \end{align*}
    we have only removed the multiplicative factor $\frac{N-1}{N}$. Setting $v = v_1(v_2)^{-1}$, we see that the expectation in Equation \eqref{eq:bilaplacetransform} is equivalent to 
     $$\exp(\phi(0)+\phi'(0))\int_{G} \exp(\phi(h_{N,R}(v^{-1})) + \phi'(h_{N,R}(v)))  \,f_N(v) \DD{v} .$$
\end{itemize}
If $v$ does not belong to $B^{(G,d_N)}(e_G,2R)$, then the triangular inequality shows that we cannot have at the same time $g \in B^{(G,d_N)}(e_G,R)$ and $gv \in  B^{(G,d_N)}(e_G,R)$. Therefore, under the condition $d_N(e_G,v) > 2R$, we have
\begin{align*}
f_N(v) &= \left(1+\int_{B^{(G,d_N)}(e_G,R)} (\exp(\phi(h_{N,R}(g)))-1) + (\exp(\phi'(h_{N,R}(g)))-1)\DD{g} \right)^{N-2} \\
&=\exp\left(\frac{1}{\vol(G)} \int_{B^{\glie}(0,R)} (\E^{\phi(t)} - 1) + (\E^{\phi'(t)}-1) \DD{t}\right) (1+o(1))
\end{align*}
and the multiplicative factor $\exp(\phi(h_{N,R}(v^{-1})) + \phi'(h_{N,R}(v)))$ is equal to $1$ under this condition. On the other hand, the contribution of the $v$'s such that $d_N(e_G,v) \leq 2R$ is asymptotically negligeable, since we are looking at a small ball of volume $O(\frac{1}{N})$. This ends the proof of the first part. \bigskip

For the second part of the proposition, by using test functions as in the proof of Theorem \ref{thm:convergencespacesandgraphs}, one gets the following easy generalisation of this theorem. Suppose given $(X,x_1,x_2,d)$ a bi-pointed proper metric space, and $M$ a random point process on it such that 
\begin{align*}
(\R_+)^2 &\to (\mathfrak{G}_{\bullet})^2 \\
(l_1,l_2) &\mapsto (\GEOM(X,x_1,d,M,l_1),\GEOM(X,x_2,d,M,l_2))
\end{align*}
is almost surely continuous at any fixed pair $(l_1,l_2)$. Then, for any pair of positive parameters $(L_1,L_2)$, the map $(\GEOM(\cdot,L_1) \circ \pi_1,\GEOM(\cdot,L_2) \circ \pi_2)$ is continuous with respect to the Benjamini--Schramm topology at the point $(X,x_1,x_2,d,M)$, where $\pi_1$ and $\pi_2$ are the two projections on $\mathbf{PMS}_{\bullet,\star}$ of the space $\mathbf{PMS}_{\bullet\bullet,\star}$ of bi-pointed proper spaces endowed with a random point process. The second item of the proposition follows now from the first item, by using the aforementioned generalisation of Theorem \ref{thm:convergencespacesandgraphs} with the bi-pointed space $(X,x_1,x_2,d,M)$ given by two disjoint and independent copies of a Poisson point process on $\R^{d}$.
\end{proof}
\bigskip

By applying Proposition \ref{prop:bordenave2}, we finally obtain:
\begin{theorem}\label{thm:poissonlimit}
Fix a ssccss $X$ and $\ell>0$, and consider the random spectral measures $\nu_N$ of the random geometric graphs $\GEOM(N,L_N)$ on $X$, with $L_N = (\ell/N)^{\frac{1}{\dim X}}$. There exists a probability measure $\mu=\mu(\dim X,\frac{\ell}{\vol(X)})$ on $\R$ such that we have the weak convergence in probability 
$$\nu_N \rightharpoonup_{N \to \infty} \mu.$$
\end{theorem}

To close this section, let us propose a slight improvement of this convergence result, which does not seem to be implied by the results from \cite{BL10,BLS11,Bor16} that we presented in Section \ref{subsec:benjaminischramm}. 
\begin{proposition}\label{prop:superpoissonlimit}
In the same setting as Theorem \ref{thm:poissonlimit}, the measure $\mu$ has moments of all order, and it is determined by its moments. We have for any $s \geq 1$
$$M_{s,N} = \int_{\R} x^s\,\nu_N(\!\DD{x}) \to_{N\to \infty} M_s = \int_{\R} x^s\,\mu(\!\DD{x}),$$
where the convergence occurs in $\leb^2$ (and therefore also in probability).
\end{proposition}

\begin{proof}
We start by examining the moments $\overline{M}_{s,N}=\int_\R x^s\,\mu_N(\!\DD{x})$ of the expected spectral measures $\mu_N$. By design, the adjacency matrices of our graphs have zeroes on their diagonal, so $M_{1,N}=\overline{M}_{1,N}=0$. Suppose now that $s \geq 2$. We can rewrite $\overline{M}_{s,N}$ as follows:
\begin{align}
\overline{M}_{s,N} &= \frac{1}{N} \,\esper\!\left[\tr\!\left((A_{\GEOM(N,L_N)})^s\right)\right] \nonumber \\
&= \frac{1}{N}\, \esper\!\left[\sum_{\substack{1\leq i_1,\ldots,i_s \leq N \\ \text{no consecutive indices are equal}}} 1_{(i_1,\ldots,i_s) \text{ is a cycle in }\GEOM(N,L_N)}\right],\label{eq:traceexpansionmoments}
\end{align}
and we have the following inequalities:
\begin{align*}
&\sum_{\substack{1\leq i_1,\ldots,i_s \leq N \\ \text{no consecutive indices are equal}}} 1_{(i_1,\ldots,i_s) \text{ is a cycle in }\GEOM(N,L_N)} \\
&\leq \sum_{1\leq i_1\neq i_2 \neq \cdots \neq i_s \leq N} 1_{i_1\leftrightarrow i_2} 1_{i_2\leftrightarrow i_3} \cdots 1_{i_{s-1}\leftrightarrow i_s}\\
&\leq (s-2)\sum_{1\leq i_1 \neq i_2 \neq \cdots \neq i_{s-1} \leq N} 1_{i_1\leftrightarrow i_2} 1_{i_2\leftrightarrow i_3} \cdots 1_{i_{s-2}\leftrightarrow i_{s-1}} + \sum_{\substack{1\leq i_1 \neq i_2 \neq \cdots \neq i_{s-1} \leq N \\ i_s \notin \{i_1,\ldots,i_{s-1}\}}} 1_{i_1\leftrightarrow i_2} 1_{i_2\leftrightarrow i_3} \cdots 1_{i_{s-1}\leftrightarrow i_{s}}.
\end{align*}
Therefore, if $\widetilde{M}_{s,N} = \frac{1}{N}\,\esper[\sum_{1\leq i_1\neq i_2 \neq \cdots \neq i_s \leq N} 1_{i_1\leftrightarrow i_2}  \cdots 1_{i_{s-1}\leftrightarrow i_s}]$, then $\overline{M}_{s,N} \leq \widetilde{M}_{s,N}$, and by using the independence of the vectors $v_i$ and the inequality $\vol_X (B(v_i,L)) \leq \vol_{\R^{\dim X}} (B(0,L))$ which holds since a sscc symmetric space has (constant) positive curvature, we obtain:
$$\widetilde{M}_{s,N} \leq \left((s-2) + \frac{c(\dim X)\,\ell}{\vol(X)}\right)\widetilde{M}_{s-1,N} .$$
By induction and since $\widetilde{M}_{2,N} \leq \frac{c(\dim X)\,\ell}{\vol(X)}$, we conclude that for any $s \geq 2$,
\begin{equation*}
\overline{M}_{s,N} \leq \widetilde{M}_{s,N} \leq \prod_{t=0}^{s-2} \left( t + \frac{c(\dim X)\,\ell}{\vol(X)} \right).
\end{equation*}
By Fatou's lemma, since $\mu_N \rightharpoonup \mu$, the same upper bound holds for the even moments of $\mu$:
$$M_{2s} \leq \liminf_{N \to \infty} \overline{M}_{2s,N} \leq \prod_{t=0}^{2s-2} \left( t + \frac{c(\dim X)\,\ell}{\vol(X)} \right) = O((2s)^{2s}),$$
where the implied constant in the $O(\cdot)$ only depends on the space $X$ and $\ell$. By Carleman's criterion (see for instance \cite[Chapter 30]{Bil95}), $\mu$ is therefore determined by its moments.\bigskip

Fix $s\geq 2$. We want to prove that $\overline{M}_{s,N} = \esper[M_{s,N}] \to M_s$ and $\esper[(M_{s,N})^2] \to (M_s)^2$; this is equivalent to the convergence in $\leb^2$. By Equation \eqref{eq:traceexpansionmoments},
\begin{align*}
\overline{M}_{s,N} &= \esper[\text{number of cycles of length $s$ starting from a random root $r_N \in \{v_1,\ldots,v_N\}$}] \\
&=\sum_{\gamma_s \in \gbullet(s)} \proba[\pi_s(\Gamma_N,r_N) = \gamma_s]\,(\text{number of $s$-cycles in $\gamma_s$ that starts and ends at the root});\\
 M_s &= \sum_{\gamma_s \in \gbullet(s)} \proba[\pi_s(\Gamma,r_\infty) = \gamma_s]\,(\text{number of $s$-cycles in $\gamma_s$ that starts and ends at the root}),
\end{align*}
where $(\Gamma_\infty,r)$ is the rooted Poisson Boolean model appearing as the limit in Theorem \ref{thm:poisson_BS}. The Benjamini--Schramm convergence ensures that each term of the series for $\overline{M}_{s,N}$ converges towards the corresponding term for $M_s$, therefore, to prove that $\overline{M}_{s,N} \to M_s$, we only need a uniform domination on the terms of these series $\overline{M}_{s,N}$. If $\gamma_s$ is an element of $\gbullet(s)$ with $k+1$ vertices, then the number of rooted $s$-cycles in $\gamma_s$ is smaller than $k^{s-1}$. On the other hand, the probability that $\pi_s(\Gamma_N,r_N)$ has $k+1$ vertices is smaller than the probability that at least $k$ vertices $v_i$ fall in $B^{X}(o,s\,L_N)$, that is
\begin{align*}
\proba\!\left[\mathcal{B}\!\left(N-1,\frac{\vol(B^X(o,s\,L_N))}{\vol(X)} \right)\geq k\right] &\leq \proba\!\left[\mathcal{B}\!\left(N,\frac{\ell \,s^{\dim X}\,c(\dim X)}{N\,\vol(X)}\right) \geq k\right] \\
&\leq \sum_{l=k}^\infty \frac{1}{l!}\left(\frac{\ell\,s^{\dim X}\,c(\dim X)}{\vol (X)}\right)^l.
\end{align*}
Therefore, if $t = \frac{\ell\,s^{\dim X}\,c(\dim X)}{\vol (X)}$, then 
\begin{align*}
&\sum_{\substack{\gamma_s \in \gbullet(s)\\ |\gamma_s|=k+1}} \proba[\pi_s(\Gamma_N,r_N) = \gamma_s]\,(\text{number of $s$-cycles in $\gamma_s$ that starts and ends at the root}) \\
&\leq k^{s-1}\,\sum_{l=k}^\infty \frac{t^l}{l!} \leq \frac{k^{s-1}\,t^k\,\E^t}{k!},
\end{align*}
and these bounds are summable with $k$. This shows the desired domination of the terms of the series $\overline{M}_{s,N}$. The proof of the convergence $\esper[(M_{s,N})^2] \to (M_s)^2$ follows the same lines, using this time the asymptotic independence from Proposition \ref{prop:macroscopicindependence} and the identity
$$
\esper[(M_{s,N})^2] = \sum_{\gamma_s,\gamma_{s}' \in \gbullet(s)} \proba[\pi_s(\Gamma_N,r_N) = \gamma_s \text{ and } \pi_s(\Gamma_N,r_N') = \gamma_s']\,C(s,\gamma_s)\,C(s,\gamma_s'),
$$
where $C(s,\gamma)$ is the number of $s$-cycles starting and ending at the root in a finite rooted graph $\gamma$. Thus, $M_{s,N} \to_{\leb^2,N\to \infty} M_s$.
\end{proof}
\begin{remark}
Standard arguments from the theory of convergence of measures show that the convergence in probability of all the moments $M_{s,N} \to M_s$ is stronger than the convergence in probability $\nu_N \rightharpoonup \mu$. In particular, the proof above can be used to bypass the general arguments from Section \ref{subsec:benjaminischramm} that connect the local convergence of graphs to the weak convergence of their spectral measures.
\end{remark}
\bigskip

\section{From Poisson geometric graphs to graph functionals of irreducible characters}\label{sec:ART}
In this last section before the appendices, we focus on the case of a sscc Lie group $G$, and we investigate the connections between:
\begin{itemize}
     \item its representation theory and the formulas obtained in Section \ref{sec:gaussian} for the asymptotics of the Gaussian regime;
     \item the limiting measure $\mu=\mu(\dim G,\frac{\ell}{\vol(G)})$ exhibited in Section \ref{sec:poisson} and that drives the asymptotics of the Poissonian regime.
 \end{itemize} 
An important objective is to obtain more information on the limiting measure $\mu$, thereby answering the following questions:
\begin{itemize}
    \item Is $\mu$ compactly supported? What is the growth rate of the moments $M_s$ of the measure $\mu$?
    \item Does the measure $\mu$ admit atoms, or is it absolutely continuous with respect to the Lebesgue measure?
\end{itemize}
In the proof of Proposition \ref{prop:superpoissonlimit}, one can try to make Equation \eqref{eq:traceexpansionmoments} more precise, and to gather the cycles that one needs to count according to the identities of indices $i_1,\ldots,i_s$ that might occur. This theory leads to a combinatorial \emph{circuit expansion} of the moments $\overline{M}_{s,N}$ and of their limit $M_s$, which we develop in Section \ref{subsec:circuit}. This combinatorial expansion of the moments involves directed graphs endowed with a distinguished traversal, and these circuits can be reduced to yield non-directed graphs possibly with loops and with labels on their edges. For instance, we shall prove that we have an expansion
$$M_5=e^{(\ell)}_{\begin{tikzpicture}[scale=0.5]
\draw (0,0.5) circle (0.5);
\fill (0,0) circle (1mm);
\fill [white] (0,1) circle (2.7mm);
\draw (0,1) circle (2.7mm);
\draw (0,1) node {\tiny $5$};
\end{tikzpicture}} + 5\,e^{(\ell)}_{\begin{tikzpicture}[scale=0.5]
\draw (0,0.5) circle (0.5);
\draw (0,-1) circle (0.5);
\fill (0,0) circle (1mm);
\fill (0,-0.5) circle (1mm);
\fill [white] (0,-1.5) circle (2.7mm);
\fill [white] (0,1) circle (2.7mm);
\draw (0,1) circle (2.7mm);
\draw (0,-1.5) circle (2.7mm);
\draw (0,1) node {\tiny $3$};
\draw (0,-1.5) node {\tiny $2$};
\end{tikzpicture}} + 5\,e^{(\ell)}_{\begin{tikzpicture}[scale=0.5]
\draw (0,0.5) circle (0.5);
\fill (0,0) circle (1mm);
\fill [white] (0,1) circle (2.7mm);
\draw (0,1) circle (2.7mm);
\draw (0,1) node {\tiny $3$};
\end{tikzpicture}},$$
each term of this expansion being a monomial in the parameter $\ell$, and corresponding to the limit of a certain observable of the random geometric graph $\GEOM(N,L_N)$. As explained in the introduction, the actual computation of each of these terms should be performed by using the flat model $\Tor^{\dim G}$, where representation theory is encoded by classical Fourier series. However, if one stays with the non-flat sscc Lie group $G$, then the same computations shed light on several important phenomena from asymptotic representation theory, and this approach allows one to understand clearly the degeneration from the Gaussian to the Poissonian regime (see Sections \ref{subsec:onepoint} and \ref{subsec:twopoints}). Even more importantly, it leads to Conjecture \ref{conj:graphfunctional}, which we detail in Section \ref{subsec:graphfunctionals}. Although we do not see yet how to solve it, we consider it to be one of the main result of our study, which is why we devoted this section to its presentation.\medskip

Before presenting the combinatorial expansion of the moments $M_s$, we should warn the reader of two things:
\begin{enumerate}
    \item We do not plan to compute here explicitly all the moments $M_s$ (or at least to obtain some precise upper bounds on them). The arguments of Section \ref{subsec:circuit} and the replacement of the space $G$ by its flat model mostly reduce these calculations to a combinatorial problem of counting graphs with certain weights, but even with these reductions, these enumerations are by no means easy to perform. We hope to address this problem in a forthcoming work.
    \item Secondly, the phenomena from asymptotic representation theory in Sections \ref{subsec:onepoint}-\ref{subsec:graphfunctionals} quickly rely on certain algebraic arguments which are more advanced than before, namely, the theory of crystals and string polytopes of Lusztig--Kashiwara and Berenstein--Zelevinsky. The precise form of our Conjecture \ref{conj:graphfunctional} also relies on this theory. In order to ease the reading of this section, we shall try to present our arguments without insisting too much on these algebraic prerequisites; they can be found in an appendix at the end of this article (Section \ref{sec:crystal}), which is a short survey of some results regarding the crystals of representations of Lie groups.
\end{enumerate}\medskip

\subsection{Circuit expansion of the expected moments}\label{subsec:circuit} Until the end of Section \ref{sec:ART}, $G$ is a connected compact Lie group endowed with a bi-invariant Riemannian structure, and starting from Subsection \ref{subsec:onepoint} we shall assume it to be simple and simply connected. We consider the Poissonian geometric graph on $G$ with parameters $N$ and $L_N$ given by Equation \eqref{eq:poissonnorm}; in particular, the parameter $\ell$ is fixed from now on, and most of the quantities manipulated hereafter implicitly depend on it (for instance, the expectations $E_{H,T,N}$ defined below). In this paragraph, we give a combinatorial expansion of $\overline{M}_{s,N}=\esper[\nu_N(x^s)] = \mu_N(x^s)$ and of $M_s = \lim_{N \to \infty} M_{s,N}$ in terms of \emph{circuits}; this is the first step towards the calculation of the moments $M_s$ of the limiting measure $\mu$. We assume $s \geq 2$ since $\overline{M}_{1,N}=M_{1}=0$. If $h_N(x,y) = 1_{d(x,y)\leq L_N}$, then
$$
\overline{M}_{s,N} = \frac{1}{N} \sum_{i_1,i_2,\ldots,i_s} \esper\left[h_N(v_{i_1},v_{i_2})\,h_N(v_{i_2},v_{i_3}) \cdots h_N(v_{i_s},v_{i_1})\right],
$$
where the $v_i$'s are independent Haar distributed random variables on $G$, and the sums run over indices $i_j \in \lle 1,N\rre$ such that two consecutive indices $i_j$ and $i_{j+1}$ are never equal (by convention, the index following $i_s$ is $i_1$).
Now, an expectation $E_{i_1,i_2,\ldots,i_s} = \esper[h_N(v_{i_1},v_{i_2})\,h_N(v_{i_2},v_{i_3}) \cdots h_N(v_{i_s},v_{i_1})]$ only depends on the possible equalities of indices. For instance, when computing $\overline{M}_{4,N}$, we have:
\begin{align*}
\overline{M}_{4,N} &= (N-1)(N-2)(N-3)\, \esper\!\left[h_N(v_1,v_2)\,h_N(v_2,v_3)\,h_N(v_3,v_4)\,h_N(v_4,v_1)\right] \\
&\quad + 2\,(N-1)(N-2)\, \esper\!\left[h_N(v_1,v_2)\,h_N(v_2,v_1)\,h_N(v_1,v_3)\,h_N(v_3,v_1)\right] \\
&\quad + (N-1)\, \esper\!\left[h_N(v_1,v_2)\,h_N(v_2,v_1)\,h_N(v_1,v_2)\,h_N(v_2,v_1)\right].
\end{align*}
The first term corresponds to the case where all the indices $i_1,i_2,i_3,i_4$ are distinct; the second term corresponds to the identities $i_1=i_3$ or $i_2=i_4$; and the last term is when $i_1=i_3$ and $i_2=i_4$ simultaneously. We associate to these four cases the \emph{circuits} of Figure \ref{fig:fourcircuit}.

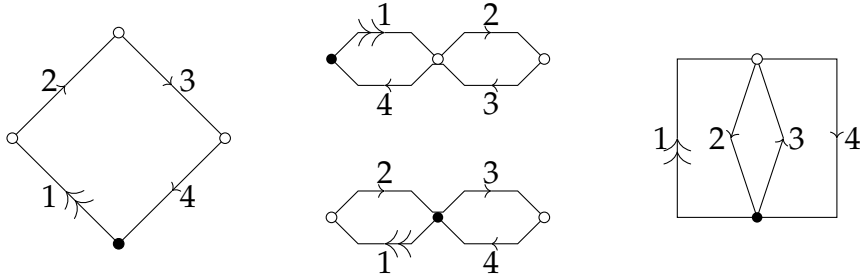
\begin{figure}[ht]
\begin{center}      
\begin{tikzpicture}[scale=0.7]
\draw [-{>[scale=2]>[scale=2]}] (0,0) -- (-1,1);
\draw (-1.3,0.9) node {$1$};
\draw (-1.3,3.1) node {$2$};
\draw (1.3,3.1) node {$3$};
\draw (1.3,0.9) node {$4$};
\draw (-1,1) -- (-2,2); 
\draw [->](-2,2) -- (-1,3);
\draw [->](-1,3) -- (0,4) -- (1,3);
\draw [->] (1,3) -- (2,2) -- (1,1);
\draw (1,1) -- (0,0);
\foreach \x in {(0,0), (-2,2), (0,4), (2,2)}
{\fill [white] \x circle (1mm); \draw \x circle (1mm);}
\fill (0,0) circle (1mm);
\begin{scope}[shift={(6,-0.5)}]
\draw [-{>[scale=2]>[scale=2]}] (0,1) -- (-0.5,0.5) -- (-1,0.5);
\draw [->] (-1,0.5) -- (-1.5,0.5) -- (-2,1) -- (-1.5,1.5) -- (-1,1.5);
\draw [->] (-1,1.5) -- (-0.5,1.5) -- (-0.1,1.1) -- (0.1,1.1) -- (0.5,1.5) -- (1,1.5);
\draw [->] (1,1.5) -- (1.5,1.5) -- (2,1) -- (1.5,0.5) -- (1,0.5);
\draw (1,0.5) -- (0.5,0.5) -- (0,1);
\foreach \x in {(-2,1),(2,1)}
{\fill [white] \x circle (1mm); \draw \x circle (1mm);}
\foreach \x in {(-2,4),(0,1)}
\fill \x circle (1mm);
\draw (-1,0.15) node {$1$};
\draw (-1,1.85) node {$2$};
\draw (1,1.85) node {$3$};
\draw (1,0.15) node {$4$};
\end{scope}
\begin{scope}[shift={(6,0.5)}]
\draw [-{>[scale=2]>[scale=2]}] (-2,3) -- (-1.5,3.5) -- (-1,3.5);
\draw (-1,3.5) -- (-0.5,3.5) -- (0,3);
\draw [->] (0,3) -- (0.5,3.5) -- (1,3.5);
\draw [->] (1,3.5) -- (1.5,3.5) -- (2,3) -- (1.5,2.5) -- (1,2.5);
\draw [->] (1,2.5) -- (0.5,2.5) -- (0.1,2.9) -- (-0.1,2.9) -- (-0.5,2.5) -- (-1,2.5);
\draw (-1,2.5) -- (-1.5,2.5) -- (-2,3);
\foreach \x in {(0,3),(2,3)}
{\fill [white] \x circle (1mm); \draw \x circle (1mm);}
\draw (-1,2.15) node {$4$};
\draw (-1,3.85) node {$1$};
\draw (1,3.85) node {$2$};
\draw (1,2.15) node {$3$};
\end{scope}
\begin{scope}[shift={(12,-1)},scale=1.5]
\draw [-{>[scale=2]>[scale=2]}] (0,1) -- (-1,1) -- (-1,2);
\draw [->] (-1,2) -- (-1,3) -- (0,3) -- (-0.33,2);
\draw [->] (-0.33,2) -- (0,1) -- (0.33,2);
\draw [->] (0.33,2) -- (0,3) -- (1,3) -- (1,2);
\draw (1,2) -- (1,1) -- (0,1);
\fill (0,1) circle (0.66mm);
\fill [white] (0,3) circle (0.66mm); \draw (0,3) circle (0.66mm);
\draw (-1.2,2) node {$1$};
\draw (-0.5,2) node {$2$};
\draw (0.5,2) node {$3$};
\draw (1.2,2) node {$4$};
\end{scope}
\end{tikzpicture}
\caption{The circuits for the calculation of $\overline{M}_{4,N}$.\label{fig:fourcircuit}}
\end{center}
\end{figure}
\vspace*{-3mm}

By circuit, we mean a directed graph $H$, possibly with multiple edges but without loops, endowed with a distinguished traversal $T$ that goes through each directed edge exactly once, and that is cyclic (the starting point is the same as the end point of the traversal). We identify two circuits $(H_1,T_1)$ and $(H_2,T_2)$ if there exists a graph isomorphism $\psi : C_1 \to C_2$ that is compatible with the traversals, that is $\psi(T_1)=T_2$. Given a circuit $(H,T)$ with $s$ edges and $k \leq s$ vertices, we associate to it the expectation of a function of $k$ independent points $v_1,\ldots,v_k$ on $G$:
$$E_{H,T,N} = \esper\!\left[\prod_{(i,j) \in T} h_N(v_i,v_j)\right].$$
Notice that $E_{H,T,N}$ only depends on $H$, and not on the particular traversal $T$, because each directed edge of $H$ appears exactly once in $T$.

\begin{lemma}\label{lem:circuit1}
For any $s \geq 0$, $\overline{M}_{s,N} = \sum_{(H,T)} (N-1)\cdots (N-|H|+1) \,E_{H,T,N}$, where the sum runs over the finite set of circuits with $s$ edges, and $k=|H|$ denotes the number of vertices in $H$.
\end{lemma}

\begin{proof}
We gather the terms of the sum $\overline{M}_{s,N} = \sum_{i_1,i_2,\ldots,i_s} * $\, according to the identities between the indices $i_1,\ldots,i_s$. Given a set of identities $I=\{i_j=i_k\}$, one can associate to it a circuit with $s$ edges by starting from the $s$-gon
\begin{center}
\begin{tikzpicture}[scale=1.8]
\draw (0:1) -- (30:1) -- (60:1) -- (90:1) -- (120:1) -- (150:1) -- (180:1) -- (210:1) -- (240:1) -- (270:1) -- (300:1) -- (330:1) -- (0:1);
\draw [-{>[scale=2]>[scale=2]}] (270:1) -- (240:1);
\foreach \x in {(0:1),(30:1),(60:1),(90:1),(120:1),(150:1),(180:1),(210:1),(240:1),(270:1),(300:1),(330:1)}
{ \fill [white] \x circle (0.5mm) ; \draw \x circle (0.5mm) ;}
\fill (270:1) circle (0.5mm);
\draw (270:1.18) node {$i_1$};
\draw (240:1.18) node {$i_2$};
\draw (210:1.18) node {$i_3$};
\draw (180:1.18) node {$i_4$};
\draw (150:1.18) node {$i_5$};
\draw (120:1.18) node {$i_6$};
\draw (90:1.18) node {$i_7$};
\draw (60:1.18) node {$i_8$};
\draw (30:1.18) node {$i_9$};
\draw (0:1.18) node {$i_{10}$};
\draw (330:1.18) node {$i_{11}$};
\draw (300:1.18) node {$i_{12}$};
\end{tikzpicture}
\end{center}
and by identifying the vertices $i_j$ and $i_k$ if the identity $i_j=i_k$ belong to the set $I$. For instance, the identities $i_2=i_5=i_7$, $i_3=i_{11}$ and $i_6=i_{12}$ give the  circuit of Figure \ref{fig:circuit}.
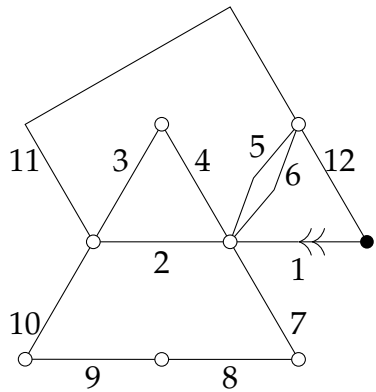
\begin{figure}[ht]
\begin{center}      
\begin{tikzpicture}[scale=1.8]
\draw [-{>[scale=2]>[scale=2]}] (0:1) -- (0:0.5);
\draw (0:0.5) -- (180:1) -- (120:1) -- (0,0) -- (70:0.5) -- (60:1) -- (50:0.5) -- (0,0) -- (300:1) -- (240:1) -- (210:1.73) -- (180:1) -- (150:1.73) -- (90:1.73) -- (60:1) -- (0:1);
\draw (0.5,-0.2) node {$1$};
\draw (-0.5,-0.15) node {$2$};
\draw (-0.8,0.6) node {$3$};
\draw (-0.2,0.6) node {$4$};
\draw (0.2,0.7) node {$5$};
\draw (0.46,0.5) node {$6$};
\draw (0.5,-0.6) node {$7$};
\draw (0,-1) node {$8$};
\draw (-1,-1) node {$9$};
\draw (-1.5,-0.6) node {$10$};
\draw (-1.5,0.6) node {$11$};
\draw (0.8,0.6) node {$12$};
\fill (0:1) circle (0.5mm);
\foreach \x in {(0,0),(180:1),(120:1),(60:1),(300:1),(240:1),(210:1.73)}
{ \fill [white] \x circle (0.5mm) ; \draw \x circle (0.5mm) ;}
\end{tikzpicture}
\caption{Circuit of length $s=12$ and with $k=8$ vertices, associated the set of identities $\{i_2=i_5=i_7,i_3=i_{11},i_6=i_{12}\}$.\label{fig:circuit}}
\end{center}
\end{figure}
\medskip

One can recover the identities from the corresponding circuit, and any circuit of length $s$ corresponds to a set of identities of indices, without identities $i_j=i_{j+1}$ since we do not allow loops. Moreover, the number of terms in the sum $\overline{M}_{s,N}$ corresponding to a circuit with $k$ vertices is $N(N-1)\cdots (N-k+1)$; and each term corresponding to a circuit $(H,T)$ is equal to $\frac{1}{N}\,E_{H,T,N}$. This ends the proof of the expansion of $\overline{M}_{s,N}$ over circuits.
\end{proof}
\medskip

The calculation of the quantities $E_{H,T,N}$ involves the operation of \emph{reduction of circuit}. Let $(H,T)$ be a circuit of length $s$. Its reduction is the labeled undirected graph which is allowed to be disconnected and to have loops, and which is obtained by performing the following operations:
\begin{itemize}
    \item forgetting the orientation of the edges of $H$;
    \item replacing any multiple edge by a single edge;
    \item putting a label $1$ on each of the (single) edges;
    \item cutting the graph at each of its cut vertices (also called articulation points), replacing a configuration
    \begin{center}
    \begin{tikzpicture}[scale=1]
    \draw (-2,-1) .. controls (-1,-1) and (1,1) .. (2,1);
    \draw (-2,1) .. controls (-1,1) and (1,-1) .. (2,-1);
    \fill (0,0) circle (1mm);
    \draw (2,-1) arc (270:450:1);
    \draw (-2,1) arc (90:270:1);
    \draw (-2,0) node {$L_1$};
    \draw (2,0) node {$L_2$};
    \end{tikzpicture}
    \end{center}
    with $L_1 \neq \emptyset$ and $L_2 \neq \emptyset$ by 
    \begin{center}
    \begin{tikzpicture}[scale=1]
    \draw (-2,-1) .. controls (-1,-1) .. (0,0);
    \draw (-2,1) arc (90:270:1);
    \draw (-2,0) node {$L_1$};
    \draw (-2,1) .. controls (-1,1) .. (0,0);
    \fill (0,0) circle (1mm);
    \begin{scope}[shift={(2,0)}]
    \draw (2,-1) .. controls (1,-1) .. (0,0);
    \draw (2,-1) arc (270:450:1);
    \draw (2,0) node {$L_2$};
    \draw (2,1) .. controls (1,1) .. (0,0);
    \fill (0,0) circle (1mm);
    \end{scope}
    \end{tikzpicture}
    \end{center}
    
    \item in the resulting connected components, removing recursively each vertex of degree $2$, replacing a configuration
    \begin{center}
    \begin{tikzpicture}[scale=1]
    \draw (-2,0) -- (2,0);
    \draw [dashed] (-2.5,0) -- (-2,0);
    \draw [dashed] (2.5,0) -- (2,0);
    \fill (0,0) circle (1mm);
    \fill (2,0) circle (1mm);
    \fill (-2,0) circle (1mm);
    \foreach \x in {(-1,0),(1,0)}
{\fill [white] \x circle (1.8mm); \draw \x circle (1.8mm);}
    \draw (-1,0) node {\footnotesize $a$};
    \draw (1,0) node {\footnotesize $b$};
    \end{tikzpicture}
    \end{center}
    by 
    \begin{center}
    \begin{tikzpicture}[scale=1]
    \draw (-2,0) -- (2,0);
    \draw [dashed] (-2.5,0) -- (-2,0);
    \draw [dashed] (2.5,0) -- (2,0);
    \fill (2,0) circle (1mm);
    \fill (-2,0) circle (1mm);
    \fill [white] (0,0) circle (4.4mm); \draw (0,0) circle (4.4mm);
    \draw (0,0) node {\footnotesize $a+b$};
    \draw (2.7,0) node {$.$};
    \end{tikzpicture}
    \end{center}

    \item finally, replacing the connected components \,
    \raisebox{-1.5mm}{\begin{tikzpicture}[scale=1]
    \draw (0,0) -- (2,0);
    \fill [white] (1,0) circle (2mm); \draw (1,0) circle (2mm);
    \draw (1,0) node {\footnotesize $1$};
    \fill (0,0) circle (1mm);
    \fill (2,0) circle (1mm);
    \end{tikzpicture}}\,
    by loops \,\raisebox{-4mm}{\begin{tikzpicture}[scale=1]
    \fill (0,0) circle (1mm);
    \draw (0.5,0) circle (0.5) ;
    \fill [white] (1,0) circle (1.8mm); \draw (1,0) circle (1.8mm); \draw (1,0) node {\footnotesize $2$};
    \end{tikzpicture}}\,.
\end{itemize}\medskip

\noindent The fourth operation in the algorithm of reduction splits the graph in its so-called \emph{biconnected components}: they are connected components which remain connected if one removes one vertex. Notice that the operation of reduction:
\begin{itemize}
    \item can send many distinct circuits to the same reduction;
    \item can create two kinds of connected components: \begin{itemize}
        \item labeled loops based at one single vertex and with a label greater than $2$;
        \item and connected loopless graphs on at least two vertices, all of them being at least of degree $3$.
    \end{itemize}
 \end{itemize}
 
\begin{example}
 The reduction of the circuit of Figure \ref{fig:circuit} appears in Figure \ref{fig:reducedcircuit}. Similarly, the reductions of the four circuits of length $4$ are drawn in Figure \ref{fig:reduction4}, with the middle one that has multiplicity $2$ (as well as two connected components).
\begin{figure}[ht]
\begin{center}      
\begin{tikzpicture}[scale=2]
\draw (0.25,0) -- (-1,0) -- (-0.8,0.25) -- (-0.2,0.25) -- (0,0) -- (60:1) -- (0,0) -- (-0.2,-0.25) -- (-0.8,-0.25) -- (-1,0) -- (150:1.73) -- (60:1) -- (45:0.97) -- (0.25,0);
\foreach \x in {(0,0),(180:1),(60:1)}
{ \fill \x circle (0.5mm) ;}
\foreach \x in {(-0.5,0.25),(-0.5,0),(-0.5,-0.25),(60:0.5),(-0.5,0.866),(0.5,0.4)}
{\fill [white] \x circle (0.9mm); \draw \x circle (0.9mm);}
\draw (-0.5,0.25) node {\footnotesize $2$};
\draw (-0.5,-0.25) node {\footnotesize $4$};
\draw (-0.5,0) node {\footnotesize $1$};
\draw (60:0.5) node {\footnotesize $1$};
\draw (-0.5,0.866) node {\footnotesize $1$};
\draw (0.5,0.4) node {\footnotesize $2$};
\end{tikzpicture}
\caption{Reduction of the circuit of Figure \ref{fig:circuit}.\label{fig:reducedcircuit}}
\end{center}
\end{figure}
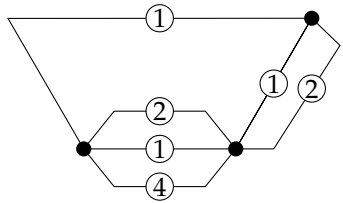
\begin{figure}[ht]
\begin{center}      
\begin{tikzpicture}[scale=1]
\draw (0,0.5) circle (0.5);
\draw (3,0.5) circle (0.5);
\draw (3,-1) circle (0.5);
\draw (6,0.5) circle (0.5);
\foreach \x in {(0,0),(3,0),(3,-0.5),(6,0)}
{ \fill \x circle (1mm) ;}
\foreach \x in {(0,1),(3,1),(3,-1.5),(6,1)}
{ \fill [white] \x circle (1.8mm) ; \draw \x circle (1.8mm);}
\draw (0,1) node {\footnotesize $4$};
\draw (3,1) node {\footnotesize $2$};
\draw (3,-1.5) node {\footnotesize $2$};
\draw (6,1) node {\footnotesize $2$};
\end{tikzpicture}
\caption{Reductions of the circuits of length $4$.\label{fig:reduction4}}
\end{center}
\end{figure}
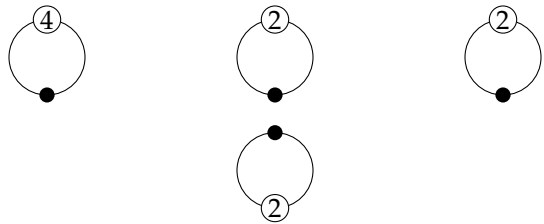 
\end{example}

If $R$ is the reduction of a circuit $(H,T)$ with $|H|=k$ vertices, then one has the identity:
\begin{equation}
     k-1 = k'-c + \sum_{e\text{ labeled edge of }R} (l_e-1), \label{eq:numbervertices}
\end{equation}
where $l_e$ is the label of an edge $e$ in $R$; $k'$ is the number of vertices of $R$; and $c$ is the number of connected components of $R$. Equation \eqref{eq:numbervertices} shows readily that $R \mapsto k-1$ is an additive map with respect to connected components of reduced circuits.

\begin{lemma}\label{lem:circuit2}
The expectation $E_{H,T,N}$ only depends on the reduction of the circuit $(H,T)$.
\end{lemma}
\begin{proof}
Consider the expectation $E_{H,T,N} = \esper[\prod_{(i,j) \in T} h_N(v_i,v_j)]$. \begin{itemize}
    \item Since $h_N$ is a symmetric kernel, it does not depend on the orientation of the edges in $H$; this allows the first step in the reduction of the circuit.

    \item Since $h_N(x,y) \in \{0,1\}$, $(h_N(x,y))^m = (h_N(x,y))$ for any $m \geq 1$, so one can replace multiple edges by simple edges in the graph.

    \item The factorisation of $E_{H,T,N}$ on the biconnected components of the graph is a consequence of the independence of the vertices $v_i$, and of the invariance of the function $h_N$ by action of the group $G$ on its two variables. Indeed, suppose that the labeled graph $L$ obtained after the three first steps of the reduction has two components $L_1$ and $L_2$ which only share one cut vertex $v_c$. Then, there are two disjoint sets of vertices $\{v_{i_1},\ldots,v_{i_r}\}$ and $\{v_{j_1},\ldots,v_{j_s}\}$ and two functions $f_1$ and $f_2$ such that
    $$E_{H,T,N} = \esper[f_1(v_{i_1},\ldots,v_{i_r},v_c)\,f_2(v_c,v_{j_1},\ldots,v_{j_s})].$$
    Moreover, the two functions $f_1$ and $f_2$ are products of functions $h_N(x,y)$, with $x$ and $y$ in the set of variables of $f_1$ or $f_2$. This implies that for any $g \in G$, $f_2(v_c,v_{j_1},\ldots,v_{j_s})=f_2(gv_c,gv_{j_1},\ldots,gv_{j_s})$. We take $g$ uniform under the Haar measure, and we set $v_c'=gv_c$ and $v_{j_k}'=gv_{j_k}$. Then,
    \begin{align*}
    E_{H,T,N} &= \int_G \esper[f_1(v_{i_1},\ldots,v_{i_r},v_c)\,f_2(gv_c,gv_{j_1},\ldots,gv_{j_s})] \DD{g} \\
    &= \esper[f_1(v_{i_1},\ldots,v_{i_r},v_c)]\,\esper[f_2(v_c',v_{j_1}',\ldots,v_{j_s}')]
    \end{align*}
    with two set of variables that are now independent.

    \item Take now a connected component after the factorisation in biconnected components. If one has in such a graph a sequence of $r$ edges 
\begin{center}
\begin{tikzpicture}[scale=1]
\foreach \x in {(0,0),(0.8,0),(1.6,0),(2.4,0),(3.2,0)}
\fill \x circle (0.7mm);
\draw (-0.5,0) node {$\cdots$};
\draw (3.7,0) node {$\cdots$};
\draw (0,0) -- (1.6,0);
\draw [dotted] (1.6,0) -- (2.4,0);
\draw (3.2,0) -- (2.4,0);
\end{tikzpicture}\,,
\end{center} 
then it corresponds to a product $h_N(v_{a},v_{b_1})\, h_N(v_{b_1},v_{b_2}) \cdots h_N(v_{b_{r-1}},v_c)$ in the expectation, with the independent random variables $v_{b_1},\ldots,v_{b_{r-1}}$ that do not appear anywhere else in the product in the expectation $E_{H,T,N}$; clearly one can encode this term by a labeled edge \,
\raisebox{-1.5mm}{\begin{tikzpicture}[scale=1]
\foreach \x in {(0,0),(1,0)}
\fill \x circle (0.7mm);
\draw (0,0) -- (1,0);
\fill [white] (0.5,0) circle (1.8mm);
\draw (0.5,0) circle (1.8mm);
\draw (0.5,0) node {\footnotesize $r$};
\end{tikzpicture}}\,. This is equivalent to the second last rule of reduction.

    \item The last step is a convention that will allow us to have only two kinds of connected components, namely, the labeled loops, and the labeled loopless graphs on at least two vertices which all have degree larger than $3$. It amounts to the obvious identity $\esper[h_N(v_i,v_j)] = \esper[h_N(v_i,v_j)\,h_N(v_j,v_i)]$. \qedhere
\end{itemize}  
\end{proof}

In the following, we denote $R(H,T)$ the reduction of a circuit $(H,T)$. Note that $R(H,T)$ depends only on $H$ (and even the underlying undirected graph). However, since we shall consider sums over circuits, it is more convenient to recall each time the pair $(H,T)$. The previous discussion leads to:
\begin{theorem}[Circuit expansion]\label{thm:circuitexpansion}
\begin{enumerate}
    \item Combinatorial expansion: for any $s \geq 0$, we have
$$\overline{M}_{s,N} = \sum_{(H,T)} (N-1)\cdots (N-|H|+1) \,E_{R(H,T),N},$$
where the sum runs over the finite set of circuits with $s$ edges, and where $E_{R(H,T),N} = E_{H,T,N}$ only depends on the circuit reduction of $(H,T)$.
\item Factorisation: if $R=R_1 \sqcup R_2 \sqcup \cdots \sqcup R_c$, then $E_{R,N} = \prod_{i=1}^c E_{R_i,N}$. 
\item Asymptotics: for any reduced circuit $R$ with parameter $k$ given by Equation \eqref{eq:numbervertices}, there exists a positive real number $e_R$ depending only on $\dim G$ such that
$$\lim_{N \to \infty} N^{k-1}\,E_{R,N} = e_R\,\left(\frac{\ell}{\vol(G)}\right)^{k-1}.$$
Therefore,
$$M_s = \sum_{(H,T)} e_{R(H,T)}\,\left(\frac{\ell}{\vol(G)}\right)^{k-1}.$$
\end{enumerate}
\end{theorem} 

\begin{proof}
The first part of the theorem comes from the combination of Lemmas \ref{lem:circuit1} and \ref{lem:circuit2}, and the second part corresponds to the fourth step of the algorithm of reduction of circuits. It remains to examine the asymptotics of $E_{R,N}$ as $N$ goes to infinity. We denote $N^{\downarrow k} = N(N-1)\cdots(N-k+1)$ a falling factorial. Given a reduced circuit $R=R(H,T)$ and a rooted graph $\gamma \in \gbullet$, we call embedding of the circuit $(H,T)$ into $\gamma$ an injective morphism of graphs $e : V(H) \to V(\gamma)$ which sends the starting and ending point of the traversal $T$ to the root of $\gamma$. We then have:
\begin{align*}
(N-1)^{\downarrow k-1}\,E_{H,T,N} &= \esper[\text{number of embeddings of $(H,T)$ in the random rooted graph $(\Gamma_N,r_N)$}] \\
&= \sum_{\gamma_d \in \gbullet(d)} \proba[\pi_d(\Gamma_N,r_N) = \gamma_d]\,\card \{\text{embeddings of $(H,T)$ into $\gamma_d$}\},
\end{align*}
where $d$ is a sufficiently large integer, namely, larger than the diameter of the circuit $(H,T)$. By Theorem \ref{thm:poisson_BS}, all the terms of this series converge, and the same kind of domination as in the proof of Proposition \ref{prop:superpoissonlimit} enables one to exchange the series and the limits. Hence,
\begin{align*}
\lim_{N \to \infty} N^{k-1}\,E_{R(H,T),N} &= \lim_{N \to \infty} (N-1)^{\downarrow k-1}\,E_{H,T,N}\\
&=\sum_{\gamma_d \in \gbullet(d)} \proba[\pi_d(\Gamma_\infty,r) = \gamma_d]\,\card \{\text{embeddings of $(H,T)$ into $\gamma_d$}\}\\
&= \esper[\text{number of embeddings of $(H,T)$ in the random rooted graph $(\Gamma_\infty,r)$}],
\end{align*}
where $(\Gamma_\infty,r)$ is the rooted Poisson Boolean model described in Theorem \ref{thm:poisson_BS}. Denote $e_{(H,T)}^{(\ell)}$ the limit that we have obtained; the last thing that remains to be shown is that $e_{(H,T)}^{(\ell)}$ depends polynomially on the intensity $\frac{\ell}{\vol(G)}$ of the Poisson point process underlying $(\Gamma_\infty,r)$. However, the expected number of embeddings can be rewritten as the integral of the $(k-1)$-th factorial moment measure $M^{(k-1)}$ of the Poisson point process $\mathcal{P}(\frac{\ell}{\vol G})$ against a certain Borel measurable subset in $(\R^{\dim G})^{k-1}$:
$$e_{(H,T)}^{(\ell)} = \int_{(\R^{\dim G})^{k-1}} \left(\prod_{(a\to b) \in T} 1_{d(x_a,x_b)\leq 1}\right) M^{k-1}(\!\DD{x_1}\,\cdots \DD{x_{k-1}})$$
where we convene that the vertices in $H$ are labelled by the integers in $\lle 0,k-1\rre$, and that $x_0=0$ in $\R^{\dim G}$. We refer to \cite[Chapter 5]{DVJ03} for details on the notion of factorial moment measure; it is well known that for the Poisson point process with intensity $\mu$, the $r$-th factorial moment measure is simply $\mu^{\otimes r}$. This proves the dependence stated in the theorem, since we have here
\begin{equation*}
  M^{(k-1)} = \left(\frac{\ell}{\vol(G)}\,\mathrm{Leb}\right)^{\otimes(k-1)}.\qedhere
 \end{equation*}
\end{proof}
\medskip

\subsection{Asymptotic contribution of a reduced circuit which is a loop}\label{subsec:onepoint}
The remainder of this section is devoted to the study of the connection between:
\begin{itemize}
     \item the coefficients $e_{(H,T)}^{(\ell)}$, which we sometimes also index by the corresponding reduced circuits $R=R(H,T)$ and denote $e_R^{(\ell)}$;
     \item the representation theory of the group $G$, which from now on will be assumed to be sscc.
 \end{itemize}
Although we know that $e_R^{(\ell)}=e_R^{(1)}\,\ell^{k-1}$, in the following it will be convenient to keep the coefficient $\ell$: it will enable one to keep track of the dimensions of various rescalings that we shall perform. The existence of the limits $e_R^{(\ell)}$ is strongly related to some interesting results or conjectures in asymptotic representation theory. To understand how the representation theory of $G$ drives the degeneration from the Gaussian to the Poissonian regime, we start by examining the case of a circuit $(H,T)$ which is a simple cycle of length $k\geq 2$, and thus has reduction 
\begin{center}
\begin{tikzpicture}[scale=1]
\draw (-2.15,0.5) node {$R=R(H,T)=$};
\draw (0.8,0.5) node {.};
\draw (0,0.5) circle (0.5);
\fill (0,0) circle (1mm);
\fill [white] (0,1) circle (1.8mm);
\draw (0,1) circle (1.8mm);
\draw (0,1) node {\footnotesize $k$};
\end{tikzpicture}
\end{center}
We have $E_{H,T,N} = \esper[h_N(v_1,v_2)\,h_N(v_2,v_3)\cdots h_N(v_k,v_1)]$, where the $v_i$'s are independent Haar distributed random variables on $G$. If $Z_{L_N}(x) = h_N(e_G,x)$, then by using the invariance of distances by the action of $G$, we can rewrite
\begin{align*}
E_{H,T,N} &= \esper\!\left[Z_{L_N}(v_1(v_2)^{-1})\,Z_{L_N}(v_2(v_3)^{-1})\cdots Z_{L_N}(v_{k-1}(v_k)^{-1})\, Z_{L_N}(v_k(v_1)^{-1})\right] \\
&= \esper\!\left[(Z_{L_N})^{*(k-1)}(v_1(v_k)^{-1}) \, Z_{L_N}(v_k(v_1)^{-1})\right] = \scal{(Z_{L_N})^{*(k-1)}}{Z_{L_N}},
\end{align*}
where the scalar product is taken in the convolution algebra $\leb^2(G,\!\DD{g})$ (and even in the subalgebra $\leb^2(G)^G$). In this Hilbert space, we have the decompositions  
\begin{align*}
Z_{L_N} &= \sum_{\lambda \in \hatG} \scal{\ch^\lambda}{Z_{L_N}}\,\ch^\lambda\qquad;\qquad
(Z_{L_N})^{*(k-1)} = \sum_{\lambda \in \hatG} \frac{(\scal{\ch^\lambda}{Z_{L_N}})^{k-1}}{(d_\lambda)^{k-2}}\,\ch^\lambda
\end{align*}
since $\ch^\lambda * \ch^\mu = \frac{\delta_{\lambda,\mu}}{d_\lambda}\,\ch^\lambda$ for any irreducible representations $\lambda,\mu \in \hatG$. Therefore,
$$E_{H,T,N} = \sum_{\lambda \in \hatG} \frac{(\scal{\ch^\lambda}{Z_{L_N}})^{k}}{(d_\lambda)^{k-2}} = \sum_{\lambda \in \hatG} (d_\lambda)^2\,(c_\lambda)^k$$
since $\scal{\ch^\lambda}{Z_{L_N}} = \int_{G} \overline{\ch^\lambda(g)}\,Z_{L_N}(g)\DD{g} = \overline{d_\lambda\,c_\lambda} = d_\lambda\,c_\lambda$. Set 
\begin{equation}
C_{\lambda,N}= d_\lambda\,c_\lambda = \frac{1}{\vol(\tlie/\tlie_{\Z})}\,\left(\frac{L_N}{\sqrt{2\pi}}\right)^{\!d} \sum_{w \in W} \eps(w)\,\tildeJ_{\R\Omega}(L_N\,(\lambda+\rho-w(\rho)))\label{eq:clambdan}
\end{equation}
where $d=\rank(G)$ denotes as in Section \ref{sec:gaussian} the rank of $G$. We put an index $N$ on $C_{\lambda,N}$ to insist on the dependence on $N$ of the Fourier coefficients of $Z_{L_N}$. We have thus shown:
$$E_{H,T,N} = \sum_{\lambda \in \hatG} \frac{(C_{\lambda,N})^k}{(d_\lambda)^{k-2}}.$$
As $N$ goes to infinity, this series will transform into a Riemann sum and converge towards an integral involving Bessel functions. Let us start by evaluating the asymptotics of $C_{\lambda,N}$ when $N$ grows and the parameter $x = L_N\,(\lambda+\rho)$ is fixed in the Weyl chamber $C$. We shall use the following properties of the function $\tildeJ_{\R\Omega}$:
\begin{itemize}
    \item it is a smooth function on $\R\Omega$ with maximum value $$\tildeJ_{\R\Omega}(0) = \frac{1}{2^{\frac{\rank(G)}{2}}\,\Gamma(1+\frac{\rank(G)}{2})};$$ 
    \item it is invariant by rotations;
    \item its asymptotics are (see \cite[Proposition 9.8.7]{Coh07})
    $$\tildeJ_{\R\Omega}(x) \simeq_{\|x\|\to \infty} \sqrt{\frac{2}{\pi}}\,\frac{1}{\|x\|^{\frac{\rank(G)+1}{2}}}\,\cos\left(\|x\| - \frac{(\rank(G)+1)\,\pi}{4}\right).$$
 \end{itemize}
 In particular, any power $k \geq 2$ of the function $\tildeJ_{\R\Omega}$ is integrable on $\R\Omega$. For $x \in C$, set
 $$\Delta_{L_N}(x) = \sum_{w \in W} \eps(w)\,\tildeJ_{\R\Omega}(x - L_N\,w(\rho)).$$
 On the other hand, for any (positive) root $\alpha$ and any smooth function $f$ on $\R\Omega$, we define the partial derivative 
 $$(\partial_\alpha f)(x) = \lim_{\eta \to 0}\left( \frac{f(x+\eta\alpha)-f(x)}{\eta}\right)=\lim_{\eta \to 0}\left( \frac{f(x+\frac{\eta\alpha}{2})-f(x-\frac{\eta\alpha}{2})}{\eta}\right).$$

\begin{lemma}\label{lem:partialderivative}
Set $\partial_{\Phi_+}=\prod_{\alpha \in \Phi_+}\partial_\alpha$, and denote $|\Phi_+| = \card\, \Phi_+$. For $x \in C+L_N\,|\Phi_+|\,\rho$, we have the estimate
$$\Delta_{L_N}(x) = (-L_N)^{|\Phi_+|}\,\left((\partial_{\Phi_+}\tildeJ_{\R\Omega})(x) + L_N\,K_{\R\Omega}(L_N,x)\right),$$
where $K_{\R\Omega}(L_N,x)$ is a function on the translated Weyl chamber $C+L_N\,|\Phi_+|\,\rho$ such that, uniformly for $L_N$ small enough,
$$|K_{\R\Omega}(L_N,x)| \leq K\,\min\left(1,\frac{1}{\|x\|^{\frac{\rank(G)+1}{2}}} \right).$$
\end{lemma}
\begin{proof}
Recall that in the group algebra of the space of weights $\R\Omega$, we have the identity
$$\prod_{\alpha \in \Phi_+} (\E^{\frac{\alpha}{2}}-\E^{-\frac{\alpha}{2}}) = \sum_{w \in W} \eps(w)\,\E^{w(\rho)},$$
see \cite[Proposition 22.7]{Bump13}. Therefore, if an element $\omega \in \R\Omega$ acts on smooth functions $f$ by 
the operator $(\E^\omega f)(x) = f(x-\omega)$, then we can write:
$$\Delta_{L_N}(x) = \left(\left(\prod_{\alpha \in \Phi_+} (\E^{\frac{L_N\alpha}{2}}-\E^{-\frac{L_N\alpha}{2}}) \right) \tildeJ_{\R\Omega}\right)(x).$$
For $k \geq 0$, let $\mathscr{F}(C+k\rho)$ be the set of smooth functions $f$ on the translated Weyl chamber $C+k\rho$, such that any partial derivative $((\prod_{i=1}^{r} \partial_{\delta_i})f)(x)$ is bounded by 
$$K(\delta_1,\ldots,\delta_r)\,\min\left(1,\frac{1}{\|x\|^{\frac{\rank(G)+1}{2}}} \right).$$
We claim that:
\begin{enumerate}
    \item The modified Bessel function $\tildeJ_{\R\Omega}$ belongs to the class $\mathscr{F}(C)$.
    \item If $f$ belongs to $\mathscr{F}(C+k\rho)$, then for any positive root $\alpha$, and any $x \in C+(k+1)\rho$
    $$\left(\left(\E^{\frac{L_N\alpha}{2}}-\E^{-\frac{L_N\alpha}{2}}) \right)f\right)(x) = -L_N\,((\partial_\alpha\,f)(x) + L_N\,g(L_N,x)),$$
    where $g(L_N,x)$ belongs to $\mathscr{F}(C+(k+1)\rho)$, and where the bounds on the partial derivatives $((\prod_{i=1}^{r} \partial_{\delta_i})g(L_N,\cdot))$ are uniform in $L_N$ (for $L_N$ small enough).
\end{enumerate}
The first claim follows from the asymptotic estimate of Bessel functions 
$$\frac{J_\beta(x)}{x^\beta}= \tildeJ_\beta(x) \simeq_{x \to +\infty} \sqrt{\frac{2}{\pi}}\,\frac{1}{x^{\beta+\frac{1}{2}}}\,\cos\left(x - (1+2\beta)\frac{\pi}{4}\right)$$
and from the recurrence relation $(\tildeJ_\beta)'(x) = -x\,\tildeJ_{\beta+1}(x)$. The second claim is obtained by a Taylor expansion of the function $f$ around $x$. To ensure that one can use it, one needs to translate $x$ a bit further inside the Weyl chamber, which is why the estimate holds only in $C+(k+1)\rho$ if $f \in \mathscr{F}(C+k\rho)$. By combining the two claims and taking $|\Phi_+|$ discrete derivatives of $\tildeJ_{\R\Omega}$, one gets the result of the lemma.
\end{proof}
\medskip

We set 
$$e_{\begin{tikzpicture}[scale=0.5]
\draw (0,0.5) circle (0.5);
\fill (0,0) circle (1mm);
\fill [white] (0,1) circle (2.5mm);
\draw (0,1) circle (2.5mm);
\draw (0,1) node {\tiny $k$};
\draw (1.1,0.5) node {$,N$};
\end{tikzpicture}}^{(\ell)}=N^{k-1}\,E_{H,T,N}=\sum_{\lambda \in \hatG} \frac{N^{k-1}\,(C_{\lambda,N})^k}{(d_\lambda)^{k-2}}$$
and for $x \in C$, $\delta(x) = \prod_{\alpha \in \Phi_+} \frac{\scal{x}{\alpha}}{\scal{\rho}{\alpha}}$. Weyl's dimension formula proves that if $x=L_N\,(\lambda+\rho)$, then $d_\lambda = (L_N)^{|\Phi_+|}\,\delta(x)$. Therefore, by using also the relation $d+2\,|\Phi_+| = \dim G$ which follows from the decomposition of the adjoint representation of $G$ in root subspaces, we obtain
\begin{align*}
&\frac{N^{k-1}\,(C_{\lambda,N})^k}{(d_\lambda)^{k-2}} \\
&= \frac{(-1)^{k\,|\Phi_+|}}{(2\pi)^{\frac{kd}{2}}\,(\vol(\tlie/\tlie_{\Z}))^k\,(\delta(x))^{k-2}} \,N^{k-1} \,(L_N)^{k\,d+(2k-2)\,|\Phi_+|} \left((\partial_{\Phi_+}\tildeJ_{\R\Omega})(x) + L_N\,K_{\R\Omega}(L_N,x)\right)^k \\
&= \frac{(-1)^{k\,|\Phi_+|}\,\ell^{k-1}}{(2\pi)^{\frac{kd}{2}}\,(\vol(\tlie/\tlie_{\Z}))^k\,(\delta(x))^{k-2}} \left((\partial_{\Phi_+}\tildeJ_{\R\Omega})(x) + L_N\,K_{\R\Omega}(L_N,x)\right)^k\,(L_N)^{d}
\end{align*}
for any $x = L_N\,(\lambda+\rho)$ falling into the translated Weyl chamber $C+L_N\,|\phi_+|\,\rho$. Therefore,
$$e_{\begin{tikzpicture}[scale=0.5]
\draw (0,0.5) circle (0.5);
\fill (0,0) circle (1mm);
\fill [white] (0,1) circle (2.5mm);
\draw (0,1) circle (2.5mm);
\draw (0,1) node {\tiny $k$};
\draw (1.1,0.5) node {$,N$};
\end{tikzpicture}}^{(\ell)}= \frac{(-1)^{k\,|\Phi_+|}\,\ell^{k-1}}{(2\pi)^{\frac{kd}{2}}\,(\vol(\tlie/\tlie_{\Z}))^k} \sum_{\substack{x \in C\\x=L_N(\lambda + \rho)}} \frac{((\partial_{\Phi_+}\tildeJ_{\R\Omega})(x))^k}{(\delta(x))^{k-2}}\,(L_N)^d + \text{remainder},
$$
with a remainder that is a $O(L_N)$, because it consists of:
\begin{itemize}
    \item the contribution of the weights $\lambda$ that are in the boundary $C \setminus (C+|\Phi_+|\rho)$ of the Weyl chamber;
    \item and terms proportional to $(L_N)^t\,K_{\R\Omega}(L_N,x)^t\,((\partial_{\Phi_+}\tildeJ_{\R\Omega})(x))^{k-t}$, with $t \geq 1$.
 \end{itemize} 
We leave the reader to check that these contributions can indeed be summed and yield a $O(L_N)$; this relies on estimates of Bessel functions similar to those previously given. Then, we are left with a standard Riemann sum over the lattice $L_N(C \cap \Z\Omega)$, whose points correspond to domains of volume $(L_N)^d\,\vol(\R\Omega/\Z\Omega)$.
We have a duality of lattices and $$\vol(\R\Omega/\Z\Omega)=\frac{1}{\vol(\tlie/\tlie_\Z)}.$$ 
We conclude:
 
\begin{theorem}\label{thm:asymptoticsonevertex}
For any $k \geq 2$, we have
\begin{align*}
e_{\begin{tikzpicture}[scale=0.5]
\draw (0,0.5) circle (0.5);
\fill (0,0) circle (1mm);
\fill [white] (0,1) circle (2.5mm);
\draw (0,1) circle (2.5mm);
\draw (0,1) node {\tiny $k$};
\end{tikzpicture}}^{(\ell)} & =\lim_{N \to \infty} e_{\begin{tikzpicture}[scale=0.5]
\draw (0,0.5) circle (0.5);
\fill (0,0) circle (1mm);
\fill [white] (0,1) circle (2.5mm);
\draw (0,1) circle (2.5mm);
\draw (0,1) node {\tiny $k$};
\draw (1.1,0.5) node {$,N$};
\end{tikzpicture}}^{(\ell)} = \left( \frac{\ell}{\vol(\tlie/\tlie_\Z)} \right)^{k-1} \int_C \left(\frac{(-1)^{|\Phi_+|}\,(\partial_{\Phi_+}\tildeJ_{\R\Omega})(x)}{(2\pi)^{d/2}}\right)^k \,\frac{\!\DD{x}}{(\delta(x))^{k-2}},
\end{align*}
where $\!\DD{x}$ is the Lebesgue measure on $\R\Omega$ associated to the scalar product of weights defined in Section \ref{subsec:weightlattice}; $d=\rank(G)$; and $\delta(x) = \prod_{\alpha \in \Phi_+} \frac{\scal{x}{\alpha}}{\scal{\rho}{\alpha}}$.
\end{theorem}
\noindent More precisely, the difference between $N^{k-1}\,E_{H,T,N}$ and its limit is a $O(L_N)$, with a constant in the $O(\cdot)$ that only depends on $G$ and $s$. In the following, since we shall always deal with the partial derivative $(-1)^{|\Phi_+|}\,\partial_{\Phi_+}=\partial_{\Phi_-} = \prod_{\alpha \in \Phi_-} \partial_\alpha$, it will be convenient to use the latter notation $\partial_{\Phi_-}$.  As an application of Theorem \ref{thm:asymptoticsonevertex}, we can compute the limits $M_{s} = \lim_{N \to \infty} \esper[\nu_N(x^s)]$ for any $s \in \{2,3,4,5\}$. Indeed, we can enumerate all the circuits of length $s \leq 5$, and all their reductions have connected components which are loops:
\begin{align*}
M_2 &= e^{(\ell)}_{\begin{tikzpicture}[scale=0.5]
\draw (0,0.5) circle (0.5);
\fill (0,0) circle (1mm);
\fill [white] (0,1) circle (2.7mm);
\draw (0,1) circle (2.7mm);
\draw (0,1) node {\tiny $2$};
\end{tikzpicture}} \!\!\quad;\quad M_3 = e^{(\ell)}_{\begin{tikzpicture}[scale=0.5]
\draw (0,0.5) circle (0.5);
\fill (0,0) circle (1mm);
\fill [white] (0,1) circle (2.7mm);
\draw (0,1) circle (2.7mm);
\draw (0,1) node {\tiny $3$};
\end{tikzpicture}}\!\!\quad;\quad
M_4 =e^{(\ell)}_{\begin{tikzpicture}[scale=0.5]
\draw (0,0.5) circle (0.5);
\fill (0,0) circle (1mm);
\fill [white] (0,1) circle (2.7mm);
\draw (0,1) circle (2.7mm);
\draw (0,1) node {\tiny $4$};
\end{tikzpicture}} + 2\,e^{(\ell)}_{\begin{tikzpicture}[scale=0.5]
\draw (0,0.5) circle (0.5);
\draw (0,-1) circle (0.5);
\fill (0,0) circle (1mm);
\fill (0,-0.5) circle (1mm);
\fill [white] (0,-1.5) circle (2.7mm);
\fill [white] (0,1) circle (2.7mm);
\draw (0,1) circle (2.7mm);
\draw (0,-1.5) circle (2.7mm);
\draw (0,1) node {\tiny $2$};
\draw (0,-1.5) node {\tiny $2$};
\end{tikzpicture}} + e^{(\ell)}_{\begin{tikzpicture}[scale=0.5]
\draw (0,0.5) circle (0.5);
\fill (0,0) circle (1mm);
\fill [white] (0,1) circle (2.7mm);
\draw (0,1) circle (2.7mm);
\draw (0,1) node {\tiny $2$};
\end{tikzpicture}}\!\!\quad;\quad
M_5=e^{(\ell)}_{\begin{tikzpicture}[scale=0.5]
\draw (0,0.5) circle (0.5);
\fill (0,0) circle (1mm);
\fill [white] (0,1) circle (2.7mm);
\draw (0,1) circle (2.7mm);
\draw (0,1) node {\tiny $5$};
\end{tikzpicture}} + 5\,e^{(\ell)}_{\begin{tikzpicture}[scale=0.5]
\draw (0,0.5) circle (0.5);
\draw (0,-1) circle (0.5);
\fill (0,0) circle (1mm);
\fill (0,-0.5) circle (1mm);
\fill [white] (0,-1.5) circle (2.7mm);
\fill [white] (0,1) circle (2.7mm);
\draw (0,1) circle (2.7mm);
\draw (0,-1.5) circle (2.7mm);
\draw (0,1) node {\tiny $3$};
\draw (0,-1.5) node {\tiny $2$};
\end{tikzpicture}} + 5\,e^{(\ell)}_{\begin{tikzpicture}[scale=0.5]
\draw (0,0.5) circle (0.5);
\fill (0,0) circle (1mm);
\fill [white] (0,1) circle (2.7mm);
\draw (0,1) circle (2.7mm);
\draw (0,1) node {\tiny $3$};
\end{tikzpicture}}.
\end{align*}
Therefore, given a sscc Lie group $G$, if we set for $k \geq 2$
$$I_k=\int_C \left(\frac{(\partial_{\Phi_-}\tildeJ_{\R\Omega})(x)}{(2\pi)^{d/2}}\right)^k \,\frac{\!\DD{x}}{(\delta(x))^{k-2}}$$ 
then, the five first asymptotic moments of the spectral measure of $\GEOM(N,L_N)$ with $L_N = (\ell/N)^{\frac{1}{\dim G}}$ are given by:
\begin{align*}
M_2 &= I_2\,\ell' ;\\
M_3 &= I_3\,( \ell' )^2 ;\\
M_4 &= I_4\,(\ell')^3 + 2\,(I_2)^2\,(\ell')^2 + I_2\,\ell' ;\\
M_5 &= I_5\,( \ell')^4 + 5\,I_3\,I_2\,( \ell')^3 + 5\,I_3\,( \ell')^2,
\end{align*}
where $\ell'=\frac{\ell}{\vol(\tlie/\tlie_\Z)}$. 
\medskip

\subsection{Asymptotic contribution of a connected reduced circuit with two vertices}\label{subsec:twopoints}
What is important in the previous paragraph is not the explicit formula that one obtains for the five first moments, but the method that leads to it: indeed, if one tries to extend it to higher moments, then one is led to new results in representation theory. These results and conjectures are related to the theory of crystals, and in order to understand this, one can try to compute the sixth moment of $\mu$ with the same method as above. The the enumeration of all the circuits of length $6$ yields
$$
M_6 =  e^{(\ell)}_{\begin{tikzpicture}[scale=0.5]
\draw (0,0.5) circle (0.5);
\fill (0,0) circle (1mm);
\fill [white] (0,1) circle (2.7mm);
\draw (0,1) circle (2.7mm);
\draw (0,1) node {\tiny $6$};
\end{tikzpicture}} 
+ 6\,e^{(\ell)}_{\begin{tikzpicture}[scale=0.5]
\draw (0,0.5) circle (0.5);
\draw (0,-1) circle (0.5);
\fill (0,0) circle (1mm);
\fill (0,-0.5) circle (1mm);
\fill [white] (0,-1.5) circle (2.7mm);
\fill [white] (0,1) circle (2.7mm);
\draw (0,1) circle (2.7mm);
\draw (0,-1.5) circle (2.7mm);
\draw (0,1) node {\tiny $4$};
\draw (0,-1.5) node {\tiny $2$};
\end{tikzpicture}} 
+ 3\,e^{(\ell)}_{\begin{tikzpicture}[scale=0.5]
\draw (0,0.5) circle (0.5);
\draw (0,-1) circle (0.5);
\fill (0,0) circle (1mm);
\fill (0,-0.5) circle (1mm);
\fill [white] (0,-1.5) circle (2.7mm);
\fill [white] (0,1) circle (2.7mm);
\draw (0,1) circle (2.7mm);
\draw (0,-1.5) circle (2.7mm);
\draw (0,1) node {\tiny $3$};
\draw (0,-1.5) node {\tiny $3$};
\end{tikzpicture}} 
+ 6\, e^{(\ell)}_{\begin{tikzpicture}[scale=0.5]
\draw (0,0.5) circle (0.5);
\fill (0,0) circle (1mm);
\fill [white] (0,1) circle (2.7mm);
\draw (0,1) circle (2.7mm);
\draw (0,1) node {\tiny $4$};
\end{tikzpicture}}
+ 6\,e^{(\ell)}_{\begin{tikzpicture}[scale=0.5]
\draw (-1.5,0.5) circle (0.5);
\draw (0,0.5) circle (0.5);
\draw (-0.75,-1) circle (0.5);
\fill (0,0) circle (1mm);
\fill (-1.5,0) circle (1mm);
\fill (-0.75,-0.5) circle (1mm);
\fill [white] (-0.75,-1.5) circle (2.7mm);
\fill [white] (0,1) circle (2.7mm);
\fill [white] (-1.5,1) circle (2.7mm);
\draw (0,1) circle (2.7mm);
\draw (-1.5,1) circle (2.7mm);
\draw (-0.75,-1.5) circle (2.7mm);
\draw (0,1) node {\tiny $2$};
\draw (-1.5,1) node {\tiny $2$};
\draw (-0.75,-1.5) node {\tiny $2$};
\end{tikzpicture}}
+ 9\,e^{(\ell)}_{\!\!\!\begin{tikzpicture}[scale=1]
\draw (0,0) -- (0,1) -- (0.35,0.75) -- (0.35,0.25) -- (0,0) -- (-0.35,0.25) -- (-0.35,0.75) -- (0,1);
\foreach \x in {(0,1),(0,0)}
{\fill \x circle (0.5mm) ;}
\foreach \x in {(0,0.5),(0.35,0.5),(-0.35,0.5)}
{ \fill [white] \x circle (1.35mm) ; \draw \x circle (1.35mm);}
\draw (0,0.5) node {\tiny $2$};
\draw (-0.35,0.5) node {\tiny $2$};
\draw (0.35,0.5) node {\tiny $1$};
\end{tikzpicture}}
+ 6\,e^{(\ell)}_{\begin{tikzpicture}[scale=0.5]
\draw (0,0.5) circle (0.5);
\draw (0,-1) circle (0.5);
\fill (0,0) circle (1mm);
\fill (0,-0.5) circle (1mm);
\fill [white] (0,-1.5) circle (2.7mm);
\fill [white] (0,1) circle (2.7mm);
\draw (0,1) circle (2.7mm);
\draw (0,-1.5) circle (2.7mm);
\draw (0,1) node {\tiny $2$};
\draw (0,-1.5) node {\tiny $2$};
\end{tikzpicture}}
+ 4\,e^{(\ell)}_{\begin{tikzpicture}[scale=0.5]
\draw (0,0.5) circle (0.5);
\fill (0,0) circle (1mm);
\fill [white] (0,1) circle (2.7mm);
\draw (0,1) circle (2.7mm);
\draw (0,1) node {\tiny $3$};
\end{tikzpicture}}
+ e^{(\ell)}_{\begin{tikzpicture}[scale=0.5]
\draw (0,0.5) circle (0.5);
\fill (0,0) circle (1mm);
\fill [white] (0,1) circle (2.7mm);
\draw (0,1) circle (2.7mm);
\draw (0,1) node {\tiny $2$};
\end{tikzpicture}},
$$
the terms being order by decreasing parameter $k$. Among these terms, there is one reduced circuit with two vertices instead of one, which is for instance obtained with the identities of indices $i_1=i_4$ and $i_2=i_5$. Indeed, these identities correspond to the circuit $(H,T)$
\begin{center}
\begin{tikzpicture}[scale=1.8]
\draw [-{>[scale=2]>[scale=2]}] (0,0) -- (-0.2,0.5);
\draw (-0.2,0.5) -- (0,1) -- (-0.866,0.5) -- (0,0) -- (0.2,0.5) -- (0,1) -- (0.866,0.5) -- (0,0);
\foreach \x in {(0,1),(-0.866,0.5),(0.866,0.5)}
{\fill [white] \x circle (0.5mm); \draw \x circle (0.5mm);}
\fill (0,0) circle (0.5mm);
\draw (-0.2,0.7) node {$1$};
\draw (-0.6,0.8) node {$2$};
\draw (-0.6,0.2) node {$3$};
\draw (0.2,0.7) node {$4$};
\draw (0.6,0.8) node {$5$};
\draw (0.6,0.2) node {$6$};
\end{tikzpicture}
\end{center}
and thus to the reduced circuit $R(H,T)$ of Figure \ref{fig:reducedsix}.
\begin{figure}[ht]
\begin{center}
\begin{tikzpicture}[scale=1.8]
\draw (0,0) -- (0,1) -- (0.25,0.75) -- (0.25,0.25) -- (0,0) -- (-0.25,0.25) -- (-0.25,0.75) -- (0,1);
\foreach \x in {(0,1),(0,0)}
{\fill \x circle (0.5mm) ;}
\foreach \x in {(0,0.5),(0.25,0.5),(-0.25,0.5)}
{ \fill [white] \x circle (0.9mm) ; \draw \x circle (0.9mm);}
\draw (0,0.5) node {\tiny $2$};
\draw (-0.25,0.5) node {\tiny $2$};
\draw (0.25,0.5) node {\tiny $1$};
\end{tikzpicture}
\end{center}
\caption{The reduced circuit corresponding to the identities $i_1=i_4$ and $i_2=i_5$.\label{fig:reducedsix}}
\end{figure}
\bigskip

The asymptotics of $E_{R,N}$ with $R$ as in Figure \ref{fig:reducedsix} are related to the asymptotics of tensor products $V^{\lambda} \otimes V^{\mu}$ when $x = L_N\,\lambda$ and $y = L_N\,\mu$ are fixed points in the interior of the Weyl chamber, and $L_N$ goes to $0$; thus, $\lambda$ and $\mu$ are very large dominant weights. Indeed, we have 
\begin{align*}
e^{(\ell)}_{R,N}=N^{3}\,E_{R,N} &= N^3 \,\esper[h_N(v_1,v_2)\,h_N(v_2,v_3)\,h_N(v_1,v_3)\,h_N(v_1,v_4)\,h_N(v_4,v_1)] \\
&= N^3\,\esper\!\left[(Z_{L_N}^{*2}(v_1(v_3)^{-1}))^2\,Z_{L_N}(v_3(v_1)^{-1})\right] = N^3\,\scal{((Z_{L_N})^{*2})^2}{Z_{L_N}}.
\end{align*}
The functions above decompose in $\leb^2(G)$ as:
$$
Z_{L_N} = \sum_{\lambda \in \hatG} C_{\lambda,N}\,\ch^\lambda \qquad;\qquad (Z_{L_N})^{*2} = \sum_{\lambda \in \hatG} \frac{(C_{\lambda,N})^2}{d_\lambda}\,\ch^\lambda.$$
Therefore, we have
$$N^{3}\,E_{R,N} = \sum_{\lambda,\mu,\nu} \frac{N^3\,(C_{\lambda,N})^2\,(C_{\mu,N})^2\,C_{\nu,N}}{d_\lambda\,d_\mu}\,\scal{\ch^\lambda\,\times\,\ch^\mu}{\ch^\nu}.$$ 
As before, the idea is to consider the sum above as a Riemann sum, and we will of course use Lemma \ref{lem:partialderivative} in order to approximate the coefficients $C_{\lambda,N}$, $C_{\mu,N}$ and $C_{\nu,N}$ by partial derivatives of Bessel functions. However, we also need to deal with $c^{\lambda,\mu}_\nu=\scal{\ch^\lambda\,\times\,\ch^\mu}{\ch^\nu}$, and the product of characters $\ch^\lambda \times \ch^\mu$ is the character of the tensor product of representations $V^{\lambda} \otimes V^{\mu}$. Therefore, we need to understand the asymptotics of the \emph{Littlewood--Richardson coefficients} $c^{\lambda,\mu}_\nu$ such that
$$V^\lambda \otimes V^\mu = \bigoplus_{\nu\in \hatG} c^{\lambda,\mu}_\nu\,V^\nu,$$
when $\lambda$ and $\mu$ are very large. More generally, if we are interested in the computations of the terms $e_R$ where $R$ is a general reduced circuit on two vertices, then we need to understand the asymptotic behavior of the Littlewood--Richardson coefficients of tensor products with more than $2$ irreducible representations. Given dominant weights $\lambda_1,\ldots,\lambda_{r-1\geq 2}$, we write
$$V^{\lambda_1} \otimes V^{\lambda_2} \otimes \cdots \otimes V^{\lambda_{r-1}} = \sum_{\lambda_r \in \hatG} c^{\lambda_1,\ldots,\lambda_{r-1}}_{\lambda_r} \,V^{\lambda_r}.$$
These generalised Littlewood--Richardson coefficients are connected to the usual one by the convolution rule:
$$c^{\lambda_1,\ldots,\lambda_{r-1}}_{\lambda_r} = \sum_{\nu_1,\ldots,\nu_{r-3} \in \hatG} c^{\lambda_1,\lambda_2}_{\nu_1} c^{\nu_1,\lambda_3}_{\nu_2} \cdots c^{\nu_{r-3},\lambda_{r-1}}_{\lambda_r}.$$

\begin{proposition}\label{prop:multilittlewood}
We denote $d$ the rank of the sscc Lie group $G$, and $l=|\Phi_+|$ its number of positive roots. We fix directions $x_1,\ldots,x_{r-1\geq 2}$ in the interior $C'$ of the Weyl chamber. There exists a compactly supported piecewise polynomial function $q_{x_1,\ldots,x_{r-1}}(z)$ on $C$ such that:
\begin{itemize}
     \item This function is non-negative and symmetric in $x_1,\ldots,x_{r-1}$.
     \item For any bounded continuous function $f$ on $C$,
     $$\lim_{\substack{t \to \infty \\ tx_{1},\ldots,tx_{r-1} \in \hatG}} \left(\frac{1}{t^{(r-2)l}} \sum_{\nu \in \hatG} c^{tx_1,\ldots,tx_{r-1}}_{\nu}\,f\left( \frac{\nu}{t} \right) \right) = \int_{C} f(z)\,q_{x_1,\ldots,x_{r-1}}(z)\DD{z}.$$
     \item The function $q_{x_1,\ldots,x_{r-1}}(z)$ is related to the functions $q_{x,y}(z)$ by the convolution rule:
     \begin{equation*}
           q_{x_1,\ldots,x_{r-1}}(z) = \int_{C^{r-3}} q_{x_1,x_2}(z_1)\,q_{x_3,z_1}(z_2)\,\cdots \,q_{x_{r-2},z_{r-4}}(z_{r-3})\,q_{x_{r-1},z_{r-3}}(z)\DD{z}_1\,\cdots\DD{z}_{r-3} .
      \end{equation*} 
     \item The function $(x_1,\ldots,x_{r-1},z)\mapsto q_{x_1,\ldots,x_{r-1}}(z)$ is locally a homogeneous polynomial function of total degree $(r-2)l - d$. The domains of polynomiality of this function are polyhedral cones in $C^r$.
     \item The total mass of the positive measure $q_{x_1,\ldots,x_{r-1}}(z)\DD{z}$ is smaller than 
     $$ \frac{1}{\max_{i \in \lle 1,r-1\rre} \delta(x_i)}\,\prod_{i=1}^{r-1} \delta(x_i).$$
 \end{itemize} 
\end{proposition}

This proposition is proved in the second appendix of this paper; the proof relies deeply on the theory of crystal bases and string polytopes. Let us give an intuitive explanation of it which does not go too much into the algebraic details. Given three dominant weights $\lambda$, $\mu$ and $\nu$, one can construct a polytope $\mathscr{P}(\lambda,\mu)$ (the \emph{relative Berenstein--Zelevinsky polytope}) in a vector space of dimension $l=|\Phi_+|$ which is determined by hyperplanes whose equations depend on $\lambda$ and $\mu$ in an affine way. Then, the Littlewood--Richardson coefficient $c^{\lambda,\mu}_\nu$ is the number of integer points which lie in the intersection of this polytope and of a vector subspace determined by $d = \rank(G)$ equations, these equations depending on $\lambda$ and $\nu$ again in an affine way (see Theorem \ref{thm:berensteinzelevinsky}). When we consider a sum
$$\sum_{\nu \in \hatG} c^{tx,ty}_\nu\,f\left(\frac{\nu}{t}\right)$$
with $x,y \in C'$, $tx,ty \in \hatG$ and $t$ going to infinity, the polytope $\mathscr{P}(tx,ty)$ scales linearly with the parameter $t$, and the counting measure of the integer points of this polytope becomes after scaling the uniform Lebesgue measure on $\mathscr{P}(x,y)$, which is of dimension $l$. The Riemann sum over the dominant weights $\nu$ becomes after scaling an integral against the affine projection of this uniform measure on a polytope of smaller dimension $d$. Since an affine projection of a uniform measure on a polytope is a piecewise polynomial function also supported by a polytope, this essentially proves Proposition \ref{prop:multilittlewood} in the case $r=3$, by using also the linearity of the various polytopes in the parameters $\lambda,\mu,\nu$. The general case $r \geq 3$ follows by using the convolution rule for multiple Littlewood--Richardson coefficients. In the sequel of this section, we use Proposition \ref{prop:multilittlewood} without insisting on its algebraic origin. The knowledge of these algebraic beginnings will only be useful in order to understand fully our conjecture on graph functionals, and why it involves the enumeration of integer points in polytopes; again, we refer to Section \ref{sec:crystal} for more details.\medskip

We now come back to the asymptotics of $E_{R,N}$ when $R$ is a connected reduced circuit with two vertices. An arbitrary connected reduced circuit on $k'=2$ vertices writes as
\vspace{2mm}
\begin{center}
\begin{tikzpicture}[scale=2.5]
\draw (0,0) -- (0,1) -- (0.25,0.75) -- (0.25,0.25) -- (0,0) -- (0.5,0.25) -- (0.5,0.75) -- (0,1) -- (1,0.75) -- (1,0.25) -- (0,0);
\foreach \x in {(0,1),(0,0)}
{\fill \x circle (0.3mm) ;}
\foreach \x in {(0,0.5),(0.25,0.5),(0.5,0.5),(1,0.5)}
{ \fill [white] \x circle (0.9mm) ; \draw \x circle (0.9mm);}
\draw (0,0.5) node {\tiny $a_1$};
\draw (0.25,0.5) node {\tiny $a_2$};
\draw (0.5,0.5) node {\tiny $a_3$};
\draw (0.75,0.5) node {\tiny $\cdots$};
\draw (1,0.5) node {\tiny $a_r$};
\draw (-0.4,0.5) node {$R=$};
\draw (1.2,0.5) node {$.$};
\end{tikzpicture}
\end{center}
with $a_1\geq a_2 \geq \cdots \geq a_r \geq 1$, at most one index $a_r=1$, and $r$ larger than $3$. In this setting, $k = (a_1+\cdots+a_r)-(r-2)$. The contribution corresponding to such a reduced circuit is:
\begin{align*}
E_{R,N} &= \scal{Z_{L_N}^{*a_1} Z_{L_N}^{*a_2} \cdots Z_{L_N}^{*a_{r-1}}}{Z_{L_N}^{*a_r}} \\
&= \sum_{\lambda_1,\ldots,\lambda_r} \frac{(C_{\lambda_1,N})^{a_1}(C_{\lambda_2,N})^{a_2}\cdots (C_{\lambda_r,N})^{a_r}}{(d_{\lambda_1})^{a_{1}-1}(d_{\lambda_2})^{a_{2}-1}\cdots (d_{\lambda_r})^{a_{r}-1}}\,\scal{\ch^{\lambda_1}\times \ch^{\lambda_2} \times \cdots \times \ch^{\lambda_{r-1}}}{\ch^{\lambda_r}} 
\end{align*}
the sum running over $r$-tuples of dominant weights. With $t_N = (L_N)^{-1}$, by combining Proposition \ref{prop:multilittlewood} and the estimates of the coefficients $C_{\lambda,N}$ given by Lemma \ref{lem:partialderivative}, we obtain:
\begin{align*}
&(2\pi)^{\frac{d}{2}\,(a_1+\cdots +a_r)}\,E_{R,N}\\
 &\simeq  \frac{(L_N)^{(\dim G)(a_1+\cdots+a_r)-lr}}{(\vol(\tlie/\tlie_\Z))^{a_1+\cdots+a_r}}\sum_{x_1,\ldots,x_r} 
\prod_{i=1}^r \frac{(\partial_{\Phi_-}\tildeJ_{\R\Omega}(x_i))^{a_i}}{(\delta(x_i))^{a_i-1}}\,c^{t_Nx_1,\ldots,t_Nx_{r-1}}_{t_Nx_r}\\
&\simeq \frac{1}{(\vol(\tlie/\tlie_\Z))^{a_1+\cdots+a_r}} \left(\frac{\ell}{N}\right)^{\!k-1}\sum_{x_1,\ldots,x_{r-1}} (L_N)^{d(r-1)}\int_C \left(\prod_{i=1}^r \frac{(\partial_{\Phi_-}\tildeJ_{\R\Omega}(x_i))^{a_i}}{(\delta(x_i))^{a_i-1}}\right) q_{x_1,\ldots,x_{r-1}}(x_r)\DD{x_r} \\
&\simeq \left(\frac{\ell}{N\,\vol(\tlie/\tlie_\Z)}\right)^{\!k-1} \int_{C^r} \left(\prod_{i=1}^r \frac{(\partial_{\Phi_-}\tildeJ_{\R\Omega}(x_i))^{a_i}}{(\delta(x_i))^{a_i-1}}\right) q_{x_1,\ldots,x_{r-1}}(x_r) \DD{x_1}\cdots \DD{x_r}, 
\end{align*}
the sums running over elements $x_1,\ldots,x_r$ which are in the Weyl chamber $C$, and which are multiple by $L_N$ of some dominant weights. The convergence of the integral on $C^r$ follows from the following argument. By using the bounds on the partial derivatives of $\tildeJ_{\R \Omega}$, and the scaling properties of the functions $\delta(x_i)$ and $q_{x_1,\ldots,x_{r-1}}(x_r)$, one sees that it suffices to prove the convergence at infinity of the integral 
$\int^\infty t^{(r-2)l - d - (a_1+\cdots+a_r)\frac{d+1}{2}-(a_1+\cdots+a_r-r)l}\,t^{r(d+1)-1}\DD{t}.$
Indeed, $q_{x_1,\ldots,x_{r-1}}(x_r)$ is homogeneous with total degree $(r-2)l-d$; we have the upper bound
$$\left|\prod_{i=1}^r \left(\partial_{\Phi_-}\tildeJ_{\R\Omega}(x_i)\right)^{a_i}\right| \leq K\, \prod_{i=1}^r \left( \frac{1}{1+\|x_i\|} \right)^{\!\frac{a_i(d+1)}{2}};$$
and $\prod_{i=1}^r (\delta(x_i))^{a_i-1}$ is homogeneous with total degree $(a_1+\cdots+a_r-r)l$. Therefore, we have to prove that
$$(r-2)l - d - (a_1+\cdots+a_r)\,\frac{d+1}{2}-(a_1+\cdots+a_r-r)l + r(d+1) <0.$$
However, the worst case is when the $a_i$'s are minimal, that is $a_1=a_2=\cdots=a_{r-1}=2$ and $a_r=1$. The left-hand side of the inequality above is then equal to
$$(r-2)l - d - (2r-1)\,\frac{d+1}{2}-(r-1)l + r(d+1) =  - \frac{d-1}{2} - l ,$$
which is clearly negative. On the other hand, the validity of the approximation of the Riemann sum by an integral follows from the smoothness of the functions considered, and from the fact that the functions $q_{x_1,\ldots,x_{r-1}}(x_r)$ are compactly supported. We have therefore proved:

\begin{theorem}\label{thm:asymptoticstwovertices}
Let $r \geq 3$, $a_1\geq a_2\geq \cdots \geq a_{r-1}\geq 2$ and $a_r \in \lle 1,a_{r-1}\rre$. We set
$$e^{(l)}_{\!\!\begin{tikzpicture}[scale=1]
\draw (0,0) -- (0,1) -- (0.4,0.75) -- (0.4,0.25) -- (0,0) -- (1.2,0.25) -- (1.2,0.75) -- (0,1) ;
\foreach \x in {(0,1),(0,0)}
{\fill \x circle (0.3mm) ;}
\foreach \x in {(0,0.5),(0.4,0.5),(1.2,0.5)}
{ \fill [white] \x circle (1.7mm) ; \draw \x circle (1.7mm);}
\draw (0,0.5) node {\tiny $a_1$};
\draw (0.4,0.5) node {\tiny $a_2$};
\draw (0.8,0.5) node {\tiny $\cdots$};
\draw (1.2,0.5) node {\tiny $a_r$};
\draw (1.7,0.5) node {$,N$};
\end{tikzpicture}} = N^{k-1}\,E_{H,T,N},$$
where $(H,T)$ is a circuit whose reduction is a connected component on two vertices with parameters $(a_1,\ldots,a_r)$. Here, $k=(a_1+\cdots+a_r)-(r-2)$. Then,
\begin{align*}
e^{(l)}_{\!\!\begin{tikzpicture}[scale=1]
\draw (0,0) -- (0,1) -- (0.4,0.75) -- (0.4,0.25) -- (0,0) -- (1.2,0.25) -- (1.2,0.75) -- (0,1) ;
\foreach \x in {(0,1),(0,0)}
{\fill \x circle (0.3mm) ;}
\foreach \x in {(0,0.5),(0.4,0.5),(1.2,0.5)}
{ \fill [white] \x circle (1.7mm) ; \draw \x circle (1.7mm);}
\draw (0,0.5) node {\tiny $a_1$};
\draw (0.4,0.5) node {\tiny $a_2$};
\draw (0.8,0.5) node {\tiny $\cdots$};
\draw (1.2,0.5) node {\tiny $a_r$};
\end{tikzpicture}} 
&= \lim_{N \to \infty} e^{(l)}_{\!\!\begin{tikzpicture}[scale=1]
\draw (0,0) -- (0,1) -- (0.4,0.75) -- (0.4,0.25) -- (0,0) -- (1.2,0.25) -- (1.2,0.75) -- (0,1) ;
\foreach \x in {(0,1),(0,0)}
{\fill \x circle (0.3mm) ;}
\foreach \x in {(0,0.5),(0.4,0.5),(1.2,0.5)}
{ \fill [white] \x circle (1.7mm) ; \draw \x circle (1.7mm);}
\draw (0,0.5) node {\tiny $a_1$};
\draw (0.4,0.5) node {\tiny $a_2$};
\draw (0.8,0.5) node {\tiny $\cdots$};
\draw (1.2,0.5) node {\tiny $a_r$};
\draw (1.7,0.5) node {$,N$};
\end{tikzpicture}} \\
& = \left(\frac{\ell}{\vol(\tlie/\tlie_\Z)}\right)^{\!k-1} \int_{C^r} \prod_{i=1}^r\left( \frac{(\partial_{\Phi_-}\tildeJ_{\R\Omega})(x_i)}{(2\pi)^{d/2}} \right)^{\!a_i} \frac{q_{x_1,\ldots,x_{r-1}}(x_r)}{\prod_{i=1}^r (\delta(x_i))^{a_i-1}} \DD{x_1}\cdots \DD{x_r},
\end{align*}
the function $q_{x_1,\ldots,x_{r-1}}(x_r)$ being related to the asymptotics of (multi-)Littlewood--Richardson coefficients by Proposition \ref{prop:multilittlewood}.
\end{theorem}

\noindent As an application, if we set
\begin{align*}
I_{(2,2,1)} &= \int_{C^3} \frac{\left(\partial_{\Phi_-}\tildeJ_{\R\Omega}(x)\right)^{\!2}\left(\partial_{\Phi_-}\tildeJ_{\R\Omega}(y)\right)^{\!2}\left(\partial_{\Phi_-}\tildeJ_{\R\Omega}(z)\right)}{(2\pi)^{5d/2}\,\delta(x)\,\delta(y)} \,q_{x,y}(z) \DD{x}\DD{y} \DD{z}
\end{align*}
then this quantity is related to the term $e_R^{(\ell)}$ with $R$ as in Figure \ref{fig:reducedsix}, and by using the circuit expansion of $M_6$ previous computed, we obtain
\begin{align*}
M_6&=I_6\,(\ell')^5 + (6\,I_4I_2 + 3\, (I_3)^2)\,(\ell')^4 + (6\,I_4+ 6\,(I_2)^3 + 9\,I_{(2,2,1)})\,(\ell')^3 \\
&\quad+ (6\,(I_2)^2 + 4\,I_3)\,(\ell')^2 + I_2\,\ell',
\end{align*}
where $\ell'=\frac{\ell}{\vol(\tlie/\tlie_\Z)}$. Again, the important point is not this exact formula, but the fact that its computation sheds light on the asymptotic properties of large representations of the group $G$. This idea culminates in the computation of the higher moments $M_{s\geq 8}$, as we shall now explain --- for $s=7$, one can check that the circuit expansion is 
\begin{align*}
M_7 &= e^{(\ell)}_{\begin{tikzpicture}[scale=0.5]
\draw (0,0.5) circle (0.5);
\fill (0,0) circle (1mm);
\fill [white] (0,1) circle (2.7mm);
\draw (0,1) circle (2.7mm);
\draw (0,1) node {\tiny $7$};
\end{tikzpicture}} 
+ 7\,e^{(\ell)}_{\begin{tikzpicture}[scale=0.5]
\draw (0,0.5) circle (0.5);
\draw (0,-1) circle (0.5);
\fill (0,0) circle (1mm);
\fill (0,-0.5) circle (1mm);
\fill [white] (0,-1.5) circle (2.7mm);
\fill [white] (0,1) circle (2.7mm);
\draw (0,1) circle (2.7mm);
\draw (0,-1.5) circle (2.7mm);
\draw (0,1) node {\tiny $5$};
\draw (0,-1.5) node {\tiny $2$};
\end{tikzpicture}}
+ 7\,e^{(\ell)}_{\begin{tikzpicture}[scale=0.5]
\draw (0,0.5) circle (0.5);
\draw (0,-1) circle (0.5);
\fill (0,0) circle (1mm);
\fill (0,-0.5) circle (1mm);
\fill [white] (0,-1.5) circle (2.7mm);
\fill [white] (0,1) circle (2.7mm);
\draw (0,1) circle (2.7mm);
\draw (0,-1.5) circle (2.7mm);
\draw (0,1) node {\tiny $4$};
\draw (0,-1.5) node {\tiny $3$};
\end{tikzpicture}}  
+ 7\, e^{(\ell)}_{\begin{tikzpicture}[scale=0.5]
\draw (0,0.5) circle (0.5);
\fill (0,0) circle (1mm);
\fill [white] (0,1) circle (2.7mm);
\draw (0,1) circle (2.7mm);
\draw (0,1) node {\tiny $5$};
\end{tikzpicture}}
+ 21\,e^{(\ell)}_{\begin{tikzpicture}[scale=0.5]
\draw (-1.5,0.5) circle (0.5);
\draw (0,0.5) circle (0.5);
\draw (-0.75,-1) circle (0.5);
\fill (0,0) circle (1mm);
\fill (-1.5,0) circle (1mm);
\fill (-0.75,-0.5) circle (1mm);
\fill [white] (-0.75,-1.5) circle (2.7mm);
\fill [white] (0,1) circle (2.7mm);
\fill [white] (-1.5,1) circle (2.7mm);
\draw (0,1) circle (2.7mm);
\draw (-1.5,1) circle (2.7mm);
\draw (-0.75,-1.5) circle (2.7mm);
\draw (0,1) node {\tiny $3$};
\draw (-1.5,1) node {\tiny $2$};
\draw (-0.75,-1.5) node {\tiny $2$};
\end{tikzpicture}}
+ 21\,e^{(\ell)}_{\!\!\!\begin{tikzpicture}[scale=1]
\draw (0,0) -- (0,1) -- (0.35,0.75) -- (0.35,0.25) -- (0,0) -- (-0.35,0.25) -- (-0.35,0.75) -- (0,1);
\foreach \x in {(0,1),(0,0)}
{\fill \x circle (0.5mm) ;}
\foreach \x in {(0,0.5),(0.35,0.5),(-0.35,0.5)}
{ \fill [white] \x circle (1.35mm) ; \draw \x circle (1.35mm);}
\draw (0,0.5) node {\tiny $2$};
\draw (-0.35,0.5) node {\tiny $3$};
\draw (0.35,0.5) node {\tiny $1$};
\end{tikzpicture}}
+ 7\,e^{(\ell)}_{\!\!\!\begin{tikzpicture}[scale=1]
\draw (0,0) -- (0,1) -- (0.35,0.75) -- (0.35,0.25) -- (0,0) -- (-0.35,0.25) -- (-0.35,0.75) -- (0,1) -- (0.7,0.75) -- (0.7,0.25) -- (0,0);
\foreach \x in {(0,1),(0,0)}
{\fill \x circle (0.5mm) ;}
\foreach \x in {(0,0.5),(0.35,0.5),(-0.35,0.5),(0.7,0.5)}
{ \fill [white] \x circle (1.35mm) ; \draw \x circle (1.35mm);}
\draw (0,0.5) node {\tiny $2$};
\draw (-0.35,0.5) node {\tiny $2$};
\draw (0.35,0.5) node {\tiny $2$};
\draw (0.7,0.5) node {\tiny $1$};
\end{tikzpicture}}\\
&\quad+ 28\,e^{(\ell)}_{\!\!\!\begin{tikzpicture}[scale=1]
\draw (0,0) -- (0,1) -- (0.35,0.75) -- (0.35,0.25) -- (0,0) -- (-0.35,0.25) -- (-0.35,0.75) -- (0,1);
\foreach \x in {(0,1),(0,0)}
{\fill \x circle (0.5mm) ;}
\foreach \x in {(0,0.5),(0.35,0.5),(-0.35,0.5)}
{ \fill [white] \x circle (1.35mm) ; \draw \x circle (1.35mm);}
\draw (0,0.5) node {\tiny $2$};
\draw (-0.35,0.5) node {\tiny $2$};
\draw (0.35,0.5) node {\tiny $1$};
\end{tikzpicture}}
+ 35\,e^{(\ell)}_{\begin{tikzpicture}[scale=0.5]
\draw (0,0.5) circle (0.5);
\draw (0,-1) circle (0.5);
\fill (0,0) circle (1mm);
\fill (0,-0.5) circle (1mm);
\fill [white] (0,-1.5) circle (2.7mm);
\fill [white] (0,1) circle (2.7mm);
\draw (0,1) circle (2.7mm);
\draw (0,-1.5) circle (2.7mm);
\draw (0,1) node {\tiny $3$};
\draw (0,-1.5) node {\tiny $2$};
\end{tikzpicture}} 
+ 21\, e^{(\ell)}_{\begin{tikzpicture}[scale=0.5]
\draw (0,0.5) circle (0.5);
\fill (0,0) circle (1mm);
\fill [white] (0,1) circle (2.7mm);
\draw (0,1) circle (2.7mm);
\draw (0,1) node {\tiny $3$};
\end{tikzpicture}}
\end{align*}
and it only involves reduced circuits on one or two vertices.
\medskip

\subsection{The conjecture on graph functionals}\label{subsec:graphfunctionals}
Starting with $s=8$, the circuit expansion from Theorem \ref{thm:circuitexpansion} yields connected reduced circuits on $3$ vertices, the smallest case being when $s=8$ and we have for instance the identities of indices $i_1=i_6$, $i_3=i_7$ and $i_4=i_8$. The corresponding reduced circuit is drawn in Figure \ref{fig:smallproblem}.

\begin{figure}[ht]
\begin{center}
\begin{tikzpicture}[scale=1.8]
\draw (0.5,1) -- (0.5,0.3) -- (0,0) -- (-0.5,1) -- (0.5,1) -- (0,0) -- (-0.5,0.3) --  (-0.5,1);
\foreach \x in {(-0.5,1),(0,0),(0.5,1)}
{\fill \x circle (0.5mm) ;}
\foreach \x in {(0.5,0.5),(-0.5,0.5),(0,1),(-0.3,0.6),(0.3,0.6)}
{ \fill [white] \x circle (0.9mm) ; \draw \x circle (0.9mm);}
\draw (0,1) node {\tiny $1$};
\draw (0.5,0.5) node {\tiny $1$};
\draw (-0.5,0.5) node {\tiny $1$};
\draw (-0.3,0.6) node {\tiny $2$};
\draw (0.3,0.6) node {\tiny $2$};
\end{tikzpicture}
\end{center}
\caption{The smallest connected reduced circuit on $3$ vertices, corresponding to the identities $i_1=i_6$, $i_3=i_7$ and $i_4=i_8$ in a circuit of length $8$.\label{fig:smallproblem}}
\end{figure}
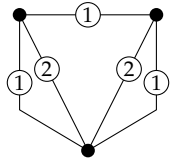

Let us explain how one should deal with the contribution $E_{R,N}$ of such a circuit. We label the $k'\geq 2$ vertices of a connected reduced circuit which is not a loop in an arbitrary order from $1$ to $k'$, and we convene to orientate each labeled edge \raisebox{-1.5mm}{\begin{tikzpicture}[scale=1]
    \draw (0,0) -- (2,0);
    \fill [white] (1,0) circle (2mm); \draw (1,0) circle (2mm);
    \draw (1,0) node {\footnotesize $l_e$};
    \draw (-0.3,0) node {\footnotesize $a$};
    \draw (2.3,0) node {\footnotesize $b$};
    \fill (0,0) circle (1mm);
    \fill (2,0) circle (1mm);
    \end{tikzpicture}}
~in the direction $a \to b$ if the index $a$ is smaller than $b$. Then,
\begin{align*}
E_{R,N} &= \int_{G^{k'}}\left( \prod_{a \to_{l_e} b \text{ labeled edge of }R} (Z_{L_N})^{*l_e}(g_a(g_b)^{-1})\right)\DD{g_1}\,\cdots \DD{g_{k'}}\\
&=\sum_{\substack{(\lambda_e)_{e \in E(R)} \text{ family}\\\text{of dominant weights}}} \left(\prod_{e \in E(R)} \frac{(C_{\lambda_e,N})^{l_e}}{(d_{\lambda_e})^{l_e-1}}\right) \int_{G^{k'}}\left( \prod_{(a \to_{l_e} b) \in E(R)} \ch^{\lambda_e}(g_a(g_b)^{-1})\right)\DD{g_1}\,\cdots \DD{g_{k'}},
\end{align*}
where $E(R)$ is the set of labeled edges of $R$. Note that the integral on the second line only depends on the unlabeled oriented graph $S$ underlying the reduced connected circuit $R$. In the following we always use the letter $S$ or the notation $S(H,T)$ for this unlabeled oriented graph, and we call \emph{graph functional} the integral
$$\mathrm{GF}_S((\lambda_e)_{e \in E(S)}) = \int_{G^{k'}}\left( \prod_{(a \to b) \in E(S)} \ch^{\lambda_e}(g_a(g_b)^{-1})\right)\DD{g_1}\,\cdots \DD{g_{k'}};$$
this is a function on the $k'$-tuples of dominant weights indexed by the edges of $S$. This definition actually makes sense for any finite graph $S$ with ordered vertices and possibly with loops; moreover, the definition immediately implies that $\mathrm{GF}_S((\lambda_e)_{e \in E(S)})$ factorises over the biconnected components of the graph $S$. We recover classical quantities when the graph $S$ has $k'=1$ or $k'=2$ vertices:
\begin{itemize}
    \item If the graph $S$ has $k'=1$ vertex, then it is a collection of loops, and we have
    $$\mathrm{GF}_{\begin{tikzpicture}[scale=0.5]
\draw (0,0.5) circle (0.5);
\fill (0,0) circle (1mm);
\end{tikzpicture}}(\lambda)= \int_G \ch^{\lambda}(e_G)\DD{g} = \dim V^\lambda.$$
    \item If the graph $S$ has $k'=2$ vertices and $r \geq 3$ edges between them, then its graph functional is a multiple Littlewood--Richardson coefficients:
    $$\mathrm{GF}_S(\lambda_1,\ldots,\lambda_r) = \int_G \ch^{\lambda_1}(g)\,\ch^{\lambda_2}(g)\cdots \ch^{\lambda_r}(g)\DD{g} = c^{\lambda_1,\ldots,\lambda_{r-1}}_{\lambda_r^*},$$
    where $\lambda^*$ denotes the highest weight of the irreducible representation which is the conjugate of $V^\lambda$ (\emph{i.e.}, $(V^\lambda)^* = V^{\lambda^*}$ and $\rho^{\lambda^*}(g) = (\rho^\lambda(g^{-1}))^t$).
\end{itemize}
Hence, the graph functionals can be considered as generalisations of the Littlewood--Richardson coefficients. For any circuit $(H,T)$, we have
\begin{equation}
    E_{H,T,N} = \sum_{(\lambda_e)_{e \in E(S(H,T))}} \left(\prod_{e \in E(S(H,T))} \frac{(C_{\lambda_e,N})^{l_e}}{(d_{\lambda_e})^{l_e-1}}\right) \mathrm{GF}_{S(H,T)}((\lambda_e)_{e \in E(S(H,T))}).\label{eq:graphfunctionalexpansion}
\end{equation}
Now, let us do a little bit of dimension analysis in order to explain how the mere existence of the limits $e_R^{(\ell)}$ suggests our Conjecture \ref{conj:graphfunctional}. To simplify a bit the discussion, we suppose in the sequel that $R(H,T)$ and $S(H,T)$ are connected. The idea is to interpret Equation \eqref{eq:graphfunctionalexpansion} as a Riemann sum, which should be asymptotically of order $O(N^{-(k-1)})$, $k$ being as usual the number of distinct indices in the circuit $(H,T)$. If $x_e = \lambda_e L_N$, then we should be able to approximate 
 $$\prod_{e \in E(R)}\frac{(C_{\lambda_e,N})^{l_e}}{(d_{\lambda_e})^{l_e-1}} \simeq_{N \to \infty} \left(\frac{\ell}{N}\right)^{k-k'}\,(L_N)^{(d+l)r}\left(\frac{1}{(2\pi)^{d/2}\,\vol(\tlie/\tlie_\Z)}\right)^{\sum_{e \in E(R)}l_e} \prod_{e \in E(R)} \frac{(\partial_{\Phi_-}\tildeJ_{\R\Omega}(x_e))^{l_e}}{(\delta(x_e))^{l_e-1}} $$
 where $r = |E(R)|$ and $k'=|V(R)|$. Assuming that this approximation is valid, we then have:
 \begin{align*}
 &\frac{N^{k-1}\,E_{H,T,N}}{(\ell')^{k-1}\,(\vol(\tlie/\tlie_\Z))^{k'-1}}\,(L_N)^{(2l+d)(k'-1)} \\
 &\simeq_{N \to \infty} \sum_{(x_e)_{e \in E(R)}} \left(\prod_{e \in E(R)}\frac{((\partial_{\Phi_-}\tildeJ_{\R\Omega})(x_e))^{l_e}}{(2\pi)^{\frac{l_e}{2}}(\delta(x_e))^{l_e-1}} \right) (L_N)^{lr}\,\,\mathrm{GF}_{S}\!\left(\left(\frac{x_e}{L_N}\right)_{\!e \in E(S)}\right)\,((L_N)^d\,\vol(\R\Omega/\Z\Omega))^r
 \end{align*}
 with $\ell' = \frac{\ell}{\vol(\tlie/\tlie_\Z)}$, and where the sum runs over elements of the lattice $(L_N\,(C\cap\Z\Omega))^r$. If we want the convergence of this Riemann sum as $N$ goes to infinity, then taking into account our Proposition \ref{prop:multilittlewood} and our Theorems \ref{thm:asymptoticsonevertex} and \ref{thm:asymptoticstwovertices}, it is natural to make the following assumption:
\vspace{2mm}

\begin{quote}
\emph{If $r$ is the number of edges of $S$ and $k,k'$ are as in Equation \eqref{eq:numbervertices}, then there exists a function $q_S : C^r \to \R_+$ which is locally polynomial, with domains of polynomiality that are polyhedral cones and with total degree $lr-(2l+d)(k'-1)$ in these cones, such that we have the asymptotics
$$\mathrm{GF}_S((tx_e)_{e \in E(S)}) \simeq_{t \to \infty} t^{lr-(2l+d)(k'-1)}\,q_S((x_e)_{e \in E(S)})$$
if the $tx_e$'s are dominant weights in the interior of the Weyl chamber.}
\end{quote}
\vspace{2mm}

\noindent Actually, this estimate cannot be true for any family of dominant weights $(tx_e)_{e \in E(S)}$, because a graph functional $\mathrm{GF}_S((\lambda_e)_{e \in E(S)})$ usually vanishes outside a full rank sublattice of $(\Z\Omega)^r$. For instance, with the Littlewood--Richardson coefficients, $c^{\lambda,\mu}_\nu=0$ if $\lambda+\mu-\nu$ does not belong to the root lattice, which is smaller than the weight lattice. Thus, the correct estimate should rather be
$$\mathrm{GF}_S((tx_e)_{e \in E(S)}) \simeq_{t \to \infty} \begin{cases}
    t^{lr-(2l+d)(k'-1)}\,q_S((x_e)_{e \in E(S)}) &\text{if }(tx_e)_{e \in E(S)} \in A_S,\\
    0 &\text{otherwise},
\end{cases}$$
where $A_S \subset (\Z\Omega)^r$ is a sublattice with maximal rank $rd$. Assuming that this is true, we would then obtain the analogue of Theorems \ref{thm:asymptoticsonevertex} and \ref{thm:asymptoticstwovertices} for any connected reduced circuit $R(H,T)$, with
\begin{align*}
e_R^{(\ell)} &= \lim_{N \to \infty} N^{k-1}\,E_{H,T,N} \\
&= (\ell')^{k-1}\,\frac{(\vol(\tlie/\tlie_\Z))^{k'-1}}{[(\Z\Omega)^r:A_S]}  \int_{C^r}\left(\prod_{e \in E(R)}\frac{((\partial_{\Phi_-}\tildeJ_{\R\Omega})(x_e))^{l_e}}{(2\pi)^{\frac{l_e}{2}}(\delta(x_e))^{l_e-1}} \right)q_S(x_1,\ldots,x_r) \DD{x_1}\,\cdots \DD{x_r}.
\end{align*}
Finally, our assumption on the asymptotic behavior of the graph functionals would follow immediately from the fact that these functionals $\mathrm{GF}_S$ count the integer points in certain polytopes $\mathscr{P}((\lambda_e)_{e\in E(S)})$ whose equations are determined by affine functions of the dominant weights; as explained just after the statement of Proposition \ref{prop:multilittlewood} and in more details in Section \ref{sec:crystal}, this is the case when $S$ consists of two vertices. We now have fully explained our Conjecture \ref{conj:graphfunctional}, except for the belonging of the polytopes to the so-called string cone of the group $G^r$; this is also expained in the appendix at the end of the paper. We hope to be able to prove the conjecture by interpreting the graph functionals in the theory of crystal bases, and by describing them in terms of string parametrisations. We close this section by two remarks.

\begin{remark}\label{remark:vanishing}
The formula $lr-(2l+d)(k'-1)$ does not give a non-negative number for any connected graph $S$; for instance, if $G=\SU(2)$ and $S$ is the connected oriented graph associated to the reduced circuit from Figure \ref{fig:smallproblem}, then $l=d=1$, $r=5$ and $k'=3$, therefore $lr-(2l+d)(k'-1)=5-6=-1$. In this situation, the corresponding polytope should be empty, and our conjecture should imply some vanishing results, which can be stated informally as follows: if one takes a graph functional of irreducible representations with too many Haar distributed random variables $g_1,\ldots,g_{k'}$ in comparison to the number $r$ of characters appearing, then this integral vanishes. This is not very surprising since $\int_G \ch^\lambda(g)\DD{g}=0$ for any non trivial representation, but our conjecture would make this much more precise.
\end{remark}

\begin{remark}\label{remark:lastone}
In the general case, it is certainly hopeless to have a beautiful closed formula for the functions $q_S$ introduced above. It should however be noticed that the analogue of Conjecture \ref{conj:graphfunctional} is trivially true when $G$ is a torus $\Tor^d$ (this is not a semisimple Lie group, but the whole theory adapts \emph{mutatis mutandis}). In this case, the weights of irreducible representations are elements of $\Z^d$, and one can check that the graph functionals are indicator functions of sublattices of $(\Z^d)^r$. This means that $M_s$ can always be written as a sum over reduced circuits of certain weights which are integrals of combinations of Bessel functions, but this time without a complicated locally polynomial function as the measure of integration. This seems a promising approach for the problem of computing a precise upper bound on $M_s$; note however that we need a good control of these weights if we want to improve substantially the arguments from Proposition \ref{prop:superpoissonlimit}.
\end{remark}
\bigskip

\section{Appendix: Geometry of the classical sscc Lie groups}\label{sec:appendix}
In this appendix, we describe for each classical case: the maximal torus $T$; the Weyl group $W$; the corresponding weight lattice $\Z\Omega$ and root system $\Phi$; the set of dominant weights $\hatG$; the dimensions of the corresponding irreducible representations. We also compute the volumes of the classical sscc Lie groups with respect to the Riemannian structure given by Equation \eqref{eq:normalisation} and the opposite of the Killing form. Most of these results can be found in the classical text books \cite{Hel78,FH91}, and we stick to the conventions of a previous paper \cite{Mel14}. Since $\Spin(n)$ is not easily described in terms of matrices, in the following we shall use numerous times the two-fold covering map $\pi : \Spin(n) \to \SO(n)$. 
\medskip

\subsection{Maximal tori and their characters}
The maximal tori are chosen as follows:
\begin{align*}
\text{type }\mathrm{A}_n:\quad G &= \SU(n+1)\\
T &= \{\diag(\E^{\I\theta_1},\E^{\I\theta_2},\ldots,\E^{\I\theta_{n+1}})\,\,|\,\,\theta_i \in [0,2\pi],\,\,\theta_1+\theta_2+\cdots+\theta_{n+1} \in 2\pi \Z\}; \\
\text{type }\mathrm{B}_n:\quad G &= \Spin(2n+1)\\
T &= \pi^{-1}\left(\{\diag(R_{\theta_1},R_{\theta_2},\ldots,R_{\theta_n},1)\,\,|\,\,\theta_i \in [0,2\pi]\}\right); \\
\text{type }\mathrm{C}_n:\quad G &= \SP(n)\\
T &= \{\diag(\E^{\I\theta_1},\E^{\I\theta_2},\ldots,\E^{\I\theta_{n}})\,\,|\,\,\theta_i \in [0,2\pi]\}; \\
\text{type }\mathrm{D}_n:\quad G &= \Spin(2n)\\
T &= \pi^{-1}\left(\{\diag(R_{\theta_1},R_{\theta_2},\ldots,R_{\theta_n})\,\,|\,\,\theta_i \in [0,2\pi]\}\right),
\end{align*}
where $R_{\theta} = \left( \begin{smallmatrix}
\cos \theta & -\sin \theta\\
\sin \theta & \cos \theta
\end{smallmatrix} \right)$. In each case, we denote $e_i$ the morphism $T \to \C^\times$ which sends an element of the torus of parameters $(\theta_i)_{i}$ to $\E^{\I \theta_i}$ (in type $\mathrm{B}_n$ and $\mathrm{D}_n$, $e_i$ factors through $\pi$). Notice that in type $\mathrm{A}_n$, by definition of the torus, $e_1+\cdots+e_{n+1} = 0$. In type $\mathrm{B}_n$, one can also define a morphism
$$\omega_n : T \to \C^\times$$
which is the highest weight of the so-called spin representation of $\Spin(2n+1)$, and which has the property that $2\omega_n = e_1+e_2+\cdots+e_n$; thus, formally, $\omega_n=\frac{1}{2}(e_1+\cdots+e_n)$. Notice that $\omega_n$ does not factor through the covering map $\pi$. Similarly, in type $\mathrm{D}_n$, one can define two morphisms
\begin{align*}
\omega_{n-1} : T &\to \C^\times; \\
\omega_{n} : T &\to \C^\times,
\end{align*}
which are the highest weights of the even and odd half-spin representations of $\Spin(2n)$, and which are formally equal to $\frac{1}{2}(e_1+\cdots+e_{n-1}+e_n)$ and $\frac{1}{2}(e_1+\cdots+e_{n-1}-e_n)$. We refer to \cite[Chapter 20]{FH91} and to \cite[\S1.8]{Jost11} for the constructions with spins.
\medskip

\subsection{Weyl groups}
The Weyl groups $W = \mathrm{Norm}(T)/T$ corresponding to the previous choices of maximal tori are:
\begin{align*}
\text{type }\mathrm{A}_n:\quad W &= \sym(n+1),\text{ acting by permutation of the angles};\\
 \text{type }\mathrm{B}_n: \quad W &= (\Z/2\Z) \wr \sym(n), \text{ acting by permutation and inversion of the angles};\\
 \text{type }\mathrm{C}_n: \quad W &= (\Z/2\Z) \wr \sym(n), \text{ acting by permutation and inversion of the angles};\\
 \text{type }\mathrm{D}_n: \quad W &= ((\Z/2\Z) \wr \sym(n))^{\mathrm{even}}, \text{ acting by permutation and inversion of the angles}.
\end{align*}
In the last case, the signed permutations, which can be represented as pairs $((\eps_1,\ldots,\eps_n),\sigma)$ with $\eps_i = \pm 1$ and $\sigma \in \sym(n)$, are allowed only when $\eps_1\eps_2\cdots \eps_n = +1$.
\medskip

\subsection{Weight lattices and root systems}
We denote $\R\Omega = \Span_{\R}(e_1,\ldots,e_n)$ in type $\mathrm{B}_n$, $\mathrm{C}_n$ and $\mathrm{D}_n$, and $\R\Omega= \Span_{\R}(e_1,\ldots,e_{n+1})/\R(e_1+\cdots+e_{n+1})$ in type $\mathrm{A}_n$. The vector space $\R\Omega$ has dimension $n$ in type $\mathrm{A}_n$, $\mathrm{B}_n$, $\mathrm{C}_n$ and $\mathrm{D}_n$. The weight lattice $\Z\Omega$ is the $\Z$-lattice of maximal rank $n$ spanned by the fundamental weights $\omega_1,\ldots,\omega_n$:
\begin{align*}
\text{type }\mathrm{A}_n:\quad &\omega_i=(e_1+\cdots+e_i) - \frac{i}{n+1}\,(e_{1}+\cdots+e_{n+1});\\
\text{type }\mathrm{B}_n:\quad &\omega_{i\leq n-1}=e_1+\cdots+e_i ;\quad \omega_n =\frac{1}{2}(e_1+\cdots+e_n);\\
\text{type }\mathrm{C}_n:\quad &\omega_i=e_1+\cdots+e_i;\\
\text{type }\mathrm{D}_n:\quad &\omega_{i\leq n-2}=e_1+\cdots+e_i;\quad \omega_{n-1,n} =\frac{1}{2}(e_1+\cdots+e_{n-1}\pm e_n).
\end{align*}
The dominant weights in $\hatG$ are the positive integer combinations of these fundamental weights. They have positive scalar products with the positive roots:
\begin{align*}
\text{type }\mathrm{A}_n:\quad &\Phi_+ = \{e_i-e_j,\,\,1\leq i<j\leq n+1\};\\
\text{type }\mathrm{B}_n:\quad &\Phi_+ = \{e_i,\,\,1\leq i \leq n\} \sqcup \{e_i-e_j,\,\,1\leq i<j\leq n\}\sqcup \{e_i+e_j,\,\,1\leq i<j\leq n\};\\
\text{type }\mathrm{C}_n:\quad &\Phi_+ = \{2e_i,\,\,1\leq i \leq n\} \sqcup \{e_i-e_j,\,\,1\leq i<j\leq n\}\sqcup \{e_i+e_j,\,\,1\leq i<j\leq n\};\\
\text{type }\mathrm{D}_n:\quad &\Phi_+ = \{e_i-e_j,\,\,1\leq i<j\leq n\}\sqcup \{e_i+e_j,\,\,1\leq i<j\leq n\}.
\end{align*}
We have drawn in Figure \ref{fig:rank2} the weight lattices, the root systems and the Weyl chambers in rank $2$.
\begin{figure}[ht]
\begin{center}      
\begin{tikzpicture}[scale=1]
 \foreach \x in {(0,0),(2,0),(-1,0),(-2,0),(0,1.733),(1,1.733),(2,1.733),(-1,1.733),(-2,1.733),(0,-1.733),(1,-1.733),(2,-1.733),(-1,-1.733),(-2,-1.733),(2.5,0.866),(2.5,-0.866),(-2.5,0.866),(-2.5,-0.866),(1.5,0.866),(-0.5,0.866),(-1.5,0.866),(0.5,-0.866),(1.5,-0.866),(-0.5,-0.866),(-1.5,-0.866)}
\fill \x circle (1.5pt);
\draw [pattern = north west lines, pattern color=black!50!white] (60:3.3) -- (0,0) -- (3.3,0);
\foreach \x in {(0,-1.733),(-1.5,-0.866),(-1.5,0.866)}
\draw [-{>[scale=2]},thick,blue!30!green] (0,0) -- \x;
\foreach \x in {(0,1.733),(1.5,-0.866),(1.5,0.866)}
\draw [-{>[scale=2]},thick,blue!70!green] (0,0) -- \x;
\foreach \x in {(1,0),(0.5,0.866)}
{\draw [red!50!black] \x circle (3pt);
\fill [red!50!black] \x circle (1.5pt);}
\draw (-3,-1) node {$\mathrm{A}_2$};
\begin{scope}[shift={(9,0)}]
\foreach \x in {-3,-2,-1,0,1,2,3}
 {\foreach \y in {-2,-1,0,1,2}
 \fill (\x,\y) circle (1.5pt);}
\foreach \x in {-2.5,-1.5,-0.5,0.5,1.5,2.5}
 {\foreach \y in {-1.5,-0.5,0.5,1.5}
 \fill (\x,\y) circle (1.5pt);}
\fill [pattern = north west lines, pattern color=black!50!white] (3.3,2.33) -- (45:3.3) -- (0,0) -- (3.3,0);
\draw (45:3.3) -- (0,0) -- (3.3,0);
\foreach \x in {(0,-1),(-1,0),(-1,-1),(-1,1)}
\draw [-{>[scale=2]},thick,blue!30!green] (0,0) -- \x;
\foreach \x in {(0,1),(1,0),(1,1),(1,-1)}
\draw [-{>[scale=2]},thick,blue!70!green] (0,0) -- \x;
\foreach \x in {(1,0),(0.5,0.5)}
{\draw [red!50!black] \x circle (3pt);
\fill [red!50!black] \x circle (1.5pt);}
\draw (-3.5,-1) node {$\mathrm{B}_2$};
\end{scope}
\begin{scope}[shift={(0,-6)}]
\foreach \x in {-3,-2,-1,0,1,2,3}
 {\foreach \y in {-2,-1,0,1,2}
 \fill (\x,\y) circle (1.5pt);}
\fill [pattern = north west lines, pattern color=black!50!white] (3.3,2.33) -- (45:3.3) -- (0,0) -- (3.3,0);
\draw (45:3.3) -- (0,0) -- (3.3,0);
\foreach \x in {(0,-2),(-2,0),(-1,-1),(-1,1)}
\draw [-{>[scale=2]},thick,blue!30!green] (0,0) -- \x;
\foreach \x in {(0,2),(2,0),(1,1),(1,-1)}
\draw [-{>[scale=2]},thick,blue!70!green] (0,0) -- \x;
\foreach \x in {(1,0),(1,1)}
{\draw [red!50!black] \x circle (3pt);
\fill [red!50!black] \x circle (1.5pt);}
\draw (-3.5,-1) node {$\mathrm{C}_2$};
\end{scope}
\begin{scope}[shift={(9,-6)}]
\foreach \x in {-3,-2,-1,0,1,2,3}
 {\foreach \y in {-2,-1,0,1,2}
 \fill (\x,\y) circle (1.5pt);}
\foreach \x in {-2.5,-1.5,-0.5,0.5,1.5,2.5}
 {\foreach \y in {-1.5,-0.5,0.5,1.5}
 \fill (\x,\y) circle (1.5pt);}
 \fill [pattern = north west lines, pattern color=black!50!white] (3.3,2.33) -- (45:3.3) -- (0,0) -- (-45:3.3) -- (3.3,-2.33);
 \draw (45:3.3) -- (0,0) -- (-45:3.3) ;
\foreach \x in {(-1,-1),(-1,1)}
\draw [-{>[scale=2]},thick,blue!30!green] (0,0) -- \x;
\foreach \x in {(1,1),(1,-1)}
\draw [-{>[scale=2]},thick,blue!70!green] (0,0) -- \x;
\foreach \x in {(0.5,-0.5),(0.5,0.5)}
{\draw [red!50!black] \x circle (3pt);
\fill [red!50!black] \x circle (1.5pt);}
\draw (-3.5,-1) node {$\mathrm{D}_2$};
\end{scope}
\end{tikzpicture}
\caption{The weight lattices in type $\mathrm{A}_2$, $\mathrm{B}_2$, $\mathrm{C}_2$ and $\mathrm{D}_2$.\label{fig:rank2}}
\end{center}
\end{figure}
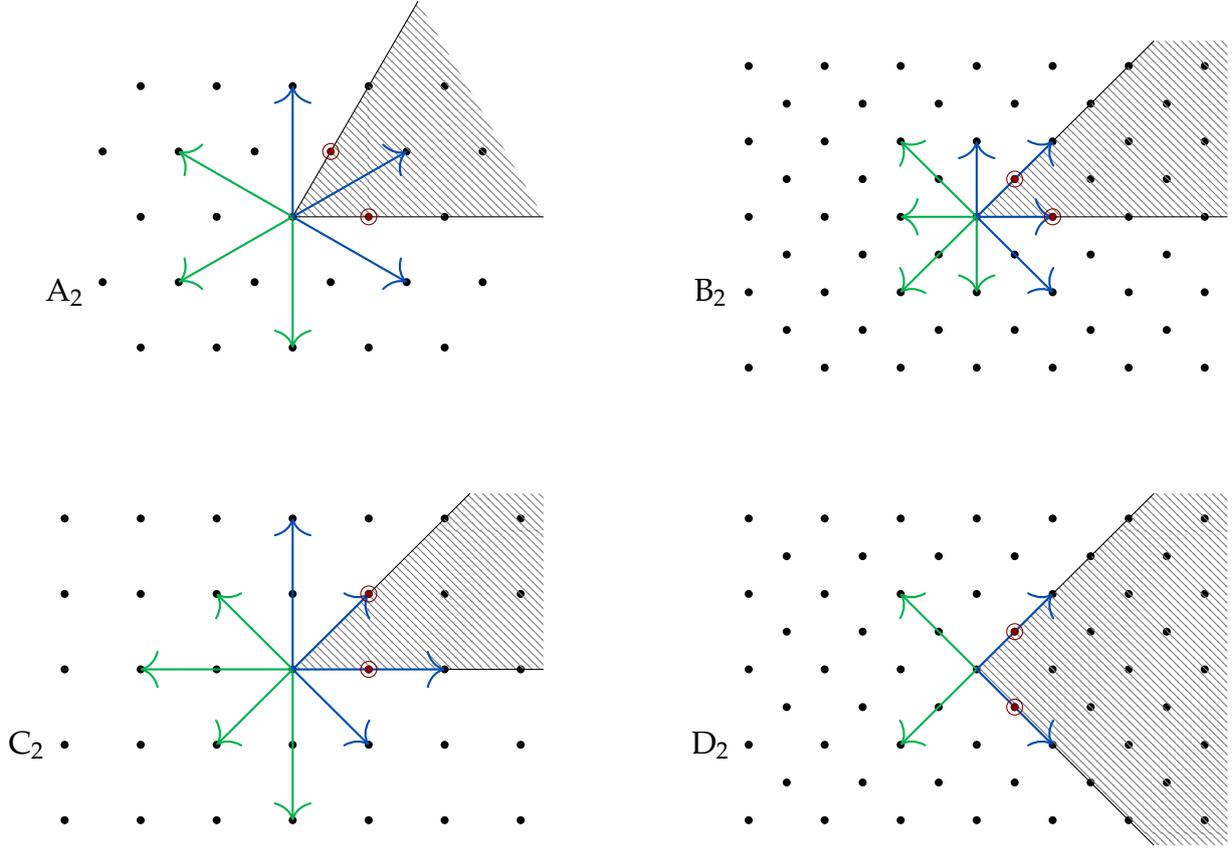

\subsection{Scalar product on the weight lattice}
Let us compute for each case the scalar product $\scal{\cdot}{\cdot}$ on $\R\Omega$ which is dual to the Killing form. 
As explained in Section \ref{subsec:normalisation}, the Killing form is given by:
\begin{align*}
 \text{type }\mathrm{A}_n:\quad B(X,Y) &= (2n+2)\,\tr(XY),\,\,\,X,Y \in \mathfrak{su}(n+1);\\
 \text{type }\mathrm{B}_n:\quad B(X,Y) &= (2n-1)\,\tr(XY),\,\,\,X,Y \in \mathfrak{so}(2n+1);\\
 \text{type }\mathrm{C}_n:\quad B(X,Y) &= (4n+4)\,\mathrm{Re}(\tr(XY)),\,\,\,X,Y \in \mathfrak{sp}(n);\\
 \text{type }\mathrm{D}_n:\quad B(X,Y) &= (2n-2)\,\tr(XY),\,\,\,X,Y \in \mathfrak{so}(2n).
\end{align*}
For each (simple) root $\alpha$, there is a unique vector $T_\alpha \in \tlie_{\C}$ such that $\alpha(t) = B(t,T_\alpha)$. The scalar product on $\R\Omega$ is then given by $\scal{\alpha}{\beta}=B(T_\alpha,T_\beta)$. One thus obtains the following scalar products:
\begin{itemize}
     \item Each time, the vectors $e_i$ form an orthogonal basis, with the following square norms:
     \begin{align*}
 \text{type }\mathrm{A}_n:\quad \|e_i\|^2 &= \frac{1}{2n+2};\\
 \text{type }\mathrm{B}_n:\quad \|e_i\|^2 &= \frac{1}{4n-2};\\
 \text{type }\mathrm{C}_n:\quad \|e_i\|^2 &= \frac{1}{4n+4};\\
 \text{type }\mathrm{D}_n:\quad \|e_i\|^2 &= \frac{1}{4n-4}.
\end{align*}
    \item In type $\mathrm{A}_n$, the weight space $\R\Omega$ is embedded in $\Span_{\R}(e_1,\ldots,e_{n+1})$ as the hyperplane $\R\Omega = \Span_{\R}(\alpha_1,\ldots,\alpha_n)$ with $\alpha_i=e_i-e_{i+1}$ (the $\alpha_i$'s are the simple roots).
    \item In type $\mathrm{B}_n$, $\mathrm{C}_n$ and $\mathrm{D}_n$, the weight space $\R\Omega$ is $\Span_{\R}(e_1,\ldots,e_n)$. 
\end{itemize} 
Note that in many representation theoretic formulas, one does not need to know exactly the normalisation of the vectors $e_i$, because one deals with quotients of scalar products (for instance, in Weyl's dimension formula). However, the knowledge of the normalisation is required for instance in Theorem \ref{thm:gaussianlimit}, and in several other theorems stated in this paper. 
\medskip

\subsection{Dominant weights and dimensions of the irreducible representations}
It is convenient to describe a dominant weight $\lambda \in \hatG$ by means of its coordinates $\lambda_1,\lambda_2,\ldots,\lambda_n$ in the basis $(e_1,\ldots,e_n)$ or $(e_1,\ldots,e_{n+1})$. Thus, we have the following descriptions of the sets $\hatG$ (see \cite[\S2.3]{Mel14}):
\begin{align*}
\text{type }\mathrm{A}_n:\quad\hatG &= \{\text{integer partitions }(\lambda_1\geq \lambda_2 \geq \cdots \geq \lambda_n \geq 0)\,\,| \,\,\forall i,\,\,\lambda_i \in \N\};\\
\text{type }\mathrm{B}_n:\quad\hatG &= \left\{\text{half-integer partitions }(\lambda_1\geq \lambda_2 \geq \cdots \geq \lambda_n \geq 0)\,\,\big| \,\,\substack{\!\!\forall i,\,\,\lambda_i \in \N\\ \text{or }\forall i,\,\,\lambda_i-\frac{1}{2} \in \N}\right\};\\
\text{type }\mathrm{C}_n:\quad\hatG &= \{\text{integer partitions }(\lambda_1\geq \lambda_2 \geq \cdots \geq \lambda_n \geq 0)\,\,| \,\,\forall i, \,\,\lambda_i \in \N\};\\
\text{type }\mathrm{D}_n:\quad\hatG &= \left\{\text{signed half-integer partitions }(\lambda_1\geq \lambda_2 \geq \cdots \geq \eps_n\lambda_n \geq 0)\,\,\big| \,\,\substack{\!\!\forall i,\,\,|\lambda_i| \in \N\\ \text{or }\forall i,\,\,|\lambda_i|-\frac{1}{2} \in \N}\right\},
\end{align*}
where in the last case the sign $\eps_n$ of $\lambda_n$ is allowed to be $\pm 1$ (unless $\lambda_n=0$). 
\begin{itemize}
    \item In type $\mathrm{A}_n$, the integer partition $\lambda$ corresponds to the highest weight
    $$\lambda_1 e_1 + \cdots + \lambda_n e_n - \frac{|\lambda|}{n+1}(e_1+\cdots+e_{n+1}), $$
    where $|\lambda| = \sum_{i=1}^{n}\lambda_i$. It is then convenient to denote $\lambda_{n+1}=0$.
    \item In type $\mathrm{B}_n$, $\mathrm{C}_n$ and $\mathrm{D}_n$, the integer or half-integer partition $\lambda$ corresponds to the highest weight $\lambda_1 e_1 + \cdots + \lambda_n e_n$.
\end{itemize}
With these conventions, the Weyl formula yields the following explicit values for the dimensions $d_\lambda$:
\begin{align*}
\text{type }\mathrm{A}_n:\quad d_\lambda &=\prod_{1\leq i < j \leq n+1} \frac{\lambda_i-\lambda_j+j-i}{j-i};\\
\text{type }\mathrm{B}_n:\quad d_\lambda &= \prod_{1\leq i < j \leq n} \frac{\lambda_i-\lambda_j+j-i}{j-i}\,\prod_{1\leq i \leq j\leq n} \frac{\lambda_i+\lambda_j + 2n+1-i-j}{2n+1-i-j};\\
\text{type }\mathrm{C}_n:\quad d_\lambda &= \prod_{1\leq i < j \leq n} \frac{\lambda_i-\lambda_j+j-i}{j-i}\,\prod_{1\leq i \leq j\leq n} \frac{\lambda_i+\lambda_j + 2n+2-i-j}{2n+2-i-j};\\
\text{type }\mathrm{D}_n:\quad d_\lambda &= \prod_{1\leq i < j \leq n} \frac{\lambda_i-\lambda_j+j-i}{j-i}\, \frac{\lambda_i+\lambda_j + 2n-i-j}{2n-i-j}.
\end{align*}

\begin{remark}\label{remark:fromspintoso}
In type $\mathrm{B}_n$ and $\mathrm{D}_n$, a sublattice of $\hatG$ corresponds to the irreducible representations that factor through the covering map $\Spin(d) \to \SO(d)$, hence are irreducible representations of the special orthogonal group. Thus, integer partitions of length $n$ label the irreducible representations of $\SO(2n+1)$, whereas signed integer partitions of length $n$ label the irreducible representations of $\SO(2n)$. This connection between representations of $\Spin(d)$ and representations of $\SO(d)$ is detailed in \cite[Chapter 31]{Bump13}. For studying random geometric graphs, everything that can be said for a random geometric graph of fixed level $L$ drawn on $\Spin(d)$ stays true for $\SO(d)$, but with the eigenvalues indexed by the dominant weights in the integral sublattice $\widehat{\SO(d)}$ of $\widehat{\mathrm{\Spin(d)}}$.
\end{remark}
\medskip

\subsection{Volume of a sscc Lie group}\label{subsec:volumes}
Let us finally explain how to use the weight lattices and the root systems in order to compute the volume of a sscc Lie group $G$ with respect to the volume form associated to opposite Killing form. This volume is involved in the renormalisation parameters of the Poissonian regime of random geometric graphs, and it is given by
$$\vol(G) = (2\sqrt{2}\,\pi)^{\dim G}\,\prod_{\alpha \in \Phi_+} \sinc\left(2\pi\,\scal{\rho}{\alpha}\right),$$
see \cite[Formula 4.32.1]{KP84}. Alternatively, if $(m_i)_{i \in \lle 1,d\rre}$ is the collection of the Chevalley exponents of the Weyl group of $G$ (see for instance \cite{Col58}), and if $2\pi \mathfrak{t}_{\Z}$ is the kernel of $\exp : \mathfrak{t} \to T$ ($T$ being the fixed maximal torus), then 
$$    \vol(G) = \vol(\mathfrak{t}/\mathfrak{t}_{\Z})\,\left(\prod_{\alpha \in \Phi_+} B(\alpha^\vee,\alpha^\vee)\right) \prod_{i=1}^d \frac{2\,\pi^{m_i+1}}{m_i!},
$$
the last product being the volume of the cartesian product of unit spheres $\prod_{i=1}^d \mathbb{S}^{2m_i+1}$; see \cite{Mac80,Hash97}. In this formula, $\alpha^\vee$ is the coroot of the root $\alpha$, that is the vector in $\tlie_{\C}$ such that $2\frac{B(t,\alpha^\vee)}{B(\alpha^\vee,\alpha^\vee)} = \alpha(t)$. In other words, $T_\alpha = \frac{2\,\alpha^\vee}{B(\alpha^\vee,\alpha^\vee)}$. As a consequence,
\begin{align*}
\vol(G) &= \vol(\mathfrak{t}/\mathfrak{t}_{\Z})\,4^{|\Phi_+|}\,\left(\prod_{\alpha \in \Phi_+} \frac{1}{\|\alpha\|^2}\right) \prod_{i=1}^d \frac{2\,\pi^{m_i+1}}{m_i!} \\
&= \vol(\mathfrak{t}/\mathfrak{t}_{\Z})\,2^{\dim G}\,\left(\prod_{\alpha \in \Phi_+} \frac{1}{\|\alpha\|^2}\right) \prod_{i=1}^d \frac{\pi^{m_i+1}}{m_i!}. 
\end{align*}
The terms of this formula are:
\vspace*{3mm}
\begin{center}
\begin{tabular}{|c|c|c|c|c|}
\hline type & $\vol(\tlie/\tlie_\Z)$ & $\dim G$ & $\prod_{\alpha \in \Phi_+} \frac{1}{\|\alpha\|^2}$ & $m_1,\ldots,m_n$ \\
\hline $\mathrm{A}_n$ & $2^{\frac{n}{2}}\,(n+1)^{\frac{n+1}{2}}$ & $n(n+2)$ & $(n+1)^{\frac{n(n+1)}{2}} $& $1,2,3,\ldots,n$\\
$\mathrm{B}_n$ & $2^{\frac{n}{2}+1}\,(2n-1)^{\frac{n}{2}}$ & $n(2n+1)$ & $2^n\, (2n-1)^{n^2}$ & $1,3,5,\ldots,2n-1$\\
$\mathrm{C}_n$ & $2^{n}\,(n+1)^{\frac{n}{2}}$  & $n(2n+1)$ & $2^{n(n-1)}\, (n+1)^{n^2}$ & $1,3,5,\ldots,2n-1$\\
$\mathrm{D}_n$ & $2^{n+1}\,(n-1)^{\frac{n}{2}}$ & $n(2n-1)$ & $2^{n(n-1)}\,(n-1)^{n(n-1)} $& $1,3,5,\ldots,2n-3,n-1$ \\
\hline
\end{tabular}
\end{center}\medskip

\noindent The term $\vol(\tlie/\tlie_\Z)$ appears in many asymptotic formulas in Section \ref{sec:ART}. As a consequence of the computations above,
\begin{align*}
\vol(\SU(n)) &= \frac{\left(2^{2n+3}\,\pi^{n+2}\right)^{\frac{n-1}{2}}\,n^{\frac{n^2}{2}}}{1!\,2!\,\cdots \,(n-1)!} ;\\
\vol(\SP(n)) &= \frac{\left(2^{3n+1}\,\pi^{n+1}\,(n+1)^{\frac{2n+1}{2}}\right)^n}{1!\,3!\,\cdots\,(2n-1)!};\\
\vol(\Spin(2n)) &= 2\,\frac{\left(2^{3n-1}\,\pi^{n}\,(n-1)^{\frac{2n-1}{2}}\right)^n}{1!\,3!\,\cdots\,(2n-3)!\,(n-1)!};\\
\vol(\Spin(2n+1)) &= 2\,\frac{\left(2^{2n+\frac{5}{2}}\, \pi^{n+1} \, (2n-1)^{\frac{2n+1}{2}}\right)^n}{1!\,3!\,\cdots\,(2n-1)!}.
\end{align*}
Of course, for $n \geq 3$, $\vol(\SO(n)) = \frac{1}{2}\, \vol(\Spin(n))$.
\bigskip

\section{Appendix: Crystals of representations and string polytopes}\label{sec:crystal}
This second appendix proves Proposition \ref{prop:multilittlewood} and gives a survey of the theory of crystals of representations. We have tried to explain it in the most pedagogical way that we were able to, and in particular we start with the path model, although it is not really required in our study. Until the end, $G$ is a fixed sscc Lie group, $\glie$ is its Lie algebra, and $d$ is the rank of $G$. The set of simple roots of $G$ is denoted $(\alpha_i)_{i \in \lle 1,d\rre}$. This is a linear basis of $\R\Omega$, and we denote $(\alpha_i^\vee)_{i \in \lle 1,d\rre}$ the basis of simple coroots in $(\R\Omega)^*$, defined by the relations
$$\forall i,j\in \lle 1, d \rre,\,\,\,\alpha_i(\alpha_j^\vee) = 2\,\frac{\scal{\alpha_i}{\alpha_j}}{\scal{\alpha_j}{\alpha_j}}.$$
The dual basis $(\omega_i)_{i \in \lle 1,d\rre}$ of the basis of coroots $(\alpha_i^\vee)_{i \in \lle 1,d\rre}$ consists in the fundamental weights, such that $\Z\Omega = \mathrm{Span}_{\Z}(\omega_1,\ldots,\omega_d)$. Fix a dominant weight $\lambda \in \hatG$, and for $\omega \in \Z\Omega$, denote $V^\lambda(\omega)$ the weight subspace of $V^\lambda$ associated to the weight $\omega$:
$$V^\lambda(\omega) = \{v \in V^\lambda\,|\,\forall t \in T,\,\,t \cdot v = \omega(t)\,v\}.$$
An element of the weight subspace $V^{\lambda}(\omega)$ is called a \emph{weight vector} of $V^\lambda$, and the irreducible representation $V^\lambda$ is the direct sum of its weight subspaces:
$$V^\lambda = \bigoplus_{\omega\,|\,V^\lambda(\omega)\neq 0}V^\lambda(\omega).$$
The set of weights with positive multiplicity in $V^\lambda$ will be denoted $\Omega(\lambda)$; it is a finite subset of $\lambda + R$, where $R$ is the root lattice of $G$, that is the sublattice of $\Z\Omega$ spanned by the (simple) roots. \medskip

\subsection{Crystals and the path model}\label{subsec:pathmodel}
The theory of crystal bases and the path model allow one to compute the multiplicities $K_{\lambda,\omega}=\dim_{\C}(V^\lambda(\omega))$ for $\omega \in \Omega(\lambda)$ (they are also called the \emph{Kostka numbers}). Let $U_q(\glie_{\C})$ be the quantum group of the complexification $\glie_{\C}$ of the Lie algebra $\glie$; it is a deformation with a complex parameter $q$ of the universal enveloping algebra $U(\glie_{\C})$, see \cite{Jim85,Jim86}. There is a corresponding deformation $V^\lambda_q$ of the irreducible module $V^\lambda$, and a notion of weight vectors in $V^\lambda_q$, such that if 
$$V^\lambda_q = \bigoplus_{\omega\,|\,V^\lambda_q(\omega)\neq 0}V^\lambda_q(\omega),$$ then the weights and the multiplicities are the same for $V^\lambda$ and for $V^\lambda_q$:
$$\forall \omega \in \Z\Omega,\,\,\,\dim_{\C} \left(V^\lambda_q(\omega)\right) = \dim_{\C} \left(V^\lambda(\omega)\right).$$
This is the Lusztig--Rosso correspondence, see the original papers \cite{Lus88,Ros88,Ros90}, and \cite[Chapter 5]{Mel17} for a detailed exposition of the case $\glie=\mathfrak{gl}(n)$. The correspondence holds for any $q$ which is not $0$ or a root of unity. Now, a \emph{crystal basis} of the irreducible representation $V^\lambda_q$ is a linear basis $\cryst(\lambda)$ of $V^\lambda_q$ that consists of weight vectors, and such that if $(e_i,f_i,q^{h_i})_{i \in \lle 1,d\rre}$ are the Chevalley generators of $U_q(\glie_{\C})$, then for any vector $v$ of the crystal basis, $e_i \cdot v$ is either $0$ or another vector $v'$ of the crystal basis; and similarly for $f_i \cdot v$. Notice that if $v \in \cryst(\lambda)$ has weight $\omega$ and $v'=e_i \cdot v$ (respectively, $v'=f_i \cdot v$) does not vanish, then $v'$ has weight $\omega + \alpha_i$ (respectively, $\omega-\alpha_i$). The \emph{crystal} of $V^\lambda_q$ is given by a crystal basis $\cryst(\lambda)$, and by the weighted labeled oriented graph:
\begin{itemize}
    \item with vertices $v \in \cryst(\lambda)$,
    \item with labeled oriented edges $v \to_{f_i} v'$ if $v'=f_i \cdot v$,
    \item with a weight map $\mathrm{wt}(\cdot)$ which associates to $v \in \cryst(\lambda)$ the corresponding weight in $\Z\Omega$.
 \end{itemize}
It has been shown independently by Lusztig and Kashiwara that crystal bases of irreducible representations of semisimple Lie algebras always exist, and that their combinatorial structure does not depend on $q$; see \cite{Lus90,Kas90}. In the sequel, we shall only work with the combinatorial object (weighted labeled oriented graph). Indeed, if one knows the crystal of an irreducible representation $V^\lambda$, then one recovers immediately the highest weight of the representation, and all the multiplicities of the weights: for any $\omega \in \Omega(\lambda)$,
$$\dim_\C \left(V^\lambda(\omega)\right) = \card\{v \in \cryst(\lambda)\,|\,v\text{ has weight }\omega\}.$$
As a consequence, we can now forget the underlying quantum groups $U_q(\glie_\C)$. 
\bigskip

There is a concrete description of the crystal $\cryst(\lambda)$ due to Littelmann, see \cite{Litt95,Litt98}, and \cite{BBO05,BBO09} for applications of this path model in probability. We call \emph{path} on the weight space $\R\Omega$ a piecewise linear map $\pi : [0,1] \to \R\Omega$ which starts at $0$, and such that $\pi(1)$ belongs to the weight lattice $\Z\Omega$. We identify two paths if they differ by a continuous reparametrisation. The set of paths is a semigroup for the operation of concatenation:
$$(\pi_1 * \pi_2)(t) = \begin{cases}
    \pi_1(2t) &\text{if }t \in [0,\frac{1}{2}],\\
    \pi_1(1) + \pi_2(2t-1) &\text{if }t \in [\frac{1}{2},1].
\end{cases}$$
Fix a simple root $\alpha$. Given a path $\pi$, we set $g_{\alpha}(\pi,t)=\pi(t)(\alpha^\vee)=2\,\frac{\scal{\pi(t)}{\alpha}}{\scal{\alpha}{\alpha}}$; this is a piecewise linear function on $[0,1]$. Suppose $$m_\alpha = \min_{t \in [0,1]} g_\alpha(\pi,t)\leq -1.$$ 

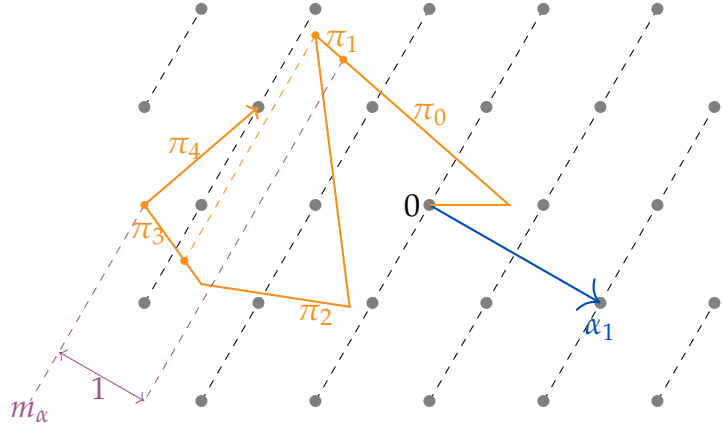
\begin{figure}[ht]
\begin{center}      
\begin{tikzpicture}[scale=1.5]
\draw [dashed] (-2.5,0.866) -- (-2,1.733);
\draw [dashed] (-2.5,-0.866) -- (-1,1.733);
\draw [dashed] (-2,-1.733) -- (0,1.733);
\draw [dashed] (-1,-1.733) -- (1,1.733);
\draw [dashed] (0,-1.733) -- (2,1.733);
\draw [dashed] (1,-1.733) -- (2.5,0.866);
\draw [dashed] (2,-1.733) -- (2.5,-0.866);
\foreach \x in {(0,0),(1,0),(2,0),(-1,0),(-2,0),(0,1.733),(1,1.733),(2,1.733),(-1,1.733),(-2,1.733),(0,-1.733),(1,-1.733),(2,-1.733),(-1,-1.733),(-2,-1.733),(2.5,0.866),(2.5,-0.866),(-2.5,0.866),(-2.5,-0.866),(1.5,0.866),(-0.5,0.866),(-1.5,0.866),(0.5,0.866),(0.5,-0.866),(1.5,-0.866),(-0.5,-0.866),(-1.5,-0.866)}
\fill [gray] \x circle (1.5pt);
\draw (-0.15,0) node {$0$};
\draw [-{>[scale=2]},thick,blue!70!green] (0,0) -- (1.5,-0.866);
\draw [blue!70!green] (1.5,-1.1) node {$\alpha_1$};
\draw [BurntOrange,thick,->] (0,0) -- (0.7,0) -- (-1,1.5) -- (-0.7,-0.9) -- (-2,-0.7) -- (-2.5,0) -- (-1.5,0.866);
\draw [thin,dashed,BurntOrange] (-1,1.5) -- (-2.15,-0.495);
\draw [thin,dashed,DarkOrchid] (-2.5,0) -- (-3.5,-1.733);
\draw [thin,dashed,DarkOrchid] (-0.755,1.285) -- (-2.5,-1.733);
\draw [<->,thin,DarkOrchid] (-2.5,-1.733) -- (-3.25,-1.3); 
\foreach \x in {(-0.755,1.285),(-1,1.5),(-2.5,0),(-2.15,-0.495)}
\fill [BurntOrange] \x circle (1pt);
\draw [BurntOrange] (0,0.8) node {$\pi_0$};
\draw [BurntOrange] (-0.77,1.45) node {$\pi_1$};
\draw [BurntOrange] (-1,-0.97) node {$\pi_2$};
\draw [BurntOrange] (-2.47,-0.25) node {$\pi_3$};
\draw [BurntOrange] (-2.15,0.5) node {$\pi_4$};
\draw [DarkOrchid] (-3.5,-1.83) node {$m_\alpha$};
\draw [DarkOrchid] (-2.9,-1.62) node {$1$};
\end{tikzpicture}
\caption{The decomposition of a path in the weight lattice of $\SU(3)$.\label{fig:decompositionpath}}
\end{center}
\end{figure}    

We cut the path $\pi$ in parts $\pi_0,\pi_1,\ldots,\pi_\ell,\pi_{\ell+1}$ such that $\pi=\pi_0*\pi_1 * \cdots * \pi_\ell*\pi_{\ell+1}$ and:
\begin{enumerate}
\setcounter{enumi}{-1}
\item $\pi_0*\pi_1*\cdots * \pi_\ell$ ends at a point with a value of $g_{\alpha}$ minimal, and it is the smallest part of the whole path $\pi$ with this property; and then $\pi_0$ is the largest part of the path $\pi_0*\pi_1*\cdots * \pi_\ell$ that ends with a value of $g_\alpha$ equal to $m_{\alpha}+1$.
\item either $g_{\alpha}(\pi)$ is strictly decreasing on the interval $[t_{i-1},t_i]$ corresponding to the part $\pi_i$, and $g_{\alpha}(\pi,s)\geq g_{\alpha}(\pi,t_{i-1})$ for $s \leq t_{i-1}$; in other words, $g_{\alpha}(\pi)$ is minimal on the segment $[t_{i-1},t_i]$.
\item or, $g_{\alpha}(\pi,t_{i-1})=g_{\alpha}(\pi,t_i)$ and $g_{\alpha}(\pi,s)\geq g_{\alpha}(\pi,t_{i-1})$ for $s \in [t_{i-1},t_i]$.
\end{enumerate}
This decomposition is better understood in a picture, see Figure \ref{fig:decompositionpath} for an example on the weight lattice of type $\mathrm{A}_2$. 
Denote $s_\alpha$ the reflection with respect to the root $\alpha$, that is the map $x \mapsto x - 2\,\frac{\scal{x}{\alpha}}{\scal{\alpha}{\alpha}}\,\alpha$. For $j \in \lle 1,\ell \rre$, we define
$$\pi_j' = \begin{cases}
    s_\alpha(\pi_j) &\text{if }\pi_j\text{ is of type }(1),\\
    \pi_j &\text{if }\pi_j\text{ is of type }(2).
\end{cases} $$
We then set 
$$e_\alpha(\pi) = \begin{cases}
    \emptyset &\text{if }m_\alpha>-1,\\
    \pi_0*(\pi_1'*\pi_2'*\cdots * \pi_\ell')*\pi_{\ell+1} &\text{if }m_\alpha \leq -1,
\end{cases}$$
and $f_\alpha(\pi) = c \circ e_\alpha \circ c(\pi)$, where $c$ is the involution on paths defined by $(c(\pi))(t) = \pi(1-t)-\pi(1)$. In this definition, $\emptyset$ is a "ghost" path. An example of action of a root operator $e_\alpha$ is in Figure \ref{fig:actionrootoperator}. 

\begin{figure}[ht]
\begin{center}      
\begin{tikzpicture}[scale=1.5]
\draw [dashed] (-2.5,0.866) -- (-2,1.733);
\draw [dashed] (-2.5,-0.866) -- (-1,1.733);
\draw [dashed] (-2,-1.733) -- (0,1.733);
\draw [dashed] (-1,-1.733) -- (1,1.733);
\draw [dashed] (0,-1.733) -- (2,1.733);
\draw [dashed] (1,-1.733) -- (2.5,0.866);
\draw [dashed] (2,-1.733) -- (2.5,-0.866);
\foreach \x in {(0,0),(1,0),(2,0),(-1,0),(-2,0),(0,1.733),(1,1.733),(2,1.733),(-1,1.733),(-2,1.733),(0,-1.733),(1,-1.733),(2,-1.733),(-1,-1.733),(-2,-1.733),(2.5,0.866),(2.5,-0.866),(-2.5,0.866),(-2.5,-0.866),(1.5,0.866),(-0.5,0.866),(-1.5,0.866),(0.5,0.866),(0.5,-0.866),(1.5,-0.866),(-0.5,-0.866),(-1.5,-0.866)}
\fill [gray] \x circle (1.5pt);
\draw (-0.15,0) node {$0$};
\draw [-{>[scale=2]},thick,blue!70!green] (0,0) -- (1.5,-0.866);
\draw [blue!70!green] (1.5,-1.1) node {$\alpha_1$};
\draw [Red,thick,->] (0,0) -- (0.7,0) -- (-0.755,1.285) -- (-0.45,1.182) -- (-0.15,-1.218) -- (-1.45,-1.018) -- (-1.6,-0.813) -- (-1,-0.866) -- (0,0);
\foreach \x in {(-0.45,1.182),(-0.755,1.285),(-1.6,-0.813),(-1,-0.866)}
\fill [Red] \x circle (1pt);
\draw [Red] (0,0.8) node {$\pi_0$};
\draw [Red] (-0.57,1.35) node {$\pi_1'$};
\draw [Red] (-0.45,-1.3) node {$\pi_2'$};
\draw [Red] (-1.25,-0.72) node {$\pi_3'$};
\draw [Red] (-0.65,-0.366) node {$\pi_4$};
\end{tikzpicture}
\caption{The action of $e_1=e_{\alpha_1}$ on the path of Figure \ref{fig:decompositionpath}.\label{fig:actionrootoperator}}
\end{center}
\end{figure}
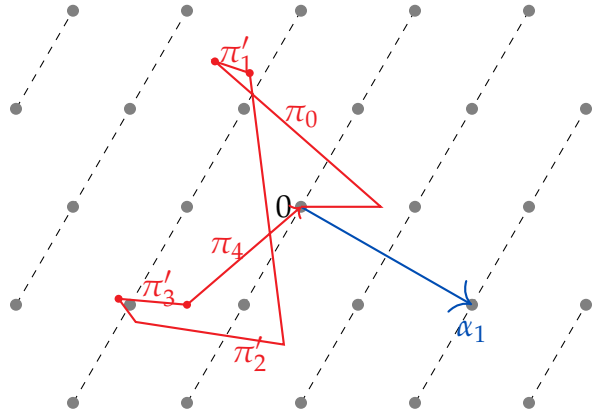
\bigskip

Notice that if $\pi$ is a path such that $e_\alpha(\pi)\neq \emptyset$, then $f_\alpha (e_\alpha(\pi))=\pi$. Similarly, if $\pi$ is a path such that $f_\alpha(\pi)\neq \emptyset$, then $e_\alpha (f_\alpha(\pi))=\pi$. On the other hand, if $\pi$ is a path and $e_\alpha(\pi) \neq 0$ (respectively, $f_\alpha(\pi)\neq 0)$, then $(e_\alpha(\pi))(1) = \pi(1)+\alpha$ (respectively, $(f_\alpha(\pi))(1) = \pi(1)-\alpha$). Therefore, paths in the weight space and root operators are a natural model for crystals. The following theorem ensures that this model is indeed adequate:
\begin{theorem}[Littelmann]
Let $\lambda$ be a dominant weight in $\hatG$, and $\pi_\lambda$ be the segment that connects $0$ to $\lambda$, considered as a path. We introduce the weighted labeled oriented graph:
\begin{itemize}
\item with vertices the paths $\pi=f_{\alpha_{i_1}}f_{\alpha_{i_2}}\cdots f_{\alpha_{i_r}}(\pi_\lambda)$ that are not the ghost path $\emptyset$, and that are obtained from $\pi_\lambda$ by applying operators $f_{\alpha_i}$;
\item with an oriented labeled edge $\pi \to_{f_i} \pi'$ if $\pi' = f_{\alpha_i}(\pi)$;
\item with the weight map $\mathrm{wt} : \pi \mapsto \pi(1)$.
\end{itemize}
The crystal $\cryst(\pi_\lambda)$ that one obtains is finite, and it is isomorphic to the crystal $\cryst(\lambda)$ of the irreducible representation $V^\lambda$. In particular,
$$\dim_\C\left( V^\lambda(\omega)\right) = \card\{\text{paths in the crystal }\cryst(\pi_\lambda)\text{ with endpoint }\omega\},$$
and the character $\mathrm{ch}^\lambda$ is given by the formula $\ch^\lambda = \sum_{\pi \in \cryst(\pi_\lambda)} \E^{\mathrm{wt}(\pi)}$.
\end{theorem}

\noindent Actually, one can take instead of $\pi_\lambda$ any path from $0$ to $\lambda$ that stays in the Weyl chamber; all these paths generate the same crystal $\cryst(\lambda)$.

\begin{example}
Consider the adjoint representation of $\SU(3)$ on $\mathfrak{su}(3)$, which has dimension $8$. The dominant weight is $\lambda = \rho = \omega_1+\omega_2$, and the crystal $\cryst(\lambda)$ is drawn in Figure \ref{fig:crystalsu3}.
\begin{figure}[ht]
\begin{center}  
\begin{tikzpicture}[scale=1.5]
\foreach \x in {(0,0),(1,0),(0.5,0.866),(2,0),(-1,0),(-2,0),(0,1.733),(1,1.733),(2,1.733),(-1,1.733),(-2,1.733),(0,-1.733),(1,-1.733),(2,-1.733),(-1,-1.733),(-2,-1.733),(2.5,0.866),(2.5,-0.866),(-2.5,0.866),(-2.5,-0.866),(1.5,0.866),(-0.5,0.866),(-1.5,0.866),(0.5,-0.866),(1.5,-0.866),(-0.5,-0.866),(-1.5,-0.866)}
\fill [gray] \x circle (1.5pt);
\draw [pattern = north west lines, pattern color=black!50!white] (60:3.3) -- (0,0) -- (3.3,0);
\draw (2.5,1.8) node {$C$};
\draw [NavyBlue] (1.6,1) node {$\lambda$};
\draw [NavyBlue, line width=2pt,->] (0,0) -- (1.5,0.866);
\draw [NavyBlue!50!DarkOrchid, line width=1.5pt,->] (0,0) -- (0,1.733);
\draw [DarkOrchid, line width=1.5pt,->] (0.05,0) -- (0.05,-0.866) -- (0.1,-0.866) -- (0.1,0);
\draw [NavyBlue!50!ForestGreen, line width=1.5pt,->] (0,0) -- (1.5,-0.866);
\draw [ForestGreen, line width=1.5pt,->] (0,0.05) -- (-0.75,0.483) -- (-0.75,0.533) -- (0,0.1);
\draw [ForestGreen!50!Red, line width=1.5pt,->] (0,0) -- (-1.5,0.866);
\draw [DarkOrchid!50!Red, line width=1.5pt,->] (0,0) -- (0,-1.733);
\draw [Red, line width=1.5pt,->] (0,0) -- (-1.5,-0.866);
\begin{scope}[shift={(5.5,0)},scale=0.6]
\node (a) at (1.5,0.866) [circle,fill,NavyBlue] {};
\node (b) at (1.5,-0.866) [circle,fill,NavyBlue!50!ForestGreen] {};
\node (c) at (0,1.733) [circle,fill,NavyBlue!50!DarkOrchid] {};
\node (d) at (-0.1,0.2) [circle,fill,ForestGreen] {};
\node (e) at (0.1,-0.2) [circle,fill,DarkOrchid] {};
\node (f) at (-1.5,0.866) [circle,fill,ForestGreen!50!Red] {};
\node (g) at (0,-1.733) [circle,fill,DarkOrchid!50!Red] {};
\node (h) at (-1.5,-0.866) [circle,fill,Red] {};
\draw [->,thick] (a) to node[auto,swap] {$f_1$} (c);
\draw [->,thick] (a) to node[auto] {$f_2$} (b);
\draw [->,thick] (b) to node[auto,swap] {$f_1$} (d);
\draw [->,thick] (d) to node[auto] {$f_1$} (f);
\draw [->,thick] (f) to node[auto,swap] {$f_2$} (h);
\draw [->,thick] (c) to node[auto] {$f_2$} (e);
\draw [->,thick] (e) to node[auto,swap] {$f_2$} (g);
\draw [->,thick] (g) to node[auto] {$f_1$} (h);
\end{scope}
\end{tikzpicture}

\caption{The crystal of the adjoint representation of $\SU(3)$, viewed as a set of paths.\label{fig:crystalsu3}}
\end{center}
\end{figure}
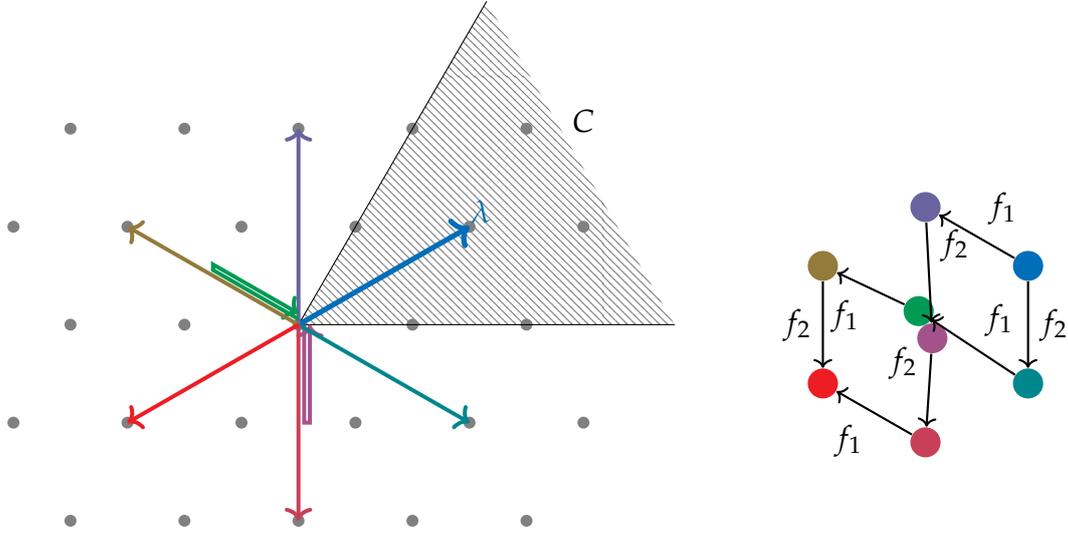
In particular, each of the six roots of $\SU(3)$ has multiplicity $1$ in the adjoint representation, whereas the weight $0$ has multiplicity $2 = \rank(\SU(3))$.
\end{example}

Let us now explain the use of the path model in order to compute tensor products. If $V^\lambda \otimes V^\mu = \sum_{\nu \in \hatG} c^{\lambda,\mu}_\nu \,V^\nu$, then the concatenation product of crystals $\cryst(\pi_\lambda)*\cryst(\pi_\mu)$ is a set of paths such that the action of the root operators on these paths generate a crystal whose connected components are isomorphic to the elements of the multiset $\{(\cryst(\pi_\nu))^{m_\nu},\nu\in\hatG\} $. Therefore, $c^{\lambda,\mu}_\nu$ is equal to the number of paths $\pi \in \cryst(\pi_\mu)$ such that $\pi_\lambda * \pi$ always stays in the Weyl chamber $C$, and $\pi_\lambda*\pi$ ends at the dominant weight $\nu$; see \cite[Proposition 2 and Corollary 1]{Litt98}. In a moment, we shall reinterpret this rule in the string polytope of $V^\mu$, see Theorem \ref{thm:berensteinzelevinsky}.

\begin{remark}
This link between the tensor product of representations and the concatenation product of crystals proves that the set of dominant weights $\nu$ such that $c^{\lambda,\mu}_{\nu} >0$ is included in $R+\lambda+\mu$, where $R$ is the root lattice. Indeed, the crystal $\cryst(\pi_\mu)$ consists of paths with weights in $\mu + R$, and if $c^{\lambda,\mu}_{\nu} >0$, then there is a path in this crystal that connects $0$ to $\nu-\lambda$.
\end{remark}

 
\subsection{The cone and the polytopes of string parametrisations}\label{subsec:conepolytope} In this paragraph, we fix a dominant weight $\lambda$, and a decomposition of the longest element $w_0$ of the Weyl group as a product of reflections $s_{\alpha_i}$ along the walls of the Weyl chamber $C$:
$$w_0 = s_{\alpha_{i_1}}s_{\alpha_{i_2}}\cdots s_{\alpha_{i_l}}.$$
Notice that $l = |\Phi_+|$ is equal to the number of positive roots of $G$. If $v$ is an element of the crystal $\cryst(\lambda)$, we call \emph{string parametrisation} of $v$ the vector of integers $(n_1,n_2,\ldots,n_l) \in \N^r$ such that:
\begin{itemize}
    \item $n_1$ is the maximal integer such that $e_{\alpha_{i_1}}^{n_1}(v) \neq 0$;
    \item if $n_1,\ldots,n_{s-1}$ are known, then $n_s$ is the maximal integer such that $e_{\alpha_{i_s}}^{n_s}\cdots e_{\alpha_{i_2}}^{n_2} e_{\alpha_{i_1}}^{n_1}(v) \neq 0$.
 \end{itemize} 

\begin{example}
For $\SU(3)$, we fix the decomposition $s_1s_2s_1$ of the longest element $w_0$ of the Weyl group $W=\sym(3)$. Then, the string parametrisation of the crystal graph of the adjoint representation appears in Figure \ref{fig:stringparametrisation}.

\begin{figure}[ht]
\begin{center}      
\begin{tikzpicture}[scale=1.8]
\node (a) at (1.5,0.866)  {(0,0,0)};
\node (b) at (1.5,-0.866)  {(0,1,0)};
\node (c) at (0,1.733)  {(1,0,0)};
\node (d) at (-0.25,0.25)  {(1,1,0)};
\node (e) at (0.15,-0.3)  {(0,1,1)};
\node (f) at (-1.5,0.866)  {(2,1,0)};
\node (g) at (0,-1.733)  {(0,2,1)};
\node (h) at (-1.5,-0.866)  {(1,2,1)};
\draw [->,thick] (a) to node[auto,swap] {$f_1$} (c);
\draw [->,thick] (a) to node[auto] {$f_2$} (b);
\draw [->,thick] (b) to node[auto,swap] {$f_1$} (d);
\draw [->,thick] (d) to node[auto] {$f_1$} (f);
\draw [->,thick] (f) to node[auto,swap] {$f_2$} (h);
\draw [->,thick] (c) to node[auto] {$f_2$} (e);
\draw [->,thick] (e) to node[auto,swap] {$f_2$} (g);
\draw [->,thick] (g) to node[auto] {$f_1$} (h);
\end{tikzpicture}
\caption{String parametrisation of the elements of the crystal of the adjoint representation of $\SU(3)$.\label{fig:stringparametrisation}}
\end{center}
\end{figure}
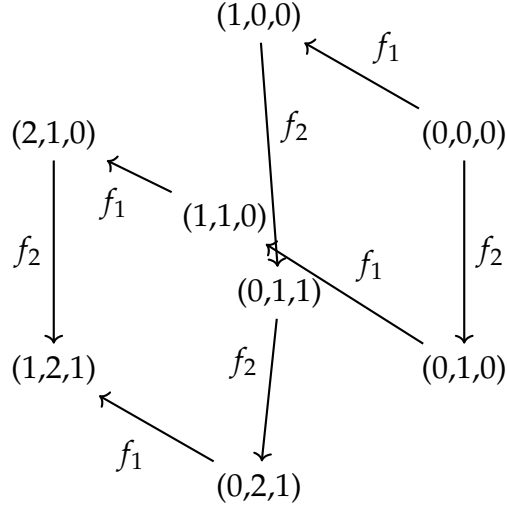
\end{example}

\noindent The string parametrisation is a natural way to label the vertices of the crystal graph: given a vertex $v$, one looks for the maximal size $n_1$ of a path in the direction $\alpha_{i_1}$ starting from $v$; then, for the maximal size $n_2$ of a path in the direction $\alpha_{i_2}$ starting from $e_{\alpha_{i_1}}^{n_1}(v)$; and so on.
 Given a vertex $v \in \cryst(\lambda)$ with string parametrisation $(n_1,\ldots,n_l)$, one has  $v = f_{\alpha_{i_1}}^{n_1}f_{\alpha_{i_2}}^{n_2}\cdots f_{\alpha_{i_l}}^{n_l}(v_\lambda)$, where $v_\lambda$ is the unique element of the crystal with weight $\lambda$. As a consequence, in this setting,
$$\mathrm{wt}(v) = \lambda - \sum_{j=1}^l n_j\alpha_{i_j}.$$
We denote $\mathscr{S}(\lambda)$ the set of all string parametrisations of elements of the crystal $\cryst(\lambda)$, and $\mathscr{S}(G) = \bigcup_{\lambda \in \hatG}\mathscr{S}(\lambda)$. We also denote $\mathscr{S}\!\mathscr{C}(G)$ the \emph{string cone} of $G$, which is the real cone (set of non-negative linear combinations) spanned by the elements of $\mathscr{S}(G)$. Finally, for $\lambda \in \hatG$, let $\mathscr{P}(\lambda)$ be the \emph{string polytope} of $\lambda$, which is the set of elements $(u_1,\ldots,u_l)$ in the string cone and such that:
\begin{align*}
u_l &\leq \lambda(\alpha_{i_l}^\vee);\\
u_{l-1} &\leq (\lambda-u_l\alpha_{i_l})(\alpha_{i_{l-1}}^\vee);\\
u_{l-2} &\leq (\lambda-u_l\alpha_{i_l}-u_{l-1}\alpha_{i_{l-1}})(\alpha_{i_{l-2}}^\vee);\\
&\vdots \qquad\qquad \vdots \\
u_{1} &\leq (\lambda-u_l\alpha_{i_l}-\cdots -u_{2}\alpha_{i_{2}})(\alpha_{i_{1}}^\vee).
\end{align*}

\begin{proposition}[Littelmann]
The string cone $\mathscr{S}\!\mathscr{C}(G)$ is a rational convex cone delimited by a finite number of hyperplanes in $\R^l$. The string parametrisations in $\mathscr{S}(G)$ are the integer points of the string cone $\mathscr{S}\!\mathscr{C}(G)$, and the string parametrisations in $\mathscr{S}(\lambda)$ are the integer points of the string polytope $\mathscr{P}(\lambda)$.
\end{proposition}
\noindent An explicit description of the string cone is given in \cite{LittCone,BZ01}; see also the remark after Theorem \ref{thm:berensteinzelevinsky}. On the other hand, the string polytope $\mathscr{P}(\lambda)$ has maximal dimension $l$ as long as $\lambda$ does not belong to the walls of the Weyl chamber.

\begin{example}
For $G=\SU(3)$, one can show that the string cone is the set of triples $(x,y,z) \in (\R_+)^3$ such that $y \geq z$; see \cite[Corollary 2]{LittCone}. The string polytope of the adjoint representation with highest weight $\lambda = \omega_1+\omega_2$ is then the subset of the string cone:
$$\mathscr{P}(\omega_1+\omega_2) = \{(x,y,z) \in (\R_+)^3\,|\,z \leq 1,\,\,z \leq y \leq 1 + z,\,\,x \leq 1-2z+y \}.$$
This polytope is drawn in Figure \ref{fig:stringpolytope}, and one can check that it contains eight integer points.
\begin{figure}[ht]
\includegraphics[scale=0.7]{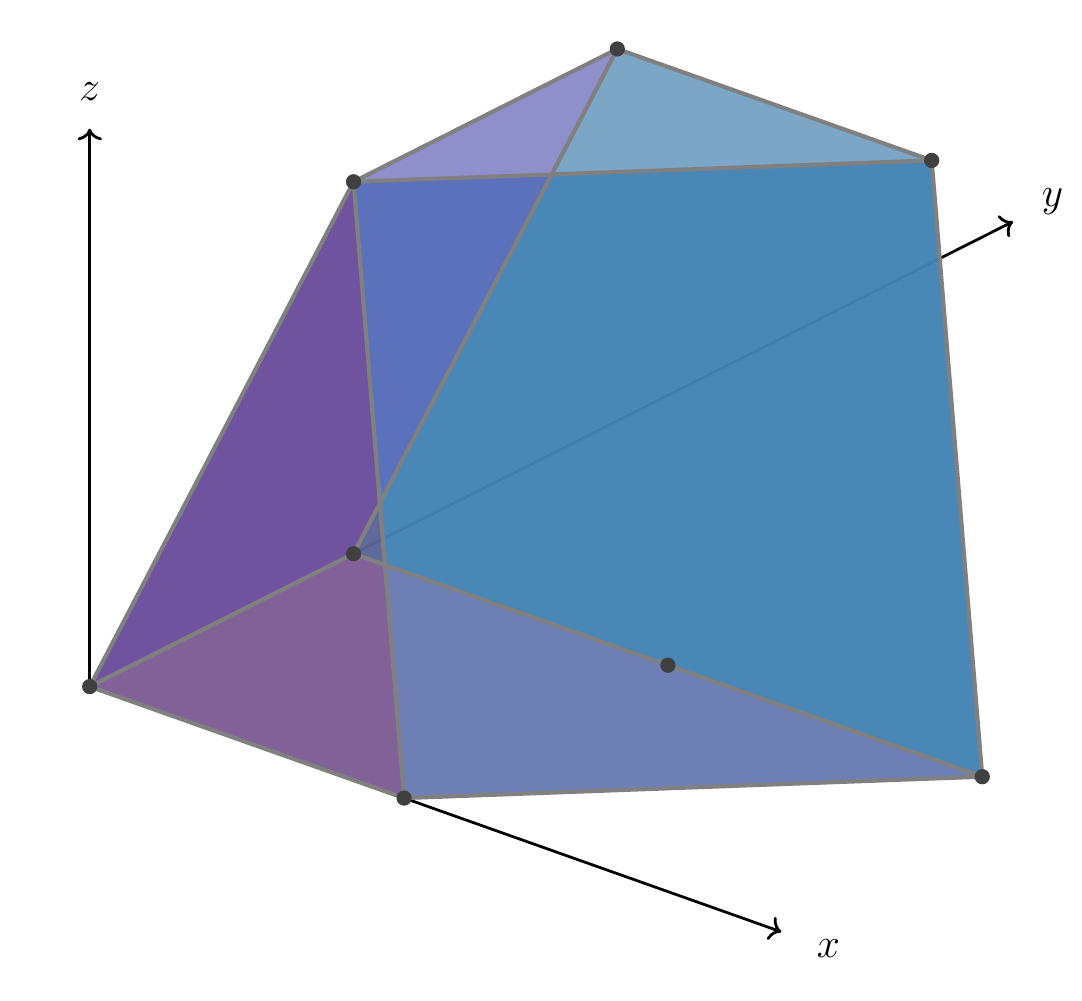}
\caption{String polytope of the adjoint representation of $\SU(3)$.\label{fig:stringpolytope}}
\end{figure}
\end{example}

We extend the weight $\mathrm{wt}(\cdot) : \cryst(\lambda) \to \R\Omega$ to a map $\Psi_\lambda$ on the whole space of string parametrisations $\R^l$, by using the same definition for real points as for integer points:
$$\Psi_\lambda(u_1,\ldots,u_r) = \lambda - \sum_{j=1}^{l} u_j\alpha_{i_j}.$$
Later we shall also consider maps $\Psi_x$ with $x$ arbitrary in $C$; the definition is the same as above, with $x=\lambda$. For any dominant weight $\lambda$, the map $\Psi_\lambda$ is affine, and $(\Psi_\lambda)_{|\cryst(\lambda)} = \mathrm{wt}$. The image of the string polytope $\mathscr{P}(\lambda)$ by $\Psi_\lambda$ is a polytope in $\R\Omega$, and one can show that it is the convex hull of the points in $W(\lambda)$; see for instance \cite[Definition 1.3]{AB04}. Moreover, the image by $\Psi_\lambda$ of the Lebesgue measure on $\mathscr{P}(\lambda)$ 
$$1_{(u_1,u_2,\ldots,u_l) \in \mathscr{P}(\lambda)}\DD{u_1}\DD{u_2}\,\cdots \DD{u_l}$$
is compactly supported by $\mathrm{Conv}(W(\lambda))$, and piecewise polynomial (this is a general property of affine images of Lebesgue measures on polytopes); see \cite[\S5.3]{BBO09}. We can then state a result of asymptotic polynomiality of the Kostka numbers $K_{\lambda,\omega} = \dim_\C (V^{\lambda}(\omega))$ (instead of the Littlewood--Richarson coefficients):

\begin{proposition}\label{prop:asymptoticsmultiplicities}
Fix a direction $x$ in the Weyl chamber $C \subset \R\Omega$, and a continuous bounded function $f$ on $\R\Omega$. We assume that $x$ does not belong to the walls of the Weyl chamber. Then, there exists a probability measure $m_{x}(y)\DD{y}$ on $\R\Omega$ that is supported by $\mathrm{Conv}(\{w(x)\,|\,w \in W\})$, that has a piecewise polynomial density $m_x$, and such that
\begin{equation}
    \lim_{\substack{t \to \infty \\ tx \in \hatG}} \left(\frac{\prod_{\alpha \in \Phi_+} \scal{\rho}{\alpha}}{t^{l}\,\prod_{\alpha \in \Phi_+} \scal{x}{\alpha}}\sum_{\omega \in \Z\Omega} K_{tx,\omega}\,\,f\!\left(\frac{\omega}{t}\right) \right) = \int_{\mathrm{Conv}(W(x))} f(y)\,m_{x}(y)\DD{y}.\label{eq:asymptoticmultiplicity}
\end{equation}
The local degree of $y \mapsto m_x(y)$ is bounded by $l-d = |\Phi_+| - \rank(G)$, and one has the scaling property
$m_{\gamma x}(\gamma y) = \frac{m_x(y)}{\gamma^{d}}.$
\end{proposition}

\noindent The probability measure $m_x(y) \DD{y}$ is a version of the \emph{Duistermaat--Heckman measure}, see in particular \cite[\S5.3]{BBO09}.

\begin{proof}
Set $\lambda = tx$. The left-hand side $L(f,x,t)$ of Equation \eqref{eq:asymptoticmultiplicity} approximates $\int_{\R\Omega} f(\frac{\omega}{t})\,\mu_{\lambda}(\omega)$, where $\mu_{\lambda}$ is the spectral probability measure of the representation $\lambda$, supported on weights and defined by
$$\mu_{\lambda} = \frac{1}{\dim_{\C} (V^\lambda)} \sum_{\omega \in \Omega(\lambda)} (\dim_{\C} (V^\lambda(\omega)))\,\delta_{\omega}.$$
Indeed, the only difference is that we approximated $\dim_{\C}(V^\lambda)$ by $t^{|\Phi_+|}\prod_{\alpha \in \Phi_+} \frac{\scal{x}{\alpha}}{\scal{\rho}{\alpha}}$, and this is valid in the limit $t \to \infty$. Now, by the previous discussion, $\mu_\lambda$ is the image of the probability measure 
$$\nu_\lambda = \frac{1}{\dim_{\C} (V^\lambda)} \sum_{(n_1,\ldots,n_l) \text{ integer points in }\mathscr{P}(\lambda)} \delta_{(n_1,\ldots,n_l)}$$
on $\R^l$ by the affine map $\Psi_\lambda$. Therefore,
$$L(f,x,t) \simeq \int_{\R^l} f\!\left(\frac{\Psi_\lambda(u)}{t}\right)\,\nu_{tx}(\!\DD{u})=\int_{\R^l} f\!\left(\Psi_x\!\left(\frac{u}{t}\right)\right)\,\nu_{tx}(\!\DD{u}).$$
As $t$ goes to infinity, the discrete measure $\nu_{x,t}(u)=\nu_{tx}(tu)$ converges in law to the uniform probability mesure $\upsilon_x$ on the polytope $\mathscr{P}(x)$, which is defined as the set of points of the string cone $\mathscr{S}(G)$ which satisfy the inequalities
\begin{align*}
u_l &\leq x(\alpha_{i_l}^\vee);\\
u_{l-1} &\leq (x-u_l\alpha_{i_l})(\alpha_{i_{l-1}}^\vee);\\
&\vdots \qquad\qquad \vdots \\
u_{1} &\leq (x-u_l\alpha_{i_l}-\cdots -u_{2}\alpha_{i_{2}})(\alpha_{i_{1}}^\vee).
\end{align*}
Therefore, 
$\lim_{t \to \infty} L(f,x,t) = \int_{\R^l} f(\Psi_x(u)) \,\upsilon_x(\DD{u})$. Finally, the image measure $m_x = (\Psi_x)_*(\upsilon_x)$ is given by a compactly supported piecewise polynomial function, of local degree smaller than $l-d$; and the obvious identities $\Psi_{\gamma x}( \gamma  u) = \gamma \,\Psi_x(u)$ and $\gamma^{l}\,\upsilon_{\gamma x}(\gamma u)\DD{u} = \upsilon_x(u)\DD{u}$ imply the scaling property.
\end{proof}


\subsection{From the string polytope to the Littlewood--Richardson coefficients}\label{subsec:littlewoodrichardson}
The theory which enables one to understand the asymptotics of Kostka numbers can be adapted to the same problem  with the Littlewood--Richardson coefficients. From the discussion at the end of Section \ref{subsec:pathmodel}, and the description of $c^{\lambda,\mu}_{\nu}$ as a number of paths in $\cryst(\pi_\mu)$ satisfying certain conditions, one can expect that there is a notion of string polytope of $V^\mu$ \emph{relatively to another dominant weight} $\lambda$ that allows to calculate these coefficients. These relative string polytopes have been constructed by Berenstein and Zelevinsky, see \cite{BZ88,BZ01}.
 A \emph{trail} from a weight $\phi$ to another weight $\pi$ of an irreducible representation $V^\mu$ of $\glie_{\C}$ is a sequence of weights $\phi=\phi_0,\phi_1,\ldots,\phi_l = \pi$ of $V^\mu$ such that:
\begin{enumerate}
    \item $\phi_{j-1}-\phi_{j} = k_j\,\alpha_{i_j}$ for any $j \in \lle 1,l\rre$, with the $k_j$'s non-negative integers;
    \item  there exist in the crystal $\cryst(\mu)$ vertices $v_\phi$ and $v_\pi$ with respective weights $\phi$ and $\pi$, and a sequence of edges 
$$v_{\phi} \to_{k_1\, f_{\alpha_{i_1}}} v_{\phi_{1}} \to_{k_{2}\,f_{\alpha_{i_{2}}}} \cdots \to_{k_{l-1}\,f_{\alpha_{i_{l-1}}}} v_{\phi_{l-1}} \to_{k_l\, f_{\alpha_{i_l}}} v_\pi,$$
where $\to_{k_j\,f_{\alpha_{i_j}}}$ stands for $k_j$ edges of label $f_{\alpha_{i_j}}$.
\end{enumerate}
 In other words, the trails are the images by the weight map of directed paths on the crystal graph. For instance, in the crystal of the adjoint representation of $\SU(3)$, there is a trail from $\omega_1+\omega_2$ to $\omega_1-2\omega_2$, since one can find the sequence of edges
$$(0,0,0) \to_{f_{\alpha_1}} (1,0,0) \to_{f_{\alpha_2}} (0,1,1) \to_{f_{\alpha_2}} (0,2,1) $$
in the crystal graph. We refer to \cite[Theorem 2.3]{BZ01} for a proof of the following result, in which we shall consider irreducible representations of the dual Langlands Lie algebra ${}^{\mathrm{L}}\glie_\C$, which is the Lie algebra obtained from $\glie_{\C}$ by exchanging roots and coroots, respectively weights and coweights.

\begin{theorem}[Berenstein--Zelevinsky]\label{thm:berensteinzelevinsky}
Let $\mathscr{S}(\lambda,\mu)$ be the subset of the set of string parametrisations $\mathscr{S}(\mu)$ that consists of strings $(n_1,n_2,\ldots,n_l)$ such that, for any $i \in \lle 1,d\rre$ and any trail $(\phi_0^\vee,\phi_1^\vee,\ldots,\phi_l^\vee)$ from $s_i(\omega_i^\vee)$ to $w_0(\omega_i^\vee)$ in the fundamental representation $V^{\omega_i^\vee}$ of ${}^{\mathrm{L}}\glie_\C$,
$$\sum_{j=1}^l n_j\,\alpha_{i_j}\!\left(\frac{\phi_{j-1}^\vee + \phi_j^\vee}{2}\right) \geq -\lambda(\alpha_i^\vee)=-l_i.$$
Then, $c^{\lambda,\mu}_{\nu}$ is the number of elements with weight $\nu-\lambda$ in $\mathscr{S}(\lambda,\mu)$. Therefore, it is the number of integer points in a slice of the Berenstein--Zelevinsky relative string polytope $\mathscr{P}(\lambda,\mu)$, which is the intersection of $\mathscr{P}(\mu)$ with the half-spaces determined by the inequalities above.
\end{theorem}

\begin{remark}
In \cite[Theorem 3.10]{BZ01}, a similar trail characterisation of the string cone $\mathscr{S}\!\mathscr{C}(G)$ is provided: it consists in all the sequences $(x_1,x_2,\ldots,x_l) \in (\R_+)^l$ such that, for any $i \in \lle 1,d\rre$ and any trail $(\phi_0^\vee,\phi_1^\vee,\ldots,\phi_l^\vee)$ from $\omega_i^\vee$ to $w_0s_i(\omega_i^\vee)$ in the fundamental representation $V^{\omega_i^\vee}$ of ${}^{\mathrm{L}}\glie_\C$,
$$\sum_{j=1}^l x_j\,\alpha_{i_j}\!\left(\frac{\phi_{j-1}^\vee + \phi_j^\vee}{2}\right) \geq 0.$$
\end{remark}

\begin{example}
In the Weyl chamber of $\SU(3)$, fix the two directions $x=\omega_1+\omega_2$ and $y = 2\omega_1+\omega_2$. We have drawn in Figure \ref{fig:multiplicitiestensorproduct} the Littlewood--Richardson coefficients for $V^\lambda \otimes V^\mu$, with $\lambda=10x$ and $\mu=10y$.
Consider for instance the weight $\nu = 10(2\omega_1+\omega_2)$; the multiplicity of $V^\nu$ in $V^\lambda \otimes V^\mu$ is equal to $11$. On the other hand, the string polytope $\mathscr{P}(\mu)$ is the set of triplets $(x,y,z) \in (\R_+)^3$ with
\begin{align*}
z \leq 20,\,\,z \leq y \leq 10 + z,\,\,x\leq 20-2z+y.
\end{align*}
For the relative string polytope $\mathscr{P}(\lambda,\mu)$, we need to add the inequalities:
$$x\leq 10,\,\,y\leq 10+x,\,\,z\leq 10.$$
Consequently, to compute the multiplicity $c^{\lambda,\mu}_{\nu}$, we need to find all the integer triplets satisfying the previous inequalities, and with $$(x+z)\alpha_1 + y\alpha_2 = \lambda+\mu-\nu = \lambda = 10(\alpha_1+\alpha_2).$$ It is easily seen that the $11$ solutions are the triplets $(k,10,10-k)$ with $k \in \lle 0,10\rre$. 
\end{example}

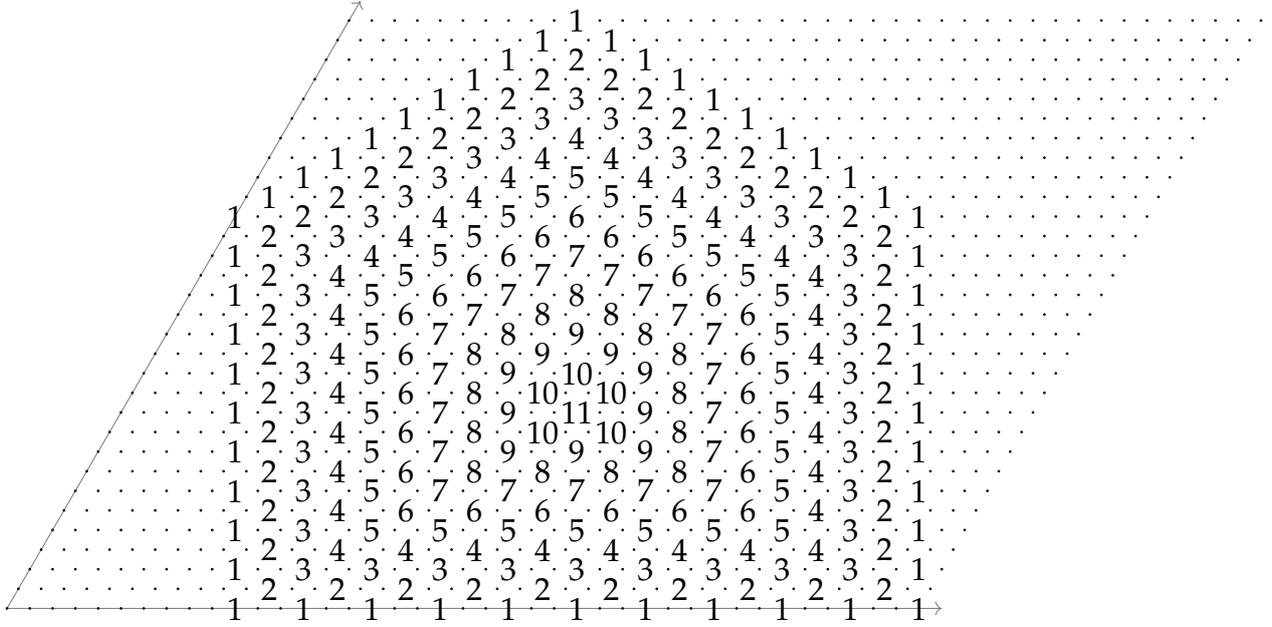
\begin{figure}[ht]
\begin{center}      
\begin{tikzpicture}[scale=0.3]
\draw [<->, gray] (15.5,26.8468) -- (0,0) -- (41,0);
\foreach \x in {    
(0.00000, 0.00000), (0.50000, 0.86602), (1.0000, 1.7320), (1.5000,
2.5981), (2.0000, 3.4641), (2.5000, 4.3301), (3.0000, 5.1962), (3.5000,
6.0622), (4.0000, 6.9282), (4.5000, 7.7942), (5.0000, 8.6602), (5.5000,
9.5263), (6.0000, 10.392), (6.5000, 11.258), (7.0000, 12.124), (7.5000,
12.990), (8.0000, 13.856), (8.5000, 14.722), (9.0000, 15.588), (9.5000,
16.454), (10.500, 18.187), (11.000, 19.053), (11.500, 19.919), (12.000,
20.785), (12.500, 21.651), (13.000, 22.517), (13.500, 23.383), (14.000,
24.249), (14.500, 25.115), (15.000, 25.981), (1.0000, 0.00000), (1.5000,
0.86602), (2.0000, 1.7320), (2.5000, 2.5981), (3.0000, 3.4641), (3.5000,
4.3301), (4.0000, 5.1962), (4.5000, 6.0622), (5.0000, 6.9282), (5.5000,
7.7942), (6.0000, 8.6602), (6.5000, 9.5263), (7.0000, 10.392), (7.5000,
11.258), (8.0000, 12.124), (8.5000, 12.990), (9.0000, 13.856), (9.5000,
14.722), (10.500, 16.454), (11.000, 17.320), (12.000, 19.053), (12.500,
19.919), (13.000, 20.785), (13.500, 21.651), (14.000, 22.517), (14.500,
23.383), (15.000, 24.249), (15.500, 25.115), (16.000, 25.981), (2.0000,
0.00000), (2.5000, 0.86602), (3.0000, 1.7320), (3.5000, 2.5981),
(4.0000, 3.4641), (4.5000, 4.3301), (5.0000, 5.1962), (5.5000, 6.0622),
(6.0000, 6.9282), (6.5000, 7.7942), (7.0000, 8.6602), (7.5000, 9.5263),
(8.0000, 10.392), (8.5000, 11.258), (9.0000, 12.124), (9.5000, 12.990),
(10.500, 14.722), (11.000, 15.588), (12.000, 17.320), (12.500, 18.187),
(13.500, 19.919), (14.000, 20.785), (14.500, 21.651), (15.000, 22.517),
(15.500, 23.383), (16.000, 24.249), (16.500, 25.115), (17.000, 25.981),
(3.0000, 0.00000), (3.5000, 0.86602), (4.0000, 1.7320), (4.5000,
2.5981), (5.0000, 3.4641), (5.5000, 4.3301), (6.0000, 5.1962), (6.5000,
6.0622), (7.0000, 6.9282), (7.5000, 7.7942), (8.0000, 8.6602), (8.5000,
9.5263), (9.0000, 10.392), (9.5000, 11.258), (10.500, 12.990), (11.000,
13.856), (12.000, 15.588), (12.500, 16.454), (13.500, 18.187), (14.000,
19.053), (15.000, 20.785), (15.500, 21.651), (16.000, 22.517), (16.500,
23.383), (17.000, 24.249), (17.500, 25.115), (18.000, 25.981), (4.0000,
0.00000), (4.5000, 0.86602), (5.0000, 1.7320), (5.5000, 2.5981),
(6.0000, 3.4641), (6.5000, 4.3301), (7.0000, 5.1962), (7.5000, 6.0622),
(8.0000, 6.9282), (8.5000, 7.7942), (9.0000, 8.6602), (9.5000, 9.5263),
(10.500, 11.258), (11.000, 12.124), (12.000, 13.856), (12.500, 14.722),
(13.500, 16.454), (14.000, 17.320), (15.000, 19.053), (15.500, 19.919),
(16.500, 21.651), (17.000, 22.517), (17.500, 23.383), (18.000, 24.249),
(18.500, 25.115), (19.000, 25.981), (5.0000, 0.00000), (5.5000,
0.86602), (6.0000, 1.7320), (6.5000, 2.5981), (7.0000, 3.4641), (7.5000,
4.3301), (8.0000, 5.1962), (8.5000, 6.0622), (9.0000, 6.9282), (9.5000,
7.7942), (10.500, 9.5263), (11.000, 10.392), (12.000, 12.124), (12.500,
12.990), (13.500, 14.722), (14.000, 15.588), (15.000, 17.320), (15.500,
18.187), (16.500, 19.919), (17.000, 20.785), (18.000, 22.517), (18.500,
23.383), (19.000, 24.249), (19.500, 25.115), (20.000, 25.981), (6.0000,
0.00000), (6.5000, 0.86602), (7.0000, 1.7320), (7.5000, 2.5981),
(8.0000, 3.4641), (8.5000, 4.3301), (9.0000, 5.1962), (9.5000, 6.0622),
(10.500, 7.7942), (11.000, 8.6602), (12.000, 10.392), (12.500, 11.258),
(13.500, 12.990), (14.000, 13.856), (15.000, 15.588), (15.500, 16.454),
(16.500, 18.187), (17.000, 19.053), (18.000, 20.785), (18.500, 21.651),
(19.500, 23.383), (20.000, 24.249), (20.500, 25.115), (21.000, 25.981),
(7.0000, 0.00000), (7.5000, 0.86602), (8.0000, 1.7320), (8.5000,
2.5981), (9.0000, 3.4641), (9.5000, 4.3301), (10.500, 6.0622), (11.000,
6.9282), (12.000, 8.6602), (12.500, 9.5263), (13.500, 11.258), (14.000,
12.124), (15.000, 13.856), (15.500, 14.722), (16.500, 16.454), (17.000,
17.320), (18.000, 19.053), (18.500, 19.919), (19.500, 21.651), (20.000,
22.517), (21.000, 24.249), (21.500, 25.115), (22.000, 25.981), (8.0000,
0.00000), (8.5000, 0.86602), (9.0000, 1.7320), (9.5000, 2.5981),
(10.500, 4.3301), (11.000, 5.1962), (12.000, 6.9282), (12.500, 7.7942),
(13.500, 9.5263), (14.000, 10.392), (15.000, 12.124), (15.500, 12.990),
(16.500, 14.722), (17.000, 15.588), (18.000, 17.320), (18.500, 18.187),
(19.500, 19.919), (20.000, 20.785), (21.000, 22.517), (21.500, 23.383),
(22.500, 25.115), (23.000, 25.981), (9.0000, 0.00000), (9.5000,
0.86602), (10.500, 2.5981), (11.000, 3.4641), (12.000, 5.1962), (12.500,
6.0622), (13.500, 7.7942), (14.000, 8.6602), (15.000, 10.392), (15.500,
11.258), (16.500, 12.990), (17.000, 13.856), (18.000, 15.588), (18.500,
16.454), (19.500, 18.187), (20.000, 19.053), (21.000, 20.785), (21.500,
21.651), (22.500, 23.383), (23.000, 24.249), (24.000, 25.981), (10.500,
0.86602), (11.000, 1.7320), (12.000, 3.4641), (12.500, 4.3301), (13.500,
6.0622), (14.000, 6.9282), (15.000, 8.6602), (15.500, 9.5263), (16.500,
11.258), (17.000, 12.124), (18.000, 13.856), (18.500, 14.722), (19.500,
16.454), (20.000, 17.320), (21.000, 19.053), (21.500, 19.919), (22.500,
21.651), (23.000, 22.517), (24.000, 24.249), (24.500, 25.115), (11.000,
0.00000), (12.000, 1.7320), (12.500, 2.5981), (13.500, 4.3301), (14.000,
5.1962), (15.000, 6.9282), (15.500, 7.7942), (16.500, 9.5263), (17.000,
10.392), (18.000, 12.124), (18.500, 12.990), (19.500, 14.722), (20.000,
15.588), (21.000, 17.320), (21.500, 18.187), (22.500, 19.919), (23.000,
20.785), (24.000, 22.517), (24.500, 23.383), (25.500, 25.115), (26.000,
25.981), (12.000, 0.00000), (12.500, 0.86602), (13.500, 2.5981),
(14.000, 3.4641), (15.000, 5.1962), (15.500, 6.0622), (16.500, 7.7942),
(17.000, 8.6602), (18.000, 10.392), (18.500, 11.258), (19.500, 12.990),
(20.000, 13.856), (21.000, 15.588), (21.500, 16.454), (22.500, 18.187),
(23.000, 19.053), (24.000, 20.785), (24.500, 21.651), (25.500, 23.383),
(26.000, 24.249), (27.000, 25.981), (13.500, 0.86602), (14.000, 1.7320),
(15.000, 3.4641), (15.500, 4.3301), (16.500, 6.0622), (17.000, 6.9282),
(18.000, 8.6602), (18.500, 9.5263), (19.500, 11.258), (20.000, 12.124),
(21.000, 13.856), (21.500, 14.722), (22.500, 16.454), (23.000, 17.320),
(24.000, 19.053), (24.500, 19.919), (25.500, 21.651), (26.000, 22.517),
(27.000, 24.249), (27.500, 25.115), (28.000, 25.981), (14.000, 0.00000),
(15.000, 1.7320), (15.500, 2.5981), (16.500, 4.3301), (17.000, 5.1962),
(18.000, 6.9282), (18.500, 7.7942), (19.500, 9.5263), (20.000, 10.392),
(21.000, 12.124), (21.500, 12.990), (22.500, 14.722), (23.000, 15.588),
(24.000, 17.320), (24.500, 18.187), (25.500, 19.919), (26.000, 20.785),
(27.000, 22.517), (27.500, 23.383), (28.500, 25.115), (29.000, 25.981),
(15.000, 0.00000), (15.500, 0.86602), (16.500, 2.5981), (17.000,
3.4641), (18.000, 5.1962), (18.500, 6.0622), (19.500, 7.7942), (20.000,
8.6602), (21.000, 10.392), (21.500, 11.258), (22.500, 12.990), (23.000,
13.856), (24.000, 15.588), (24.500, 16.454), (25.500, 18.187), (26.000,
19.053), (27.000, 20.785), (27.500, 21.651), (28.500, 23.383), (29.000,
24.249), (29.500, 25.115), (30.000, 25.981), (16.500, 0.86602), (17.000,
1.7320), (18.000, 3.4641), (18.500, 4.3301), (19.500, 6.0622), (20.000,
6.9282), (21.000, 8.6602), (21.500, 9.5263), (22.500, 11.258), (23.000,
12.124), (24.000, 13.856), (24.500, 14.722), (25.500, 16.454), (26.000,
17.320), (27.000, 19.053), (27.500, 19.919), (28.500, 21.651), (29.000,
22.517), (30.000, 24.249), (30.500, 25.115), (31.000, 25.981), (17.000,
0.00000), (18.000, 1.7320), (18.500, 2.5981), (19.500, 4.3301), (20.000,
5.1962), (21.000, 6.9282), (21.500, 7.7942), (22.500, 9.5263), (23.000,
10.392), (24.000, 12.124), (24.500, 12.990), (25.500, 14.722), (26.000,
15.588), (27.000, 17.320), (27.500, 18.187), (28.500, 19.919), (29.000,
20.785), (30.000, 22.517), (30.500, 23.383), (31.000, 24.249), (31.500,
25.115), (32.000, 25.981), (18.000, 0.00000), (18.500, 0.86602),
(19.500, 2.5981), (20.000, 3.4641), (21.000, 5.1962), (21.500, 6.0622),
(22.500, 7.7942), (23.000, 8.6602), (24.000, 10.392), (24.500, 11.258),
(25.500, 12.990), (26.000, 13.856), (27.000, 15.588), (27.500, 16.454),
(28.500, 18.187), (29.000, 19.053), (30.000, 20.785), (30.500, 21.651),
(31.500, 23.383), (32.000, 24.249), (32.500, 25.115), (33.000, 25.981),
(19.500, 0.86602), (20.000, 1.7320), (21.000, 3.4641), (21.500, 4.3301),
(22.500, 6.0622), (23.000, 6.9282), (24.000, 8.6602), (24.500, 9.5263),
(25.500, 11.258), (26.000, 12.124), (27.000, 13.856), (27.500, 14.722),
(28.500, 16.454), (29.000, 17.320), (30.000, 19.053), (30.500, 19.919),
(31.500, 21.651), (32.000, 22.517), (32.500, 23.383), (33.000, 24.249),
(33.500, 25.115), (34.000, 25.981), (20.000, 0.00000), (21.000, 1.7320),
(21.500, 2.5981), (22.500, 4.3301), (23.000, 5.1962), (24.000, 6.9282),
(24.500, 7.7942), (25.500, 9.5263), (26.000, 10.392), (27.000, 12.124),
(27.500, 12.990), (28.500, 14.722), (29.000, 15.588), (30.000, 17.320),
(30.500, 18.187), (31.500, 19.919), (32.000, 20.785), (33.000, 22.517),
(33.500, 23.383), (34.000, 24.249), (34.500, 25.115), (35.000, 25.981),
(21.000, 0.00000), (21.500, 0.86602), (22.500, 2.5981), (23.000,
3.4641), (24.000, 5.1962), (24.500, 6.0622), (25.500, 7.7942), (26.000,
8.6602), (27.000, 10.392), (27.500, 11.258), (28.500, 12.990), (29.000,
13.856), (30.000, 15.588), (30.500, 16.454), (31.500, 18.187), (32.000,
19.053), (33.000, 20.785), (33.500, 21.651), (34.000, 22.517), (34.500,
23.383), (35.000, 24.249), (35.500, 25.115), (36.000, 25.981), (22.500,
0.86602), (23.000, 1.7320), (24.000, 3.4641), (24.500, 4.3301), (25.500,
6.0622), (26.000, 6.9282), (27.000, 8.6602), (27.500, 9.5263), (28.500,
11.258), (29.000, 12.124), (30.000, 13.856), (30.500, 14.722), (31.500,
16.454), (32.000, 17.320), (33.000, 19.053), (33.500, 19.919), (34.500,
21.651), (35.000, 22.517), (35.500, 23.383), (36.000, 24.249), (36.500,
25.115), (37.000, 25.981), (23.000, 0.00000), (24.000, 1.7320), (24.500,
2.5981), (25.500, 4.3301), (26.000, 5.1962), (27.000, 6.9282), (27.500,
7.7942), (28.500, 9.5263), (29.000, 10.392), (30.000, 12.124), (30.500,
12.990), (31.500, 14.722), (32.000, 15.588), (33.000, 17.320), (33.500,
18.187), (34.500, 19.919), (35.000, 20.785), (35.500, 21.651), (36.000,
22.517), (36.500, 23.383), (37.000, 24.249), (37.500, 25.115), (38.000,
25.981), (24.000, 0.00000), (24.500, 0.86602), (25.500, 2.5981),
(26.000, 3.4641), (27.000, 5.1962), (27.500, 6.0622), (28.500, 7.7942),
(29.000, 8.6602), (30.000, 10.392), (30.500, 11.258), (31.500, 12.990),
(32.000, 13.856), (33.000, 15.588), (33.500, 16.454), (34.500, 18.187),
(35.000, 19.053), (36.000, 20.785), (36.500, 21.651), (37.000, 22.517),
(37.500, 23.383), (38.000, 24.249), (38.500, 25.115), (39.000, 25.981),
(25.500, 0.86602), (26.000, 1.7320), (27.000, 3.4641), (27.500, 4.3301),
(28.500, 6.0622), (29.000, 6.9282), (30.000, 8.6602), (30.500, 9.5263),
(31.500, 11.258), (32.000, 12.124), (33.000, 13.856), (33.500, 14.722),
(34.500, 16.454), (35.000, 17.320), (36.000, 19.053), (36.500, 19.919),
(37.000, 20.785), (37.500, 21.651), (38.000, 22.517), (38.500, 23.383),
(39.000, 24.249), (39.500, 25.115), (40.000, 25.981), (26.000, 0.00000),
(27.000, 1.7320), (27.500, 2.5981), (28.500, 4.3301), (29.000, 5.1962),
(30.000, 6.9282), (30.500, 7.7942), (31.500, 9.5263), (32.000, 10.392),
(33.000, 12.124), (33.500, 12.990), (34.500, 14.722), (35.000, 15.588),
(36.000, 17.320), (36.500, 18.187), (37.500, 19.919), (38.000, 20.785),
(38.500, 21.651), (39.000, 22.517), (39.500, 23.383), (40.000, 24.249),
(40.500, 25.115), (41.000, 25.981), (27.000, 0.00000), (27.500,
0.86602), (28.500, 2.5981), (29.000, 3.4641), (30.000, 5.1962), (30.500,
6.0622), (31.500, 7.7942), (32.000, 8.6602), (33.000, 10.392), (33.500,
11.258), (34.500, 12.990), (35.000, 13.856), (36.000, 15.588), (36.500,
16.454), (37.500, 18.187), (38.000, 19.053), (38.500, 19.919), (39.000,
20.785), (39.500, 21.651), (40.000, 22.517), (40.500, 23.383), (41.000,
24.249), (41.500, 25.115), (42.000, 25.981), (28.500, 0.86602), (29.000,
1.7320), (30.000, 3.4641), (30.500, 4.3301), (31.500, 6.0622), (32.000,
6.9282), (33.000, 8.6602), (33.500, 9.5263), (34.500, 11.258), (35.000,
12.124), (36.000, 13.856), (36.500, 14.722), (37.500, 16.454), (38.000,
17.320), (39.000, 19.053), (39.500, 19.919), (40.000, 20.785), (40.500,
21.651), (41.000, 22.517), (41.500, 23.383), (42.000, 24.249), (42.500,
25.115), (43.000, 25.981), (29.000, 0.00000), (30.000, 1.7320), (30.500,
2.5981), (31.500, 4.3301), (32.000, 5.1962), (33.000, 6.9282), (33.500,
7.7942), (34.500, 9.5263), (35.000, 10.392), (36.000, 12.124), (36.500,
12.990), (37.500, 14.722), (38.000, 15.588), (39.000, 17.320), (39.500,
18.187), (40.000, 19.053), (40.500, 19.919), (41.000, 20.785), (41.500,
21.651), (42.000, 22.517), (42.500, 23.383), (43.000, 24.249), (43.500,
25.115), (44.000, 25.981), (30.000, 0.00000), (30.500, 0.86602),
(31.500, 2.5981), (32.000, 3.4641), (33.000, 5.1962), (33.500, 6.0622),
(34.500, 7.7942), (35.000, 8.6602), (36.000, 10.392), (36.500, 11.258),
(37.500, 12.990), (38.000, 13.856), (39.000, 15.588), (39.500, 16.454),
(40.500, 18.187), (41.000, 19.053), (41.500, 19.919), (42.000, 20.785),
(42.500, 21.651), (43.000, 22.517), (43.500, 23.383), (44.000, 24.249),
(44.500, 25.115), (45.000, 25.981), (31.500, 0.86602), (32.000, 1.7320),
(33.000, 3.4641), (33.500, 4.3301), (34.500, 6.0622), (35.000, 6.9282),
(36.000, 8.6602), (36.500, 9.5263), (37.500, 11.258), (38.000, 12.124),
(39.000, 13.856), (39.500, 14.722), (40.500, 16.454), (41.000, 17.320),
(41.500, 18.187), (42.000, 19.053), (42.500, 19.919), (43.000, 20.785),
(43.500, 21.651), (44.000, 22.517), (44.500, 23.383), (45.000, 24.249),
(45.500, 25.115), (46.000, 25.981), (32.000, 0.00000), (33.000, 1.7320),
(33.500, 2.5981), (34.500, 4.3301), (35.000, 5.1962), (36.000, 6.9282),
(36.500, 7.7942), (37.500, 9.5263), (38.000, 10.392), (39.000, 12.124),
(39.500, 12.990), (40.500, 14.722), (41.000, 15.588), (41.500, 16.454),
(42.000, 17.320), (42.500, 18.187), (43.000, 19.053), (43.500, 19.919),
(44.000, 20.785), (44.500, 21.651), (45.000, 22.517), (45.500, 23.383),
(46.000, 24.249), (46.500, 25.115), (47.000, 25.981), (33.000, 0.00000),
(33.500, 0.86602), (34.500, 2.5981), (35.000, 3.4641), (36.000, 5.1962),
(36.500, 6.0622), (37.500, 7.7942), (38.000, 8.6602), (39.000, 10.392),
(39.500, 11.258), (40.500, 12.990), (41.000, 13.856), (41.500, 14.722),
(42.000, 15.588), (42.500, 16.454), (43.000, 17.320), (43.500, 18.187),
(44.000, 19.053), (44.500, 19.919), (45.000, 20.785), (45.500, 21.651),
(46.000, 22.517), (46.500, 23.383), (47.000, 24.249), (47.500, 25.115),
(48.000, 25.981), (34.500, 0.86602), (35.000, 1.7320), (36.000, 3.4641),
(36.500, 4.3301), (37.500, 6.0622), (38.000, 6.9282), (39.000, 8.6602),
(39.500, 9.5263), (40.500, 11.258), (41.000, 12.124), (41.500, 12.990),
(42.000, 13.856), (42.500, 14.722), (43.000, 15.588), (43.500, 16.454),
(44.000, 17.320), (44.500, 18.187), (45.000, 19.053), (45.500, 19.919),
(46.000, 20.785), (46.500, 21.651), (47.000, 22.517), (47.500, 23.383),
(48.000, 24.249), (48.500, 25.115), (49.000, 25.981), (35.000, 0.00000),
(36.000, 1.7320), (36.500, 2.5981), (37.500, 4.3301), (38.000, 5.1962),
(39.000, 6.9282), (39.500, 7.7942), (40.500, 9.5263), (41.000, 10.392),
(41.500, 11.258), (42.000, 12.124), (42.500, 12.990), (43.000, 13.856),
(43.500, 14.722), (44.000, 15.588), (44.500, 16.454), (45.000, 17.320),
(45.500, 18.187), (46.000, 19.053), (46.500, 19.919), (47.000, 20.785),
(47.500, 21.651), (48.000, 22.517), (48.500, 23.383), (49.000, 24.249),
(49.500, 25.115), (50.000, 25.981), (36.000, 0.00000), (36.500,
0.86602), (37.500, 2.5981), (38.000, 3.4641), (39.000, 5.1962), (39.500,
6.0622), (40.500, 7.7942), (41.000, 8.6602), (41.500, 9.5263), (42.000,
10.392), (42.500, 11.258), (43.000, 12.124), (43.500, 12.990), (44.000,
13.856), (44.500, 14.722), (45.000, 15.588), (45.500, 16.454), (46.000,
17.320), (46.500, 18.187), (47.000, 19.053), (47.500, 19.919), (48.000,
20.785), (48.500, 21.651), (49.000, 22.517), (49.500, 23.383), (50.000,
24.249), (50.500, 25.115), (51.000, 25.981), (37.500, 0.86602), (38.000,
1.7320), (39.000, 3.4641), (39.500, 4.3301), (40.500, 6.0622), (41.000,
6.9282), (41.500, 7.7942), (42.000, 8.6602), (42.500, 9.5263), (43.000,
10.392), (43.500, 11.258), (44.000, 12.124), (44.500, 12.990), (45.000,
13.856), (45.500, 14.722), (46.000, 15.588), (46.500, 16.454), (47.000,
17.320), (47.500, 18.187), (48.000, 19.053), (48.500, 19.919), (49.000,
20.785), (49.500, 21.651), (50.000, 22.517), (50.500, 23.383), (51.000,
24.249), (51.500, 25.115), (52.000, 25.981), (38.000, 0.00000), (39.000,
1.7320), (39.500, 2.5981), (40.500, 4.3301), (41.000, 5.1962), (41.500,
6.0622), (42.000, 6.9282), (42.500, 7.7942), (43.000, 8.6602), (43.500,
9.5263), (44.000, 10.392), (44.500, 11.258), (45.000, 12.124), (45.500,
12.990), (46.000, 13.856), (46.500, 14.722), (47.000, 15.588), (47.500,
16.454), (48.000, 17.320), (48.500, 18.187), (49.000, 19.053), (49.500,
19.919), (50.000, 20.785), (50.500, 21.651), (51.000, 22.517), (51.500,
23.383), (52.000, 24.249), (52.500, 25.115), (53.000, 25.981), (39.000,
0.00000), (39.500, 0.86602), (40.500, 2.5981), (41.000, 3.4641),
(41.500, 4.3301), (42.000, 5.1962), (42.500, 6.0622), (43.000, 6.9282),
(43.500, 7.7942), (44.000, 8.6602), (44.500, 9.5263), (45.000, 10.392),
(45.500, 11.258), (46.000, 12.124), (46.500, 12.990), (47.000, 13.856),
(47.500, 14.722), (48.000, 15.588), (48.500, 16.454), (49.000, 17.320),
(49.500, 18.187), (50.000, 19.053), (50.500, 19.919), (51.000, 20.785),
(51.500, 21.651), (52.000, 22.517), (52.500, 23.383), (53.000, 24.249),
(53.500, 25.115), (54.000, 25.981), (40.500, 0.86602), (41.000, 1.7320),
(41.500, 2.5981), (42.000, 3.4641), (42.500, 4.3301), (43.000, 5.1962),
(43.500, 6.0622), (44.000, 6.9282), (44.500, 7.7942), (45.000, 8.6602),
(45.500, 9.5263), (46.000, 10.392), (46.500, 11.258), (47.000, 12.124),
(47.500, 12.990), (48.000, 13.856), (48.500, 14.722), (49.000, 15.588),
(49.500, 16.454), (50.000, 17.320), (50.500, 18.187), (51.000, 19.053),
(51.500, 19.919), (52.000, 20.785), (52.500, 21.651), (53.000, 22.517),
(53.500, 23.383), (54.000, 24.249), (54.500, 25.115), (55.000, 25.981)}
\draw \x circle (1pt);
\draw (19.000, 5.1962) node {$7$}; \draw (17.500, 16.454) node {$4$};
\draw (32.500, 12.990) node {$6$}; \draw (17.500, 4.3301) node {$6$};
\draw (28.000, 17.320) node {$5$}; \draw (14.500, 4.3301) node {$4$};
\draw (10.000, 6.9282) node {$1$}; \draw (34.000, 10.392) node {$5$};
\draw (28.000, 13.856) node {$7$}; \draw (20.500, 19.919) node {$3$};
\draw (31.000, 10.392) node {$7$}; \draw (40.000, 15.588) node {$1$};
\draw (13.000, 15.588) node {$3$}; \draw (13.000, 3.4641) node {$3$};
\draw (31.000, 17.320) node {$4$}; \draw (25.000, 8.6602) node {$11$};
\draw (29.500, 12.990) node {$7$}; \draw (26.500, 25.115) node {$1$};
\draw (32.500, 0.86602) node {$2$}; \draw (35.500, 16.454) node {$3$};
\draw (19.000, 15.588) node {$5$}; \draw (17.500, 0.86602) node {$2$};
\draw (25.000, 20.785) node {$4$}; \draw (16.000, 6.9282) node {$5$};
\draw (17.500, 21.651) node {$1$}; \draw (20.500, 6.0622) node {$8$};
\draw (32.500, 4.3301) node {$6$}; \draw (22.000, 17.320) node {$5$};
\draw (32.500, 14.722) node {$5$}; \draw (11.500, 14.722) node {$2$};
\draw (14.500, 0.86602) node {$2$}; \draw (31.000, 22.517) node {$1$};
\draw (29.500, 14.722) node {$6$}; \draw (31.000, 15.588) node {$5$};
\draw (14.500, 9.5263) node {$4$}; \draw (38.500, 14.722) node {$2$};
\draw (17.500, 14.722) node {$5$}; \draw (25.000, 0.00000) node {$1$};
\draw (13.000, 5.1962) node {$3$}; \draw (16.000, 13.856) node {$5$};
\draw (34.000, 5.1962) node {$5$}; \draw (10.000, 1.7320) node {$1$};
\draw (32.500, 19.919) node {$2$}; \draw (40.000, 17.320) node {$1$};
\draw (11.500, 12.990) node {$2$}; \draw (34.000, 12.124) node {$5$};
\draw (22.000, 8.6602) node {$9$}; \draw (22.000, 1.7320) node {$3$};
\draw (23.500, 0.86602) node {$2$}; \draw (37.000, 1.7320) node {$3$};
\draw (11.500, 18.187) node {$1$}; \draw (38.500, 9.5263) node {$2$};
\draw (37.000, 17.320) node {$2$}; \draw (26.500, 6.0622) node {$8$};
\draw (20.500, 0.86602) node {$2$}; \draw (29.500, 2.5981) node {$4$};
\draw (19.000, 20.785) node {$2$}; \draw (13.000, 19.053) node {$1$};
\draw (13.000, 0.00000) node {$1$}; \draw (16.000, 0.00000) node {$1$};
\draw (25.000, 3.4641) node {$5$}; \draw (34.000, 3.4641) node {$5$};
\draw (25.000, 17.320) node {$6$}; \draw (35.500, 7.7942) node {$4$};
\draw (25.000, 5.1962) node {$7$}; \draw (10.000, 8.6602) node {$1$};
\draw (26.500, 12.990) node {$8$}; \draw (17.500, 18.187) node {$3$};
\draw (19.000, 0.00000) node {$1$}; \draw (38.500, 2.5981) node {$2$};
\draw (16.000, 15.588) node {$4$}; \draw (37.000, 0.00000) node {$1$};
\draw (28.000, 20.785) node {$3$}; \draw (20.500, 16.454) node {$5$};
\draw (20.500, 4.3301) node {$6$}; \draw (26.500, 16.454) node {$6$};
\draw (38.500, 18.187) node {$1$}; \draw (23.500, 21.651) node {$3$};
\draw (23.500, 19.919) node {$4$}; \draw (29.500, 0.86602) node {$2$};
\draw (38.500, 16.454) node {$2$}; \draw (25.000, 25.981) node {$1$};
\draw (11.500, 2.5981) node {$2$}; \draw (40.000, 0.00000) node {$1$};
\draw (26.500, 14.722) node {$7$}; \draw (10.000, 15.588) node {$1$};
\draw (17.500, 12.990) node {$6$}; \draw (35.500, 6.0622) node {$4$};
\draw (14.500, 6.0622) node {$4$}; \draw (32.500, 18.187) node {$3$};
\draw (40.000, 8.6602) node {$1$}; \draw (37.000, 13.856) node {$3$};
\draw (13.000, 17.320) node {$2$}; \draw (22.000, 12.124) node {$8$};
\draw (16.000, 17.320) node {$3$}; \draw (17.500, 6.0622) node {$6$};
\draw (22.000, 3.4641) node {$5$}; \draw (11.500, 7.7942) node {$2$};
\draw (28.000, 15.588) node {$6$}; \draw (20.500, 2.5981) node {$4$};
\draw (19.000, 10.392) node {$7$}; \draw (31.000, 5.1962) node {$7$};
\draw (11.500, 6.0622) node {$2$}; \draw (23.500, 14.722) node {$7$};
\draw (22.000, 20.785) node {$3$}; \draw (13.000, 6.9282) node {$3$};
\draw (38.500, 4.3301) node {$2$}; \draw (23.500, 12.990) node {$8$};
\draw (16.000, 3.4641) node {$5$}; \draw (37.000, 6.9282) node {$3$};
\draw (14.500, 14.722) node {$4$}; \draw (31.000, 8.6602) node {$7$};
\draw (14.500, 19.919) node {$1$}; \draw (22.000, 6.9282) node {$9$};
\draw (31.000, 12.124) node {$7$}; \draw (23.500, 9.5263) node {$10$};
\draw (25.000, 6.9282) node {$9$}; \draw (28.000, 1.7320) node {$3$};
\draw (35.500, 4.3301) node {$4$}; \draw (14.500, 11.258) node {$4$};
\draw (32.500, 6.0622) node {$6$}; \draw (35.500, 11.258) node {$4$};
\draw (34.000, 1.7320) node {$3$}; \draw (25.000, 22.517) node {$3$};
\draw (29.500, 16.454) node {$5$}; \draw (31.000, 19.053) node {$3$};
\draw (40.000, 13.856) node {$1$}; \draw (20.500, 21.651) node {$2$};
\draw (40.000, 1.7320) node {$1$}; \draw (25.000, 10.392) node {$10$};
\draw (14.500, 12.990) node {$4$}; \draw (26.500, 7.7942) node {$10$};
\draw (10.000, 12.124) node {$1$}; \draw (25.000, 19.053) node {$5$};
\draw (29.500, 7.7942) node {$8$}; \draw (28.000, 3.4641) node {$5$};
\draw (19.000, 6.9282) node {$7$}; \draw (34.000, 0.00000) node {$1$};
\draw (29.500, 6.0622) node {$8$}; \draw (16.000, 5.1962) node {$5$};
\draw (11.500, 4.3301) node {$2$}; \draw (20.500, 9.5263) node {$8$};
\draw (13.000, 8.6602) node {$3$}; \draw (40.000, 3.4641) node {$1$};
\draw (28.000, 8.6602) node {$9$}; \draw (13.000, 1.7320) node {$3$};
\draw (26.500, 2.5981) node {$4$}; \draw (25.000, 13.856) node {$8$};
\draw (13.000, 10.392) node {$3$}; \draw (26.500, 4.3301) node {$6$};
\draw (37.000, 8.6602) node {$3$}; \draw (29.500, 19.919) node {$3$};
\draw (23.500, 18.187) node {$5$}; \draw (14.500, 2.5981) node {$4$};
\draw (34.000, 13.856) node {$5$}; \draw (38.500, 7.7942) node {$2$};
\draw (35.500, 0.86602) node {$2$}; \draw (37.000, 19.053) node {$1$};
\draw (22.000, 10.392) node {$9$}; \draw (26.500, 21.651) node {$3$};
\draw (16.000, 8.6602) node {$5$}; \draw (38.500, 12.990) node {$2$};
\draw (14.500, 7.7942) node {$4$}; \draw (40.000, 6.9282) node {$1$};
\draw (16.000, 10.392) node {$5$}; \draw (23.500, 6.0622) node {$8$};
\draw (34.000, 17.320) node {$3$}; \draw (28.000, 10.392) node {$9$};
\draw (32.500, 16.454) node {$4$}; \draw (40.000, 10.392) node {$1$};
\draw (37.000, 3.4641) node {$3$}; \draw (28.000, 19.053) node {$4$};
\draw (34.000, 15.588) node {$4$}; \draw (31.000, 20.785) node {$2$};
\draw (25.000, 15.588) node {$7$}; \draw (34.000, 8.6602) node {$5$};
\draw (20.500, 14.722) node {$6$}; \draw (23.500, 4.3301) node {$6$};
\draw (28.000, 22.517) node {$2$}; \draw (19.000, 8.6602) node {$7$};
\draw (25.000, 12.124) node {$9$}; \draw (26.500, 18.187) node {$5$};
\draw (32.500, 21.651) node {$1$}; \draw (17.500, 19.919) node {$2$};
\draw (26.500, 23.383) node {$2$}; \draw (28.000, 5.1962) node {$7$};
\draw (37.000, 10.392) node {$3$}; \draw (26.500, 19.919) node {$4$};
\draw (35.500, 18.187) node {$2$}; \draw (11.500, 9.5263) node {$2$};
\draw (11.500, 11.258) node {$2$}; \draw (40.000, 12.124) node {$1$};
\draw (16.000, 20.785) node {$1$}; \draw (17.500, 7.7942) node {$6$};
\draw (16.000, 12.124) node {$5$}; \draw (22.000, 5.1962) node {$7$};
\draw (29.500, 18.187) node {$4$}; \draw (22.000, 13.856) node {$7$};
\draw (29.500, 4.3301) node {$6$}; \draw (31.000, 1.7320) node {$3$};
\draw (32.500, 2.5981) node {$4$}; \draw (29.500, 11.258) node {$8$};
\draw (10.000, 0.00000) node {$1$}; \draw (20.500, 7.7942) node {$8$};
\draw (10.000, 5.1962) node {$1$}; \draw (29.500, 21.651) node {$2$};
\draw (19.000, 1.7320) node {$3$}; \draw (17.500, 9.5263) node {$6$};
\draw (31.000, 0.00000) node {$1$}; \draw (34.000, 20.785) node {$1$};
\draw (11.500, 16.454) node {$2$}; \draw (22.000, 19.053) node {$4$};
\draw (32.500, 7.7942) node {$6$}; \draw (38.500, 0.86602) node {$2$};
\draw (23.500, 25.115) node {$1$}; \draw (10.000, 13.856) node {$1$};
\draw (35.500, 9.5263) node {$4$}; \draw (23.500, 11.258) node {$9$};
\draw (22.000, 0.00000) node {$1$}; \draw (40.000, 5.1962) node {$1$};
\draw (32.500, 9.5263) node {$6$}; \draw (31.000, 6.9282) node {$7$};
\draw (28.000, 6.9282) node {$9$}; \draw (29.500, 9.5263) node {$8$};
\draw (14.500, 18.187) node {$2$}; \draw (35.500, 2.5981) node {$4$};
\draw (19.000, 19.053) node {$3$}; \draw (20.500, 23.383) node {$1$};
\draw (23.500, 7.7942) node {$10$}; \draw (19.000, 22.517) node {$1$};
\draw (11.500, 0.86602) node {$2$}; \draw (25.000, 24.249) node {$2$};
\draw (13.000, 12.124) node {$3$}; \draw (34.000, 6.9282) node {$5$};
\draw (38.500, 11.258) node {$2$}; \draw (19.000, 12.124) node {$7$};
\draw (17.500, 11.258) node {$6$}; \draw (28.000, 24.249) node {$1$};
\draw (35.500, 14.722) node {$4$}; \draw (20.500, 12.990) node {$7$};
\draw (16.000, 19.053) node {$2$}; \draw (22.000, 24.249) node {$1$};
\draw (28.000, 0.00000) node {$1$}; \draw (20.500, 18.187) node {$4$};
\draw (38.500, 6.0622) node {$2$}; \draw (23.500, 2.5981) node {$4$};
\draw (31.000, 3.4641) node {$5$}; \draw (14.500, 16.454) node {$3$};
\draw (10.000, 17.320) node {$1$}; \draw (28.000, 12.124) node {$8$};
\draw (20.500, 11.258) node {$8$}; \draw (37.000, 15.588) node {$3$};
\draw (26.500, 0.86602) node {$2$}; \draw (26.500, 9.5263) node {$10$};
\draw (19.000, 3.4641) node {$5$}; \draw (37.000, 12.124) node {$3$};
\draw (29.500, 23.383) node {$1$}; \draw (22.000, 22.517) node {$2$};
\draw (19.000, 13.856) node {$6$}; \draw (35.500, 12.990) node {$4$};
\draw (17.500, 2.5981) node {$4$}; \draw (35.500, 19.919) node {$1$};
\draw (31.000, 13.856) node {$6$}; \draw (23.500, 23.383) node {$2$};
\draw (23.500, 16.454) node {$6$}; \draw (10.000, 10.392) node {$1$};
\draw (26.500, 11.258) node {$9$}; \draw (10.000, 3.4641) node {$1$};
\draw (13.000, 13.856) node {$3$}; \draw (34.000, 19.053) node {$2$};
\draw (32.500, 11.258) node {$6$}; \draw (25.000, 1.7320) node {$3$};
\draw (37.000, 5.1962) node {$3$}; \draw (19.000, 17.320) node {$4$};
\draw (22.000, 15.588) node {$6$}; \draw (16.000, 1.7320) node {$3$};
\end{tikzpicture}
\caption{The multiplicities $c^{\lambda,\mu}_{\nu}$ with $\lambda=10(\omega_1+\omega_2)$ and $\mu=10(2\omega_1+\omega_2)$.\label{fig:multiplicitiestensorproduct}}
\end{center}
\end{figure}

On Figure \ref{fig:multiplicitiestensorproduct}, one sees that the Littlewood--Richardson coefficients for a tensor product of two large irreducible representations are almost given by a piecewise polynomial function (and even piecewise affine for this example). This is the analogue of Proposition \ref{prop:asymptoticsmultiplicities} for Littlewood--Richardson coefficients.

\begin{theorem}\label{thm:asymptoticstensor}
Fix two directions $x$ and $y$ in the Weyl chamber $C$. We assume that $x$ and $y$ do not belong to the walls of $C$. There exists a finite positive measure $p_{x,y}(z)\DD{z}$ on $C$:
\begin{itemize}
    \item compactly supported by a polytope $\mathscr{P}(x,y)$ whose boundary is determined by affine functions of $x,y,z$;
    \item with a mass smaller than $1$, and a density given by a piecewise polynomial function in $z$ of local degree bounded by $l-d = |\Phi_+| - \rank(G)$;
    \item such that, for any function $f$ that is continuous and bounded on $C$,
    $$\lim_{\substack{t \to \infty\\tx,ty \in \hatG}} \left(\frac{\prod_{\alpha \in \Phi_+} \scal{\rho}{\alpha}}{t^{l}\,\prod_{\alpha \in \Phi_+} \scal{y}{\alpha}} \sum_{\nu \in \hatG} c^{tx,ty}_\nu\,f\!\left(\frac{\nu}{t}\right) \right)= \int_{\mathscr{P}(x,y)} f(z)\,p_{x,y}(z)\DD{z}.$$
 \end{itemize} 
Moreover, one has the scaling property $p_{\gamma x,\gamma y}(\gamma z) = \frac{p_{x,y}(z)}{\gamma^d}.$
\end{theorem}

\begin{proof}
We set $\lambda = tx$ and $\mu = ty$, and we introduce the discrete measure
$$\rho_{\lambda,\mu} = \frac{1}{\dim_{\C}(V^\mu)} \sum_{(n_1,\ldots,n_l)\text{ integer points in }\mathscr{P}(\lambda,\mu)} \delta_{(n_1,\ldots,n_l)};$$
it has mass smaller than $1$. 
We have
\begin{align*}
\sum_{\nu \in \hatG} c^{\lambda,\mu}_{\nu}\,f\!\left(\frac{\nu}{t}\right) &= (\dim_\C(V^\mu)) \int_{\R^l} f\!\left(\frac{\psi_{\lambda+\mu}(u)}{t}\right)\,\rho_{\lambda,\mu}(\!\DD{u}) \\
&\simeq  \frac{t^l\,\prod_{\alpha \in \Phi_+} \scal{y}{\alpha}}{\prod_{\alpha \in \Phi_+} \scal{\rho}{\alpha}} \int_{\R^l} f\!\left(\psi_{x+y}\!\left(\frac{u}{t}\right)\right)\,\rho_{tx,ty}(\!\DD{u}).
\end{align*}
As $t$ goes to infinity, the discrete measure $\rho_{tx,ty}(tu)$ converges in law to the measure 
$$\varrho_{x,y}(\!\DD{u})=\frac{1_{u \in \mathscr{P}(x,y)}\DD{u}}{\mathrm{vol}(\mathscr{P}(y))}$$ 
on the relative polytope $\mathscr{P}(x,y)$, which is the subset of $\mathscr{P}(y)$ that consists in parameters $u$ such that
$$\sum_{j=1}^l u_j\,\alpha_{i_j}\!\left(\frac{\phi_{j-1}^\vee + \phi_j^\vee}{2}\right)\geq -x(\alpha_i^\vee)$$
for any trail $(\phi_0^\vee,\phi_1^\vee,\ldots,\phi_l^\vee)$ as in Theorem \ref{thm:berensteinzelevinsky}. Therefore,
$$\lim_{\substack{t \to \infty\\tx,ty \in \hatG}} \left(\frac{\prod_{\alpha \in \Phi_+} \scal{\rho}{\alpha}}{t^{l}\,\prod_{\alpha \in \Phi_+} \scal{y}{\alpha}} \sum_{\nu \in \hatG} c^{tx,ty}_\nu\,f\!\left(\frac{\nu}{t}\right) \right)= \int_{\R^l} f(\psi_{x+y}(u))\,\rho_{x,y}(\!\DD{u}).$$
The piecewise polynomial measure $p_{x,y}$ of the statement of the theorem is the image measure $(\psi_{x+y})_*\varrho_{x,y}$, and it is indeed piecewise polynomial since $\varrho_{x,y}$ is proportional to the Lebesgue measure on a polytope, and $\psi_{x+y}$ is a affine map. The scaling property is proven in the same way as in Proposition \ref{prop:asymptoticsmultiplicities}.
\end{proof}

\begin{remark}
The hypothesis that $x$ and $y$ do not belong to the walls of $C $ are used in order to ensure that the relative polytope $\mathscr{P}(x,y)$ has maximal dimension $l$, and that the dimension $\dim_{\C} (V^{ty})$ is of order $O(t^l)$. Otherwise, one might need to consider a different renormalisation of the Littlewood--Richardson coefficients, but the asymptotic polynomiality stays true.
\end{remark}

In the sequel, it will be convenient to have a more symmetric version of Theorem \ref{thm:asymptoticstensor}. We set $q_{x,y}(z) = \delta(y)\,p_{x,y}(z)$, where as before $\delta(y) = \prod_{\alpha \in \Phi_+} \frac{\scal{y}{\alpha}}{\scal{\rho}{\alpha}}$. The function $q_{x,y}(z)$ is:
\begin{itemize}
    \item a compactly supported piecewise polynomial function in $z$, of total integral smaller than $\min(\delta(x),\delta(y))$,
    \item symmetric in $x$ and $y$, 
    \item such that for any bounded continuous function $f$ on $C$,
$$\lim_{\substack{t \to \infty \\ tx,ty \in \hatG}} \left(\frac{1}{t^l}\,\sum_{\nu \in \hatG} c^{tx,ty}_\nu\,f\!\left(\frac{\nu}{t}\right) \right)= \int_{C} f(z)\, q_{x,y}(z)\DD{z}.$$
 \end{itemize}  
The symmetry in $x$ and $y$ comes from the symmetry of the Littlewood--Richardson coefficients $c^{\lambda,\mu}_{\nu}$ in $\lambda$ and $\mu$. On the other hand, since $y \mapsto \delta(y)$ is an homogeneous polynomial in  the coordinates of $y$ with degree $l$, the function $q$ satisfies the scaling property $q_{\gamma x,\gamma y}(\gamma z)=\gamma^{l-d}\,q_{x,y}(z)$. Actually, a bit more is true:

\begin{proposition}\label{prop:homogeneouspolynomial}
Let $C'$ denote the interior of the Weyl chamber. The function of three variables $(x,y,z) \in (C')^3 \mapsto q_{x,y}(z) \in \R_+$ is: \begin{itemize}
    \item piecewise polynomial and locally homogeneous of total degree $l-d$ in $(x,y,z)$, 

    \item  with domains of polynomiality that are polyhedral cones in $(C')^3$ (subsets that are stable by $(x,y,z) \mapsto (\gamma x,\gamma y,\gamma z)$ and that are bounded by a finite number of affine hyperplanes).
\end{itemize}
\end{proposition}

\begin{proof}
From Theorem \ref{thm:berensteinzelevinsky}, we know that the equations that determine $\mathscr{P}(x,y)$ are affine maps of $x$ and $y$. Therefore, the (local) coefficients of the polynomial function $z \mapsto q_{x,y}(z)$ are polynomials in $x,y$. In other words, for any sufficiently small open subset $U \subset (C')^3$, the restriction of the map $(x,y,z) \mapsto q_{x,y}(z)$ to $U$ is given by a polynomial in (the coordinates of) $x,y,z$. The scaling property of this map forces then the polynomials to be homogeneous of degree $l-d$. Finally, the form of the domains of polynomiality comes from the following fact. If one projects by an affine map $\pi$ a compact polytope $P$ in dimension $l$ to a space of dimension $d \leq l$, then the non-empty intersections of images of the faces of the polytope $P$ partition the image $\pi(P)$ into a finite number of polytopes. On each of these non-empty intersections, the image of the uniform measure on $P$ is polynomial, hence the second part of the proposition. 
\end{proof}

Proposition \ref{prop:multilittlewood} is the immediate generalisation of Proposition \ref{prop:homogeneouspolynomial}, and it is proved by applying it recursively and by using the convolution rule for multiple Littlewood--Richardson coefficients, which turns into a convolution rule for the functions $q_{x_1,\ldots,x_{r-1}}(z)$.
\medskip

\begin{example}
Consider the trivial example where $G=\SU(2)$. In this case, a tensor product $V^{k \omega} \otimes V^{l \omega}$ is given by the Clebsch--Gordan rules:
$$V^{k\omega} \otimes V^{l\omega} = V^{(k+l)\omega} \oplus V^{(k+l-2)\omega} \oplus V^{(k+l-4)\omega} \oplus \cdots \oplus V^{|k-l|\omega} = \bigoplus_{\substack{m \equiv k+l \bmod 2 \\ |k-l|\leq m \leq k+l}} V^{m\omega}.$$
The limit when $k=tx$, $l=ty$ and $t \to +\infty$ of this rule is obviously given by the locally constant function
$$q_{x\omega,y\omega}(z\omega) \DD{(z\omega)}= \frac{1}{2}\,1_{|x-y|\leq z \leq x+y}\DD{z}.$$
This agrees with the previous discussion, since $l=d=1$ and thus $l-d=0$. The domain where $q_{x,y}(z) \neq 0$ is drawn in Figure \ref{fig:polyhedron} hereafter, and it is indeed a polyhedral cone.
\begin{figure}[ht]
\begin{center}      
\includegraphics[scale=0.7]{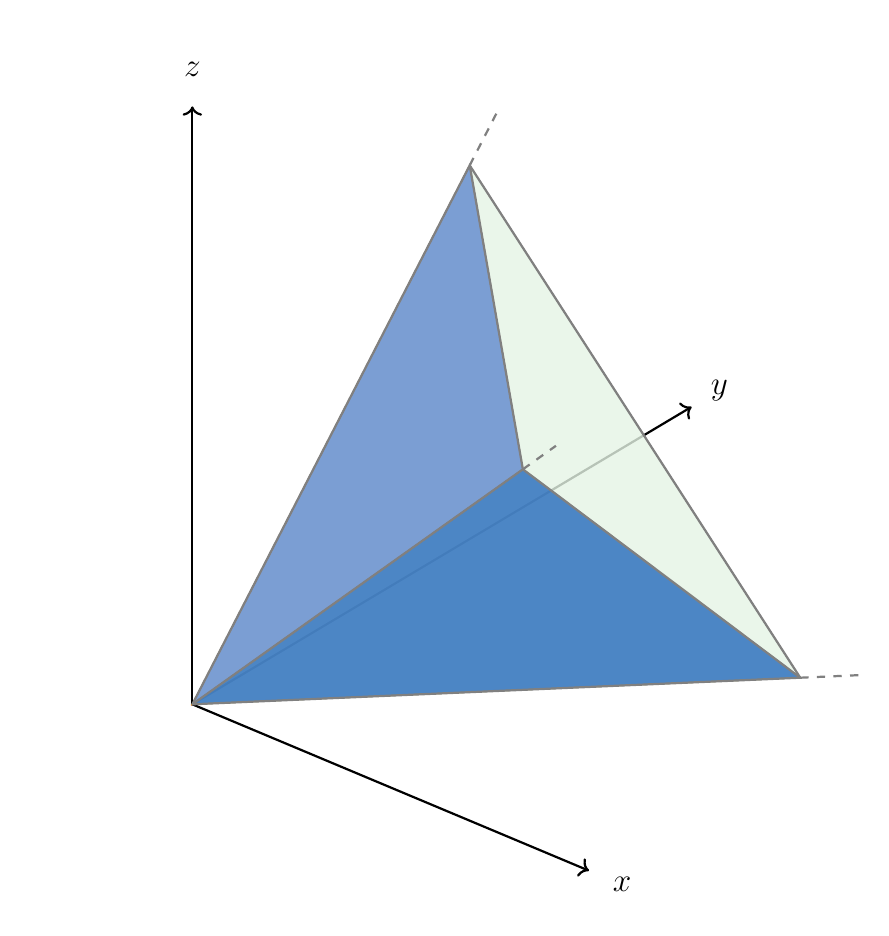}
\caption{The domains of polynomiality of the functions $(x,y,z) \mapsto q_{x,y}(z)$ are polyhedral cones in $(C')^3$.\label{fig:polyhedron}}
\end{center}
\end{figure}
\end{example}


\begin{example}
Let us detail the case $G=\SU(3)$. If $y=y_1\omega_1+y_2\omega_2$, then the polytope $\mathscr{P}(y)$ is the set of real triplets $(u_1,u_2,u_3)$ with
$$ 0\leq u_1 \leq y_1-2u_3+u_2\qquad;\qquad 0 \leq u_3 \leq y_1 \qquad;\qquad u_3 \leq u_2 \leq u_3+y_2.$$
If $x=x_1\omega_1+x_2\omega_2$, then the relative polytope $\mathscr{P}(x,y)$ is the subset of $\mathscr{P}(y)$ that consists in triplets such that
$$ u_1 \leq x_1\qquad;\qquad u_2 \leq u_1+x_2 \qquad;\qquad u_3 \leq x_2.$$
The measure $q_{x,y}(z)\DD{z}$ is the image of the non-normalised Lebesgue measure $1_{u \in \mathscr{P}(x,y)}\DD{u}$ by the affine map $u \mapsto \psi_{x+y}(u) = x+y - (u_1+u_3)\alpha_1 - u_2\alpha_2$. Set $v=u_1+u_3$, $w=u_2$. We have:
\begin{align*}
\int_{C} f(z)\,q_{x,y}(z)\DD{z} &= \int_{(\R_+)^3} 1_{u \in \mathscr{P}(x,y)}\,f(\psi_{x+y}(u)) \DD{u_1}\DD{u_2}\DD{u_3} \\
&= \int_{(\R_+)^3} 1_{m\leq u_1\leq M}\,f(x+y-v\alpha_1-w\alpha_2) \DD{u_1}\DD{v} \DD{w}
\end{align*}
where 
\begin{align*}
m&= \max(0,v-y_1,v-w,v-x_2,w-x_2,2v-w-y_1);\\
M&= \min(x_1,v,v-w+y_2).
\end{align*}
Integrating the variable $u_1$ and making the change of variables $z = z_1\omega_1+z_2\omega_2 = \psi_{x+y}(u)$ yields:
$$\int_{C} f(z)\,q_{x,y}(z)\DD{z} = \frac{1}{3} \int_{(\R_+)^2} f(z) \,(M(x,y,z)-m(x,y,z))_+\DD{z_1}\DD{z_2}$$
with
\begin{align*}
m(x,y,z)&= \max\left(0,x_1-z_1,\frac{2(x_1-z_1)+x_2+y_2-z_2-y_1}{3},\frac{x_1+y_1+z_2-x_2-y_2-z_1}{3},\right.\\
&\left.\qquad\qquad \frac{2(x_1+y_1-z_1-x_2)+y_2-z_2}{3},\frac{x_1+y_1-z_1-x_2+2(y_2-z_2)}{3}\right);\\
M(x,y,z) &= \min\left(x_1,\frac{2(x_1+y_1-z_1)+x_2+y_2-z_2}{3},\frac{2y_2+x_1+y_1+z_2-x_2-z_1}{3}\right).
\end{align*}
On each subset of $(C')^3$ where $m(x,y,z)$ and $M(x,y,z)$ are affine maps, the non-negative part $(M(x,y,z)-m(x,y,z))_+$ is either $0$, or an homogeneous polynomial function in $x,y,z$ of degree $l-d=3-2=1$. This agrees with the general statement of Proposition \ref{prop:homogeneouspolynomial}. One thing that is absolutely not obvious with this expression of $q_{x,y}(z)$ is the symmetry in $x$ and $y$. As far as we know, there is no way to write a symmetric relative string polytope $\mathscr{P}(\lambda,\mu)$, and one gets back the symmetry only after projection of this polytope to the space of weights.
\end{example}
\bigskip

\bibliographystyle{alpha}
\bibliography{randomgeom.bib}

\begin{thebibliography}{EKYY13}

\bibitem[AB04]{AB04}
V.~Alexeev and M.~Brion.
\newblock Toric degenerations of spherical varieties.
\newblock {\em Selecta Math. (N.S.)}, 10(4):453--478, 2004.

\bibitem[Ada96]{Adam96}
J.~F. Adams.
\newblock {\em Lectures on exceptional {L}ie groups}.
\newblock Chicago Lectures in Mathematics. Chicago University Press, 1996.

\bibitem[AH10]{AH10}
K.~Atkinson and W.~Han.
\newblock {\em Spherical Harmonics and Approximations on the Unit Sphere: An
  Introduction}, volume 2044 of {\em Lecture Notes in Mathematics}.
\newblock Springer-Verlag, 2010.

\bibitem[Ald91]{Ald91}
D.~J. Aldous.
\newblock Asymptotic fringe distributions for general families of random trees.
\newblock 1(2):228--266, 1991.

\bibitem[AS04]{AS04}
D.~J. Aldous and J.~M. Steele.
\newblock The objective method: probabilistic combinatorial optimization and
  local weak convergence.
\newblock In H.~Kesten, editor, {\em Probability on Discrete Structures},
  volume 110 of {\em Encyclopaedia Math. Sci.}, pages 1--72. Springer, 2004.

\bibitem[ATV11]{ATV11}
M.~Ab\'ert, A.~Thom, and B.~Vir\'ag.
\newblock Benjamini--schramm convergence and pointwise convergence of the
  spectral measure, 2011.
\newblock Unpublished.

\bibitem[AY97]{AY97}
K.~Abe and I.~Yokota.
\newblock Volumes of compact symmetric spaces.
\newblock 20(1):87--105, 1997.

\bibitem[Bae02]{Baez02}
J.~C. Baez.
\newblock The octonions.
\newblock {\em Bull. Amer. Math. Soc.}, 39:145--205, 2002.

\bibitem[BBO05]{BBO05}
P.~Biane, P.~Bougerol, and N.~O'Connell.
\newblock Littelmann paths and {B}rownian paths.
\newblock {\em Duke Math. J.}, 130(1):127--167, 2005.

\bibitem[BBO09]{BBO09}
P.~Biane, P.~Bougerol, and N.~O'Connell.
\newblock Continuous crystal and {D}uis\-ter\-maat–-{H}eck\-man measure for
  {C}oxeter groups.
\newblock {\em Adv. Math.}, 221(5):1522--1583, 2009.

\bibitem[BEJJ06]{BEJJ06}
P.~Blackwell, M.~Edmondson-Jones, and J.~Jordan.
\newblock Spectra of adjacency matrices of random geometric graphs.
\newblock Unpublished, 2006.

\bibitem[Bil95]{Bil95}
P.~Billingsley.
\newblock {\em Probability and Measure}.
\newblock Wiley Series in Probabilistic and Mathematical Statistics. John Wiley
  and Sons, 3rd edition edition, 1995.

\bibitem[Bil99]{Bil99}
P.~Billingsley.
\newblock {\em Convergence of Probability Measures}.
\newblock John Wiley and Sons, 2nd edition edition, 1999.

\bibitem[BL10]{BL10}
C.~Bordenave and M.~Lelarge.
\newblock Resolvent of large random graphs.
\newblock 37(3):332--352, 2010.

\bibitem[BLS11]{BLS11}
C.~Bordenave, M.~Lelarge, and J.~Salez.
\newblock The rank of diluted random graphs.
\newblock 39(3):1097--1121, 2011.

\bibitem[Bor08]{Bor08}
C.~Bordenave.
\newblock Eigenvalues of euclidean random matrices.
\newblock {\em Random Structures \& Algorithms}, 33(4):515--532, 2008.

\bibitem[Bor16]{Bor16}
C.~Bordenave.
\newblock Spectrum of random graphs, 2016.

\bibitem[Bou81]{Bour81}
N.~Bourbaki.
\newblock {\em Espaces vectoriels topologiques 1-5}.
\newblock Masson, 1981.

\bibitem[BS01]{BS01}
I.~Benjamini and O.~Schramm.
\newblock Recurrence of distributional limits of finite planar graphs.
\newblock {\em Electronic Journal of Probability}, 6(23):1--13, 2001.

\bibitem[Bum13]{Bump13}
D.~Bump.
\newblock {\em Lie Groups}, volume 225 of {\em Graduate Texts in Mathematics}.
\newblock Springer-Verlag, 2nd edition edition, 2013.

\bibitem[BZ88]{BZ88}
A.~Berenstein and A.~Zelevinsky.
\newblock Tensor product multiplicities and convex polytopes in partition
  space.
\newblock {\em J. Geom. Phys.}, 5(3):453--472, 1988.

\bibitem[BZ01]{BZ01}
A.~Berenstein and A.~Zelevinsky.
\newblock Tensor product multiplicities, canonical bases and totally positive
  varieties.
\newblock {\em Invent. Math.}, 143(1):77--128, 2001.

\bibitem[Chu97]{Chung97}
F.~R.~K. Chung.
\newblock {\em Spectral Graph Theory}, volume~92 of {\em CBMS Regional
  Conference Series in Mathematics}.
\newblock Conf. Board Math. Sci., Washington, DC, 1997.

\bibitem[Coh07]{Coh07}
H.~Cohen.
\newblock {\em Number Theory. Volume II: Analytic and Modern Tools}, volume 240
  of {\em Graduate Texts in Mathematics}.
\newblock Springer-Verlag, 2007.

\bibitem[Col58]{Col58}
A.~J. Coleman.
\newblock The betti numbers of the simple lie groups.
\newblock 10:349--356, 1958.

\bibitem[CSST08]{CSST08}
T.~Ceccherini-Silberstein, F.~Scarabotti, and F.~Tolli.
\newblock {\em Harmonic Analysis on Finite Groups}, volume 108 of {\em
  Cambridge Studies in Advanced Mathematics}.
\newblock Cambridge University Press, 2008.

\bibitem[DGK16]{DGK16}
C.~P. Dettmann, O.~Georgiou, and G.~Knight.
\newblock Spectral statistics of random geometric graphs.
\newblock \texttt{arXiv:1608.01154v1}, 2016.

\bibitem[DVJ03]{DVJ03}
D.~J. Daley and D.~Vere-Jones.
\newblock {\em An Introduction to the Theory of Point Processes}, volume I:
  Elementary Theory and Methods.
\newblock Springer-Verlag, second edition edition, 2003.

\bibitem[EKYY12]{EKYY12}
L.~Erdös, A.~Knowles, H.-T. Yau, and J.~Yin.
\newblock Spectral statistics of {E}rdös--{R}ényi graphs {II}: eigenvalues
  spacings and the extreme eigenvalues.
\newblock {\em Commun. Math. Phys.}, 314:587--640, 2012.

\bibitem[EKYY13]{EKYY13}
L.~Erdös, A.~Knowles, H.-T. Yau, and J.~Yin.
\newblock Spectral statistics of {E}rdös--{R}ényi graphs {I}: local
  semicircle law.
\newblock {\em Ann. Probab.}, 41(3B):2279--2375, 2013.

\bibitem[ER59]{ER59}
P.~Erdös and A.~Rényi.
\newblock On random graphs. {I}.
\newblock {\em Publ. Math. Debrecen}, 6:290--297, 1959.

\bibitem[FH91]{FH91}
W.~Fulton and J.~Harris.
\newblock {\em Representation Theory. A First Course}, volume 129 of {\em
  Graduate Texts in Mathematics}.
\newblock Springer--Verlag, 1991.

\bibitem[GK00]{GK00}
E.~Gin\'e and V.~Koltchinskii.
\newblock Random matrix approximation of spectra of integral operators.
\newblock {\em Bernoulli}, 6(1):113--167, 2000.

\bibitem[GR01]{GR01}
C.~Godsil and G.~Royle.
\newblock {\em Algebraic Graph Theory}, volume 207 of {\em Graduate Texts in
  Mathematics}.
\newblock Springer-Verlag, 2001.

\bibitem[Gri83]{Grin83}
E.~L. Grinberg.
\newblock Spherical harmonics and integral geometry on projective spaces.
\newblock {\em Trans. Amer. Math. Soc.}, 279(1):187--203, 1983.

\bibitem[Gro07]{Gro07}
M.~Gromov.
\newblock {\em Metric structures for Riemannian and non-Riemannian spaces}.
\newblock Modern Birk\-häuser Classics. Birkhäuser, 2007.

\bibitem[GW09]{GW09}
R.~Goodman and N.~R. Wallach.
\newblock {\em Symmetry, Representations, and Invariants}, volume 255 of {\em
  Graduate Texts in Mathematics}.
\newblock Springer-Verlag, 2009.

\bibitem[Has97]{Hash97}
Y.~Hashimoto.
\newblock On macdonald's formula for the volume of a compact lie group.
\newblock 72:660--662, 1997.

\bibitem[Hel70]{Hel70}
S.~Helgason.
\newblock A duality for symmetric spaces with applications to group
  representations.
\newblock {\em Adv. Math.}, 5:1--154, 1970.

\bibitem[Hel78]{Hel78}
S.~Helgason.
\newblock {\em Differential Geometry, Lie Groups, and Symmetric Spaces}.
\newblock Academic Press, 1978.

\bibitem[Hel84]{Hel84}
S.~Helgason.
\newblock {\em Groups and Geometric Analysis. Integral Geometry, Invariant
  Differential Operators, and Spherical Functions}.
\newblock Academic Press, 1984.

\bibitem[Jim85]{Jim85}
M.~Jimbo.
\newblock A $q$-difference analogue of ${U}(\mathfrak{g})$ and the
  {Y}ang--{B}axter equation.
\newblock {\em Letters in Math. Phys.}, 10:63--69, 1985.

\bibitem[Jim86]{Jim86}
M.~Jimbo.
\newblock A $q$-analogue of ${U}(\glie(n+1))$, {H}ecke algebra and the
  {Y}ang--{B}axter equation.
\newblock {\em Letters in Math. Phys.}, 11:247--252, 1986.

\bibitem[Joh76]{John76}
K.~Johnson.
\newblock Composition series and intertwining operators for the spherical
  principal series {II}.
\newblock {\em Trans. Amer. Math. Soc.}, 215:269--283, 1976.

\bibitem[Jos11]{Jost11}
J.~Jost.
\newblock {\em Riemannian Geometry and Geometric Analysis}.
\newblock Universitext. Sprin\-ger-Verlag, 6th edition edition, 2011.

\bibitem[Kal02]{Kal02}
O.~Kallenberg.
\newblock {\em Foundations of Modern Probability}.
\newblock Probability and Its Applications. Springer-Verlag, 2nd edition
  edition, 2002.

\bibitem[Kas90]{Kas90}
M.~Kashiwara.
\newblock Crystalizing the $q$-analogue of universal enveloping algebras.
\newblock {\em Com\-mun. Math. Phys.}, 133(2):249--260, 1990.

\bibitem[KP84]{KP84}
V.~G. Kac and D.~H. Peterson.
\newblock Infinite-dimensional lie algebras, theta functions and modular forms.
\newblock 53:125--264, 1984.

\bibitem[Lan93]{Lang93}
S.~Lang.
\newblock {\em Real and Functional Analysis}, volume 142 of {\em Graduate Texts
  in Mathematics}.
\newblock Sprin\-ger-Verlag, 1993.

\bibitem[Lit95]{Litt95}
P.~Littelmann.
\newblock Paths and root operators in representation theory.
\newblock {\em Ann. Math.}, 142(3):499--525, 1995.

\bibitem[Lit98a]{LittCone}
P.~Littelmann.
\newblock Cones, crystals and patterns.
\newblock {\em Transformation Groups}, 3(2):145--179, 1998.

\bibitem[Lit98b]{Litt98}
P.~Littelmann.
\newblock The path model, the quantum {F}robenius map and standard monomial
  theory.
\newblock In {\em Algebraic Groups and their Representations}, volume 517 of
  {\em NATO ASI Series}, pages 175--212. Kluwer Academic Publishers, 1998.

\bibitem[Lus88]{Lus88}
G.~Lusztig.
\newblock Quantum deformations of certain simple modules over enveloping
  algebras.
\newblock {\em Adv. Math.}, 70:237--249, 1988.

\bibitem[Lus90]{Lus90}
G.~Lusztig.
\newblock Canonical bases arising from quantized enveloping algebras.
\newblock {\em Journal of the American Mathematical Society}, 3(2):447--498,
  1990.

\bibitem[Mac80]{Mac80}
I.~G. Macdonald.
\newblock The volume of a compact lie group.
\newblock 56:93--95, 1980.

\bibitem[M{\'e}l14]{Mel14}
P.-L. M{\'e}liot.
\newblock The cut-off phenomenon for {B}rownian motions on compact symmetric
  spaces.
\newblock {\em Potential Analysis}, 40(4):427--509, 2014.

\bibitem[M{\'e}l17]{Mel17}
P.-L. M{\'e}liot.
\newblock {\em Representation Theory of Symmetric Groups}.
\newblock CRC Press, 2017.

\bibitem[MR96]{MR96}
R.~Meester and R.~Roy.
\newblock {\em Continuum Percolation}, volume 119 of {\em Cambridge Tracts in
  Mathematics}.
\newblock Cambridge University Press, 1996.

\bibitem[Pen03]{Pen03}
M.~Penrose.
\newblock {\em Random geometric graphs}, volume~5 of {\em Oxford Studies in
  Probability}.
\newblock Oxford University Press, 2003.

\bibitem[Pet06]{Pet06}
P.~Petersen.
\newblock {\em Riemannian Geometry}, volume 171 of {\em Graduate Texts in
  Mathematics}.
\newblock Springer-Verlag, 2006.

\bibitem[Ros88]{Ros88}
M.~Rosso.
\newblock Finite dimensional representations of the quantum analog of the
  enveloping algebra of a complex simple lie algebra.
\newblock {\em Commun. Math. Phys.}, 117:581--593, 1988.

\bibitem[Ros90]{Ros90}
M.~Rosso.
\newblock Analogue de la forme de {K}illing et du th{\'e}or{\`e}me
  d'{H}arish--{C}handra pour les groupes quantiques.
\newblock {\em Annales Scientifiques de l'{\'E}cole Normale Sup{\'e}rieure
  4{\`e}me s{\'e}rie}, 23(3):445--467, 1990.

\bibitem[Sch12]{Sch12}
K.~Schmüdgen.
\newblock {\em Unbounded Self-adjoint Operators on Hilbert Space}, volume 265
  of {\em Graduate Texts in Mathematics}.
\newblock Springer-Verlag, 2012.

\bibitem[Sug62]{Sug62}
M.~Sugiura.
\newblock Representations of compact groups realized by spherical functions on
  symmetric spaces.
\newblock {\em Proc. Japan Acad.}, 38(3):111--113, 1962.

\bibitem[VV09]{Vol09}
V.~V. Volchkov and V.~V. Volchkov.
\newblock {\em Harmonic Analysis of Mean Periodic Functions on Symmetric Spaces
  and the Heisenberg Group}.
\newblock Springer Monographs in Mathematics. Springer-Verlag, 2009.

\bibitem[Wol67]{Wolf67}
J.~A. Wolf.
\newblock {\em Spaces of Constant Curvature}.
\newblock McGraw-Hill, 1967.

\end{thebibliography}

\end{document}